\documentclass[11pt,letterpaper]{amsart}
\usepackage[foot]{amsaddr}

\usepackage{mathtools}
\usepackage[T1]{fontenc}
\usepackage[latin9]{inputenc}
\usepackage{geometry}
\usepackage{wrapfig}
\usepackage{multicol}
\usepackage{graphicx}
\usepackage{soul}
\usepackage{xcolor}
\usepackage{amssymb}
\usepackage{placeins}
\usepackage{cite}
\usepackage{caption}
\usepackage{enumerate}
\usepackage{afterpage}
\usepackage{enumitem}
\usepackage{bmpsize}
\usepackage{hyperref}
\usepackage{tabu}
\usepackage{enumitem}
\numberwithin{equation}{section}
\usepackage{stmaryrd}
\usepackage{tikz}
\usetikzlibrary{matrix,graphs,arrows,positioning,calc,decorations.markings,shapes.symbols}

\makeatletter
\renewcommand{\email}[2][]{%
  \ifx\emails\@empty\relax\else{\g@addto@macro\emails{,\space}}\fi%
  \@ifnotempty{#1}{\g@addto@macro\emails{\textrm{(#1)}\space}}%
  \g@addto@macro\emails{#2}%
}
\makeatother

\newtheorem{theorem}{Theorem}[section]
\newtheorem{lemma}[theorem]{Lemma}
\newtheorem{proposition}[theorem]{Proposition}
\newtheorem{conjecture}[theorem]{Conjecture}
\newtheorem{corollary}[theorem]{Corollary}
{ \theoremstyle{definition}
\newtheorem{definition}[theorem]{Definition}}
{ \theoremstyle{remark}
\newtheorem{remark}[theorem]{Remark}}

\def\Re{ \mathsf{Re}}
\def\Im{ \mathsf{Im}}
\def\fB{ \mathsf{B}}
\def\LB{ M}

\newcommand{\R}{\mathbb{R}}

\renewcommand{\P}{\mathbb{P}}

\newcommand{\topc}{K}
\newcommand{\g}{G}
\newcommand{\tsigma}{d}

\title{Spatial tightness at the edge of Gibbsian line ensembles}

\author[G. Barraquand]{Guillaume Barraquand}
\address{G. Barraquand,
Laboratoire de physique de l'{\'e}cole normale sup\'erieure, ENS, Universit{\'e} PSL, CNRS, Sorbonne Universit{\'e}, Universit{\'e} de Paris, Paris, France}
\email{guillaume.barraquand@ens.fr}
\author[I. Corwin]{Ivan Corwin}
\address{I. Corwin, Department of Mathematics, Columbia University, New York, NY 10027, USA} \email{ivan.corwin@gmail.com}
\author[E. Dimitrov]{Evgeni Dimitrov}
\address{E. Dimitrov, Department of Mathematics, Columbia University, New York, NY 10027, USA} \email{esd2138@columbia.edu}

\begin{document}

\maketitle

\begin{abstract}
Consider a sequence of Gibbsian line ensembles, whose lowest labeled curves (i.e., the edge) have tight one-point marginals. Then, given certain technical assumptions on the nature of the Gibbs property and underlying random walk measure, we prove that the entire spatial process of the edge is tight. We then apply this black-box theory to the log-gamma polymer Gibbsian line ensemble, which we construct. The edge of this line ensemble is the transversal free energy process for the polymer, and our theorem implies tightness with the ubiquitous KPZ class $2/3$ exponent, as well as Brownian absolute continuity of all the subsequential limits.

A key technical innovation which fuels our general result is the construction of a continuous grand monotone coupling of Gibbsian line ensembles with respect to their boundary data (entrance and exit values, and bounding curves). {\em Continuous} means that the Gibbs measure varies continuously with respect to varying the boundary data, {\em grand} means that all uncountably many boundary data measures are coupled to the same probability space, and {\em monotone} means that raising the values of the boundary data likewise raises the associated measure. This result applies to a general class of Gibbsian line ensembles where the underlying random walk measure is discrete time, continuous valued and log-convex, and the interaction Hamiltonian is nearest neighbor and convex.
\end{abstract}

\tableofcontents

%
\section{Introduction and main results}

%

Gibbs measures are ubiquitous in statistical mechanics and probability theory. Subject to given boundary conditions, they are measures, which are proportional to the exponential of a sum of local energy contributions. Gibbsian line ensembles are a special class of Gibbs measures, which have received considerable attention in the past two decades owing, in part, to their occurrence in integrable probability.

A Gibbsian line ensemble can be thought of as a collection of labeled random walks, whose joint law is reweighed by a Radon-Nikodym derivative, proportional to the exponential of the sum of local interaction energies between consecutively labeled curves. {\em Local} means that the energies only depend on the values of nearby curves both in terms of the time and label. A simple example of a Gibbsian line ensemble is a collection of random walks conditioned not to touch or cross each other (e.g. level lines of random rhombus or domino tilings). In this case the local energy is infinity or zero depending on whether the touching or crossing occurs or does not. Dyson Brownian motion  with $\beta=2$ is a continuous space and time limit of such ensembles \cite{EK08}. 

Besides providing a compact way to describe a large class of measures, the structure of a Gibbsian line ensemble can be utilized to great benefit when studying their asymptotic scaling limits. Starting with \cite{CorHamA}, there has been a fruitful development of techniques, which leverage the Gibbs property of Gibbsian line ensembles to prove their tightness under various scalings, given only one-point tightness information about their top curve -- see for instance \cite{CorHamK, CD, CGH19, HamA, HamB, HamC, HamD, CHH19, DV18, DNV19, CIW, caputo2019, Wu19, DREU}. In \cite{CorHamA}, this program was initiated through the study of $N$ one-dimensional Brownian bridges, conditioned to start at time $-N$ and end at time $N$ at the origin and not intersect in the time interval $(-N,N)$. These measures are called {\em Brownian watermelons} and they are closely related to Dyson Brownian motion with $\beta=2$. The limiting line ensemble, which arises in that case, is the {\em Airy line ensemble}. Among its many distinctions, this line ensemble forms the foundation of the entire Kardar-Parisi-Zhang (KPZ) fixed point through its role in the construction of the {\em Airy sheet} in \cite{DOV18}.

Since \cite{CorHamA}, a number of other important examples of Gibbsian line ensembles have arisen. One natural context is in describing the level-lines of two-dimensional interfaces conditioned to stay positive \cite{CIW, caputo2019}. Another is in models arising in integrable probability where the Gibbs property is born in the branching structure of the symmetric polynomials, from which integrable models are defined -- see for example \cite{CD} in the case of Hall-Littlewood processes.

Our present study is prompted by our interest in the {\em log-gamma polymer} \cite{Sep12}, which, through a connection to Whittaker processes \cite{COSZ}, can be related to the lowest labeled curve of a Gibbsian line ensemble (see Section \ref{sec:introlog}). The structure of the local energy in this model is considerably more complicated than that of non-touching or crossing random walks.  
\vspace{2mm}

The primary aim of our work is to develop a black-box theory (Theorem \ref{informalintrotightness}), which proves tightness and Brownian absolute continuity of the lowest labeled curve of a Gibbsian line ensemble given tightness of its one-point marginal distribution. We develop this theory for a general class of line ensembles, in which the underlying random walk measure has continuous jumps and scales diffusively to Brownian motion, and in which the interaction energy is such that a key stochastic monotonicity property holds (Lemma \ref{MonCoup}). The first subsection of this introduction, Section \ref{sec:introtight}, contains a statement of our black-box theory (various definitions and terminology are introduced in more detail in the main text). 

The secondary aim of our work is to apply our black-box theory to the log-gamma polymer line ensemble, that we construct in Section \ref{sec:introlog}, and conclude that the polymer free energy has transversal fluctuation exponent $2/3$, as expected by KPZ universality. In Section \ref{sec:introlog} we recall the definition of the log-gamma polymer as well as describe the nature of the Gibbsian line ensemble, into which it embeds. Combining this with the one-point tightness, proved recently in \cite{BCDA}, we apply our black-box theory and arrive at the advertised transversal fluctuation behavior, see Theorem \ref{ThmTight}.

%
\subsection{Tightness and Brownian absolute continuity for general Gibbsian line ensembles}\label{sec:introtight}

Gibbsian line ensembles have been studied mostly in the context of interacting Brownian paths \cite{CorHamK, CGH19, HamA, HamB, HamC, HamD, CHH19, DV18, CIW, caputo2019} or in the fully discrete setting where the path's time and value are indexed by the integers \cite{CD, DNV19, DREU}. In this paper, we will be dealing with discrete time, continuous valued Gibbsian line ensembles -- see Figure \ref{LineEnsembleFig} for an illustration and Section \ref{Section4.1} for a precise definition. Informally, these are measures on collections of curves $\mathfrak{L}= \big(L_i\big)_{i\in \llbracket 1,\topc\rrbracket}$ so that each $L_i$ is the linear interpolation of a function from $\llbracket T_0, T_1 \rrbracket\to \mathbb{R}$ for some $\topc\geq 2$ and some integer interval $\llbracket T_0, T_1 \rrbracket\subset \mathbb{Z}$. Here and throughout the paper, we write $\llbracket a,b \rrbracket = \{a, a+1, \dots, b\}$ for two integers $b \geq a$. The key property, which these measures enjoy, is a resampling invariance, which we refer to as the {\em (partial) $(H,H^{RW})$-Gibbs property}. The function $H^{RW}$ is the {\em random walk Hamiltonian}, and the function $H$ is the {\em interaction Hamiltonian}. We describe this Gibbs property informally here. For any
$k_1\leq k_2$ with $k_1,k_2 \in \llbracket 1,\topc-1 \rrbracket$ and any $a<b$ with $a, b\in\llbracket T_0+1, T_1-1 \rrbracket$, the law of the curves $L_{k_1},\ldots, L_{k_2}$ on the interval $\llbracket a, b \rrbracket$ is a reweighting of random walk bridges according to a specific Radon-Nikodym derivative. The random walk bridges have starting and ending values to match the values of $L_{k_1},\ldots, L_{k_2}$ at $a$ and $b$, respectively, and have jump increments with density proportional to $G(x)=e^{-H^{RW}(x)}$. The Radon-Nikodym derivative, which reweighs this measure is proportional to
\begin{equation}
\prod_{i=k_1-1}^{k_2}\prod_{m=a}^{b-1} e^{-H\left(L_{i+1}(m+1)-L_i(m)\right)}.
\label{RadonNikodymIntroeq}
\end{equation}
The fact that we assume that $k_2\leq \topc-1$ means that we are fixing the curve indexed by $\topc$ and not resampling it. This is why we use the term {\em partial} in describing this Gibbs property. In this text we will primarily be concerned with the behavior of $L_1$, the lowest labeled curve. If we restrict our attention to just the few lowest labeled curves, i.e. $L_1$ through  $L_{\topc}$, of a Gibbsian line ensemble with more curves, the Gibbs property transfers to the restriction provided that we do not resample the $\topc$-th curve. We will generally drop the term {\em partial} and just refer to this as the Gibbs property.

The above definition implies that on a given domain the law of the line ensemble is determined only by the boundary values of the domain, and independent of what lies outside. In this sense, this is similar to a spatial version of the Markov property. Section \ref{Section4.1} contains a precise definition of the $(H,H^{RW})$-Gibbsian line ensembles, that we have described above.

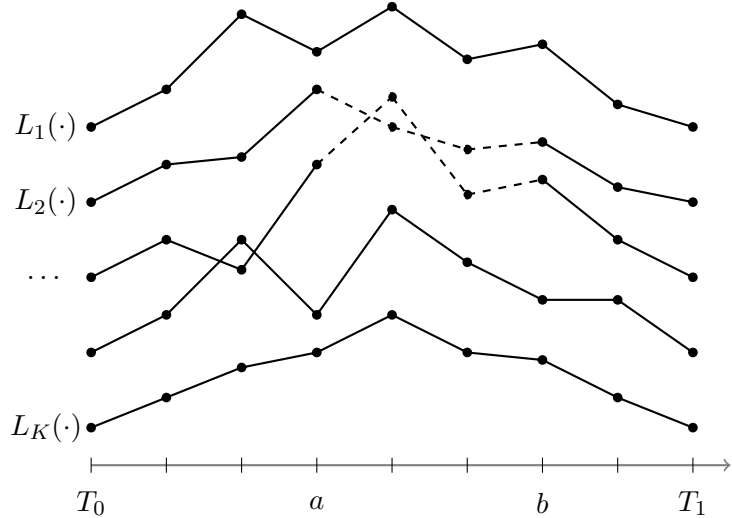
\begin{figure}
\begin{tikzpicture}[scale=1]
\filldraw[thick] (-1,4.5) circle(0.05) -- (0,5) circle(0.05) -- (1,6) circle(0.05)  -- (2,5.5) circle(0.05) -- (3,6.1) circle(0.05) -- (4,5.4) circle(0.05) --  (5,5.6) circle(0.05) -- (6,4.8) circle(0.05) -- (7,4.5) circle(0.05);
\filldraw[thick] (-1,3.5) circle(0.05) -- (0,4) circle(0.05) -- (1,4.1) circle(0.05)  -- (2,5) circle(0.05);
\filldraw[thick, dashed] (2,5) -- (3,4.5) circle(0.05) -- (4,4.2) circle(0.05) -- (5,4.3);
\filldraw[thick]  (5,4.3) circle(0.05) -- (6,3.7) circle(0.05)-- (7,3.5) circle(0.05);
\filldraw[thick]  (-1,2.5) circle(0.05) -- (0,3) circle(0.05) -- (1,2.6) circle(0.05)  -- (2,4) circle(0.05);
\filldraw[thick, dashed] (2,4)  -- (3,4.9) circle(0.05) -- (4,3.6) circle(0.05) -- (5,3.8) ;
\filldraw[thick] (5,3.8) circle(0.05)-- (6,3) circle(0.05)-- (7,2.5) circle(0.05);
\filldraw[thick] (-1,1.5) circle(0.05) -- (0,2) circle(0.05) -- (1,3) circle(0.05)  -- (2,2) circle(0.05) -- (3,3.4) circle(0.05) -- (4,2.7) circle(0.05) -- (5,2.2) circle(0.05) -- (6,2.2) circle(0.05)-- (7,1.5) circle(0.05);
\filldraw[thick] (-1,0.5) circle(0.05) -- (0,0.9) circle(0.05) -- (1,1.3) circle(0.05)  -- (2,1.5) circle(0.05) -- (3,2) circle(0.05) -- (4,1.5) circle(0.05) -- (5,1.4) circle(0.05) -- (6,0.9) circle(0.05)-- (7,0.5) circle(0.05);
\draw[gray, thick, -> ] (-1,0)  -- (7.5,0);
\foreach \x in {-1,...,7}
\draw (\x, 0.1) -- (\x,-0.1);
\draw (-1,-0.5) node{$T_0$};
\draw (7,-0.5) node{$T_1$};
\draw (2,-0.5) node{$a$};
\draw (5,-0.5) node{$b$};
\draw (-1.6,4.5) node{$ L_1(\cdot)$};
\draw (-1.6,3.5) node{$ L_2(\cdot)$};
\draw (-1.6,2.5) node{$\dots$};
\draw (-1.6,0.5) node{$ L_\topc(\cdot)$};
\end{tikzpicture}
\caption{A discrete line ensemble $\mathfrak{L}$. We illustrate the Gibbs property. The distribution of lines $L_i$ indexed by $i\in \llbracket 2,3 \rrbracket$ on the interval $\llbracket a,b\rrbracket$ (corresponding to the portion of lines that are dashed in the picture) is absolutely continuous with respect to the law of random walk bridges (with law determined by $H^{RW}$) joining $L_i(a)$ to $L_i(b)$, with Radon-Nikodym derivative proportional to \eqref{RadonNikodymIntroeq}.}
\label{LineEnsembleFig}
\end{figure}

\medskip
We can now state our main result on general Gibbsian line ensembles. In words, our theorem says that the Gibbs property propagates one-point tightness to spatial tightness of the lowest labeled curve in a sequence of general Gibbsian line ensembles and that all subsequential limits are absolutely continuous with respect to a suitably scaled Brownian bridge measure. The diffusive scaling in defining $f_N(s)$ in \eqref{eq:fnsintro} is present (and expected), because the underlying random walk measure dictated by $H^{RW}$ is assumed to converge to Brownian motion under diffusive scaling.

Since our theorem pertains to the lowest labeled curve of a Gibbsian line ensemble, it is phrased in terms of line ensembles with two curves $L_1$ and $L_2$. This could arise as the marginal of an ensemble on more curves. Since we are working with the {\em (partial)} Gibbs property, taking this marginal preserves our ability to apply the Gibbs resampling to the $L_1$ curve.

In the statement of our theorem, we make certain assumptions on $H$ and $H^{RW}$, that are given in Definitions \ref{AssH} and  \ref{AssHR}, respectively. Most of these are mild growth bounds that hold in typical examples. The fundamental assumption that we make on both $H$ and $H^{RW}$ is convexity. This is key because it implies that our $(H,H^{RW})$-Gibbsian line ensembles enjoy a monotone coupling, whereby shifting the boundary data for the line ensemble up results in the measure shifting up. This result is shown as Lemma \ref{MonCoup}. 

The following definition is useful in stating our main theorem.

\begin{definition}\label{def:intro}
 Assume that $H^{RW}$ satisfies the conditions of Definition \ref{AssHR} and $H$ those of Definition \ref{AssH}. Fix $\alpha > 0$, $p \in \mathbb{R}$ and $T > 0$. Suppose we are given a sequence $\{T_N\}_{N = 1}^\infty$ with $T_N \in \mathbb{N}$ and that $\{\mathfrak{L}^N\}_{N = 1}^\infty$ is a sequence of $\llbracket 1, 2 \rrbracket \times \llbracket -T_N, T_N \rrbracket$-indexed line ensembles $\mathfrak{L}^N = (L^N_1, L^N_2)$. We say that the sequence $\big\{\mathfrak{L}^N\big\}_{N=1}^{\infty}$ is $(\alpha,p,T)$--{\em good}, if there exists $N_0=N_0(\alpha,p,T)>0$, such that for $N \geq N_0$
\begin{itemize}
\item $T_N > T N^{\alpha} + 1$ and $\mathfrak{L}^N$ satisfies the $(H, H^{RW})$-Gibbs property;
\item for each $s \in [-T, T]$ the sequence $ N^{-\alpha/2}\big(L_1^N(\lfloor sN^{\alpha} \rfloor ) - p s N^{\alpha}\big) $ is tight.\\
(In other words, we have one-point tightness of the top curve under scaling of space by $N^{\alpha}$ and fluctuations by $N^{\alpha/2}$.)
\end{itemize}
\end{definition}
In words, the above definition states that $L^N_1$ is a sequence of random curves, which globally have a slope $p \in \mathbb{R}$; moreover, when the line of slope $p$ is subtracted from $L^N_1$ the resulting sequence of random curves scaled horizontally by $N^{-\alpha}$ and vertically by $N^{-\alpha/2}$ has tight one-point marginals over a fixed interval $[-T,T]$. The assumption $T_N > T N^{\alpha} + 1$ is merely there to ensure that the rescaled lines (and consequently their marginals) are well-defined on $[-T,T]$. 

It is possible to formulate Definition \ref{def:intro} for line ensembles containing more than two curves; however, if we have such a line ensemble, the nature of the Gibbs property allows us to restrict it to the top two curves, and then this restricted line ensemble of two curves will satisfy the same assumptions as above. Since all of our results describe the behavior of $L_1^N$, there is no loss in generality in assuming that our line ensemble has exactly (rather than at least) two curves.

We may now state our main black-box theorem, whose proof is a combination of Theorems \ref{PropTightGood} and \ref{ACBB} in the main text.
\begin{theorem}\label{informalintrotightness}
Fix $\alpha, T > 0$ and $p \in \mathbb{R}$ and let $\big\{\mathfrak{L}^N = (L^N_1, L^N_2)\big\}_{N=1}^{\infty}$ be an $(\alpha, p, T+3)$--good sequence of line ensembles. For $N \geq N_0(\alpha, p, T + 3)$ (where $N_0(\alpha, p, T + 3)$  is as in Definition \ref{def:intro} and exists by our assumption of being $(\alpha, p, T+3)$--good) let $f_N(x)$ be given by
\begin{equation}\label{eq:fnsintro}
f_N(x) := N^{-\alpha/2}\big(L_1^N( xN^{\alpha}  ) - p x N^{\alpha}\big)
\end{equation}
Let $\mathbb{P}_N$ denote the law of $f_N$ as a random variable in $(C[-T,T], \mathcal{C})$, where $\mathcal{C}$ is the Borel $\sigma$-algebra coming from the topology of uniform convergence in $C[-T,T]$. Then, the sequence of distributions $\mathbb{P}_N$ is tight in $N$. 
Furthermore, all subsequential limits of $\mathbb{P}_N$ are absolutely continuous with respect to the Brownian bridge with variance $2T\sigma_p^2$ ( the absolute continuity statement is explained in Definition \ref{DACB} and $\sigma_p$ is defined in terms of $H^{RW}$ in Definition \ref{AssHR} and represents the diffusion coefficient of the Brownian bridge whose domain is $[-T,T]$, hence the factor of $2T$).
\end{theorem}
\begin{remark} With a bit of work, the assumption that $\big\{\mathfrak{L}^N = (L^N_1, L^N_2)\big\}_{N=1}^{\infty}$ is $(\alpha, p, T+3)$-good can be replaced with being $(\alpha, p, T+\epsilon)$-good for some $\epsilon > 0$. In words, we can ensure the tightness and Brownian subsequential limits of the restrictions of our curves to $[-T,T]$ starting from a one-point marginal tightness on a slightly bigger interval $[-T -\epsilon, T+\epsilon]$. Our choice of $\epsilon = 3$ is purely cosmetic and made to simplify the notation and proofs in the main text.
\end{remark}

\subsubsection{Comparison to previous literature}

Since there is now a fairly large literature studying Gibbsian line ensembles and their tightness, we briefly describe how our work fits into and extends this literature. In particular, the two main innovations of this paper are that (1) we deal with a very general class of Gibbs properties and (2) we provide a completely new approach to proving the key stochastic monotonicity (see Lemma \ref{MonCoup}). Regarding the second point, our new approach allows us to construct a monotone coupling, which compared to previous results is completely explicit, holds in greater generality and enjoys remarkable topological properties -- we elaborate on these statements in this section as well as in the remarks that follow Lemma \ref{MonCoup}.

Previous work on Gibbsian line ensembles have mainly focused on systems where the underlying random walk is a Brownian motion
\cite{CorHamK, CGH19, HamA, HamB, HamC, HamD, CHH19, DV18, CIW, caputo2019}. Recently, there have been some studies of discrete underlying random walks, which have jumps that are Bernoulli or geometric \cite{CD,DNV19, DREU}. To move to general random walks, we utilize a recently developed bridge extension of the KMT strong coupling \cite{KMT1}, which was developed in \cite{DW19}. That work was, in fact, developed for application to this present paper and the related work of \cite{Wu19}.

 \cite[Chapter 3]{Wu19} works with a different but related form of our $(H,H^{RW})$-Gibbsian line ensemble, in which the interaction Hamiltonian $H$ depends on the scaling parameter $N$ (which also indexes a sequence of line ensembles as in Theorem \ref{informalintrotightness}) and converges in a suitable $N$-dependent scale to an exponential function. The main result of \cite[Chapter 3]{Wu19} is that if the lowest labeled curve is tight in the same $N$-dependent scale, in which $H$ becomes an exponential, then the entire line ensemble is tight in that scaling. Moreover, \cite[Chapter 3]{Wu19} shows that all subsequential limits enjoy the same exponential Brownian Gibbs property that was introduced in \cite{CorHamK} in the context of the KPZ line ensemble.

In contrast to the work of \cite{Wu19}, we deal with general interaction Hamiltonians $H$, which are not scaling with $N$. Though we presently only prove tightness of the lowest labeled curve, we expect the entire edge of the line ensemble is similarly tight and that all subsequential limits enjoy the non-intersecting Brownian Gibbs property introduced in \cite{CorHamA} in the context of the Airy line ensemble.

The other main innovation in our current work is a completely new approach to proving stochastic monotonicity (see Lemma \ref{MonCoup}). This property is key to the entire Gibbsian line ensemble machinery, since it enables us to reduce various interacting systems of random walks to estimates about single random walks or interacting random walks with simpler boundary conditions. Until now, the only approach that people have taken to proving stochastic monotonicity for Gibbsian line ensembles is through Markov chain Monte Carlo (MCMC) methods. Specifically, given the boundary data for a Gibbsian line ensemble, the measure on the curves can be sampled by running a MCMC until it reaches its stationary state, which is the desired measure. The key to proving the stochastic monotonicity is to show that for a pair of ordered boundary data the MCMC can be coupled to maintain ordering. In more detail, we start both chains off at their lowest possible configurations (which should be ordered due to the ordering of the boundary data) and run them until they reach stationarity. This provides a coupling of the two measures, which clearly satisfies the right ordering to imply stochastic monotonicity.

The MCMC approach to proving stochastic monotonicity was first implemented in \cite{CorHamA} in the context of non-intersecting Bernoulli random walks (see also \cite{DREU}) and then extended to non-intersecting Brownian bridges in  \cite{CorHamA} via a limit transition. The treatment in \cite{CorHamA} of this limit from discrete to continuous was terse and short on details, though this issue has been since remedied in \cite{DM20}.
In the case of $H$-Brownian Gibbs line ensembles (such as the KPZ line ensemble) \cite{CorHamK} implemented a similar scheme (also short on details), which relied on proving a stochastic monotonicity for Bernoulli random walks subject to a convex interaction Hamiltonian and then transferring that monotonicity to the Brownian setting through a diffusive scaling limit. 

In the context of the $(H,H^{RW})$-Gibbsian line ensembles, which we consider herein, \cite{Wu19} implemented the MCMC approach. As in the work of \cite{CorHamA,CorHamK}, \cite{Wu19} first worked with a discrete approximation (though no longer Bernoulli random walks since the aim was to access general continuous jump distributions). In the discrete case the MCMC approach provides the desired coupling for any two ordered boundary data, and in the proof of \cite[Lemma 3.1.11]{Wu19} it was shown that this monotonicity transfers to the continuous limit. In the present paper, we introduce a new method, which works directly with the distribution function of the line ensemble. This way there is no need to approximate or pass to limits.

Our construction, which is the content of Lemma \ref{MonCoup}, provides a continuous grand monotone coupling of Gibbsian line ensembles with respect to their boundary data (entrance and exit values, and bounding curves). {\em Continuous} means that the Gibbs measure varies continuously with respect to varying the boundary data, {\em grand} means that all uncountably many boundary data measures are coupled to the same probability space, and {\em monotone} means that raising the values of the boundary data likewise raises the associated measure. This result applies to a general class of Gibbsian line ensembles where the underlying random walk measure is discrete time, continuous valued and log-convex, and the interaction Hamiltonian is nearest neighbor and convex. 

One advantage of our continuous grand monotone coupling is that, unlike previous works that showed that two line ensembles with fixed boundary data (that are ordered) can be monotonically coupled, our result shows that {\em all} line ensembles with {\em all possible} boundary data can be simultaneously monotonically coupled. One might be able to improve the MCMC argument in the discrete setting to prove a similar statement; however, there is some delicacy in showing that the convergence of the discrete approximations to the limit happens simultaneously for all boundary data, since the latter form an uncountable set. 

An additional advantage of our continuous grand monotone coupling is that it is completely explicit, and the space, on which the line ensembles are coupled, is a standard Borel space (in fact it is nothing but the unit cube $(0,1)^n$ of appropriate dimension, with the Borel $\sigma$-algebra and Lebesgue measure). The explicit realization of the coupling probability space as a topological space allows one to probe the topological properties of the coupling. In this direction we establish, as part of Lemma 2.10, the continuity of the map that takes as input the boundary data and an elementary outcome in our probability space (i.e. a point in $(0,1)^n$) and gives as output a line ensemble evaluated at this outcome. While we do not use this statement in our proofs, we hope that such a statement can find future applications. The general point here is that our construction makes it possible to extract topological information regarding our line ensemble and its dependence on the boundary data, which was previously unattainable via the MCMC coupling techniques.

Our continuous grand monotone coupling result presently only holds for one curve, which suffices for the purposes of the present paper. It would be interesting to see an extension of our construction to arbitrary number of curves although presently it does not seem to be straightforward.
\vspace{-2mm}

\subsubsection{Natural extensions}
We close out our discussion on general Gibbsian line ensembles by identifying a few natural extensions to our results and methods that we believe merit further investigation.

This paper focuses on the lowest labeled curve of general $(H,H^{RW})$-Gibbsian line ensembles. This is because our application (see Section \ref{sec:introlog}) only requires such control. However, it would be natural to extend our tightness result to the entire edge of the line ensemble and, moreover, to show that all subsequential limits enjoy the non-intersecting Brownian Gibbs property. The reason for the non-intersecting Gibbs property is because the curves should separate in the $N^{\alpha}$-scale and by our assumption that $H(x)$ goes to infinity as $x$ does, we should see a limiting hard-wall potential emerge.  Lemma \ref{MonCoup}, our monotone coupling result, is restricted to only deal with the lowest labeled curve though we expect that a more general coupling for arbitrarily many curves should be provable via an extension of our new method.

Our $(H,H^{RW})$-Gibbsian line ensembles are not the most general for which one could hope to prove tightness results. For instance, the underlying random walk Hamiltonian $H^{RW}$ could be inhomogeneous, varying within the line ensemble. The Radon-Nikodym derivative in \eqref{RadonNikodymIntroeq} could also involve a more general type of local interaction than just between pairs $L_i(m)$ and $L_{i+1}(m+1)$. Alternatively, one could release the condition on convexity of either $H^{RW}$ or $H$. All of these variations  arise in some form when considering the Gibbsian line ensemble associated to the stochastic vertex models introduced in \cite{CP16}. We believe it is worth further study to determine how stochastic monotonicity is affected by these variations and to classify the general hypotheses, under which it holds. In the case of a non-convex interaction Hamiltonian, it is already known from \cite{CD} that the stochastic monotonicity can be lost. However, \cite{CD} showed that a weaker version of that property still holds and is sufficient to prove tightness of the lowest labeled curved of the Hall-Littlewood Gibbsian line ensemble considered therein.

\vspace{-2mm}

%
\subsection{Application to the log-gamma polymer}\label{sec:introlog}

The main motivation behind our black-box (Theorem \ref{informalintrotightness}) is its application to the log-gamma polymer \cite{Sep12}. We start this section by introducing the model. The connection between the log-gamma polymer and an $(H,H^{RW})$-Gibbsian line ensembles is recorded as Corollary \ref{cor:loggammalineensemble}. This is a corollary of Proposition \ref{PropLOG}, which is stated and proved in Section \ref{Section8}, and follows with some work from the results of \cite{COSZ}.  We then state a one-point tightness (actually limit theorem) result that is proved in our companion paper \cite{BCDA}. We close out this section by combining the line ensemble interpretation with the one-point tightness (using Theorem \ref{informalintrotightness}) to show transversal tightness of the log-gamma polymer free energy with the ubiquitous $2/3$ KPZ universality class exponent. We also briefly mention some other applications of the Gibbs property.

\subsubsection{The log-gamma polymer}
Recall that a random variable $X$ is said to have the inverse-gamma distribution with parameter $\theta > 0 $ if its density against Lebesgue measure is given by 
$$f_\theta(x) = {\bf 1} \{ x > 0 \}\Gamma(\theta)^{-1}  x^{-\theta - 1} \exp( - x^{-1}).$$
 Fixing $\theta > 0$, we let $d = \left( d_{i,j} : i \geq 1, j \geq 1\right)$ denote the semi-infinite random matrix of i.i.d. entries $d_{i,j}$ that are inverse-gamma distributed with the same $\theta$ parameter. A  directed lattice path is an up-right path on $\mathbb{Z}^2$, which makes unit steps in the positive coordinate directions (see Figure \ref{halfspaceloggamma}). Given $n,N \geq 1$, we let $\Pi_{n,N}$ denote the set of directed paths $\pi$ in $\mathbb{Z}^2$ from $(1,1)$ to $(n,N)$. Given a directed path $\pi$, we define its weight $w(\pi)$ to be the product of all $d_{i,j}$ where $(i,j)$ are vertices contained in the path $\pi$:
\begin{equation}\label{PathWeight}
w(\pi) = \prod_{(i,j) \in \pi} d_{i,j}
\end{equation}
From this we define the partition function $Z^{n,N}$ to be the sum over all weights
\begin{equation}\label{PartitionFunct}
Z^{n,N} = \sum_{ \pi \in \Pi_{n,N}} w(\pi).
\end{equation}
The logarithm of the partition function is called the free energy.

\begin{figure}
	\begin{center}
		\begin{tikzpicture}[scale=0.8]
		\begin{scope}
		\draw (0,0) node[anchor = north]{{\footnotesize $(1,1)$}};
		\draw (9,6) node[anchor=south]{{\footnotesize $(n,N)$}};
		\draw[->, >=stealth'] (9,2.5) node[anchor=west]{{\footnotesize $d_{i,j}$}} to[bend right] (7,2);
		\draw[dotted, gray] (0,0) grid (9,6);
		\draw[ultra thick] (0,0) -- (1,0) -- (2,0) -- (2,1) -- (2,2) -- (3,2) -- (4,2) -- (4,3) -- (5,3) -- (6,3) -- (6,4) -- (7,4) -- (8,4) -- (8,6) -- (9,6);
		\end{scope}
		\end{tikzpicture}
	\end{center}
	\caption{A directed lattice path $\pi\in \Pi_{n,N}$ in the log-gamma polymer model. }
	\label{halfspaceloggamma}
\end{figure}
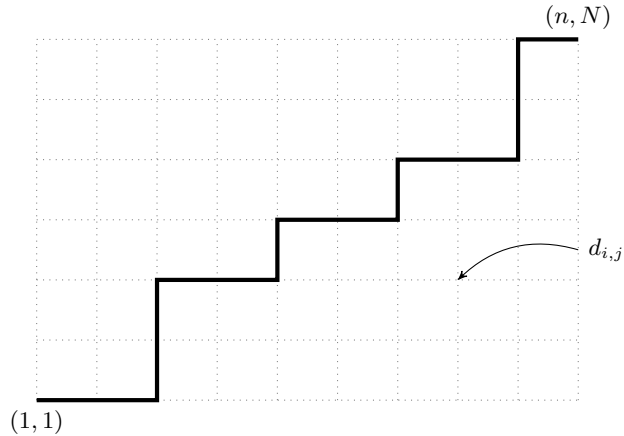

\subsubsection{Embedding the log-gamma polymer in a line ensemble} In Section \ref{Section8.2} we prove that the log-gamma polymer can be embedded as the lowest labeled curve in a discrete line ensemble, that satisfies the $(H,H^{RW})$-Gibbs property, where
\begin{equation}\label{Hamiltonianintro}
H^{RW}(x) = \theta x + e^{-x} + \log \Gamma (\theta)\qquad \textrm{and}\qquad H(x) = e^x,
\end{equation}
with $\theta > 0$ as in the definition of the log-gamma polymer model. This result, stated as Proposition \ref{PropLOG}, implies that the other lines in this log-gamma line ensemble have meaning in terms of polymer partition functions for multiple paths in the log-gamma environment. Similar results have appeared in \cite[Section 3.4]{Wu19} and \cite{JO20} as well. For this introduction, we will simply record a corollary of Proposition \ref{PropLOG} as needed to apply our black-box tightness result.

\begin{corollary}\label{cor:loggammalineensemble}
Let $H,H^{RW}$ be as in (\ref{Hamiltonianintro}). Fix $\topc, N \in \mathbb{N}$ with $N \geq \topc \geq 2$ and Let $T_0, T_1 \in \mathbb{N}$ be such that $T_0 < T_1$ and $T_0 \geq \topc$. Then, we can construct a probability space with measure $\mathbb{P}$ that supports a $\llbracket 1, \topc \rrbracket \times \llbracket T_0, T_1 \rrbracket$-indexed line ensemble $\mathfrak{L} = (L_{1}, \dots, L_\topc)$ such that:
\begin{enumerate}
\item the $\mathbb{P}$-distribution of $(L_1(n): n\in\llbracket T_0, T_1 \rrbracket )$ is the same as that of $ ( \log Z^{n,N}: n\in\llbracket T_0, T_1 \rrbracket )$;
\item $\mathfrak{L}$ satisfies the $(H,H^{RW})$-Gibbs property.
\end{enumerate}
\end{corollary}

\begin{proof}
This follows immediately from Proposition \ref{PropLOG} by identifying the notation $z_{N,1}(n)$ used there with $Z^{n,N}$ defined this introduction.
\end{proof}

\begin{remark}\label{S1RemWS}
The strict-weak polymer model \cite{CSS15, OO15} also enjoys a relationship to a similar $(H,H^{RW})$-Gibbsian line ensemble. We expect that all of the results, which are proved in our present work for the log-gamma polymer, can be likewise proved for the strict-weak polymer. The only technical input, which would need to be developed, are analogous asymptotic results to those proved for the log-gamma polymer in \cite{BCDA} (see Section \ref{sec:asyBCDA} for the precise results).
\end{remark}

\subsubsection{Asymptotic fluctuations of the log-gamma polymer free energy}\label{sec:asyBCDA}

For each $N$, Corollary \ref{cor:loggammalineensemble} provides an embedding of the log-gamma polymer free energy $\log Z^{n,N}$, as a process in $n$, as the top curve of a $(H,H^{RW})$-Gibbsian line ensemble.
We are interested in the large $N$ and $n$ limit of this process. In order to apply our black-box theory (Theorem \ref{informalintrotightness}), we need information about the one-point fluctuations of $\log Z^{n,N}$. This is accomplished in \cite{BCDA}. We first recall the necessary notation and then recall the relevant result proved therein.

\begin{definition}
Let $\Psi(x)$ denote the digamma function, defined by
\begin{equation}\label{digammaS1}
\Psi(z) = \frac{\Gamma'(z)}{\Gamma(z)} = - \gamma_{E} + \sum_{n = 0}^\infty \left[\frac{1}{n + 1} - \frac{1}{n+z} \right],
\end{equation}
with $\gamma_{E}$ denoting the Euler constant.
Define the function
\begin{equation}\label{DefLittleg}
g_\theta(z) = \frac{\sum_{n =0}^\infty \frac{1}{(n+\theta - z)^2}}{ \sum_{n = 0}^\infty \frac{1}{(n+z)^2}} = \frac{\Psi'(\theta -z)}{\Psi'(z)},
\end{equation}
and observe that it is a smooth, strictly increasing bijection from $(0, \theta)$ to $(0, \infty)$. The inverse function $g_\theta^{-1}: (0, \infty) \rightarrow (0,\theta)$ is also a strictly increasing smooth bijection. For $x \in (0,\infty)$, define the function
\begin{equation}\label{HDefLLN}
h_\theta(x) = x \cdot  \Psi(g_\theta^{-1}(x)) + \Psi( \theta - g_\theta^{-1}(x)),
\end{equation}
which is easily seen to be a smooth function on $(0, \infty)$. Finally, for  $x \in (0,\infty)$, define the function
\begin{equation}\label{DefSigma}
\tsigma_\theta(x) = \left[ \sum_{n = 0}^\infty \frac{x}{\big(n+g_\theta^{-1}(x)\big)^3} +  \sum_{n = 0}^\infty \frac{1}{\big(n+\theta - g_\theta^{-1}(x)\big)^3} \right]^{1/3}.
\end{equation}
\end{definition}

\medskip
We now consider $\log  Z^{n,N}$ for $n$ and $N$ going to infinity with a ratio which is approximately $r \in (0, \infty)$. We recall one of the main results of \cite{BCDA} which shows that as $N$ and $n$ tend to $\infty$ the one-point marginals of a properly centered and scaled version of $\log  Z^{n,N}$ tend to the GUE Tracy-Widom distribution \cite{TWPaper}. For $n,N\geq 1$ define the rescaled free energy
\begin{equation}\label{ResFE}
\mathcal{F}(n,N) := \frac{ \log Z^{n,N} + N h_{\theta}(n/N)}{N^{1/3} \tsigma_{\theta}(n/N)}  .
\end{equation}

\begin{proposition}\cite[Theorem 1.2]{BCDA}\label{thm:BCDA}
Let $\theta ,r>0$ be given. Assume that $n$ and $N$ go to infinity in such a way that the sequence $n/N$ converges to $r$. Then, for all $y\in \R$,
$$
\lim_{N\to \infty} \P\big(\mathcal{F}(n,N)\leq y\big) = F_{GUE}(y).
$$
\end{proposition}
\begin{remark} We mention here that \cite[Theorem 1.2]{BCDA} was formulated with $\mathcal{F}(n,N)$ defined by
$$\frac{ \log Z^{n,N} + n h_{\theta}(N/n)}{n^{1/3} \tsigma_{\theta}(N/n)};$$
however, this is readily seen to agree with (\ref{ResFE}) once we utilize the fact that $g_\theta^{-1}(1/x) = \theta - g_\theta^{-1}(x).$
\end{remark}

In light of Proposition \ref{thm:BCDA}, we are lead to define a centered and scaled spatial process $f_N^{LG}(\cdot)$ for the free energy.

\begin{definition}\label{RandomCurve} Fix any  $T >0$, $\theta > 0$ and $r \in (0,\infty)$. Suppose that $N$ is sufficiently large so that $rN \geq 2 + TN^{2/3}$. For each $x \in [-T-N^{-2/3}, T + N^{-2/3}]$ such that $xN^{2/3}$ is an integer, we define $n = \lfloor rN \rfloor + xN^{2/3}$ and
\begin{equation}\label{LGCurve}
f_N^{LG}(x)=N^{-1/3}\Big(\log  Z^{n,N} + h_{\theta}(r)N + h_\theta'(r) x N^{2/3}\Big),
\end{equation}
and then extend $f_N^{LG}$ to all points $x \in [-T, T]$, by linear interpolation.
The above construction provides a random continuous curve in the space $(C[-T,T], \mathcal{C})$ -- the space of continuous functions on $[-T, T]$ with the uniform topology and Borel $\sigma$-algebra $\mathcal{C}$ (see e.g. Chapter 7 in \cite{Bill}) -- and we denote its law by $\mathbb{P}_N$.
\end{definition}

\medskip
We now combine the $(H,H^{RW})$-Gibbs property for the log-gamma line ensemble, constructed in Corollary \ref{cor:loggammalineensemble}, with the convergence in Proposition \ref{thm:BCDA}. These provide the input to apply Theorem \ref{informalintrotightness} and lead to the following transversal tightness and Brownian absolute continuity result for the log-gamma polymer free energy.
\begin{theorem}\label{ThmTight}
Fix any $T, \theta, r >0$. Then, the laws $\P_N$ of $f_N^{LG}\big([-T,T]\big)$ (see Definition \ref{RandomCurve}) form a tight sequence in $N$. Moreover, any subsequential limit $\P_{\infty}$ is absolutely continuous with respect to the Brownian bridge with variance $ 2T\Psi'(g_{\theta}^{-1}(r))$ (see Definition \ref{DACB}) . 
\end{theorem}
\begin{proof}[Proof sketch] Here we provide a sketch of the proof of Theorem \ref{ThmTight}. The goal is to explain how the different statements in the introduction fit together to produce the result. A complete proof to the theorem can be found in Section \ref{Section8.3}.

Let $\LB = \lfloor rN + (T+3)N^{2/3} + 2 \rfloor$ and fix any $\topc\geq 2$. For each $N\geq \topc$, Corollary \ref{cor:loggammalineensemble} provides us with a $\llbracket 1, \topc \rrbracket \times \llbracket K, \LB \rrbracket$-indexed line ensemble, which we will denote $\mathfrak{\tilde{L}}^N$, whose lowest labeled curve $\big(\tilde{L}^N_1(n):n\in \llbracket \topc,\LB \rrbracket \big)$ has the same law as $\big(\log Z^{n,N}:n\in \llbracket\topc, \LB \rrbracket \big)$. Moreover, this line ensemble enjoys the $(H,H^{RW})$-Gibbs property with $H$ and $H^{RW}$ given in \eqref{Hamiltonianintro}.

We define the $\llbracket 1, 2 \rrbracket \times \llbracket -T_N, T_N \rrbracket$-indexed line ensemble $\mathfrak{L}^N$ by setting
$$  L_i^N(x) = \tilde{L}_i^N(x+ \lfloor rN\rfloor) + N h_\theta(r) \mbox{ for $i = 1,2$ and $x \in \llbracket -T_N, T_N\rrbracket$},$$
where $T_N = \lfloor (T+3)N^{2/3} + 2 \rfloor$. If $f_N$ denotes the function in (\ref{eq:fnsintro}) for the line ensemble $\mathfrak{L}^N$ with $\alpha = 2/3$ and $p = -  h_\theta'(r)$, then one observes that $f_N$ and $f_N^{LG}\big([-T,T]\big)$ have the same distribution. Consequently, Theorem \ref{informalintrotightness} would imply the present theorem provided we can show that the sequence $ \mathfrak{L}^N$ is $(2/3, - h_\theta'(r), T+3)$-good in the sense of Definition \ref{def:intro}. 

The strength of Theorem \ref{informalintrotightness} is that it reduces our problem to verifying that $ \mathfrak{L}^N$ satisfies all the assumptions in Definition \ref{def:intro}. We first need to show that $H$ and $H^{RW}$ in (\ref{Hamiltonianintro}) satisfy the assumptions in Definitions \ref{AssHR} and \ref{AssH}. For example, one of these assumptions is that $H$ and $H^{RW}$ are both convex, which is clear from (\ref{Hamiltonianintro}). There are more assumptions, which we will not discuss presently, but verifying all of them is straightforward due to the explicit nature of $H$ and $H^{RW}$ in (\ref{Hamiltonianintro}) and takes only several lines (see Section \ref{Section8.3}).

The second thing we need to check is that $\mathfrak{L}^N$ satisfies the $(H,H^{RW})$-Gibbs property, which is immediate as we know the latter to be true for $\mathfrak{\tilde{L}}^N$ and the Gibbs property is maintained upon horizontal and vertical shifts. Finally, we need to check that $f_N^{LG}$ has one-point tight marginals, which is a consequence of Proposition \ref{thm:BCDA}. 

To summarize, if one ignores the technical assumptions on $H$ and $H^{RW}$ (which can be verified), there are two key parts that need to be checked in Definition \ref{def:intro} -- the $(H,H^{RW})$-Gibbs property and the one-point tightness. In our case, the former is seen to hold by Corollary \ref{cor:loggammalineensemble} and the latter by Proposition \ref{thm:BCDA}. Once these two pieces are in place, our black-box result (Theorem \ref{informalintrotightness}) completely takes over and establishes the tightness and Brownian continuity of all subsequential limits of $f_N^{LG}$. 
\end{proof}
\begin{remark}
There are other KPZ class models whose spatial processes can be embedded into Gibbsian line ensembles. For some of these models, similar tightness and Brownian absolute continuity results (like Theorem \ref{ThmTight}) have been demonstrated. In particular, there are similar results for:
Brownian LPP \cite{CorHamA}, the O'Connell-Yor polymer model and KPZ equation \cite{CorHamK}, and the asymmetric simple exclusion process and stochastic six vertex model \cite{CD}. For the integrable models of last passage percolation (which are related to discrete Gibbsian line ensembles with Bernoulli, geometric and exponential jump distributions) \cite{DNV19} addresses the question of tightness assuming finite dimensional convergence to the Airy line ensemble. In the case when the Gibbsian line ensembles have the structure of avoiding Bernoulli random walkers, \cite{DREU} proves tightness of the full line ensembles assuming one-point tightness of its lowest indexed curve. For the log-gamma polymer, the recent work \cite{Wu19} applies the Gibbs property to the weak-noise scaled free energy. This means that the parameter $\theta$, which controls the inverse-gamma distributions is tuned to go to infinity in a suitable manner as the dimensions of the polymer $N$ and $n$ go to infinity. In terms of the line ensemble, this means that the Gibbs property is changing with $N$ and in the limit becomes the exponential Brownian Gibbs property that was introduced in \cite{CorHamK} in the context of the KPZ line ensemble.
\end{remark}
\begin{remark}\label{RemKPZ}
Theorem \ref{ThmTight} states that when we view $\log Z^{n,N}$ as spatial processes in $n$, then as $N$ tends to infinity this sequence of processes (properly shifted) forms a tight sequence of non-trivial random continuous curves under a transversal scaling by $N^{2/3}$ and fluctuation scaling by $N^{1/3}$.  This demonstrates that the ubiquitous KPZ exponents hold for this model. The transversal $2/3$ exponent was previously demonstrated (in terms of non-trivial fluctuations of the polymer measure) for the log-gamma polymer with stationary boundary conditions in \cite{Sep12}. The information (e.g. tightness and Brownian absolute continuity), contained in Theorem \ref{ThmTight}, is of a rather different nature than the results proved in \cite{Sep12}.
\end{remark}

By KPZ universality, one expects that the sequence $f^{LG}_N$ in Theorem \ref{ThmTight} is not only tight but in fact convergent to some affine transformation of the Airy$_2$ process (shifted by a parabola). In the remainder of this section we formulate a precise conjecture (Conjecture \ref{Conj}) that details this convergence. The statement of Conjecture \ref{Conj} involves certain constants which we will introduce presently. After we state the conjecture we explain how our choice of constants supports its validity using the results of this paper. In Section \ref{appendixKPZscaling} we explain the KPZ scaling theory from  \cite{ krug1992amplitude,spohn2012kpz} for the log-gamma polymer and show that it also agrees with Conjecture \ref{Conj}.

Let us denote by $A_{\theta}(r)=\Psi'(g_{\theta}^{-1}(r))$ the diffusion coefficient appearing in Theorem \ref{ThmTight}. Notice that the function $d_{\theta}$ defined in \eqref{DefSigma} can be expressed in terms of $A_{\theta}(r)$ and $h_{\theta}''(r)$ as 
\begin{equation}\label{Defdvariant}
d_{\theta}(r) = \left( \frac{A_{\theta}(r)^2}{2 h_\theta''(r)} \right)^{1/3}.
\end{equation}
One can deduce the latter by noticing that $h'(r) = \Psi(g_\theta^{-1}(r))$ and $g_\theta(r) = \Psi'(\theta-r) /\Psi'(r).$

It is convenient to define another function 
\begin{equation}\label{DefKappa}
\kappa_{\theta}(r) = \left(  \frac{2A_{\theta}(r)}{h_{\theta}''(r)^2}\right)^{1/3}, 
\end{equation}
so that we have the relations 
\begin{equation}\label{simplification}
\frac{A_{\theta}(r) \kappa_{\theta}(r)}{2d_{\theta}(r)^2} = \frac{h_\theta''(r)\kappa_{\theta}(r)^2}{2 d_{\theta}(r)}=1.
\end{equation}

\begin{definition}\label{DefFLG2} Fix any $\theta,r > 0$. Let $\tilde{T}_N = \lfloor N^{2/3} \log N \rfloor$, $A_N = \kappa_\theta(r)^{-1} \tilde{T}_N N^{-2/3}$ and suppose that $N$ is sufficiently large, so that $rN \geq \tilde{T}_N + 2$.  For each $ x \in [-A_N, A_N] $, such that $\kappa_\theta(r) xN^{2/3}$ is an integer, we define $n = \lfloor rN \rfloor + \kappa_\theta(r) xN^{2/3}$ and
\begin{equation}\label{LGCurve2}
\tilde{f}_N^{LG}(x)=2^{-1/2}d_\theta(r)^{-1} N^{-1/3}\Big(\log  Z^{n,N} + h_{\theta}(r)N + h_\theta'(r) \kappa_\theta(r) x N^{2/3}\Big),
\end{equation}
and then extend $\tilde{f}_N^{LG}$ to $\mathbb{R}$ by:
\begin{itemize}
\item linear interpolation on the interval $ [- A_N, A_N]$;
\item constant extension outside the interval $ [- A_N, A_N]$, i.e. we put $\tilde{f}_N^{LG}(x) = \tilde{f}_N^{LG}(A_N)$ when $x \geq A_N$ and $\tilde{f}_N^{LG}(x) = \tilde{f}_N^{LG}(-A_N)$ when $x \leq -A_N$.
\end{itemize}
In this way, $\tilde{f}_N^{LG}$ becomes a random variable taking value in the space $C(\mathbb{R})$ of continuous functions on $\mathbb{R}$ with the topology of uniform convergence over compacts and Borel $\sigma$-algebra $\mathcal{C}$. 
\end{definition}

In words, the conjecture below states that if $\tilde{f}_N^{LG}$ are as in Definition \ref{DefFLG2}, then they converge (as random variables in $(C(\mathbb{R}), \mathcal{C})$) to a suitably scaled and parabolically shifted version of the Airy$_2$ process from \cite{prahofer2002scale}. We mention that the Airy$_2$ process is a random continuous process in $C(\mathbb{R})$, and the extension we performed in Definition \ref{DefFLG2} was to embed $ \tilde{f}_N^{LG}$ (initially defined at a restricted set of lattice sites) into $C(\mathbb{R})$. The precise definition of $\tilde{T}_N$ and the extension outside $[-A_N, A_N]$ is not important since we are dealing with the topology of uniform convergence over compacts. In particular, all that matters is that the sequence of intervals $[- A_N, A_N]$ increases to $\mathbb{R}$.

\begin{conjecture}\label{Conj} Fix any $\theta, r >0$ and let $\tilde{f}_N^{LG}$ be as in Definition \ref{DefFLG2}. Then, as $N \rightarrow \infty$ the random functions $\tilde{f}_N^{LG}$ converge weakly in $(C(\mathbb{R}), \mathcal{C})$ to $\mathcal{L}^{Airy}_1(x) = 2^{-1/2} (\mathcal{A}(x) - x^2)$, where $\mathcal{A}$ is the Airy$_2$ process from \cite{prahofer2002scale}.
\end{conjecture}
Here we give some credence to the conjecture by appealing to our Gibbsian line ensemble interpretation of the log-gamma polymer and the results of this paper.

Firstly, we observe that by definition we have $\tilde{f}_N^{LG}(x) = 2^{-1/2}d_\theta(r)^{-1} f_N^{LG}(\kappa_\theta x)$ and so by Theorem \ref{ThmTight} we know that $\tilde{f}_N^{LG}$ form a tight sequence of random curves in $C(\mathbb{R})$. As we mentioned a few times before, the free energy $\log Z^{N,n}$ can be embedded as the lowest labeled curve in a discrete line ensemble that satisfies the $(H,H^{RW})$-Gibbs property. The results of the present paper show that the lowest indexed curve of this ensemble is tight; however, one expects that the full line ensemble is tight and moreover that all subsequential limits satisfy the non-intersecting Brownian Gibbs property introduced in \cite{CorHamA}. The latter Gibbs property is the natural limit of the $(H,H^{RW})$-Gibbs property we deal with, and roughly states that the local structure of paths is that of non-intersecting Brownian bridges with a fixed diffusion parameter. If we assume that the latter tightness statement for the full log-gamma line ensemble is true, then Proposition \ref{thm:BCDA} and Theorem \ref{ThmTight} would imply that any subsequential limit of $\tilde{f}_N^{LG}$ can be realized as the lowest indexed curve of a line ensemble $\mathcal{L} = \{ \mathcal{L}_i \}_{i =1}^\infty$ such that:
\begin{enumerate}
\item $\mathcal{L}$ satisfies the non-intersecting Brownian Gibbs property of \cite{CorHamA} (where the Brownian bridges have diffusion parameter $1$);
\item  the random variables $\mathcal{L}_1$ and $\mathcal{L}_1^{Airy}$ have the same one-point marginal distribution.
\end{enumerate}
The first property can be deduced from the fact that the diffusion coefficient becomes $\frac{\kappa_{\theta}(r)}{2d_{\theta}(r)^2}$ times the diffusion coefficient in Theorem \ref{ThmTight} (in view of the relation $\tilde{f}_N^{LG}(x) = 2^{-1/2}d_\theta(r)^{-1} f_N^{LG}(\kappa_\theta x)$), which equals to $1$ by \eqref{simplification}. The second property can be deduced from Proposition \ref{thm:BCDA} and the fact that for each $x \in \mathbb{R}$ the random variable $\mathcal{A}(x)$ has the Tracy-Widom distribution.
 
If one replaces condition (2) above with the stronger condition that
\begin{enumerate}
\item[(2')]  the random variables $\mathcal{L}_1$ and $\mathcal{L}_1^{Airy}$ have the same finite dimensional distribution,
\end{enumerate}
then \cite{DM20} showed that $\mathcal{L}$ is equal to the parabolic Airy line ensemble $\mathcal{L}^{Airy} = \{\mathcal{L}_i^{Airy}\}_{i = 1}^\infty$ of \cite{CorHamA}. It seems plausible that conditions (1) and (2) uniquely pinpoint $\mathcal{L}^{Airy}$; however, this has not been proved so far. If true the latter statement together with the (also unproved) tightness and Brownian Gibbs structure of all subsequential limits of the log-gamma line ensemble would establish Conjecture \ref{Conj} and its natural generalization to the full line ensemble. It may also be possible that Conjecture \ref{Conj} can be approached by the method announced recently in \cite{virag2020heat}.

Characterizing the Airy line ensemble by conditions of the form (1) and (2) dates back to \cite[Conjecture 3.2]{CorHamA}, which suggested yet another characterization by condition (1) above and the condition that
\begin{enumerate}
\item[(2'')]  $\mathcal{L}$ is {\em extremal}, {\em shift-invariant} and $\mathbb{E}[\mathcal{L}_1(0)] = \mathbb{E}[\mathcal{L}^{Airy}_1(0)]$.
\end{enumerate}
We refer the interested reader to \cite[Conjecture 3.2]{CorHamA}, \cite[Conjecture 1.7]{CorwinSun} and \cite[Conjecture 1.4]{DM20} for further discussion of the latter conjecture and definition of the terms ``extremal'' and ``shift-invariant''.

Overall, the above discussion suggests that the scaling we have performed in Conjecture \ref{Conj} ensures the correct diffusion parameter and one-point marginal of the limit, which gives some  credence to its validity. In Section \ref{appendixKPZscaling}, we give some further support for the validity of Conjecture \ref{Conj}, based on the KPZ scaling theory of the log-gamma polymer.
\subsection*{Outline}
Section \ref{Section4} contains a number of foundational definitions and results about $(H,H^{RW})$-Gibbsian line ensembles. Section \ref{Section4.1} provides a detailed definition of discrete $(H,H^{RW})$-Gibbsian line ensembles (Definition \ref{DefLGGP}). Section \ref{Section4.2} contains the statement and proof of our continuous grand monotone coupling result (Lemma \ref{MonCoup}). Finally, Section \ref{Section4.3} contains additional technical assumptions (see Definition \ref{AssHR}) that we make on the random walk Hamiltonian $H^{RW}$ to be able to strongly couple it to a Brownian bridge. A number of the highly technical (and at times, measure theoretic) proofs from Section \ref{Section4} are deferred to Section \ref{Section11}.

Section \ref{Section5} contains a restatement and proof of the first part (tightness) of Theorem \ref{informalintrotightness} (see Theorem \ref{PropTightGood} therein). In the course of that proof we utilize Lemmas \ref{PropSup}, \ref{PropSup2} and \ref{LemmaAP1}. These three key lemmas are proved in \ref{Section6}.
Section \ref{Section7} contains a restatement and proof of the second part (Brownian absolute continuity) of Theorem \ref{informalintrotightness} (see Theorem \ref{ACBB}).

Section \ref{Section8} finally pivots back to the log-gamma polymer. Section \ref{Section8.1} recalls some of the results in \cite{COSZ}, namely a Markov chain formulation for the image of the log-gamma polymer weight matrix under the geometric RSK correspondence. In Section \ref{Section8.2}, we prove that this Markov chain has the structure of a $(H,H^{RW})$-Gibbsian line ensemble. In particular, the polymer free energy arises as the lowest labeled curve. The complete proof of Theorem \ref{ThmTight} is given in Section \ref{Section8.3}. In Section \ref{appendixKPZscaling}, we explain the KPZ scaling theory for the log-gamma polymer and show that it is in agreement with Conjecture \ref{Conj} and the discussion that follows it in Section \ref{sec:asyBCDA} above.

\subsection*{Acknowledgments}
The authors wish to thank their anonymous referees for many instances of helpful feedback. I.C. is partially supported by the NSF grants DMS:1811143, DMS:1664650 and DMS:1937254, as well as a Packard Foundation Fellowship for Science and Engineering, a W. M. Keck Foundation Science and Engineering Grant, a Simons Foundation Fellowship and a Visiting Professorship at the Miller Institute for Basic Research in Science.. G.B. was partially supported by NSF grant DMS:1664650 as well. E.D. is partially supported by the Minerva Foundation Fellowship and NSF grant DMS:2054703.

%
\section{Gibbsian line ensembles}\label{Section4}
In this section, we introduce the notion of a discrete $(H,H^{RW})$-Gibbsian line ensemble and establish some of its properties.

%
\subsection{Discrete $(H,H^{RW})$-Gibbsian line ensembles}\label{Section4.1}
In this section, we introduce the notion of a discrete line ensemble and the $(H,H^{RW})$-Gibbs property. Our discussion will parallel that of \cite[Section 3.1]{Wu19}, which in turn goes back to \cite[Section 2.1]{CorHamK} and \cite[Section 3.1]{CD}.

\begin{definition}\label{YVec} For a finite set $J \subset \mathbb{Z}^2$ we let $Y(J)$ denote the space of functions $f: J \rightarrow \mathbb{R}$ with the Borel $\sigma$-algebra $\mathcal{D}$ coming from the natural identification of ${Y}(J)$ with $\mathbb{R}^{|J|}$. We think of an element of $Y(J)$ as a $|J|$-dimensional vector, whose coordinates are indexed by $J$. In particular, if $f(j) = x_j \in \mathbb{R}$ for $j \in J$, we will denote this vector by $(x_j: j \in J)$.
\end{definition}
If $a ,b \in \mathbb{Z}$ satisfy $a < b$, we let $\llbracket a,b \rrbracket$ denote the set $\{a, a+1, \dots, b\}$. We will use this $\llbracket\cdot,\cdot\rrbracket$ notation in general in this paper.
\begin{definition}\label{DefDLE}
Let $k_1, k_2, T_0, T_1 \in \mathbb{Z}$ with $k_1 \leq k_2$, $T_0 < T_1$  and denote $\Sigma = \llbracket k_1, k_2 \rrbracket$. A $\Sigma \times \llbracket T_0, T_1 \rrbracket$-{\em indexed discrete line ensemble $\mathfrak{L}$ }  is a random variable, defined on a probability space $(\Omega, \mathcal{B}, \mathbb{P})$, taking values in $Y(\Sigma \times \llbracket T_0, T_1 \rrbracket)$ as in Definition \ref{YVec}, such that $\mathfrak{L}$ is a $(\mathcal{B}, \mathcal{D})$-measurable function.
\end{definition}

The way we think of a $\Sigma \times \llbracket T_0, T_1 \rrbracket$-indexed discrete line ensemble $\mathfrak{L}$ is as a random $(k_2 - k_1 +1) \times (T_1 - T_0 + 1)$ matrix, whose rows are indexed by $\Sigma$ and whose columns are indexed by $\llbracket T_0, T_1 \rrbracket$. For $i \in \Sigma$ we let $L_i(\omega)$ denote the $i$-th row of this random matrix, and then $L_i$ is a $Y(\llbracket T_0, T_1 \rrbracket)$-valued random variable on $(\Omega, \mathcal{B}, \mathbb{P})$. Conversely, if we are given $k_2 - k_1 +1$ random $Y(\llbracket T_0, T_1 \rrbracket)$-valued random variables $L_{k_1}, \dots, L_{k_2}$, defined on the same probability space, then we can define a $\Sigma \times \llbracket T_0, T_1 \rrbracket$-indexed discrete line ensemble $\mathfrak{L}$ through $\mathfrak{L}(\omega)(i,j) = L_i(\omega)(j)$. Consequently, a $\Sigma \times \llbracket T_0, T_1 \rrbracket$-indexed discrete line ensemble $\mathfrak{L}$ is equivalent to having $k_2 - k_1 +1$ random $Y(\llbracket T_0, T_1 \rrbracket)$-valued random variables $L_{k_1}, \dots, L_{k_2}$ on the same probability space and depending on the context we will switch between these two formulations. For $i \in \Sigma$ and $j \in \llbracket T_0, T_1 \rrbracket$, we denote by $L_i(j): \Omega \rightarrow \mathbb{R}$ the function $L_i(j)(\omega) = L_i(\omega)(j)$ and observe that the latter are real random variables on $(\Omega, \mathcal{B}, \mathbb{P})$. If $A \subset \Sigma \times \llbracket T_0, T_1 \rrbracket$ we write $\mathfrak{L}\vert_{A} : \Omega \rightarrow Y(A)$ to denote the function $\mathfrak{L} \vert_A(\omega)(a) = \mathfrak{L}(\omega)(a)$ for $a \in A$. If $\llbracket a, b \rrbracket \subset \llbracket T_0, T_1 \rrbracket$ and $i \in \Sigma$, we denote the random vector $(L_i(a), \dots, L_i(b)) \in Y(\llbracket a, b\rrbracket)$ by $L_i\llbracket a, b \rrbracket$.

Observe that one can view an indexed set of real numbers $L(j)$ for $j \in \llbracket T_0, T_1 \rrbracket $ as a continuous curve by linearly interpolating the points $(j, L(j))$ -- see Figure \ref{LineEnsembleFig} for an illustration of such an interpolation for a discrete line ensemble. This allows us to define $ (\mathfrak{L}(\omega)) (i, s)$ for non-integer $s \in [T_0,T_1]$ by linear interpolation and to view discrete line ensembles as line ensembles in the sense of \cite{CorHamA}. Specifically, by linear interpolation we can extend $L_i(\omega)$ to a continuous curve on $[T_0, T_1]$ and in this way we can view it as a random variable on $(\Omega, \mathcal{B}, \mathbb{P})$ taking values in $(C[T_0,T_1], \mathcal{C})$  -- the space of continuous functions on $[T_0,T_1]$ with the uniform topology and Borel $\sigma$-algebra $\mathcal{C}$ (see e.g. Chapter 7 in \cite{Bill}). We will denote this random continuous curve by $L_i[T_0,T_1]$. We will often slightly abuse notation and suppress the $\omega$ from the above notation as one does for usual random variables, writing for example  $\{L_i(j) \in A\}$ in place of either $\{\omega \in \Omega: L_i(j)(\omega) \in A\}$ or $\{\omega \in \Omega: L_i(\omega)(j) \in A\}$ (notice that these sets are the same and in general the definitions are consistent so that the suppression of $\omega$ does not lead to any ambiguity).

\begin{definition}\label{DefBridge} Let $H^{RW}: \mathbb{R} \rightarrow \mathbb{R}$ be a continuous function and $\g(x) = e^{-H^{RW}(x)}$. We assume that $\g(x)$ is bounded and $\int_{\mathbb{R}} \g(x) dx = 1$. Let $Y_1, Y_2, \dots$ be i.i.d. random variables with density $\g(\cdot)$ and let $S^x_n =x + Y_1 + \cdots + Y_n$ denote the random walk with jumps $Y_m$ started from $x$. We denote by $\g^x_n(\cdot)$ the density of $S^x_n$ and note that
\begin{equation}\label{RWN}
\g^x_n(y) = \g^0_n(y - x) = \int_\mathbb{R} \cdots \int_{\mathbb{R}} \g(y_1) \cdots \g(y_{n-1}) \cdot \g(y - x - y_1 - \cdots - y_{n-1}) dy_1 \cdots dy_{n-1}.
\end{equation}
Given $x, y \in \mathbb{R}$ and $a, b\in \mathbb{Z}$ with $a < b$, we let $S(x,y; a, b) = \{ S_m(x,y;a,b) \}_{m = a}^b$ denote the process with the law of $\{S_m^x\}_{m = 0}^{b-a}$, conditioned so that $S_{b-a}^x = y$. We call this process an $H^{RW}$ {\em random walk bridge} between the points $(a,x)$ and $(b,y)$. Explicitly, viewing $S(x,y; a, b)$ as a random vector taking values in $Y(\llbracket a,b \rrbracket)$, we have that its distribution is given by the density
\begin{equation}\label{RWB}
\g(y_a, \dots, y_b; x,y; a,b) = \frac{\delta_x(y_a) \cdot \delta_y(y_b) \cdot \prod_{m = a+1}^{b} \g(y_m - y_{m-1}) }{\g_{b-a}^x(y)},
\end{equation}
 where we recall that $\delta_z$ is the Dirac delta measure at $z$. As before, we can also view $S(x,y; a, b)$ as a random continuous curve between the points $(a,x)$ and $(b,y)$ once we linearly interpolate the points $(m, S_m(x,y; a, b))$ for $m \in \llbracket a, b \rrbracket$.
\end{definition}

\begin{definition}\label{Pfree}
Let $H^{RW}$ be as in Definition \ref{DefBridge}. Fix $k_1 \leq k_2$, $a < b$ with $k_1, k_2, a,b \in \mathbb{Z}$ and two vectors $\vec{x}, \vec{y} \in \mathbb{R}^{k_2 - k_1 + 1}$. A $\llbracket k_1, k_2 \rrbracket \times \llbracket a, b\rrbracket$-indexed discrete line ensemble $L_{k_1}, \dots, L_{k_2}$ is called a {\em free $H^{RW}$ bridge line ensemble} with entrance data $\vec{x}$ and exit data $\vec{y}$ if its law $\mathbb{P}_{H^{RW}}^{k_1, k_2, a, b, \vec{x}, \vec{y}}$ is that of $k_2 - k_1 + 1$ independent $H^{RW}$ random walk bridges indexed by $\llbracket k_1, k_2 \rrbracket$ with the $i$-th bridge $L_{k_1 +i-1}$ being between the points $(a, x_i)$ and $(b, y_i)$ for $i \in \llbracket 1, k_2 -k_1 + 1\rrbracket$, see \eqref{RWB}. We write $\mathbb{E}_{H^{RW}}^{k_1, k_2, a, b, \vec{x}, \vec{y}}$ for the expectation with respect to this measure. When the parameters $k_1, k_2, a,b, \vec{x}, \vec{y}$ are clear from context we will drop them from the notation and simply write $\mathbb{P}_{H^{RW}}$ and $\mathbb{E}_{H^{RW}}$. Observe that the measure remains unchanged upon replacing $(k_1, k_2)$ with $(k_1 + m, k_2+m)$ for some $m \in \mathbb{Z}$, except for a reindexing of the $L_i$'s.

An {\em interaction Hamiltonian} $H$ is defined to be any continuous function $H: [-\infty, \infty) \rightarrow [0, \infty)$ such that $H(- \infty) = 0$. Suppose we are given an interaction Hamiltonian $H$ and two functions $f: \llbracket a, b \rrbracket \rightarrow \mathbb{R} \cup \{ \infty \}$ and $g : \llbracket a,b \rrbracket \rightarrow \mathbb{R} \cup \{-\infty\}$. We define the $\llbracket k_1, k_2 \rrbracket \times \llbracket a, b \rrbracket$-indexed $(H, H^{RW})$ line ensemble with entrance data $\vec{x}$ and exit data $\vec{y}$ and boundary data $(f,g)$ to be the law $\mathbb{P}_{H,H^{RW}}^{k_1, k_2, a,b, \vec{x}, \vec{y}, f, g}$ on $L_{k_1}, \dots, L_{k_2} : \llbracket a, b \rrbracket \rightarrow \mathbb{R}$ given in terms of the following Radon-Nikodym derivative (with respect to the free $H^{RW}$ bridge line ensemble $\mathbb{P}_{H^{RW}}^{k_1, k_2, a, b, \vec{x}, \vec{y}}$):
\begin{equation}\label{RND}
 \frac{d \mathbb{P}_{H,H^{RW}}^{k_1, k_2, a ,b, \vec{x}, \vec{y},f,g}}{d\mathbb{P}_{H^{RW}}^{k_1, k_2, a, b, \vec{x}, \vec{y}}} (L_{k_1}, \dots, L_{k_2}) = \frac{ W_{H}^{k_1, k_2, a ,b,f,g} (L_{k_1}, \dots, L_{k_2}) }{Z_{H,H^{RW}}^{k_1, k_2, a ,b, \vec{x}, \vec{y},f,g}}.
\end{equation}
Here, we call $L_{k_1 - 1} = f$ and $L_{k_2 + 1} = g$ and define the {\em Boltzmann weight}
\begin{equation}\label{WH}
W_{H}^{k_1, k_2, a ,b,f,g} (L_{k_1}, \dots, L_{k_2}) : = \exp \left( - \sum_{i = k_1 - 1}^{k_2}  \sum_{ m = a}^{b-1} { H} (L_{i + 1}(m + 1) - L_{i}(m)) \right),
\end{equation}
and the {\em normalizing constant}
\begin{equation}\label{AccProb}
Z_{H,H^{RW}}^{k_1, k_2, a ,b, \vec{x}, \vec{y},f,g} := \mathbb{E}_{H^{RW}}^{k_1, k_2, a, b, \vec{x}, \vec{y}} \Big[ W_{H}^{k_1, k_2, a ,b,f,g} (L_{k_1}, \dots, L_{k_2})  \Big],
\end{equation}
where we recall that on the right side in \eqref{AccProb} the vectors $L_{k_1}, \dots, L_{k_2}$ are distributed according to the measure $\mathbb{P}_{H^{RW}}^{k_1, k_2, a, b, \vec{x}, \vec{y}}$. Notice that, by our assumption on $f$ and $g$, we have that the argument of $H$ in \eqref{WH} is always in $[-\infty, \infty)$ and so $W_{ H}^{k_1, k_2, a ,b,f,g}  \in (0,1]$ almost surely, which implies that $Z_{H,H^{RW}}^{k_1, k_2, a ,b, \vec{x}, \vec{y},f,g} \in (0,1]$ and we can indeed divide by this quantity in (\ref{RND}). We write the expectation with respect to $\mathbb{P}_{H,H^{RW}}^{k_1, k_2, a ,b, \vec{x}, \vec{y},f,g}$ as $\mathbb{E}_{H,H^{RW}}^{k_1, k_2, a ,b, \vec{x}, \vec{y},f,g}$.
\end{definition}

The key definition of this section is the following (partial) $(H, H^{RW})$-Gibbs property. The term {\em (partial)} means that we do not allow resampling of highest labeled curve $L_{\topc}$. The full Gibbs property would allow for resampling that without changing the overall measure. This partial Gibbs property is nice because it is preserved under restricting curves labeled by $1,\ldots, \topc$ to curves labeled by $1,\ldots \topc'$ for $\topc'<\topc$. Since we will be entirely making use of this partial Gibbs property, we will drop the term {\em (partial)} throughout the paper, besides in the below definition.

\begin{definition}\label{DefLGGP}
Let $H^{RW}$ and $H$ be as in Definition \ref{Pfree}. Fix $\topc \geq 1$, two integers $T_0 < T_1$ and set $\Sigma = \llbracket 1,\topc \rrbracket$. Suppose that $\mathbb{P}$ is the probability distribution of a $\Sigma\times \llbracket T_0, T_1 \rrbracket$-indexed discrete line ensembles $\mathfrak{L} = (L_1, \dots, L_\topc)$ and adopt the convention that $L_0 = \infty$. We say that $\mathbb{P}$ satisfies the {\em(partial) $(H, H^{RW})$-Gibbs property} if the following holds. Fix any $k_1, k_2 \in \llbracket 1, \topc-1\rrbracket$ with $k_1 \leq k_2$ and $a,b \in \llbracket T_0, T_1 \rrbracket$ with $a < b$ and set $\mathbf{k} = \llbracket k_1, k_2 \rrbracket$. Then, we have the following distributional equality
$$\mbox{Law} \left( \mathfrak{L}\vert_{\mathbf{k}  \times \llbracket a, b \rrbracket} \mbox{ conditional on } \mathfrak{L}\vert_{\Sigma \times \llbracket T_0, T_1 \rrbracket \setminus \mathbf{k}  \times \llbracket a + 1, b - 1 \rrbracket} \right) =\mathbb{P}_{H, H^{RW}}^{k_1, k_2, a ,b, \vec{x}, \vec{y},f,g} .$$
Here, we have set $f = L_{k_1 - 1}$, $g = L_{k_2 + 1}$, $\vec{x} = (L_{k_1}(a), \dots L_{k_2}(a))$ and $\vec{y} = (L_{k_1}(b), \dots L_{k_2}(b))$.

Let us elaborate on what the above statement means. A $\Sigma \times \llbracket T_0, T_1 \rrbracket$-indexed line ensemble $\mathfrak{L}$ enjoys the $(H, H^{RW})$-Gibbs property if and only if for any $\mathbf{k}  = \llbracket k_1, k_2 \rrbracket \subset \llbracket 1, \topc-1 \rrbracket$ and $\llbracket a, b\rrbracket \subset \llbracket T_0, T_1 \rrbracket$ and any bounded Borel-measurable function $F$ from $Y( \mathbf{k}  \times \llbracket a, b \rrbracket)$ (here $Y$ is as in Definition \ref{YVec}) to $\mathbb{R}$, we have $\mathbb{P}$-almost surely
\begin{equation}\label{GibbsEq}
\mathbb{E} \Big[ F \big( \mathfrak{L}\vert_{\mathbf{k}  \times \llbracket a, b \rrbracket} \big)  \big{\vert} \mathcal{F}_{ext} (\mathbf{k}  \times \llbracket a + 1, b - 1 \rrbracket) \Big]  = \mathbb{E}_{H,H^{RW}}^{k_1, k_2, a ,b, \vec{x}, \vec{y},f,g}  \big[F( \tilde{\mathfrak{L}}) \big],
\end{equation}
where $\vec{x}, \vec{y}, f$ and $g$ are defined in the previous paragraph and the $\sigma$-algebra $\mathcal{F}_{ext}$ is defined as
\begin{equation}\label{GibbsCond}
 \mathcal{F}_{ext} (\mathbf{k}  \times \llbracket a + 1, b - 1 \rrbracket) := \sigma \left( L_i(s): (i,s) \in \Sigma \times \llbracket T_0, T_1 \rrbracket \setminus \mathbf{k}  \times  \llbracket a + 1, b - 1 \rrbracket \right).
\end{equation}
On the right side of (\ref{GibbsEq}), the variable $\tilde{\mathfrak{L}}$ has law $\mathbb{P}_{H,H^{RW}}^{k_1, k_2, a ,b, \vec{x}, \vec{y},f,g} $.
\end{definition}
\begin{remark} \label{CondGibbs} It is worth mentioning that the right side of (\ref{GibbsEq}) is measurable with respect to  $\mathcal{F}_{ext} (\mathbf{k}  \times \llbracket a + 1, b - 1 \rrbracket)$ and thus equation (\ref{GibbsEq}) makes sense. Indeed, we will show in Lemma \ref{ContinuousGibbsCond} that for any bounded measurable function $F$ on $Y(\llbracket k_1, k_2 \rrbracket \times \llbracket a, b\rrbracket)$ we have that $\mathbb{E}_{H,H^{RW}}^{k_1, k_2, a ,b, \vec{x}, \vec{y},f,g}  \big[F( \tilde{\mathfrak{L}}) \big]$ is a measurable function of $(\vec{x}, \vec{y}, f, g) \in Y(V_L) \times Y(V_R) \times Y(V_T) \times Y(V_B) $, where $V_L =  \llbracket k_1, k_2 \rrbracket  \times \{a \}$, $V_R = \llbracket k_1, k_2 \rrbracket  \times \{b \}$, $V_T =  \{k_1 - 1\} \times \llbracket a,b \rrbracket$ and $V_B = \{k_2 + 1 \} \times \llbracket a, b\rrbracket$ (the $L$, $R$, $T$ and $B$ stand for left, right, top and bottom, respectively). In particular, the right side of (\ref{GibbsEq}) is measurable with respect to $\sigma( L_i(j): (i,j) \in V_L \cup V_R \cup V_T \cup V_B) \subset  \mathcal{F}_{ext} (\mathbf{k}  \times \llbracket a + 1, b - 1 \rrbracket) $.
\end{remark}
\begin{remark} \label{restrict} From Definition \ref{DefLGGP}, it is clear that for $\topc^\prime \leq \topc$ and $\llbracket a,b \rrbracket \subset \llbracket T_0, T_1\rrbracket$ we have that the induced law on $L_i(j)$ for $(i,j) \in \llbracket 1, \topc^\prime \rrbracket \times \llbracket a,b \rrbracket $ from $\mathbb{P}$ also satisfies the $(H, H^{RW})$-Gibbs property, as an $\llbracket 1, \topc^\prime \rrbracket \times \llbracket a, b \rrbracket$-indexed line ensemble. Also, if $\topc = 1$, then the conditions in Definition \ref{DefLGGP} are void, meaning that any $\{1 \} \times \llbracket T_0, T_1\rrbracket$-indexed line ensemble satisfies the $(H, H^{RW})$-Gibbs property.
\end{remark}

In the remainder of this section, we present two foundational results, whose proofs are postponed to Section \ref{Section11.1}. Lemma \ref{S4AltGibbs} provides another formulation of the $(H, H^{RW})$-Gibbs property, and Lemma \ref{S4WeakGibbs} shows that the  $(H, H^{RW})$-Gibbs property survives weak limits.

\begin{lemma}\label{S4AltGibbs} Let $H^{RW}$ and $H$ be as in Definition \ref{Pfree}. Fix $\topc \geq 2$, two integers $T_0 < T_1$ and set $\Sigma = \llbracket 1,\topc \rrbracket$. Define sets $A =  \llbracket 1, \topc-1 \rrbracket \times \llbracket T_0 + 1, T_1 - 1\rrbracket$ and $B =\Sigma\times \llbracket T_0, T_1 \rrbracket \setminus A$.  Suppose that $\mathbb{P}$ is a probability distribution on  a $\Sigma\times \llbracket T_0, T_1 \rrbracket$-indexed discrete line ensemble $\mathfrak{L} = (L_1, \dots, L_\topc)$. Then, the following two statements are equivalent:
\begin{enumerate}
\item $\mathbb{P}$ satisfies the $(H, H^{RW})$-Gibbs property;
\item For any bounded continuous functions $f_{i,j}$ on $\mathbb{R}$ with $(i,j) \in \llbracket 1, \topc \rrbracket \times \llbracket T_0 , T_1 \rrbracket$, we have
\begin{equation}\label{S4towerFun}
\begin{split}
& \mathbb{E}\Bigg[\prod_{i = 1}^\topc \prod_{j = T_0}^{T_1} f_{i,j}(  L_i(j)  )\Bigg]  = \\
& \mathbb{E} \Bigg[ \prod_{(i,j) \in B}f_{i,j}( L_i(j) ) \cdot  \mathbb{E}_{H, H^{RW}}^{1, \topc-1, T_0, T_1, \vec{x}, \vec{y},\infty,L_\topc\llbracket T_0,T_1\rrbracket} \bigg[\prod_{(i,j) \in A} f_{i,j}( \tilde{L}_i(j)  ) \bigg]  \Bigg],
\end{split}
\end{equation}
where $\vec{x} =  (L_{1}(T_0), \dots, L_{\topc-1}(T_0))$, $\vec{y} = (L_{1}(T_1), \dots, L_{\topc-1}(T_1))$ and $\tilde{\mathfrak{L}} = (\tilde{L}_1, \dots, \tilde{L}_{\topc-1})$ is distributed according to
 $ \mathbb{P}_{H, H^{RW}}^{1, \topc-1, T_0, T_1, \vec{x}, \vec{y},\infty,L_\topc\llbracket T_0,T_1\rrbracket} $.
\end{enumerate}
Moreover, if $\vec{z} \in [-\infty, \infty)^{T_1 - T_0 +1}$ and $\vec{x}, \vec{y} \in \mathbb{R}^{\topc-1}$, then $\mathbb{P}_{H, H^{RW}}^{1, \topc-1, T_0, T_1, \vec{x}, \vec{y},\infty,\vec{z}}$  from Definition \ref{Pfree} satisfies the $(H,H^{RW})$-Gibbs property in the sense that (\ref{GibbsEq}) holds for all $1 \leq k_1 \leq k_2 \leq \topc-1$, $T_0 \leq a < b \leq T_1$ and bounded Borel-measurable $F$ on $Y( \llbracket k_1, k_2 \rrbracket \times \llbracket a, b \rrbracket)$.
\end{lemma}
Lemma \ref{S4AltGibbs} provides a sufficient condition for a line ensemble to satisfy the $(H,H^{RW})$-Gibbs property, which is easier to verify in practice. In addition, it shows that for any deterministic vectors $\vec{x}, \vec{y} \in \mathbb{R}^{\topc-1}$ and $\vec{z} \in  [-\infty, \infty)^{T_1 - T_0 +1}$ the distribution $\mathbb{P}_{H, H^{RW}}^{1, \topc-1, T_0, T_1, \vec{x}, \vec{y},\infty,\vec{z}}$ satisfies the $(H,H^{RW})$-Gibbs property. The latter fact for $\vec{z} \in Y(\{\topc\} \times \llbracket T_0, T_1 \rrbracket)$ (i.e. when all the entries of $\vec{z}$ are finite) follows from (\ref{S4towerFun}), and the content of the second part of the lemma is that one can replace some (or all) of the entries of $\vec{z}$ by $-\infty$, while still retaining the Gibbs property. Finally, the lemma shows a certain self-consistency of Definition \ref{DefLGGP}, that we explain here. If a $\Sigma \times   \llbracket T_0, T_1 \rrbracket$-indexed line ensemble $\mathfrak{L}$ satisfies (\ref{GibbsEq}) with $k_1 = 1, k_2 = \topc-1, a = T_0$ and $b= T_1$, then it satisfies (\ref{S4towerFun}) and so, by Lemma \ref{S4AltGibbs}, we conclude that (\ref{GibbsEq}) holds for any choice of $k_1, k_2, a,b$. In plain words, satisfying the conditional distribution equality of (\ref{GibbsEq}) for a rectangular box $K \times \llbracket a,b\rrbracket$ implies that it holds for all rectangular sub-boxes, and this consistency of the definition means that to prove that a line ensemble satisfies the $(H,H^{RW})$-Gibbs property it suffices to check it for the largest box, which is essentially the first statement of Lemma \ref{S4AltGibbs}.

\begin{lemma}\label{S4WeakGibbs} Let $H$ and $H^{RW}$ be as in Definition \ref{Pfree}. Fix $\topc \geq 2$, two integers $T_0 < T_1$ and set $\Sigma = \llbracket 1, \topc \rrbracket$. Suppose that $\mathbb{P}_n$ is a sequence of probability distributions on  $\Sigma\times \llbracket T_0, T_1 \rrbracket$-indexed discrete line ensembles, such that for each $n$ we have that $\mathbb{P}_n$ satisfies the $(H,H^{RW})$-Gibbs property. If $\mathbb{P}_n$ converges weakly to a measure $\mathbb{P}$, then $\mathbb{P}$ also satisfies the $(H,H^{RW})$-Gibbs property.
\end{lemma}
Let us explain the significance of Lemma \ref{S4WeakGibbs} for this paper. In Section \ref{Section8}, we will demonstrate a way to interpret the log-gamma polymer as a $(H,H^{RW})$-Gibbsian line ensemble $\mathfrak{L}$ for a certain choice of $H^{RW}$ and $H$. This will be done by taking a limit of a sequence of line ensembles $\mathfrak{L}^n$ that weakly converges to $\mathfrak{L}$ as $n \rightarrow \infty$. It will be easy to check that each of the $\mathfrak{L}^n$ satisfies the $(H,H^{RW})$-Gibbs property and then Lemma \ref{S4WeakGibbs} will imply that so does the limit $\mathfrak{L}$.

%
\subsection{Continuous grand monotone coupling lemma}\label{Section4.2} The goal of this section is to establish a continuous grand monotone coupling lemma for  $Y(\llbracket T_0, T_1 \rrbracket)$-valued random variables, whose laws are given by $\mathbb{P}_{H, H^{RW}}^{1, 1, T_0 ,T_1, x, y, \infty ,\vec{z}}$ as in Definition \ref{Pfree}. This result is given as Lemma \ref{MonCoup} and is one of the main results we prove about line ensembles, satisfying the $(H,H^{RW})$-Gibbs property.

The laws $\mathbb{P}_{H, H^{RW}}^{1, 1, T_0 ,T_1, x, y, \infty ,\vec{z}}$ arise as the marginal $L_1$ of a $\llbracket 1, k \rrbracket \times \llbracket T_0,T_1 \rrbracket$ -indexed line ensemble $\mathfrak{L}$, that satisfies the $(H,H^{RW})$-Gibbs property with $\vec{z} = L_2\llbracket T_0, T_1 \rrbracket$, given in terms of the second labeled curve $L_2$. In order to simplify our notation, we write $\mathbb{P}_{H,H^{RW}}^{T_0,T_1, x,y,\vec{z}}$ in place of $\mathbb{P}_{H,H^{RW}}^{1, 1, T_0 ,T_1, x, y, \infty ,\vec{z}}$. We will also denote $W_{H}^{1, 1, T_0 ,T_1,\infty ,\vec{z}}(\ell)$ by $W_H(T_0, T_1, \ell,\vec{z})$ and $Z_{H,H^{RW}}^{1, 1, T_0 ,T_1, x, y,\infty,\vec{z}}$ by $Z_{H,H^{RW}}(T_0, T_1, x,y,\vec{z})$. The random vector, whose law is $\mathbb{P}_{H, H^{RW}}^{T_0,T_1, x,y,\vec{z}}$, will typically be denoted by $\ell$ and as in Section \ref{Section4.1} we can think of it as a random continuous curve in $(C[T_0,T_1], \mathcal{C})$ by linearly interpolating the points $(i, \ell(i))$ for $i = T_0, \dots, T_1$ (see the discussion after Definition \ref{DefDLE}). We refer to $\mathbb{P}_{H,H^{RW}}^{T_0,T_1, a,b,\vec{z}}$-distributed random curves $\ell$ as $(H, H^{RW})$-curves.

The main result of this section is the following continuous grand monotone coupling. 

\begin{lemma}\label{MonCoup} Let $T \in \mathbb{N}$ satisfy $T \geq 2$ and assume that $H, H^{RW}$ are as in Definition \ref{Pfree}. Then, the following statements hold.
\begin{enumerate}[label={\Roman*)}]
 \item (Grand coupling) There exists a probability space $( \Omega^T, \mathcal{F}^T,\mathbb{P}^T)$ that supports random vectors $\ell^{T,x,y, \vec{z}} \in \mathbb{R}^{T}$ for all $x,y \in \mathbb{R}$ and $\vec{z} \in [-\infty, \infty)^{T}$, such that under $\mathbb{P}^T$ the random vector $\ell^{T,x,y, \vec{z}}$ has law $\mathbb{P}_{H, H^{RW}}^{ 1,T,x, y, \vec{z}}$ as in the beginning of this section.
 \item (Monotone coupling) Moreover, if we further suppose that $H$ and $H^{RW}$ are convex and $H$ is increasing, then for any fixed $x,y,x',y' \in \mathbb{R}$ with $x \leq x'$, and $y \leq y'$, and $\vec{z}, \vec{z}^{\hspace{0.5mm}\prime} \in [-\infty, \infty)^T$ with $z_i \leq z'_i$ for $i = 1, \dots, T$, we have $\mathbb{P}^{T}$-almost surely that $\ell^{T,x,y, \vec{z}}(i) \leq \ell^{T,x',y', \vec{z}^{\hspace{0.5mm}\prime}}(i)$ for $i = 1, \dots, T$.
 \item (Continuous coupling) If $T \geq 3$, the probability space $( \Omega^T, \mathcal{F}^T,\mathbb{P}^T)$ in part I can be taken to be $(0,1)^{T-2}$ with the Borel $\sigma$-algebra and Lebesgue measure. If $T = 2$, then $( \Omega^T, \mathcal{F}^T,\mathbb{P}^T)$ can be taken to be the space with a single point $\omega_0$, discrete $\sigma$-algebra and the measure that assigns unit mass to the point $\omega_0$. 

Furthermore, the construction in part I can be made so that the map $\Phi^T: \mathbb{R} \times \Omega^T \times \mathbb{R} \times [-\infty, \infty)^T \rightarrow \mathbb{R}^{T} \times  [-\infty, \infty)^T$, defined by
$$\Phi^T(x,\omega,y, \vec{z}) = (\ell^{T,x,y, \vec{z}}(\omega), \vec{z})$$ 
is a homeomorphism between the spaces $ \mathbb{R} \times \Omega^T \times \mathbb{R} \times [-\infty, \infty)^T $ and $ \mathbb{R}^{T} \times  [-\infty, \infty)^T$. In the last statement, $\Omega^T = (0,1)^{T-2}$ and we endow it with the subspace topology from $\mathbb{R}^{T-2}$ (and discrete topology if $T = 2$) and the space $\mathbb{R} \times \Omega^T \times \mathbb{R} \times [-\infty, \infty)^T $ has the product topology.
\end{enumerate}
\end{lemma}
\begin{remark}
We observe that when $\vec{z} = (- \infty)^T$, then $\mathbb{P}_{H, H^{RW}}^{ 1,T,x, y,\vec{z}}$ is precisely $\mathbb{P}_{H^{RW}}^{1, T, x,y}$ -- the law of a $H^{RW}$ random walk bridge between the points $(1,x)$ and $(T,y)$. In this case, Lemma \ref{MonCoup} generalizes the main theorem in \cite{Efron}.
\end{remark}
\begin{remark}\label{S2MonCR}
Let us briefly explain the main idea behind the proof of Lemma \ref{MonCoup}. We will construct the probability spaces $( \Omega^T, \mathcal{F}^T,\mathbb{P}^T)$ and the random variables $\ell^{T,x,y, \vec{z}}$ by induction on $T$, with base case $T = 2$ being trivial. Assuming the construction for $T = k$ and going to $T = k +1$, we will construct $\ell^{k+1,x,y, \vec{z}}$ by first constructing a probability space that supports the random variables $\ell^{k+1,x,y,\vec{z}}(k)$, whose law is the marginal of $\ell^{k+1,x,y, \vec{z}}$ in the $k$-th coordinate. These random variables will be built from the inverse conditional CDFs (depending on $k+1, x,y, \vec{z}$), applied to the {\em same} uniform random variable. A consequence of Lemma \ref{TechMonL}, which we prove below, is that the variables $\ell^{k+1,x,y,\vec{z}}(k)$ will be stochastically monotone in the variables $(x,y,\vec{z})$. Once we have the marginal laws $\ell^{k+1,x,y,\vec{z}}(k)$ constructed, we construct the rest of the entries $\ell^{k+1,x,y,\vec{z}}(\cdot )$ by appealing to our available by induction hypothesis construction of $\ell^{k,x,y,\vec{z}}$. The construction of $( \Omega^T, \mathcal{F}^T,\mathbb{P}^T)$ and $\ell^{k+1,x,y,\vec{z}}$ with the desired laws does not rely on the convexity of $H$ and $H^{RW}$. The convexity of $H$ and $H^{RW}$ is used in proving that the coupling is monotone, and at its core is Lemma \ref{TechMonL}, which as mentioned above allows us to monotonically couple the marginal laws $\ell^{k+1,x,y, \vec{z}}(k)$, whose densities are the functions $h^{c, \vec{z}}_{k}$ appearing in that lemma.
\end{remark}

Before we go to the proof of Lemma \ref{MonCoup} we establish the following preliminary result, which is central for the proof.
\begin{lemma}\label{TechMonL} Assume that $H^{RW}$ and $H$ are as in Definition \ref{Pfree}, are convex, and that $H$ is increasing. Given $c,y \in \mathbb{R}$, $n \in \mathbb{N}$ with $n\geq 2$ and $\vec{z} \in [-\infty, \infty)^{n+1}$, we define
\begin{equation}\label{S22DefH}
h^{c, \vec{z}}_{n }(y) = \int_{\mathbb{R}^{n-1}}  \prod_{i = 1}^{n}\g(x_i - x_{i-1})  e^{-H(z_{i+1} - x_i)} dx_1 \cdots dx_{n-1},
\end{equation}
where $x_0 = c$, $x_n = y$. We also set $h^c_1(y) = \g(y-c) e^{-H(z_{2} - y)}$ if $n = 1$.

Suppose that $a,b,c,d,s,t \in \mathbb{R}$ with $a \leq c$, $b \leq d$, $s \leq t$ and $\vec{u}, \vec{v} \in [-\infty,\infty)^{n+1}$ with $u_i \leq v_i$ for $i = 1, \dots, n$. Then, we have
\begin{equation}\label{conv1}
\frac{\int_{-\infty}^s h^{a,\vec{u}}_n(x) \g(b-x)dx }{\int_{-\infty}^s h^{c,\vec{v}}_n(x)\g (d-x)dx} \geq \frac{\int_{-\infty}^t h^{a,\vec{u}}_n(x) \g(b-x)dx }{\int_{-\infty}^t h^{c,\vec{v}}_n(x) \g(d-x)dx}.
\end{equation}
In particular, we have
\begin{equation}\label{conv2}
\frac{\int_{-\infty}^s h^{a,\vec{u}}_n(x) \g(b-x)dx }{\int_{-\infty}^\infty h^{a,\vec{u}}_n(x) \g(b-x)dx } \geq \frac{\int_{-\infty}^s h^{c,\vec{v}}_n(x)\g (d-x)dx}{\int_{-\infty}^\infty h^{c,\vec{v}}_n(x) \g(d-x)dx}.
\end{equation}
\end{lemma}
\begin{proof}

Observe first that all the integrals are well-defined by our assumption on $H,H^{RW}$. Also it is clear that we can obtain (\ref{conv2}) from (\ref{conv1}) upon taking the limit as $t \rightarrow \infty$. We thus focus on establishing (\ref{conv1}).

Note that if $F$ is convex, $\alpha \leq \beta$ and $\Delta \geq  0$, then
\begin{equation}\label{convexFun}
F(\alpha + \Delta) - F(\alpha) \leq F(\beta+\Delta) -F(\beta),
\end{equation}
which can be deduced from \cite[Theorem 24.1 and Corollary 24.2.1]{Rockafellar}. In particular, using that $H^{RW}$ is convex we deduce from (\ref{convexFun}) that for $\alpha,\beta,\gamma,\delta \in \mathbb{R}$, $\alpha \leq \beta$ and $\gamma \leq \delta$ we have 
\begin{equation}\label{S6GConv}
\begin{split}
&\g(\delta - \beta) \g(\gamma-\alpha) \geq \g(\delta-\alpha) \g(\gamma-\beta),
\end{split}
\end{equation}
and using the convexity of $H$, and the fact that it is increasing, we see for $\alpha, \beta \in \mathbb{R}$, $\gamma,\delta \in [-\infty, \infty)$ with $\alpha \leq \beta$ and $\gamma \leq \delta$
\begin{equation}\label{S6HConv}
\begin{split}
&e^{-H(\delta - \beta)} e^{-H(\gamma-\alpha)} \geq e^{-H(\delta-\alpha)} e^{-H(\gamma-\beta)}.
\end{split}
\end{equation}

Define the function $A(y; a,b, \vec{z}) := \log \big( \int_{-\infty}^y h_n^{a, \vec{z}}(x) \g(b-x)dx \big)$ and notice that, for fixed $a,b, \vec{z}$, the latter function is differentiable in $y$ and its derivative is given by
$$\frac{d A(y; a,b, \vec{z})}{dy} = \frac{ h_n^{a, \vec{z}}(y) \g(b-y)}{ \int_{-\infty}^y h_n^{a, \vec{z}}(x) \g(b-x)dx}.$$
Upon taking logarithms on both sides of (\ref{conv1}), we see that our goal is to show that
$$A(s; a,b, \vec{u}) - A(s; c,d, \vec{v}) \geq A(t;a,b, \vec{u} ) - A(t; c,d, \vec{v}),$$
and so it suffices to show that $\partial_y\big[A(y;a,b, \vec{u}) - A(y;c,d, \vec{v})\big] \leq 0$ or, equivalently,
$$ \frac{ h_n^{a, \vec{u}}(y) \g(b-y)}{ \int_{-\infty}^y h_n^{a, \vec{u}}(x) \g(b-x)dx} \leq \frac{ h_n^{c, \vec{v}}(y) \g(d-y)}{ \int_{-\infty}^y h_n^{c, \vec{v}}(x) \g(d-x)dx}$$
for all $y \in \mathbb{R}$. Cross-multiplying the above, we see that it suffices to show that for all $x,y \in \mathbb{R}$ with $x \leq y$ we have
\begin{equation*}
\begin{split}
h_n^{c, \vec{v}}(x) \g(d-x)h_n^{a, \vec{u}}(y) \g(b-y) \leq h_n^{a, \vec{u}}(x) \g(b-x)h_n^{c, \vec{v}}(y) \g(d-y).
\end{split}
\end{equation*}
In view of (\ref{S6GConv}), applied to $\alpha= x, \beta = y, \gamma = b, \delta = d$ (recall that $d \geq b$ by assumption), the last inequality would hold if for $x \geq y$ we have
\begin{equation}\label{IndIneq}
\begin{split}
h_n^{c, \vec{v}}(x) h_n^{a, \vec{u}}(y) \leq h_n^{a, \vec{u}}(x) h_n^{c, \vec{v}}(y).
\end{split}
\end{equation}

We proceed to prove (\ref{IndIneq}) by induction on $n \geq 1$. When $n = 1$, (\ref{IndIneq}) is equivalent to
\begin{equation*}
 \g(x-c) e^{-H(v_{2} - x)} \g(y -a) e^{-H(u_{2} - y)} \leq  \g(x-a) e^{-H(u_{2} - x)} \g(y -c) e^{-H(v_{2} - y)} ,
\end{equation*}
which holds from (\ref{S6GConv}), applied to $\alpha = x, \beta = y, \gamma = a, \delta = c$ (recall that $a \leq c$ by assumption), and (\ref{S6HConv}), applied to $\alpha = x, \beta = y, \gamma = u_2, \delta = v_2$ (recall that $u_2 \leq v_2$ by assumption).

Suppose that we now know (\ref{IndIneq}) for $n = k$ and we wish to show it for $n = k+1$. Using (\ref{S22DefH}), we can rewrite (\ref{IndIneq}) for $n = k+1$ as
\begin{equation*}
\begin{split}
&\int_{ \mathbb{R}} h_k^{c, \vec{v}_{k+1}}(x_k) \g(x - x_k) e^{-H(v_{k+2} - x )} dx_k \cdot \int_{ \mathbb{R}} h_k^{a, \vec{u}_{k+1}}(y_k) \g(y - y_k) e^{-H(u_{k+2} - y)} dy_k \leq \\
&\int_{ \mathbb{R}} h_k^{a, \vec{u}_{k+1}}(x_k) \g(x - x_k) e^{-H(u_{k+2} - x)} dx_k \cdot \int_{ \mathbb{R}} h_k^{c, \vec{v}_{k+1}}(y_k) \g(y - y_k) e^{-H(v_{k+2} -  y)} dy_k,
\end{split}
\end{equation*}
where $\vec{v}_{k+1} = (v_1, \dots, v_{k+1})$ and $\vec{u}_{k+1} = (u_1, \dots, u_{k+1})$. Applying (\ref{S6HConv}) with $\alpha = x, \beta = y, \gamma = u_{k+2}, \delta = v_{k+2}$, we see that the above inequality would hold if we can show
\begin{equation*}
\begin{split}
&\iint_{ \mathbb{R}^2} h_k^{c, \vec{v}_{k+1}}(x_k) \g(x - x_k)  h_k^{a, \vec{u}_{k+1}}(y_k) \g(y - y_k)  dx_k dy_k \leq \\
&\iint_{ \mathbb{R}^2} h_k^{a, \vec{u}_{k+1}}(x_k) \g(x - x_k)   h_k^{c, \vec{v}_{k+1}}(y_k) \g(y - y_k)  dx_kdy_k.
\end{split}
\end{equation*}
Splitting the above integrals over $\{x_k < y_k\}$ and $\{y_k < x_k\}$ and swapping the $x_k, y_k$ labels over the region $\{y_k < x_k\}$, we see that  (\ref{IndIneq}) for $n = k+1$ would follow if we can show
\begin{equation}\label{IndIneq2}
\begin{split}
\iint_{x_k < y_k } &\Big{(}h_k^{c, \vec{v}_{k+1}}(x_k) \g(x - x_k)   h_k^{a, \vec{u}_{k+1}}(y_k) \g(y - y_k)+  \\
& h_k^{c, \vec{v}_{k+1}}(y_k) \g(x -y_k)  h_k^{a, \vec{u}_{k+1}}(x_k) \g(y - x_k) \Big{)}dx_kdy_k \leq \\
\iint_{x_k < y_k } &\Big{(}h_k^{a, \vec{u}_{k+1}}(x_k) \g(x - x_k)   h_k^{c, \vec{v}_{k+1}}(y_k) \g(y - y_k)  + \\
&h_k^{a, \vec{u}_{k+1}}(y_k) \g(x - y_k)  h_k^{c, \vec{v}_{k+1}}(x_k) \g(y - x_k)\Big{)} dx_kdy_k.
\end{split}
\end{equation}

In order to prove (\ref{IndIneq2}), it suffices to show that the integrand on the left-hand side is pointwise dominated by that on the right-hand side. This is equivalent to proving that if $x_k < y_k$ we have
\begin{equation*}
\begin{split}
A B XY + CD ZW \leq CD XY + ABZW, \mbox{ or equivalently } (AB - CD) (XY - ZW) \leq 0, \mbox{ where }\\
\end{split}
\end{equation*}
$$A = h_k^{c, \vec{v}_{k+1}}(x_k) ,  B=  h_k^{a, \vec{u}_{k+1}}(y_k) ,  X = \g(x - x_k) , Y = \g(y - y_k) , $$
$$ C = h_k^{c, \vec{v}_{k+1}}(y_k)  , D=   h_k^{a, \vec{u}_{k+1}}(x_k), Z =  \g(x -y_k), W =  \g(y - x_k) .$$
We observe that 
\begin{equation*}
XY - ZW \geq 0 \iff  \g(x - x_k) \g(y - y_k)  \geq \g(x -y_k) \g(y - x_k) , 
\end{equation*}
and the latter holds from (\ref{S6GConv}), applied to $\alpha = x_k, \beta = y_k, \gamma = x, \delta = y$ (recall that $x \leq y$ and $x_k \leq y_k$ by assumption). In addition, we have by the induction hypothesis (\ref{IndIneq}) for $n = k$
\begin{equation*}
AB = h_k^{c, \vec{v}_{k+1}}(x_k)h_k^{a, \vec{u}_{k+1}}(y_k)  \leq h_k^{c, \vec{v}_{k+1}}(y_k)  h_k^{a, \vec{u}_{k+1}}(x_k) =  CD.
\end{equation*}
The last two inequalities give $AB \leq CD$ and $XY \geq ZW$ so that $(AB - CD) (XY - ZW) \leq 0$ or, equivalently, $A B XY + CD ZW \leq CD XY + ABZW$. This proves (\ref{IndIneq2}), and so (\ref{IndIneq}) holds for $n = k+1$. This concludes the induction step and we conclude (\ref{IndIneq}) for all $n \in \mathbb{N}$, which completes the proof of the lemma.
\end{proof}

\begin{proof}[Proof of Lemma \ref{MonCoup}] In Remark \ref{S2MonCR} we gave an outline of the main ideas of the proof of the lemma and here we present the details. For clarity, we split the proof into several steps. In the first step, we explain our construction of the probability space $( \Omega^T, \mathcal{F}^T,\mathbb{P}^T)$ and the random vectors $\ell^{T,x,y, \vec{z}} \in \mathbb{R}^{T}$ for all $x,y \in \mathbb{R}$ and $\vec{z} \in [-\infty, \infty)^{T}$ on this space. In the second step, we make two claims about the function $\Phi^T$ in the statement of the lemma and assuming the validity of these claims prove the parts I and II of the lemma. The two claims are proved in Steps 3 and 4, and in Step 5 we conclude the proof of part III of the lemma.\\

{\bf \raggedleft Step 1.} In this step, we explain how to construct the probability space $( \Omega^T, \mathcal{F}^T,\mathbb{P}^T)$ and the random vectors $\ell^{T,x,y, \vec{z}} \in \mathbb{R}^{T}$ by induction on $T \geq 2$. If $T = 2$ we take $\Omega^T$ to be a set with one point $\omega_0$, $\mathcal{F}^T$ to be the discrete $\sigma$-algebra, and $\mathbb{P}^T$ to be the unit mass at $\omega_0$. The random vectors $\ell^{T,x,y, \vec{z}}$ are then defined by $\ell^{T,x,y, \vec{z}} = (x,y)$ and clearly satisfy the conditions of the lemma.

Suppose we have constructed our desired space for $T = k \geq 2$. We now explain how the construction goes for $T = k+1$.

Notice that a $\mathbb{P}_{H, H^{RW}}^{ 1,T,x, y,\vec{z}}$-distributed random vector has density $h(y_1, \dots, y_T; x,y,\vec{z})$, given by
\begin{equation}\label{S24RWB}
 \frac{\delta_x(y_1) \cdot \delta_y(y_T) \cdot  \prod_{i = 2}^{T-1}\g(y_i - y_{i-1})  e^{-H(z_{i+1} - y_i)}\g(y_T-y_{T-1})}{\int_{\mathbb{R}^{T-2}}   \prod_{i = 2}^{T-1}\g(y_i - y_{i-1})  e^{-H(z_{i+1} - y_i)}\g(y_T-y_{T-1}) dy_2 \cdots dy_{T-1} },
\end{equation}
where $y_1 = x$ and $y_T = y$ in the denominator. Define for $\xi_1, \xi_2 \in \mathbb{R}$ and $\vec{z} \in \mathbb{R}^{T}$ the function
\begin{equation}\label{EqS2F}
F_{\xi_1,\xi_2}^{\vec{z},k} (s): = \frac{\int_{-\infty}^s h^{\xi_1,\vec{z}}_k(r) \g(\xi_2-r)dr }{\int_{-\infty}^\infty h^{\xi_1,\vec{z}}_k(r) \g(\xi_2-r)dr },
\end{equation}
where $h^{\xi_1,\vec{z}}_k$ is as in (\ref{S22DefH}). Then, $F_{\xi_1,\xi_2}^{\vec{z},k} (s)$ is precisely the marginal cumulative distribution function of $y_k$ under $h(y_1, \dots, y_T; x,y,\vec{z})$ with $ y_1 = \xi_1$ and $y_{T} = \xi_2$.

We now construct a probability space as follows. Let $\big((0,1), \mathcal{B}((0,1)),\lambda \big)$  be the space $(0,1)$ with the Borel $\sigma$-algebra and usual Lebesgue measure. This space supports the uniform random variable $U_{k-1}(r) = r$. We take the product space of the probability spaces $(\Omega^k, \mathcal{F}^k, \mathbb{P}^k)$ (we have this by the induction hypothesis) and the space $\big((0,1), \mathcal{B}((0,1)), \lambda \big)$. This will be our space $( \Omega^{k+1}, \mathcal{F}^{k+1},\mathbb{P}^{k+1})$. We next show how to construct $\ell^{k+1,x,y, \vec{z}}$ with the desired properties.

Given $x,y \in \mathbb{R}$, we construct $\ell^{k+1,x,y,\vec{z}}$ as follows:
\begin{enumerate}
\item set $\ell^{k+1,x,y,\vec{z}}(k+1) = y$;
\item set $\ell^{k+1,x,y,\vec{z}}(k) = \big[F^{\vec{z},k}_{x, y}\big]^{-1}(U_{k-1}) =:Y_k$;
\item set $\ell^{k+1,x,y,\vec{z}}(i) = \ell^{k,x,Y_k,\vec{z}_k}(i)$ for $i = 1, \dots, k-1$, where $\ell^{k,x,y, \vec{z}_k}$ are the random variables on $(\Omega^k, \mathcal{F}^k,\mathbb{P}^k )$, that have been constructed by induction hypothesis and $\vec{z}_k = (z_1, \dots, z_k)$.
\end{enumerate}
Notice that by assumption we know that $F^{\vec{z},k}_{x,y}$ is strictly increasing and defines a bijection between $(0,1)$ and $\mathbb{R}$. In particular, $\big[F^{\vec{z},k}_{a, b}\big]^{-1}(U_{k-1})$ is well-defined. This concludes the construction when $T = k +1$ and the general construction now proceeds by induction on $k$. \\

{\bf \raggedleft Step 2.} In this step we show that the construction of Step 1 satisfies parts I and II of the lemma. From our construction in Step 1, it is clear that $( \Omega^T, \mathcal{F}^T,\mathbb{P}^T)$ is nothing but $(0,1)^{T-2}$ with the Borel $\sigma$-algebra and Lebesgue measure (with the convention we had in the statement of the lemma for $T = 2$). We make the following two claims about the function $\Phi^T$ from the statement of the lemma. We claim that
\begin{equation}\label{S2Phi1}
\mbox{ the function $\Phi^T$ is a bijection between $\mathbb{R} \times \Omega^T \times \mathbb{R} \times [-\infty, \infty)^T $ and $ \mathbb{R}^{T} \times  [-\infty, \infty)^T$ }
\end{equation}
and if $[\Phi^T]^{-1}$ denotes its inverse function then for any sequence $(\vec{w}^n, \vec{z}^n)$ converging to $(\vec{w}^\infty, \vec{z}^\infty)$ in $ \mathbb{R}^{T} \times  [-\infty, \infty)^T$ we have that 
\begin{equation}\label{S2Phi2}
\lim_{n \rightarrow \infty} [\Phi^T]^{-1} (\vec{w}^n, \vec{z}^n) = [\Phi^T]^{-1}(\vec{w}^\infty, \vec{z}^\infty).
\end{equation}
The claims in (\ref{S2Phi1}) and (\ref{S2Phi2}) will be proved in Steps 3 and 4 below. Here, we assume their validity and conclude the proof of the first part of the lemma. \\

Notice that if we fix $\vec{z} \in [-\infty, \infty)^T$, then by (\ref{S2Phi1}) and (\ref{S2Phi2}) we know that the function $[\Phi^T]^{-1} (\vec{w}, \vec{z})$ defines a continuous bijection between $\mathbb{R}^T$ and $\mathbb{R} \times (0,1)^{T-2} \times \mathbb{R}$ as a function of $\vec{w}$. By the Invariance of domain theorem \cite[Theorem 36.5]{Munk2} we see that $\Phi^T$ is also continuous and, hence, by restriction for fixed $x,y \in \mathbb{R}$, we have that $\ell^{T,x,y, \vec{z}}(\omega)$ is a continuous function of $\omega$. In particular, all the vector-valued functions $\ell^{T,x,y, \vec{z}}$ we defined in Step 1 are random vectors.

We next check that $\ell^{T,x,y, \vec{z}}(\omega)$ has the law $\mathbb{P}_{H, H^{RW}}^{ 1,T,x, y, \vec{z}}$ as in the beginning of this section. We establish this by induction on $T \geq 2$, with base case $T = 2$ being trivially true. Assuming the result for $T = k$ and going to $T = k +1$, we note that what we have done in our construction from Step 1 is set $\ell^{k+1,x,y,\vec{z}}(k+1)  = y$ and sampled $\ell^{k+1,x,y,\vec{z}}(k) $ from the marginal law of $y_k$ under $h(y_1, \dots, y_{k+1}; x,y, \vec{z})$. Subsequently, we sampled a conditionally on $\ell^{k+1,x,y,\vec{z}}(k)$ independent  random vector $(\ell^{k+1,x,y,\vec{z}}(1), \dots, \ell^{k+1,x,y,\vec{z}}(k-1))$ whose law is the same as the marginal law of $(y_1, \dots, y_{k-1})$ under $h(y_1, \dots, y_{k+1}; a,b, \vec{z})$ conditioned on $y_k$. All of this implies that $\ell^{k+1,x,y,\vec{z}}$ indeed has law $\mathbb{P}_{H, H^{RW}}^{ 1 ,k+1,x, y,\vec{z}}$, and so the result holds for $T = k +1$ and by induction for all $T \geq 2$. The last three paragraphs establish part I of the lemma.

In the remainder of this step, we prove part II of the lemma. As before, we argue by induction on $T \geq 2$, with base case $T = 2$ being trivially true. Assuming the result for $T = k$, we verify the monotonicity when $T = k +1$.

Suppose that $x,y,x',y' \in \mathbb{R}$ with $x \leq x'$ and $y \leq y'$ and $\vec{z}, \vec{z}^{\hspace{0.5mm}\prime} \in [-\infty, \infty)^{k+1}$ with $z_i \leq z'_i$ for $i = 1, \dots, k+1$ are given. We want to show that $\ell^{k+1,x, y, \vec{z}}(i) \leq \ell^{k+1,x', y', \vec{z}^{\hspace{0.5mm}\prime}}(i)$ for $i = 1, \dots, k+1$. We know that $\ell^{k+1,x, y, \vec{z}}(k+1)= y \leq y' =\ell^{k+1,x', y', \vec{z}^{\hspace{0.5mm}\prime}}(k+1)$ by construction. In addition, by (\ref{conv2}) (here we use our assumption that $H$ and $H^{RW}$ are convex and $H$ is increasing) we know that
\begin{equation*}
\begin{split}
&F_{x,y}^{\vec{z},k} (s) \leq F_{x',y'}^{\vec{z}^{\hspace{0.5mm}\prime},k}(s),
\end{split}
\end{equation*}
which implies that
$$Y_k : = \ell^{k+1,x, y, \vec{z}}(k) = \big[F^{\vec{z},k}_{x, y}\big]^{-1}(U_{k-1})  \leq\big[F^{\vec{z}^{\hspace{0.5mm}\prime},k}_{x', y'}\big]^{-1}(U_{k-1})  =  \ell^{k+1,x' y', \vec{z}^{\hspace{0.5mm}\prime}}(k) =: Y_k'.$$
 Finally, we have by the induction hypothesis and our construction that
$$\ell^{k+1,x,y,\vec{z}}(i)  = \ell^{k, x, Y_k, \vec{z}_k}(i) \leq \ell^{k, x', Y_k', \vec{z}^{\hspace{0.5mm}\prime}_k}(i) = \ell^{k+1,x',y',\vec{z}^{\hspace{0.5mm}\prime}}(i) $$
 for $i = 1, \dots, k-1$, where $\vec{z}_k = (z_1, \dots, z_k)$ and $\vec{z}^{\hspace{0.5mm}\prime}_k = (z_1', \dots, z_k')$. This proves that the random vectors $\ell^{T,x,y,\vec{z}}$ satisfy the monotonicity conditions in part II when $T = k +1$, and the general result now follows by induction.\\

{\bf \raggedleft Step 3.} In this step, we prove (\ref{S2Phi1}). We define functions $\Psi^T:  \mathbb{R}^{T} \times  [-\infty, \infty)^T \rightarrow \mathbb{R} \times (0,1)^{T-2} \times \mathbb{R} \times [-\infty, \infty)^T $, with $(0,1)^0$ denoting the set containing the single point $\omega_0$, by induction on $T \geq 2$ as follows. If $T = 2$, then the map is given by 
$$\Psi^T(\vec{w}, \vec{z}) = (w_1, \omega_0, w_2, \vec{z}),$$
where $\vec{w} = (w_1, \dots, w_T)$. Assuming that we have defined $\Psi^{k}$, we define $\Psi^{k+1}$ for $k \geq 2$ through 
$$\Psi^{k+1}(\vec{w}, \vec{z}) = ( \Psi^k(\vec{w}_k, \vec{z}_k) \vert_{k-1} , F^{\vec{z},k}_{w_1, w_{k+1}}(w_k), w_{k+1}  ,\vec{z}),$$
where $F_{\xi_1,\xi_2}^{\vec{z},k}$ is as in (\ref{EqS2F}), $\vec{w}_k = (w_1, \dots, w_k)$, $\vec{z}_k = (z_1, \dots, z_k)$ and $\Psi^k(\vec{w}_k, \vec{z}_k) \vert_{k-1}$ denotes the image of $(\vec{w}_k, \vec{z}_k)$ under the (inductively constructed) $\Psi^k$ projected to the first $k-1$ coordinates. This gives the definition of $\Psi^{k+1}$ and the general construction proceeds by induction on $k$.

We claim that for all $T \geq 2$ the function $\Psi^T$ is a left and right inverse to the function $\Phi^T$. If true, the latter will clearly imply (\ref{S2Phi1}).

When $T = 2$, the latter is trivial, since the maps are basically the identity map, except that $\Psi^2$ inserts the coordinate $\omega_0$ between $w_1$ and $w_2$ in the vector $(\vec{w}, \vec{z})$, while $\Phi^2$ removes it. Suppose we know the result when $T = k \geq 2$ and wish to show it for $T = k +1$. We want to show that 
\begin{equation}\label{ASDF1}
\Phi^{k+1} \left( \Psi^{k+1}(\vec{w}, \vec{z}) \right) = (\vec{w}, \vec{z}) \mbox{ and }\Psi^{k+1} \left( \Phi^{k+1}( w_1, \vec{u}, w_{k+1}, \vec{z})\right)  = ( w_1, \vec{u},  w_{k+1}, \vec{z}),
\end{equation}
for all $\vec{w} \in \mathbb{R}^{k+1}$, $\vec{z} \in [-\infty, \infty)^{k+1}$ and $\vec{u} \in (0,1)^{k-1}$. 

Using the inductive definition of $\Psi^{k+1}$, we see that to show the first equality in (\ref{ASDF1}) it suffices to show that 
$$\Phi^{k+1} \left( ( \Psi^k(\vec{w}_k, \vec{z}_k) \vert_{k-1} , F^{\vec{z},k}_{w_1, w_{k+1}}(w_k), w_{k+1}  ,\vec{z}) \right) =  (\vec{w}, \vec{z}),$$
where $\vec{w}_k$ and $\vec{z}_k$ are as above. Using the inductive definition of $\Phi^{k+1}$ and the fact that 
$$[F^{\vec{z},k}_{w_1, w_{k+1}}]^{-1} (F^{\vec{z},k}_{w_1, w_{k+1}}(w_k)) = w_k,$$
we see that the latter is equivalent to 
$$\left( \Phi^{k}( \Psi^k(\vec{w}_k, \vec{z}_k) \vert_{k-1},w_k, \vec{z}_k) \vert_{k-1},  w_k, w_{k+1}, \vec{z} \right) =(\vec{w}, \vec{z}),$$
where $\Phi^k(\cdot) \vert_{k-1}$ denotes the projection to the first $k-1$ coordinates. Since by induction hypothesis we know that 
$$ \Phi^{k}( \Psi^k(\vec{w}_k, \vec{z}_k) \vert_{k-1}, w_k, \vec{z}_k) \vert_{k-1} = \Phi^{k}( \Psi^k(\vec{w}_{k}, \vec{z}_k)) \vert_{k-1} = (w_1, \dots, w_{k-1}),$$ 
we see that the left side of (\ref{ASDF1}) is satisfied for $k+1$ and the general result now follows by induction.

Similarly, using the inductive definition of $\Phi^{k+1}$, we see that to show the second equality in (\ref{ASDF1}) it suffices to show that 
$$\Psi^{k+1} \left( \Phi^{k}( w_1, \vec{u}_{k-2}, [F^{\vec{z},k}_{w_1, w_{k+1}}]^{-1}(u_{k-1}), \vec{z}_k)\vert_{k-1}, [F^{\vec{z},k}_{w_1, w_{k+1}}]^{-1}(u_{k-1}), w_{k+1} , \vec{z})\right)  = ( w_1, \vec{u},  w_{k+1}, \vec{z}),$$
where $\Phi^k(\cdot) \vert_{k-1}$ is as above, $\vec{u} = (u_1, \dots, u_{k-1})$ and 
$$\vec{u}_{k-2} = \begin{cases} (u_1, \dots, u_{k-2}) &\mbox{ if $k \geq 3$ and } \\
                                 \omega_0 &\mbox{ if $k = 2$.}
\end{cases}
$$
Using the inductive definition of $\Psi^{k+1}$ and the fact that 
$$F^{\vec{z},k}_{w_1, w_{k+1}} \left( [F^{\vec{z},k}_{w_1, w_{k+1}}]^{-1}(u_{k-1})\right)= u_{k-1},$$
we see that it suffices to show
$$\left(\Psi^{k} \left( \Phi^{k}( w_1,  \vec{u}_{k-2}, [F^{\vec{z},k}_{w_1, w_{k+1}}]^{-1}(u_{k-1}), \vec{z}_k)\right) \vert_{k-1}, u_{k-1}, w_{k+1} , \vec{z}\right)  = ( w_1, \vec{u},  w_{k+1}, \vec{z}),$$
where $\Psi^k(\cdot)\vert_{k-1}$ is as above. Since by induction hypothesis we know that 
$$ \Psi^{k} \left( \Phi^{k}( w_1, \vec{u}_{k-2}, [F^{\vec{z},k}_{w_1, w_{k+1}}]^{-1}(u_{k-1}), \vec{z}_k)\right) \vert_{k-1} = (w_1, \vec{u}_{k-2}),$$
we see that the right side of (\ref{ASDF1}) is satisfied for $k+1$ and the general result now follows by induction. \\

{\bf \raggedleft Step 4.} In this step, we prove (\ref{S2Phi2}). In view of our work in Step 3, we know that $[\Phi^T]^{-1}$ is nothing but the function $\Psi^T$ we constructed in that step. Thus, we want to prove that 
\begin{equation}\label{S2Phi3}
\lim_{n \rightarrow \infty} \Psi^T(\vec{w}^n, \vec{z}^n) =\Psi^T(\vec{w}^\infty, \vec{z}^\infty),
\end{equation}
provided that $(\vec{w}^n, \vec{z}^n) \rightarrow (\vec{w}^\infty, \vec{z}^\infty)$ in $ \mathbb{R}^{T} \times  [-\infty, \infty)^T$.

As usual, we prove (\ref{S2Phi3}) by induction on $T \geq 2$, with base case $T = 2$ being trivially true by the definition of $\Psi^T$. Assuming the result for $T = k$, we show that it holds when $T = k+1$. Using the inductive definition of $\Psi^{k+1}$, we see that to show (\ref{S2Phi3}) it suffices to prove that
$$\lim_{n \rightarrow \infty} ( \Psi^k(\vec{w}^n_k, \vec{z}^n_k) \vert_{k-1} , F^{\vec{z}^n,k}_{w_1^n, w^n_{k+1}}(w^n_k), w^n_{k+1}  ,\vec{z}^n) = ( \Psi^k(\vec{w}^\infty_k, \vec{z}^\infty_k) \vert_{k-1} , F^{\vec{z}^\infty,k}_{w_1^\infty, w^\infty_{k+1}}(w^\infty_k), w^\infty_{k+1}  ,\vec{z}^\infty),$$
where we recall that $\Psi^k(\cdot) \vert_{k-1}$ is the projection to the first $k-1$ coordinates. Our assumption that $(\vec{w}^n, \vec{z}^n) \rightarrow (\vec{w}^\infty, \vec{z}^\infty)$ and our induction hypothesis reduce the validity of the last statement to
\begin{equation}\label{S2Phi3.5}
 \lim_{n \rightarrow \infty}F^{\vec{z}^n,k}_{w_1^n, w^n_{k+1}}(w^n_k) = F^{\vec{z}^\infty,k}_{w_1^\infty, w^\infty_{k+1}}(w^\infty_k),
\end{equation}
which by the definition of $F_{\xi_1,\xi_2}^{\vec{z}} $ in (\ref{EqS2F}) is equivalent to
\begin{equation}\label{S2Phi4}
\lim_{n \rightarrow \infty}  \frac{\int_{-\infty}^{w^n_k} h^{w_1^n,\vec{z}^n}_{k}(r) \g(w_{k+1}^n-r)dr }{\int_{-\infty}^\infty h^{w_1^n,\vec{z}^n}_{k}(r) \g(w_{k+1}^n-r)dr }  = \frac{\int_{-\infty}^{w^\infty_k} h^{w_1^\infty,\vec{z}^\infty}_{k}(r) \g(w_{k+1}^\infty-r)dr }{\int_{-\infty}^\infty h^{w_1^\infty,\vec{z}^\infty}_{k}(r) \g(w_{k+1}^\infty-r)dr }.
\end{equation}
We show that the numerators and denominators on the left side of (\ref{S2Phi4}) converge to the numerator and the denominator on the right side, respectively. As the proofs are very similar, we only prove this statement for the numerators, which using the definition of $h^{c, \vec{z}}_{n }$ in (\ref{S22DefH}) boils down to
\begin{equation*}
\begin{split}
&\lim_{n \rightarrow \infty}  \int_{-\infty}^{w^n_k}  \int_{\mathbb{R}^{k-1}}  \prod_{i = 1}^{k+1}\g(x_i - x_{i-1})  \prod_{i = 1}^{k}e^{-H(z^n_{i+1} - x_i)} dx_1 \cdots dx_{k} = \\
&  \int_{-\infty}^{w^\infty_k}  \int_{\mathbb{R}^{k-1}}  \prod_{i = 1}^{k+1}\g(x_i - x_{i-1}) \prod_{i = 1}^{k} e^{-H(z^\infty_{i+1} - x_i)} dx_1 \cdots dx_{k} ,
\end{split}
\end{equation*}
where on the left $x_0 = w_1^n$ and $x_{k+1} = w^n_{k+1}$, while on the right $x_0 = w_1^\infty$ and $x_{k+1} = w^{k+1}_\infty$. Applying the change of variables $y^n_i = x_i +w_k^n$ in the top line above, and $y^\infty_i = x_i + w_k^\infty$ in the second, we see that it suffices to prove
\begin{equation}\label{S2Phi5}
\begin{split}
&\lim_{n \rightarrow \infty}  \int_{-\infty}^{0}  \int_{\mathbb{R}^{k-1}}  \prod_{i = 1}^{k+1}\g(y^n_i - y^n_{i-1})  \prod_{i = 1}^{k}e^{-H(z^n_{i+1} - y^n_i + w_k^n)} dy^n_1 \cdots dy^n_{k} = \\
&  \int_{-\infty}^{0}  \int_{\mathbb{R}^{k-1}}  \prod_{i = 1}^{k+1}\g(y^\infty_i - y^\infty_{i-1}) \prod_{i = 1}^{k} e^{-H(z^\infty_{i+1} - y^\infty_i + w_k^\infty)} dy^\infty_1 \cdots dy^\infty_{k},
\end{split}
\end{equation}
where $y^n_0 =w_1^n + w_k^n$ and $y^n_{k+1} = w_{k+1}^n + w_k^n$ for $n \in \mathbb{N} \cup \{\infty\}$.

Notice that by the continuity of $\g$ and $H$ we know that the integrands on the top line of (\ref{S2Phi5}) converge pointwise to the integrand on the bottom. The fact that the integrals also converge then follows from the Generalized dominated convergence theorem (see \cite[Theorem 4.17]{Royden}) with dominating functions
$$f_n (y_1, \dots, y_k) = \prod_{i = 1}^{k+1}\g(y_i - y_{i-1}), \mbox{ where $y_0 =w_1^n + w_k^n$ and $y_{k+1} = w_{k+1}^n + w_k^n$ } .$$
Let us elaborate on the last argument briefly. Since $H \geq 0$ by assumption, we know that $f_n$ dominate the integrands on the top line of (\ref{S2Phi5}). Furthermore, by the continuity of $\g$ we conclude that $f_n$ converge pointwise to $f_\infty$, which has the same form as $f_n$ with $y_0 =w_1^\infty + w_k^\infty$ and $y_{k+1} = w_{k+1}^\infty + w_k^\infty$. To conclude the application of the Generalized dominated convergence theorem, we need to show
$$\lim_{n \rightarrow \infty}  \int_{-\infty}^{0}  \int_{\mathbb{R}^{k-1}} \prod_{i = 1}^{k+1}\g(y^n_i - y^n_{i-1})  dy^n_1 \cdots dy^n_k =  \int_{-\infty}^{0}  \int_{\mathbb{R}^{k-1}} \prod_{i = 1}^{k+1}\g(y^\infty_i - y^\infty_{i-1})  dy^\infty_1 \cdots dy^\infty_k.$$
Changing variables $\tilde{y}_i = y_i^n - y_{i-1}^n$ for $i = 1, \dots, k$, we see that the latter is equivalent to 
\begin{equation*}
\begin{split}
\lim_{n \rightarrow \infty}  & \int_{\mathbb{R}^{k}} \prod_{i = 1}^{k}\g(\tilde{y}_i) \cdot \g \left(w_{k+1}^n - \sum_{i = 1}^k \tilde{y}_i - w_1^n \right)\cdot {\bf 1} \left\{ w_1^n + \sum_{i = 1}^k \tilde{y_i} \leq 0 \right\} d\tilde{y}_1 \cdots d\tilde{y}_k = \\
&    \int_{\mathbb{R}^{k}} \prod_{i = 1}^{k} \g(\tilde{y}_i) \cdot \g \left(w_{k+1}^\infty - \sum_{i = 1}^k \tilde{y}_i  - w_1^\infty \right) \cdot{\bf 1} \left\{ w_1^\infty + \sum_{i = 1}^k \tilde{y_i} \leq 0 \right\}  d\tilde{y}_1 \cdots d\tilde{y}_k.
\end{split}
\end{equation*}
The last equation is now a consequence of the dominated convergence theorem (see \cite[Theorem 4.16]{Royden}) with dominating function $\|\g \|_\infty \cdot \prod_{i = 1}^{k}\g(\tilde{y}_i) $ (note that the latter is integrable on $\mathbb{R}^k$ as $\g$ is bounded and integrable on $\mathbb{R}$). We thus conclude that the Generalized dominated convergence theorem is applicable and implies (\ref{S2Phi5}). This proves that the numerators in (\ref{S2Phi4}) converge and one can analogously show the same holds for the denominators, which concludes the proof of (\ref{S2Phi3}) when $T = k+1$. The general result now follows by induction.\\

{\bf \raggedleft Step 5.} In this step, we prove part III of the lemma. We already observed in Step 2 that our construction gives for $T \geq 3$ that $( \Omega^T, \mathcal{F}^T,\mathbb{P}^T)$ is $(0,1)^{T-2}$ with the Borel $\sigma$-algebra and Lebesgue measure, and when $T = 2$ it is the trivial probability space with a single point. Furthermore, we showed in Step 3 that $\Phi^T$ is a bijection and in Step 4 that its inverse $\Psi^T$ is continuous. Thus, we only need to prove that $\Phi^T$ is itself continuous. As usual, we establish this statement by induction on $T \geq 2$, with base case $T = 2$ being trivially true by the definition of $\Phi^T$. Assuming the result for $T = k$, we now prove it for $T = k +1$, which boils down to establishing
\begin{equation}\label{S2Phi6}
\begin{split}
&\lim_{n \rightarrow \infty}  \Phi^{k+1}(w_1^n, \vec{u}^n, w_{k+1}^n, \vec{z}^n) = \Phi^{k+1}(w_1^\infty, \vec{u}^\infty, w_{k+1}^\infty, \vec{z}^\infty), 
\end{split}
\end{equation}
provided that $ \lim_{n \rightarrow \infty} (w_1^n, \vec{u}^n, w_{k+1}^n, \vec{z}^n) = (w_1^\infty, \vec{u}^\infty, w_{k+1}^\infty, \vec{z}^\infty)$ in $\mathbb{R} \times (0,1)^{T-2} \times \mathbb{R} \times [-\infty, \infty)^T $. 

\begin{equation*}
\begin{split}
\lim_{n \rightarrow \infty} &\left( \Phi^{k}( w^n_1, \vec{u}^n_{k-2}, [F^{\vec{z}^n,k}_{w^n_1, w^n_{k+1}}]^{-1}(u^n_{k-1}), \vec{z}^n_k)\vert_{k-1}, [F^{\vec{z}^n,k}_{w^n_1, w^n_{k+1}}]^{-1}(u^n_{k-1}), w^n_{k+1} , \vec{z}^n)\right) = \\
&\left( \Phi^{k}( w^\infty_1, \vec{u}^\infty_{k-2}, [F^{\vec{z}^\infty,k}_{w^\infty_1, w^\infty_{k+1}}]^{-1}(u^\infty_{k-1}), \vec{z}^\infty_k)\vert_{k-1}, [F^{\vec{z}^\infty,k}_{w^\infty_1, w^\infty_{k+1}}]^{-1}(u^\infty_{k-1}), w^\infty_{k+1} , \vec{z}^\infty)\right).
\end{split}
\end{equation*}
Our assumption that $(w_1^n, \vec{u}^n, w_{k+1}^n, \vec{z}^n) \rightarrow (w_1^\infty, \vec{u}^\infty, w_{k+1}^\infty, \vec{z}^\infty)$ and our induction hypothesis reduce the validity of the last statement to 
\begin{equation}\label{S2Phi7}
\begin{split}
&\lim_{n \rightarrow \infty}  A_n = A_\infty, \mbox{where } A_n =   [F^{\vec{z}^n,k}_{w^n_1, w^n_{k+1}}]^{-1}(u^n_{k-1}) \mbox{ for $n \in \mathbb{N} \cup \{\infty \}$}.
\end{split}
\end{equation}

We now show that $A_n$ is a bounded sequence and all its subsequential limits are equal to $A_\infty$, which proves (\ref{S2Phi7}). Suppose first that $A_{n_m}$ converges to $\infty$ along some subsequence $n_m$. Then, by monotonicity of the function $F^{\vec{z}^n,k}_{w^n_1, w^n_{k+1}}$ we know that for any $a \in \mathbb{R}$
$$u^\infty_{k-1} = \lim_{m \rightarrow \infty} u^{n_m}_{k-1} =\lim_{m \rightarrow \infty}  F^{\vec{z}^{n_m},k}_{w^{n_m}_1, w^{n_m}_{k+1}}(A_{n_m}) \geq \limsup_{m \rightarrow \infty}  F^{\vec{z}^{n_m},k}_{w^{n_m}_1, w^{n_m}_{k+1}}(a) = F^{\vec{z}^{\infty},k}_{w^{\infty}_1, w^{\infty}_{k+1}}(a),$$ 
where in the last equality we used (\ref{S2Phi3.5}). Letting $a \rightarrow \infty$ above, we see $u^\infty_{k-1} \geq 1$, which is a contradiction as $u^\infty_{k-1}  \in (0,1)$.

Analogously, if $A_{n_m}$ converges to $-\infty$ along some subsequence $n_m$ then we have for any $a \in \mathbb{R}$
$$u^\infty_{k-1} = \lim_{m \rightarrow \infty} u^{n_m}_{k-1} =\lim_{m \rightarrow \infty}  F^{\vec{z}^{n_m},k}_{w^{n_m}_1, w^{n_m}_{k+1}}(A_{n_m}) \leq \liminf_{m \rightarrow \infty}  F^{\vec{z}^{n_m},k}_{w^{n_m}_1, w^{n_m}_{k+1}}(a) = F^{\vec{z}^{\infty},k}_{w^{\infty}_1, w^{\infty}_{k+1}}(a),$$ 
where in the last equality we used (\ref{S2Phi3.5}). Letting $a \rightarrow -\infty$, we see $u^\infty_{k-1} \leq 0$, which is a contradiction as $u^\infty_{k-1}  \in (0,1)$. 

Finally, suppose that $A_{n_m}$ converges to $B_\infty$ along some subsequence $n_m$. Then,
$$u^\infty_{k-1} = \lim_{m \rightarrow \infty} u^{n_m}_{k-1} =\lim_{m \rightarrow \infty}  F^{\vec{z}^{n_m},k}_{w^{n_m}_1, w^{n_m}_{k+1}}(A_{n_m})  = F^{\vec{z}^{\infty},k}_{w^{\infty}_1, w^{\infty}_{k+1}}(B_\infty),$$ 
where in the last equality we used (\ref{S2Phi3.5}). Applying $[F^{\vec{z}^\infty,k}_{w^\infty_1, w^\infty_{k+1}}]^{-1}$ to both sides, we see that $A_\infty = B_\infty$, as desired. This shows that all subsequential limits of $A_n$ are equal to $A_\infty$, which together with the boundedness of the sequence proves (\ref{S2Phi7}) and thus (\ref{S2Phi6}) holds. The general result now follows by induction.
\end{proof}

%
\subsection{Properties of $H^{RW}$-random walk bridges}\label{Section4.3} A special case of the measures $\mathbb{P}_{H,H^{RW}}^{T_0,T_1, x,y,\vec{z}}$ we considered in Section \ref{Section4.2} is when $\vec{z} = (-\infty)^{T_1 - T_0}$. In this case, our assumption that $H(-\infty) = 0$ implies that $\mathbb{P}_{H, H^{RW}}^{T_0,T_1, x,y,\vec{z}}$ becomes the law of a $H^{RW}$-random walk bridge between the points $(T_0, x)$ and $(T_1,y)$, see (\ref{RWB}). We denote such measures by $\mathbb{P}^{T_0,T_1,x,y}_{H^{RW}}$ and write $\mathbb{E}^{T_0,T_1,x,y}_{H^{RW}}$ for their expectation. In this section, we derive several results about the measures $\mathbb{P}^{T_0,T_1,x,y}_{H^{RW}}$, that rely on a strong coupling between random walk bridges and Brownian bridges from \cite{DW19} -- recalled here as Proposition \ref{KMT}. The advantage of this strong coupling is that it allows us to estimate the probabilities of various events under $\mathbb{P}^{T_0,T_1,x,y}_{H^{RW}}$ by comparing them to ones involving a Brownian bridge, for which exact computations are easier. In order to apply Proposition \ref{KMT}, we need to make several assumptions on the function $H^{RW}$, summarized in the following definition. 

\begin{definition}\label{AssHR} We make the following five assumptions on $H^{RW}$.

{\bf \raggedleft Assumption 1.} We assume that $H^{RW}: \mathbb{R} \rightarrow \mathbb{R}$ is a continuous convex function and $\g(x) = e^{-H^{RW}(x)}$. We assume that $\g(x)$ is bounded and $\int_{\mathbb{R}} \g(x) dx = 1$.\\

If $X$ is a random variable with density $g$, we denote
\begin{equation}\label{S2S1E1}
M_X(t) := \mathbb{E} \big[ e^{t X} \big], \hspace{3mm}  \phi_X(t) := \mathbb{E} \big[ e^{itX} \big], \hspace{3mm} \Lambda(t) :=\log M_X(t), \hspace{3mm} \mathcal{D}_\Lambda := \{ x: \Lambda(x) < \infty\}.
\end{equation}

{\bf \raggedleft Assumption 2.} We assume that $\mathcal{D}_\lambda$ contains an open neighborhood of $0$.

It is easy to see that $\mathcal{D}_{\Lambda}$ is a connected set and, hence, an interval. We denote $(A_{\Lambda}, B_{\Lambda})$ the interior of $\mathcal{D}_\Lambda$, where $A_{\Lambda} < 0$ and $B_{\Lambda} > 0$ by Assumption 2. We write $M_X(u)$ for all $u \in D = \{u \in \mathbb{C}: A_\Lambda< \Re(u) < B_\Lambda \}$ to mean the (unique) analytic extension of $M_X(x)$ to $D$, afforded by \cite[Lemma 2.1]{DW19}.\\

{\bf \raggedleft Assumption 3.} We assume that the function $\Lambda(\cdot)$ is lower semi-continuous on $\mathbb{R}$.\\

Under Assumptions 1,2 and 3 for a given $p \in \mathbb{R}$ the quantity $\Lambda''( (\Lambda')^{-1}(p))$ is well-defined -- see \cite[Section 2.1]{DW19}. For brevity, we write $\sigma_p^2 := \Lambda''( (\Lambda')^{-1}(p))$. \\

{\raggedleft \bf Assumption 4.} We assume that for every $B_\Lambda > t> s > A_\Lambda$ there exist constants $K(s,t)>0$ and $p(s,t) > 0$, such that $\left|M_X(z)\right| \leq \frac{K(s,t)}{(1 + |\Im(z)|)^{p(s,t)}}$, provided $s\leq \Re(z) \leq t$.\\

{\bf \raggedleft Assumption 5.} We suppose that there are constants $ D, d > 0$, such that at least one of the following statements holds
\begin{equation}\label{S2S1E2}
\mbox{1. }\g(x) \leq De^{-dx^2} \mbox{ for all $x \geq 0$ \qquad or\qquad  2. }\g(x) \leq De^{-dx^2} \mbox{ for all $x \leq 0$}.
\end{equation}
\end{definition}
\begin{remark} As mentioned before, our goal is to use a strong coupling result for random walk bridges from \cite{DW19}, which is a certain analogue of the classical {\em KMT}-coupling result from \cite{KMT1, KMT2}. Assumption 1 (except for the convexity part) is essentially ensuring that the random walk underlying the random walk bridge has a single interval of support. This allows one to condition on its endpoints and make sure that the corresponding bridge law in (\ref{RWB}) is well-defined. The convexity assumption on $H^{RW}$ is made so that we can apply our monotone coupling Lemma \ref{MonCoup}. Assumptions 2 and 4 are also somewhat natural as they were also needed in $KMT$'s original work \cite{KMT1,KMT2}. Assumptions 3 and 5 are a bit more technical and we refer to \cite[Section 2.3]{DW19} for a more detailed discussion of their significance. We also mention that any convex $H^{RW}$ such that $\liminf_{|x| \rightarrow \infty} x^{-2} H^{RW}(x) > 0$ satisfies the assumptions in Definition \ref{AssHR}.
\end{remark}

If $W_t$ denotes a standard one-dimensional Brownian motion and $\sigma > 0$, then the process
$$B^{\sigma}_t = \sigma (W_t - t W_1), \hspace{5mm} 0 \leq t \leq 1,$$
is called a {\em Brownian bridge (conditioned on $B_0 = 0, B_1 = 0$)  with variance $\sigma^2$.}  With the above notation we state the strong coupling result we use.
\begin{proposition}\label{KMT}
Suppose $H^{RW}$ satisfies the assumptions of Definition \ref{AssHR}. Let $p \in \mathbb{R}$ and $\sigma_p^2$ be as in that definition. There exist constants $0 < C, a, \tilde{\alpha} < \infty$ (depending on $p$ and $H^{RW}$), such that for every positive integer $T$, there is a probability space on which are defined a Brownian bridge $B^\sigma$ with variance $\sigma^2 = \sigma_p^2$ and a family of random curves $\ell^{(T,z)}$ on $[0,T]$, which is parameterized by $z \in \mathbb{R}$, such that $\ell^{(T,z)}$  has law $\mathbb{P}^{0,T,0,z}_{H^{RW}}$ and
\begin{equation}\label{KMTeq}
\mathbb{E}\big[ e^{a \Delta(T,z)} \big] \leq C e^{\tilde{\alpha}  (\log T)^2}e^{|z- p T|^2/T}, \mbox{ where $\Delta(T,z)=  \sup_{0 \leq t \leq T} \big|\hspace{-1mm} \sqrt{T} B^\sigma_{t/T} + \frac{t}{T}z - \ell^{(T,z)}(t) \big|.$}
\end{equation}
Here we recall that $\ell^{(T,z)}(s)$ was defined for non-integer $s$ by linear interpolation.
\end{proposition}
\begin{proof}
This result is a special case of \cite[Theorem 2.3]{DW19}. Indeed, Assumptions 1-5 in Definition \ref{AssHR} imply Assumptions C1-C5 in \cite[Section 2.1]{DW19}. In addition, Assumption C6 in \cite[Section 2.1]{DW19} is satisfied in view of \cite[Lemma 7.2]{DW19} and the fact that $H^{RW}$ is convex.
\end{proof}

In the lemmas below we consider measures $\mathbb{P}^{T_0,T_1,x,y}_{H^{RW}}$, with $H^{RW}$ satisfying the above assumptions. The random variable, whose law is $\mathbb{P}^{T_0,T_1,x,y}_{H^{RW}}$, will usually be denoted by $\ell$. We recall that this is a $Y(\llbracket T_0, T_1 \rrbracket)$-valued random variable and for $i \in \llbracket T_0, T_1\rrbracket$ we denote its $i$-th entry by $\ell(i)$. As explained in Section \ref{Section4.1}, we also think of $\ell$ as a random continuous curve on $[T_0, T_1]$, formed by linearly interpolating the points $(i, \ell(i))$ for $i \in \llbracket T_0, T_1\rrbracket$. \\

Below, we list several lemmas, whose proofs are postponed until Section \ref{Section11.2}. We provide a brief informal explanation of what each result says after it is stated. After we state all the lemmas we explain the underlying theme behind their proofs.
\begin{lemma}\label{LemmaHalfS4} Let $\ell$ have distribution $\mathbb{P}^{0,T,x,y}_{H^{RW}}$, with $H^{RW}$ satisfying the assumptions in Definition \ref{AssHR}. Let $M_1, M_2 \in \mathbb{R}$ and $p \in \mathbb{R}$ be given. Then, we can find $W_0 = W_0(p,M_2 - M_1) \in \mathbb{N}$, such that for $T \geq W_0$, $x \geq M_1 T^{1/2}$, $y \geq pT + M_2 T^{1/2}$ and $s \in [0,T]$ we have
\begin{equation}\label{halfEq1S4}
\mathbb{P}^{0,T,x,y}_{H^{RW}}\Big( \ell(s)  \geq \frac{T-s}{T} \cdot M_1 T^{1/2} + \frac{s}{T} \cdot \big(p T + M_2 T^{1/2}\big) - T^{1/4} \Big) \geq \frac{1}{3}.
\end{equation}
\end{lemma}
\begin{remark}
If $M_1, M_2 = 0$, then Lemma \ref{LemmaHalfS4} states that if a random walk bridge $\ell$ is started from $(0,x)$ and terminates at $(T,y)$, which are above the straight line of slope $p$, then at any given time $s \in [0,T]$ the probability that $\ell(s)$ goes a modest distance below the straight line of slope $p$ is upper bounded by $ 2/3$.
\end{remark}

\begin{lemma}\label{LemmaMinFreeS4} Let $\ell$ have distribution $\mathbb{P}^{0,T,0,y}_{H^{RW}}$, with $H^{RW}$ satisfying the assumptions in Definition \ref{AssHR}. Let $M > 0$, $p \in \mathbb{R}$ and $\epsilon > 0$ be given. Then, we can find $W_1=W_1(M,p, \epsilon) \in \mathbb{N}$ and $A=A(M,p, \epsilon) > 0$, such that for $T \geq W_1$, $ y \geq p T -  MT^{1/2}$ we have
\begin{equation}\label{minFree1S4}
\mathbb{P}^{0,T,0,y}_{H^{RW}}\Big( \inf_{s \in [ 0, T]}\big( \ell(s) -  ps \big) \leq -AT^{1/2} \Big) \leq \epsilon.
\end{equation}
\end{lemma}
\begin{remark} Roughly, Lemma \ref{LemmaMinFreeS4} states that if a random walk bridge $\ell$ is started from $(0,0)$ and terminates at $(T,y)$ with $(T,y)$ not significantly lower than the straight line of slope $p$, then the event that $\ell$ goes significantly below the straight line of slope $p$ is very unlikely.
\end{remark}

\begin{lemma}\label{LemmaTailS4} Let $\ell$ have distribution $\mathbb{P}^{0,T,x,y}_{H^{RW}}$, with $H^{RW}$ satisfying the assumptions in Definition \ref{AssHR}. Let $M_1,M_2 > 0$ and $p \in \mathbb{R}$ be given. Then, we can find $W_2 = W_2(M_1,M_2,p) \in \mathbb{N}$, such that for $T \geq W_2$, $ x \geq -M_1T^{1/2}$, $ y \geq pT -  M_1T^{1/2}$ and $\rho \in \{-1, 0 , 1\}$ we have
\begin{equation}\label{halfEq2S4}
\mathbb{P}^{0,T,x,y}_{H^{RW}}\bigg( \ell(\lfloor T/2\rfloor  + \rho)  \geq \frac{M_2T^{1/2} + p T}{2} - T^{1/4} \bigg) \geq (1/2) (1 - \Phi^{v}(M_1 + M_2) ),
\end{equation}
where $\Phi^{v}$ is the cumulative distribution function  of a Gaussian random variable with mean $0$ and variance $v = \sigma_p^2/4$.
\end{lemma}
\begin{remark} Lemma \ref{LemmaTailS4} states that  if a random walk bridge $\ell$ is started from $(0,x)$ and terminates at $(T,y)$, with these points not significantly lower than the straight line of slope $p$, then its mid-point would lie well above the straight line of slope $p$, at least with some quantifiably tiny probability.
\end{remark}

\begin{lemma}\label{LemmaAwayS4} Let $\ell$ have distribution $\mathbb{P}^{0,T,x,y}_{H^{RW}}$ with $H^{RW}$ satisfying the assumptions in Definition \ref{AssHR}. Let $p \in \mathbb{R}$ be given. Then we can find $W_3 = W_3(p) \in \mathbb{N}$ such that for $T \geq W_3$, $ x \geq T^{1/2}$, $ y \geq pT +  T^{1/2}$
\begin{equation}\label{awayS4}
\mathbb{P}^{0,T,x,y}_{H^{RW}}\Big( \inf_{s \in [0,T]} \big( \ell(s) -ps \big)+ T^{1/4} \geq 0 \Big) \geq \frac{1}{2} \left(1 - \exp\left(\frac{-2}{\sigma_p^2}\right)\right).
\end{equation}
\end{lemma}
\begin{remark}
Lemma \ref{LemmaAwayS4} states that  if a random walk bridge $\ell$ is started from $(0,x)$ and terminates at $(T,y)$ with $(0,x)$ and $(T,y)$ well above the line of slope $p$ then at least with some positive probability $\ell$ will not fall significantly below the line of slope $p$.
\end{remark}

We need the following definition for our next result. For a function $f \in C[a,b]$ we define its {\em modulus of continuity} by
\begin{equation}\label{MOCS4}
w(f,\delta) = \sup_{\substack{x,y \in [a,b]\\ |x-y| \leq \delta}} |f(x) - f(y)|.
\end{equation}
\begin{lemma}\label{MOCLemmaS4} Let $\ell$ have distribution $\mathbb{P}^{0,T,0,y}_{H^{RW}}$, with $H^{RW}$ satisfying the assumptions in Definition \ref{AssHR}. Let $M > 0$ and $p \in \mathbb{R}$ be given. For each positive $\epsilon$ and $\eta$, there exist a $\delta > 0$ and $W_4 = W_4(M, p, \epsilon, \eta) \in \mathbb{N}$, such that  for $T \geq W_4$ and $|y - pT| \leq MT^{1/2}$ we have
\begin{equation}\label{MOCeqS4}
\mathbb{P}^{0,T,0,y}_{H^{RW}}\Big( w\big({f^\ell},\delta\big) \geq \epsilon \Big) \leq \eta,
\end{equation}
where $f^\ell(u) = T^{-1/2}\big(\ell(uT) - puT\big)$  for $u \in [0,1]$.
\end{lemma}
\begin{remark}
Lemma \ref{MOCLemmaS4} states that if $\ell$ is a random walk bridge that is started from $(0,0)$ and terminates at $(T,y)$, with $y$ close to $pT$ (i.e. with well-behaved endpoints), then the modulus of continuity of $\ell$ is also well-behaved with high probability.
\end{remark}

The above five lemmas are proved in a similar fashion. For the first four lemmas, one observes that the event, whose probability is being estimated, is monotone in $\ell$. This allows, by Lemma \ref{MonCoup}, to replace $x,y$ in the statements of the lemmas with the extreme values of the ranges specified in each. Once the choice of $x$ and $y$ is fixed, one can use our strong coupling result of $\ell$ and a Brownian bridge to reduce each of the lemmas to an analogous one with $\ell$ replaced by a Brownian bridge with some prescribed variance. The latter statements are then easily confirmed as one has exact formulas for all of the probabilities in the above lemmas whenever $\ell$ is replaced by a Brownian bridge.\\

We end this section with the following result for $(H,H^{RW})$-random curves. Its proof will also be provided in Section \ref{Section11.2}.
\begin{lemma}\label{NoExplodeS4}Let $H$ be as in Definition \ref{Pfree} and suppose it is convex,  increasing and  $\lim_{x \rightarrow \infty} H(x) = \infty$. For such a choice of $H$, we let $\ell$ have law $\mathbb{P}_{H,H^{RW}}^{0, 2T, x ,y, \vec{z}}$ as in Section \ref{Section4.2}, where $H^{RW}$ satisfies the assumptions in Definition \ref{AssHR}. Let $M, \epsilon > 0$ and $p \in \mathbb{R}$ be given. Then, we can find a constant $W_5 = W_5(M,p,\epsilon) \in \mathbb{N}$ so that the following holds. If $T \geq W_5$, $\vec{z} \in [-\infty, \infty)^{2T+1}$ with $z_{ T+1} \geq pT + 2MT^{1/2}$ and $x,y \in \mathbb{R}$ with $x \geq -MT^{1/2}$ and $y \geq  - MT^{1/2} + 2pT$, then we have
\begin{equation}\label{NE1v2S4}
\mathbb{P}_{H,H^{RW}}^{0,2T,x,y, \vec{z}} \big( \ell(T) \leq pT+ MT^{1/2} \big) \leq \epsilon.
\end{equation}
\end{lemma}
\begin{remark}What Lemma \ref{NoExplodeS4} states is that if $(T,z_{T+1})$ is well-above the straight segment of slope $p$, then $(T,\ell(T))$ is also well-above the straight segment of slope $p$ with very high probability.
\end{remark}

It is not too surprising that if $\vec{z}$ is high then $\ell$ is also forced to be high, because the definition of $\mathbb{P}_{H,H^{RW}}^{0, 2T, x ,y, \vec{z}}$ exponentially penalizes $\ell$'s that go below $\vec{z}$. What is somewhat surprising is that it is enough for only $z_{T+1}$ (a single entry of $\vec{z}$) to be high to force $\ell(T)$ to become high with it. The reason it works out this way is that we are performing a type of diffusive scaling to the curve $\ell$ and $z_{T+1}$ is high in the order $T^{1/2}$ of this scaling. While the curves are living on order $T^{1/2}$ the interaction Hamiltonian $H$ is not scaled with $T$ at all. In particular, as $T$ becomes large the $H$ interaction on a $T^{1/2}$ scale starts to look like the indicator function that $\ell$ lies above $\vec{z}$. This, in particular, makes the proof of Lemma \ref{NoExplodeS4} very easy compared to its analogue \cite[Proposiition 7.6]{CorHamK} in the context of the KPZ line ensemble, where the interaction Hamiltonian is also influenced by the diffuse scaling.

%
\section{Tightness of simple $(H,H^{RW})$-Gibbsian line ensembles}\label{Section5}
We call a $\{1,2\} \times \llbracket T_0, T_1 \rrbracket$-indexed line ensemble {\em simple} (i.e., a simple line ensemble has only two curves, indexed by $1$ and $2$). In this section, we describe a general framework, that can be used to prove tightness for the top curve of a sequence of simple line ensembles, that satisfy the $(H,H^{RW})$-Gibbs property. We start by summarizing our assumptions on $H$ in the following definition.
\begin{definition}\label{AssH} We let $H: [-\infty, \infty) \rightarrow [0,\infty)$ be continuous, increasing, and convex. We assume that $\lim_{x \rightarrow \infty} H(x) = \infty$, and $\lim_{x \rightarrow \infty} x^2 H(-x) = 0$.
\end{definition}

In this section, we also require from Definition \ref{def:intro} the notion of an $(\alpha, p, T)$-good sequence. The main technical result of this section is as follows.
\begin{theorem}\label{PropTightGood}
Fix $\alpha, r > 0$, and $p \in \mathbb{R}$, and let $\big\{\mathfrak{L}^N = (L^N_1, L^N_2)\big\}_{N=1}^{\infty}$  be a sequence of (random) simple $\llbracket 1, 2 \rrbracket \times \llbracket -T_N, T_N \rrbracket$-indexed line ensembles that is $(\alpha, p, r+3)$--good.  For $N \geq N_0(\alpha, p, r + 3)$ (where $N_0(\alpha, p, r + 3)$  is afforded by Definition \ref{def:intro}, owing to our assumption of being $(\alpha, p, r+3)$--good), let $f_N(s)$ be given by
$$f_N(s) = N^{-\alpha/2}\big(L_1^N( sN^{\alpha}  ) - p s N^{\alpha}\big),$$
whenever $sN^{\alpha}$ is an integer. For all other values in $s \in [-r, r]$, we define $f_N$ by linear interpolation. Let $\mathbb{P}_N$ denote the law of $f_N$ as a random variable in $(C[-r,r], \mathcal{C})$. Then, the sequence of distributions $\mathbb{P}_N$ is tight.
\end{theorem}
\begin{remark}
Roughly, Theorem \ref{PropTightGood} states that if a process can be viewed as the top curve of a $(H, H^{RW})$-Gibbsian discrete line ensemble and under some shift and diffusive scaling the process's one-point marginals are tight, then under the same shift and scaling the trajectory of the process is tight in the space of continuous curves.
\end{remark}

The goal of this section is to prove Theorem \ref{PropTightGood} and for the remainder we assume that
\begin{equation}\label{eqalphagood}
\big\{\mathfrak{L}^N = (L^N_1,L^N_2)\big\}_{N=1}^{\infty} \textrm{ is an } (\alpha,p,r + 3)\textrm{--good sequence of simple line ensembles},
\end{equation}
defined on a probability space with measure $\mathbb{P}$. The main technical result we will require is contained in Proposition \ref{PropMain} below, and its proof is the content of Section \ref{Section5.1}. The proof of Theorem \ref{PropTightGood} is given in Section \ref{Section5.2} and relies on Proposition \ref{PropMain}, and Lemma \ref{MOCLemmaS4}.
%
\subsection{Bounds on $Z_{H,H^{RW}}$}\label{Section5.1}
 The main result in this section is presented as Proposition \ref{PropMain} below. In it, and the lemmas after it, we assume that \eqref{eqalphagood} holds. In other words, for fixed $\alpha, r>0$ and $p\in \mathbb{R}$, we have that $\big\{(L^N_1,L^N_2)\big\}_{N=1}^{\infty}$ is an $N$-indexed $(\alpha,p,r + 3)$--good sequence of line ensembles. We will also adopt the notation
\begin{equation}\label{eqsts}
t^{\pm}_1 =\lfloor \pm (r+1) N^{\alpha} \rfloor,\quad t^{\pm}_2 = \lfloor \pm (r+2)N^{\alpha} \rfloor,\quad \textrm{and } t_3^{\pm} = \lfloor \pm (r+3)N^{\alpha} \rfloor.
\end{equation}
The assumption that $(L^N_1,L^N_2)$ is $(\alpha,p,r + 3)$-good implies that there exists a function $R_0: (0,\infty) \rightarrow (0, \infty)$, such that for any $\epsilon > 0$ we have
\begin{equation}\label{TailFun}
\sup_{N \geq N_0(\alpha, p,r+3)} \mathbb{P} \left( \max_{i \in \{1,2,3\}, j \in \{+, -\}} N^{-\alpha/2} \left|L_1^N(t_i^j) - pt_i^j \right| \geq R_0(\epsilon) \right) < \epsilon.
\end{equation}
\begin{proposition}\label{PropMain} For any $\epsilon > 0$ and any $(\alpha,p,r + 3)$-good sequence of simple line ensembles $\big\{ (L^N_1,L^N_2)\big\}_{N=1}^{\infty}$
there exist $\delta > 0$ and $N_1$ (depending on $\epsilon$ as well as $H, H^{RW}, \alpha, p, r, N_0$, and the function $R_0$ in (\ref{TailFun})), such that for all $N \geq N_1$
we have
$$\mathbb{P}\Big(Z_{H,H^{RW}}\big( t_1^-,t_1^+,L_1^N(t_1^-),L_1^N(t_1^+), L_2\llbracket t_1^-, t_1^+\rrbracket\big) < \delta\Big) < \epsilon,$$
where $Z_{H,H^{RW}}$ is the normalizing constant in (\ref{AccProb}) (we recall that this alternative notation for $ Z_{H,H^{RW}}$ was introduced in Section \ref{Section4.2}).
\end{proposition}
\begin{remark}\label{ZMeas} In Lemma \ref{ContinuousGibbsCond}, we show that $Z_{H,H^{RW}}\big( T_0, T_1,x,y,\vec{z}\big)$ is a continuous function of $(x,y, \vec{z}) \in \mathbb{R} \times \mathbb{R} \times Y(\llbracket T_0, T_1 \rrbracket)$  bounded above by $1$ and below by $0$. In particular, the event in Proposition \ref{PropMain} is measurable and its probability well-defined.
\end{remark}
\begin{remark} The inequality in Proposition \ref{PropMain} implies that the Radon-Nikodym derivative of $L^N_1$ with respect to a suitable $H^{RW}$ random walk bridge is lower bounded. This will ultimately allow us to show that the sequence $f_N $ is tight by comparing it with a sequence of random walk bridges, for which tightness is easier to establish.
\end{remark}

The general strategy we use to prove Proposition \ref{PropMain} is inspired by the proof of Proposition 6.5 in \cite{CorHamK}. We begin by stating three key lemmas that will be required. Their proofs are postponed to Section \ref{Section6}. All constants in the statements below will depend implicitly on $\alpha$, $r$, $p$, $N_0$, $H$, $H^{RW}$, and the function $R_0$ from (\ref{TailFun}), which are fixed throughout. We will not list this dependence explicitly.

Lemma \ref{PropSup} controls the deviation of the curve $L^N_1(s)$ from the line $ps$ in the scale $N^{\alpha/2}$.
\begin{lemma}\label{PropSup} For each $\epsilon > 0$ there exist $R_1=R_1(\epsilon) > 0$ and $N_2= N_2(\epsilon)$, such that for $N \geq N_2$
$$\mathbb{P}\Big( \sup_{s \in [ -t_3^{-}, t_3^{+}] }\big| L^N_1(s) - p s \big| \geq  R_1N^{\alpha/2} \Big) < \epsilon.$$
\end{lemma}

Lemma \ref{PropSup2} controls the upper deviation of the curve $L^N_2(s)$ from the line $ps$ in the scale $N^{\alpha/2}$.
\begin{lemma}\label{PropSup2} For each $\epsilon > 0$ there exist $R_2=R_2( \epsilon) > 0$ and $N_3=N_3(\epsilon)$, such that for $N \geq N_3$
$$\mathbb{P}\Big( \sup_{s \in [ t_2^-, t_2^+ ]}\big(L^N_2(s) - p s \big) \geq  R_2N^{\alpha/2} \Big) < \epsilon.$$
\end{lemma}

Lemma \ref{LemmaAP1} states that if one is given a bottom bounding curve $\ell_{bot} \in Y(\llbracket t_2^-,t_2^+ \rrbracket)$ which is not too high and if $x,y \in \mathbb{R}$ are not too low, then under $\mathbb{P}^{t^-_2, t^+_2, x,y, \ell_{bot} }_{H,H^{RW}}$ the random variable $Z_{H,H^{RW}}\big(  t_1^-, t_1^+, \ell(t_1^-) ,\ell(t^+_1), \ell_{bot}\llbracket t_1^-, t_1^+\rrbracket\big)$ is tiny with very small probability.
\begin{lemma}\label{LemmaAP1} Fix $M_1, M_2 > 0$, $\ell_{bot} \in Y(\llbracket t_2^-,t_2^+ \rrbracket)$, and $x,y \in \mathbb{R}$, such that
\begin{enumerate}
\item $\sup_{s \in [ t_2^-,t_2^+]}\big(\ell_{bot}(s)  - ps \big)  \leq M_2 (t_2^+ - t_2^-)^{1/2}$,
\item  $ x \geq  pt^-_2- M_1 (t_2^+ - t_2^-)^{1/2},$
\item $ y \geq  p t^+_2- M_1(t_2^+ - t_2^-)^{1/2}.$
\end{enumerate}
Define the constants $g$ and $h$ (depending on $ M_1, M_2$) via
$$g =  \frac{1}{4} \left( 1 - \exp \left( \frac{-2}{\sigma_p^2} \right) \right) \mbox{ and } h = (1/18) \cdot \Big(1 - \Phi^{v}\big(10(2 +r)^2 (M_1 + M_2 + 10) \big)\Big),$$
where $\sigma_p$ is specified in terms of $H^{RW}$ as in Definition \ref{S2S1E1}, and $\Phi^v$ is the cumulative distribution function of a Gaussian random variable with mean zero and variance $v = \sigma_p^2/4$.

Then, there exists $N_4 = N_4(M_1,M_2) \in \mathbb{N}$, such that for any $\tilde{\epsilon}  > 0$ and $N \geq N_4$ we have
\begin{equation}\label{eqn60}
\mathbb{P}^{t^-_2, t^+_2, x,y, \ell_{bot} }_{H,H^{RW}} \Big( {Z_{H,H^{RW}}}\big(  t_1^-, t_1^+, \ell(t_1^-) ,\ell(t^+_1), \ell_{bot}\llbracket t_1^-, t_1^+\rrbracket\big) \leq  gh \tilde{\epsilon}   \Big) \leq \tilde{\epsilon},
\end{equation}
where $\ell_{bot}\llbracket t_1^-, t_1^+\rrbracket$ is the vector in $Y(\llbracket t_1^-, t_1^+\rrbracket)$, whose coordinates match those of $\ell_{bot}$ on $\llbracket t_1^-, t_1^+\rrbracket$.
\end{lemma}

In the remainder, we prove Proposition \ref{PropMain}, assuming the validity of Lemmas \ref{PropSup}, \ref{PropSup2} and \ref{LemmaAP1}. The arguments we present are similar to those used in the proof of Proposition 6.5 in \cite{CorHamK}.
\begin{proof}[Proof of Proposition \ref{PropMain}] Let $\epsilon > 0$ be given. Define the event
\begin{equation*}
\begin{split}
&E_N = \bigcap_{\varsigma\in \{\pm \}} \Big\{    L_1^N(  t^{\varsigma}_2)-  pt^{\varsigma}_2 \geq  - M_1 (t_2^+ - t_2^-)^{1/2}\Big\} \cap \Big\{ \sup_{s \in [ t^{-}_2, t^+_2]} \big( {L}^N_2(s) - p s \big)\leq M_2  (t_2^+ - t_2^-)^{1/2} \Big\},
\end{split}
\end{equation*}
where $M_1$ and $M_2$ are sufficiently large so that for all large $N$ we have $\mathbb{P}(E_N^c) <  \epsilon / 2$. The existence of such $M_1$ and $M_2$ is assured from Lemmas \ref{PropSup} and \ref{PropSup2}.

Let $\delta = (\epsilon/2) \cdot g h$, where $g,h$ are as in Lemma \ref{LemmaAP1} for the values $M_1, M_2$ as above and $r$ as in the statement of the proposition.
We denote
$$V = \Big\{Z_{H,H^{RW}}\big( t_1^-, t_1^+,{L}_1^N(t_1^-),{L}_1^N(t_1^+),{L}^N_2\llbracket t_1^-, t_1^+\rrbracket \big) < \delta\Big\}$$
and make the following deduction
\begin{equation*}
\begin{split}
&\mathbb{P}\big( V \cap E_N \big) =\mathbb{E} \bigg[    \mathbb{E}\Big[{\bf 1}_{E_N} \cdot {\bf 1}_{V} \Big{|} \mathcal{F}_{ext} \big( \{1\} \times \llbracket t_2^- + 1,t_2^+ - 1\rrbracket \big)\Big] \bigg] = \\
&\mathbb{E} \bigg[ {\bf 1}_{E_N} \cdot   \mathbb{E}\Big[ {\bf 1} \{ Z_{H,H^{RW}}\big( t_1^-, t_1^+,{L}_1^N(t_1^-),{L}_1^N(t_1^+),{L}^N_2\llbracket t_1^-, t_1^+\rrbracket \big) < \delta\}   \Big{|} \mathcal{F}_{ext} \big( \{1\} \times \llbracket t_2^- + 1,t_2^+ - 1\rrbracket \big)\Big] \bigg]  = \\
&\mathbb{E} \left[ {\bf 1}_{E_N} \cdot  \mathbb{E}^{t_2^-, t_2^+, L_1^N(t_2^-), L_1^N(t_2^+), L^2_N\llbracket t^-_2, t_2^+\rrbracket }_{H,H^{RW}}\left[ {\bf 1} \{ Z_{H,H^{RW}}\big( t_1^-, t_1^+, \ell(t_1^-),\ell(t_1^+),{L}^N_2\llbracket t_1^-, t_1^+\rrbracket \big) < \delta\} \right] \right] \leq \\
&  \mathbb{E} \left[ {\bf 1}_{E_N} \cdot  \epsilon/2 \right] \leq \epsilon/2.
\end{split}
\end{equation*}
The first equality follows from the tower property for conditional expectations. The second equality uses the fact that ${\bf 1}_{E_N} $ is $\mathcal{F}_{ext} \big( \{1\} \times \llbracket t_2^- + 1,t_2^+ - 1\rrbracket$-measurable and can thus be taken outside of the conditional expectation as well as the definition of $V$. The third equality uses the $(H,H^{RW})$-Gibbs property (\ref{GibbsEq}), applied to $F(\ell) = {\bf 1} \{ Z_{H,H^{RW}}\big( t_1^-, t_1^+, \ell(t_1^-),\ell(t_1^+),{L}^N_2\llbracket t_1^-, t_1^+\rrbracket \big) < \delta\} $. The inequality on the third line uses Lemma \ref{LemmaAP1} with $\tilde{\epsilon} = \epsilon/2$ as well as the fact that on the event $E_N^c$ the random variables $L^N_1(t_2^-), L^N_1(t_2^+)$ and $L^N_2 \llbracket t^{-}_2, t^+_2 \rrbracket$ (that play the roles of $x, y$ and $\ell_{bot}$) satisfy the inequalities
$$L^N_1(t_2^-) \geq  pt^-_2- M_1 (t_2^+ - t_2^-)^{1/2},  L^N_1(t_2^+) \geq  p t^+_2- M_1(t_2^+ - t_2^-)^{1/2}, \sup_{s \in [ t_2^-,t_2^+]} \hspace{-2mm}\big(L^N_2(s)  - ps \big)  \leq M_2(t_2^+ - t_2^-)^{1/2}.$$
The last inequality is trivial.

Combining the above inequality with $\mathbb{P}(E_N^c) <  \epsilon/2$, we see that for all large $N$ we have
$$\mathbb{P}\left( V  \right) = \mathbb{P}(V \cap E_N) + \mathbb{P}(V \cap E_N^c) \leq \epsilon/2 + \mathbb{P}(E_N^c) < \epsilon,$$
which completes the proof.
\end{proof}

%
\subsection{Proof of Theorem \ref{PropTightGood} }\label{Section5.2}
For clarity, we split the proof of Theorem \ref{PropTightGood} into three steps. In Step 1, we reduce the statement of the theorem to establishing a certain estimate on the modulus of continuity of the curves $L_1^N$. In Step 2, we show that it is enough to establish these estimates under the additional assumption that $(L^N_1, L^N_2)$ are well-behaved (in particular, well-behaved implies that  $Z_{H,H^{RW}}\big(t_1^-,t^+_1,L_1^N(t^-_1),L_1^N(t^+_1), L_2^N\llbracket t_1^-,t_1^+\rrbracket\big)$ is lower bounded and it is here that we use Proposition \ref{PropMain}). The fact that the $Z_{H,H^{RW}}$ is lower bounded is exploited in Step 3 to effectively reduce the estimates on the modulus of continuity of $L_1^N$ to those of a $H^{RW}$ random walk bridge. The latter estimates are then derived by appealing to Lemma \ref{MOCLemmaS4}.
\vspace{2mm}

{\raggedleft \bf Step 1.} Recall from (\ref{MOCS4}) that the modulus of continuity of $f \in C([-r,r])$ is defined by
$$ w(f,\delta) = \sup_{\substack{x,y \in [-r,r]\\ |x-y| \leq \delta}} |f(x) - f(y)|.$$
As an immediate generalization of  \cite[Theorem 7.3]{Bill}, in order to prove the theorem, it suffices for us to show that the sequence of random variables $f_N(0)$ is tight and that for each positive $\epsilon$ and $\eta$ there exist $\delta' > 0 $ and $N' \in \mathbb{N}$, such that for $N \geq N'$ we have
\begin{equation}\label{PTG2}
\mathbb{P} \big( w(f_N, \delta') \geq \epsilon \big) \leq \eta.
\end{equation}
The tightness of $f_N(0)$ is immediate from our assumption that $\big\{(L^N_1,L^N_2)\big\}_{N=1}^{\infty}$ is an $(\alpha, p, r+3)$--good sequence (it is true by the second condition in Definition \ref{def:intro} for $s =0$). Consequently, we are left with verifying (\ref{PTG2}).\\

Suppose $\epsilon,\eta > 0$ are given and also recall $t^{\pm}_1$ from \eqref{eqsts}. We claim that we can find $  \delta > 0$, such that for all $N$ sufficiently large we have
\begin{equation}\label{PTG2.5}
\mathbb{P} \bigg(  \sup_{\substack{ x,y \in [ t^-_1, t_1^+] \\  |x - y| \leq   \delta (t_1^+ - t_1^-) }}  \big| L_1^N(x) - L_1^N(y) - p(x - y) \big| \geq \frac{\epsilon (t_1^+ -t_1^-)^{1/2}}{2 (2r+2)^{1/2}}  \bigg) \leq \eta.
\end{equation}
Here, as usual, we are treating $L_1^N$ as the continuous curve, which linearly interpolates between its values on integers. Let us assume the validity of (\ref{PTG2.5}) and deduce (\ref{PTG2}).

If we set $\delta' = \delta$, and observe that for all large enough $N$ we have $(t_1^+ - t_1^-) N^{-\alpha} \geq 1$, we get
\begin{equation}\label{eqwfn}
\mathbb{P} \big( w(f_N, \delta') \geq \epsilon \big) \leq \mathbb{P} \bigg(  \sup_{\substack{ x,y \in [ t_1^-, t_1^+]  \\  |x - y| \leq \delta(t_1^+ - t_1^-) }}  \big| L_1^N(x) - L_1^N(y) - p(x - y) \big| \geq \epsilon N^{\alpha/2} \bigg).
\end{equation}
Since $t^{\pm}_1 = \lfloor \pm (r+1)N^{\alpha} \rfloor$, we see that $\frac{\epsilon (t_1^+ - t_1^-)^{1/2}}{2 (2r+2)^{1/2}} \sim (\epsilon /2) N^{\alpha/2}$ as $N$ becomes large and so we conclude that for all sufficiently large $N$ we have $\frac{\epsilon (t_1^+ -t_1^-)^{1/2}}{2 (2r+2)^{1/2}} < \epsilon N^{\alpha/2}$. This, together with (\ref{PTG2.5}), implies that the right-hand side of \eqref{eqwfn} bounded from above by $\eta$, which is what we wanted. \\

{\raggedleft \bf Step 2.} In Step 1, we reduced the proof of the theorem to establishing (\ref{PTG2.5}). This step sets up notation needed in the subsequent step in order to prove (\ref{PTG2.5}).

From Lemma \ref{PropSup} we can find $M_1 > 0$ sufficiently large so that for all large $N$ we have
$$\mathbb{P} (E_1) \geq 1 - \eta/4,\quad \mbox{where}\quad  E_1 = \bigg\{ \max \Big(\big| L^N_1(t^-_1) - pt^-_1 \big| , \big| L^N_1(t^+_1) - pt^+_1 \big| \Big) \leq M_1N^{\alpha/2} \bigg\}.$$
In addition, by Proposition \ref{PropMain} we can find $\delta_1 > 0$, such that for all sufficiently large $N$ we have
$$\mathbb{P} (E_2) \geq 1 - \eta/4,\quad \mbox{where}\quad   E_2 = \Big\{ Z_{H,H^{RW}} ( t_1^-,t_1^+,L_1^N(t^-_1),L_1^N(t^+_1), L_2^N\llbracket t^-_1, t_1^+\rrbracket ) > \delta_1  \Big\}.$$
For $\delta > 0$ and any continuous curve $\ell$ on $[t_1^-, t_1^+]$, we define
$$V(\delta, \ell) = \sup_{\substack{ x,y \in [t^-_1, t^+_1] \\  |x - y| \leq  \delta (t_1^+ -t_1^-)}}  \left| \ell(x) - \ell(y) - p(x - y) \right|.$$
We assert that we can find $\delta > 0$, such that for all large $N$ we have
\begin{equation}\label{PTG7.5}
\mathbb{P} \Big{(} V(\delta, L_1^N[t_1^-, t_1^+])\geq A \big\} \cap E_1\cap E_2 \Big{)} \leq \eta/2,\quad \mbox{where}\quad  A = \frac{\epsilon (t_1^+ -t_1^-)^{1/2}}{2 (2r+2)^{1/2}}.
\end{equation}
In the above, $L_1^N[t_1^-, t_1^+]$ denotes the restriction of $L_1^N$ to the interval $[t_1^-, t_1^+].$

Let us assume the validity of (\ref{PTG7.5}) and deduce (\ref{PTG2.5}). We have
$$\mathbb{P}\Big{(} V(\delta, L_1^N[t_1^-, t_1^+] )\geq A \Big{)} \leq \mathbb{P} \Big{(}\big\{V(\delta, L_1^N[t_1^-, t_1^+] )\geq A \big\} \cap E_1\cap E_2\Big{)} + \eta/2 < \eta$$
where we used that $\mathbb{P}( E^c_1 ) \leq \eta/4$ and $\mathbb{P}( E^c_2) \leq \eta/4$. Identifying  $\mathbb{P}\big(V(\delta, L_1^N[t_1^-, t_1^+])\geq A \big)$ with  the left-hand side of \eqref{PTG2.5}, we see that the last inequality implies (\ref{PTG2.5}).\\

{\raggedleft \bf Step 3.} In this step, we establish (\ref{PTG7.5}). Let us write $F_{\delta} = \{V(\delta, L_1^N[t_1^-, t_1^+] )\geq A \big\}$.
Using the $(H,H^{RW})$-Gibbs property (see (\ref{GibbsEq})), we know that
\begin{equation}\label{PTG10}
\begin{split}
&\mathbb{P} \Big{(}\big\{V(\delta, L_1^N[t_1^-, t_1^+] )\geq A \big\} \cap E_1\cap E_2\Big{)} = \mathbb{E} \left[ \mathbb{E} \left[ {\bf 1}_{F_\delta}  {\bf 1}_{E_1}  {\bf 1}_{E_2}  \big{\vert} \mathcal{F}_{ext} ( \{1\} \times \llbracket t_1^- \hspace{-1mm}+ 1, t_1^+ \hspace{-1mm} - 1 \rrbracket) \right] \right] = \\
&\mathbb{E} \left[  {\bf 1}_{E_1}  {\bf 1}_{E_2} \mathbb{E} \left[ {\bf 1}_{F_\delta}   \big{\vert} \mathcal{F}_{ext} ( \{1\} \times \llbracket t_1^- \hspace{-1mm}+ 1, t_1^+ \hspace{-1mm} - 1 \rrbracket) \right] \right] =\mathbb{E} \left[  {\bf 1}_{E_1} \cdot {\bf 1}_{E_2} \cdot \mathbb{E}_{H, H^{RW}}\left[ {\bf 1}\{ V(\delta, \ell) \geq A\} \right]  \right],
\end{split}
\end{equation}
where we have written $\mathbb{E}_{H, H^{RW}}$ to stand for $\mathbb{E}_{H, H^{RW}}^{t_1^-, t_1^+,L_1^N(t_1^-), L_1^N(t_1^+), L_2^N \llbracket t_1^-, t_1^+ \rrbracket}$ to ease the notation (recall that this notation for $\mathbb{E}_{H, H^{RW}}$ was introduced in Section \ref{Section4.2}) and the random variable with respect to which we are taking the expectation in $\mathbb{E}_{H, H^{RW}}$ is denoted by $\ell$. In addition, from (\ref{RND}) we have
\begin{equation}\label{PTG11}
\begin{split}
 \mathbb{E}_{H, H^{RW}}\left[ {\bf 1}\{ V(\delta, \ell) \geq A\} \right]  = \frac{\mathbb{E}_{H^{RW}}^{t_1^-, t_1^+,L_1^N(t_1^-), L_1^N(t_1^+)} \left[ W_H(\ell) \cdot {\bf 1}\{ V(\delta, \ell) \geq A\}  \right]}{Z_{H,H^{RW}} ( t_1^-,t_1^+,L_1^N(t^-_1),L_1^N(t^+_1), L_2^N\llbracket t^-_1, t_1^+\rrbracket )},
\end{split}
\end{equation}
where we have written $W_H(\ell)$ in place of $W_H(t_1^-, t_1^+, \ell , L_2^N \llbracket t_1^-, t_1^+ \rrbracket)$ to ease the notation.  We next use the fact that $W_H \in [0,1]$ and $Z_{H,H^{RW}} ( t_1^-,t_1^+,L_1^N(t^-_1),L_1^N(t^+_1), L_2^N\llbracket t^-_1, t_1^+\rrbracket ) > \delta_1$ on $E_2$ by definition to conclude that
\begin{equation}\label{PTG12}
\begin{split}
 {\bf 1}_{E_1} {\bf 1}_{E_2} \cdot  \mathbb{E}_{H, H^{RW}}\left[ W_H(\ell)  \cdot{\bf 1}\{ V(\delta, \ell) \geq A\} \right]  \leq  {\bf 1}_{E_1} {\bf 1}_{E_2} \cdot  \frac{\mathbb{P}_{H^{RW}}^{t_1^-, t_1^+,L_1^N(t_1^-), L_1^N(t_1^+)} \left( V(\delta, \ell) \geq A  \right)}{\delta_1}.
\end{split}
\end{equation}
We now observe that
\begin{equation}\label{PTG13}
\begin{split}
\mathbb{P}_{H^{RW}}^{t_1^-, t_1^+,L_1^N(t_1^-), L_1^N(t_1^+)} \left( V(\delta, \ell) \geq A  \right) = \mathbb{P}_{H^{RW}}^{0, t_1^+ - t_1^-,0, L_1^N(t_1^+) - L_1^N(t_1^-)} \left(w(f^{\ell},\delta ) \geq \frac{\epsilon}{2(2r+2)^{1/2}}  \right),
\end{split}
\end{equation}
where on the right side $\ell$ is $\mathbb{P}_{H^{RW}}^{0, t_1^+ - t_1^-,0, L_1^N(t_1^+) - L_1^N(t_1^-)} $-distributed, and we used the notation $f^\ell$ from Lemma \ref{MOCLemmaS4}. In deriving the above equation, we used the definition of $A$, as well as the fact that if two random cruves $\ell_1$ and $\ell_2$ are distributed according to $\mathbb{P}_{H^{RW}}^{t_1, t_2,x,y}$ and $\mathbb{P}_{H^{RW}}^{0, t_2 - t_1,0,y - x}$, then they have the same distribution except for a re-indexing and a vertical shift by $x$ -- hence their modulus of continuity has the same distribution. Notice that on the event $E_1$ we have that
$$
| L_1^N(t_1^+) - L_1^N(t_1^-)  - p(t_1^+ - t_1^-)| \leq 2M_1 N^{\alpha/2} \leq 2M_1 (t_1^+- t_1^-)^{1/2}.
$$
The latter and Lemma \ref{MOCLemmaS4} (applied to $\eta = \delta_1 (\eta/2)$, $\epsilon =  \frac{\epsilon}{2(2r+2)^{1/2}}$, $M = 2M_1$, $p$ as in the statement of the theorem, and $T = t_1^+ - t_1^-$) together imply that we can find $\delta > 0$ sufficiently small, such that for all large enough $N$ we have
\begin{equation}\label{PTG14}
\begin{split}
{\bf 1}_{E_1} \cdot \mathbb{P}_{H^{RW}}^{0, t_1^+ - t_1^-,0, L_1^N(t_1^+) - L_1^N(t_1^-)} \left(w(f^{\ell},\delta ) \geq \frac{\epsilon}{2(2r+2)^{1/2}}  \right) \leq{\bf 1}_{E_1} \cdot \delta_1 \eta/2.
\end{split}
\end{equation}
Combining (\ref{PTG12}), (\ref{PTG13}), (\ref{PTG14}), we see that
\begin{equation*}
\begin{split}
 {\bf 1}_{E_1} \cdot {\bf 1}_{E_2} \cdot  \mathbb{E}_{H, H^{RW}}\left[ {\bf 1}\{ V(\delta, \ell) \geq A\} \right]  \leq  {\bf 1}_{E_1} \cdot {\bf 1}_{E_2} \cdot  \eta/2,
\end{split}
\end{equation*}
which together with (\ref{PTG10}) implies (\ref{PTG7.5}). This suffices for the proof.

%
\section{Proof of three key lemmas}\label{Section6}
Here, we prove the three key lemmas from Section \ref{Section5.1}. The arguments we use below heavily depend on the results from Section \ref{Section4}, and use the key notation and definitions from Section \ref{Section5} (e.g. Definition \ref{def:intro} of an $(\alpha,p,r)$--good line ensemble, assumption \eqref{eqalphagood} on $\big\{(L^N_1,L^N_2)\big\}_{N=1}^{\infty}$, and the notation $t_1^{\pm}, t_2^{\pm}$ and $t_3^{\pm}$ in equation \eqref{eqsts}). In the proofs below there are various constants that depend on
\begin{equation}\label{blah}
\alpha, p ,r, N_0, H, H^{RW} \mbox{ and } R_0 \mbox{ as in (\ref{TailFun}),}
\end{equation}
which are as in Section \ref{Section5}.

%
%
\subsection{Proof of Lemma \ref{PropSup}}\label{Section6.1}

We split the proof into two parts. In the first, we show that we can find $R_1'> 0$, such that for all large $N$
\begin{equation}\label{EqK5}
 \mathbb{P}\Big( \sup_{s \in [ t_3^-, t_3^+]} \big( L_1^N(s) -p s \big) \geq  R_1'N^{\alpha/2}\Big)<\epsilon/2,
\end{equation}
and in the second, we show we can find $R_1'' > 0$, such that for all large $N$
\begin{equation}\label{EqK6}
\mathbb{P}\Big( \inf_{s \in [ t_3^-, t_3^+]} \big( L_1^N(s) -p s \big) \leq - R_1''N^{\alpha/2}\Big)<\epsilon/2.
\end{equation}
Clearly the statement of the lemma follows from (\ref{EqK5}) and (\ref{EqK6}) with $R_1 = \max(R_1', R_1'')$.

\medskip
{\bf \raggedleft Proof of \eqref{EqK5}.}
We will prove that we can find $R_1'$ sufficiently large, so that for all large $N$
\begin{equation}\label{EKG1}
\mathbb{P}\Big( \sup_{s \in [ 0, t_3^+]} \big( L_1^N(s) -p s \big) \geq  R_1'N^{\alpha/2}\Big)<\epsilon/4 \mbox{ and } \mathbb{P}\Big( \sup_{s \in [ t_3^-, 0]} \big( L_1^N(s) -p s \big) \geq  R_1'N^{\alpha/2}\Big)<\epsilon/4,
\end{equation}
which clearly implies (\ref{EqK5}). Since the proofs of the above two statements are completely analogous, we only focus on proving the first inequality in (\ref{EKG1}). 

Define the $N$-indexed events (as indicated below, we will generally drop the superscript  $N$ on these events to ease the notation)
\begin{align*}
E(M)=E^N(M) &:= \Big\{ \big| {L}^N_1(t_3^-) - p t_3^-\big| > MN^{\alpha/2} \Big\},\\
F(M)=F^N(M) &:= \Big\{ L_1^N(t_1^-) > pt_1^- + MN^{\alpha/2} \Big\},\\
G(M)=G^N(M) &:=  \Big\{  \sup_{s \in [ 0,t_3^+]} \big( L_1^N(s) -p s \big) >  (4r + 17)(M+1)N^{\alpha/2}\Big\}.
\end{align*}

We claim that we can find $M$ sufficiently large, such that for all large $N$
\begin{equation}\label{eqnGMeps}
\mathbb{P}\big(G(M)\big) <\epsilon/4,
\end{equation}
which, if true, would imply \eqref{EKG1} with $R_1' =  (4r + 17)(M+1)$. In what remains, we prove \eqref{eqnGMeps}. For the sake of clarity, we split the proof into two steps. In the first step, we specify the choice of $M, N_2$, such that if $N \geq N_2$ then (\ref{eqnGMeps}) holds. The way we deduce (\ref{eqnGMeps}) is by writing $G(M)$ as a finite disjoint union of events, namely we split the event over the first index $n$ such that $L_1^N(n) -p n  >  (4r + 17)(M+1)N^{\alpha/2}$ and then we look at the line ensemble on the interval $[t_3^-, n]$. For such an ensemble, we know that its starting point is well-behaved and its terminal point is high. This allows us to conclude that an intermediate point, namely $L_1^N(t_1^-) -p t^-_1$, is also quite at least $1/3$ of the time whenever $L_1^N(n) -p n$ is high -- for this we use (\ref{EqK3}), which is proved in Step 2 by essentially appealing to Lemma \ref{LemmaHalfS4}. Now, (\ref{EqK3}) is formulated in terms of a free $H^{RW}$ random walk bridge, while the curve $L_1^N$ interacts with the second curve of the ensemble, which pushes it even higher. Thus (\ref{EqK3}) through stochastic monotonicity, Lemma \ref{MonCoup}, allows us to show that the estimate in (\ref{EqK3}) for a free $H^{RW}$ random walk bridge extends to $L_1^N$, and this is carefully explained in (\ref{eqnGFEM}). Overall, equation (\ref{eqnGFEM}) shows that on the event $G(M) \cap E^c(M)$ it is quite likely that $F(M)$ occurs. But since $F(M)$ and $E(M)$ are both unlikely, once we pick $M$ large enough, this would imply that $G(M)$ is quite unlikely. \\

{\bf \raggedleft Step 1.} In this step, we specify the choice of $M$ and $N_2$ and prove (\ref{eqnGMeps}) for this $M$ and $N \geq N_2$, modulo a certain statement given in (\ref{EqK3}), whose proof is postponed until the next step.

We pick $M > 0$ sufficiently large, so that for every $N \geq N_0$ (as in (\ref{blah})) we have
\begin{equation}\label{S3eqN1}
\mathbb{P}\big{(} E(M) \big{)}  <  \epsilon/8\quad\textrm{and}\quad\mathbb{P}\big{(} F(M) \big{)}  < \epsilon/24.
\end{equation}
Observe that such a choice is possible by (\ref{TailFun}). This fixes our choice of $M$. Next we pick $N_2 \in \mathbb{N}$ sufficiently large, depending on $M$ and the constants in (\ref{blah}), so that $N_2 \geq N_0$ and for $N \geq N_2$ the following inequalities all hold
\begin{equation}\label{IneqCoeffN}
 t_1^- - t_3^- \geq N^{\alpha}, \hspace{2mm} t_3^+ - t_3^- \leq (2r+8) N^{\alpha}, \hspace{2mm} (2r +8)^{1/4} N^{\alpha/4} \leq N^{\alpha/2}, N \geq W_0(p, 2\sqrt{2r+8} (M+1)),
\end{equation}
where $W_0$ is as in Lemma \ref{LemmaHalfS4}. With this, our choice of $M$ and $N_2$ is fixed.

Now, suppose that $x,y\in \mathbb{R}$ and $s \in \llbracket 0, t_3^+\rrbracket$ are chosen so as to satisfy the inequalities
\begin{equation}\label{eqnabM}
|x - pt^-_3| \leq MN^{\alpha/2}\quad \textrm{and}\quad y - ps > (4r + 17)(M+1)N^{\alpha/2}.
\end{equation}
Then, recalling the measure $\mathbb{P}^{T_0,T_1,x,y}_{H^{RW}}$ on a curve $\ell$ (introduced at the beginning of Section \ref{Section4.2}), we claim the following inequality for all $N \geq N_2$
\begin{equation}\label{EqK3}
\mathbb{P}^{t_3^-, s,x,y}_{H^{RW}} \left(\ell(t_1^{-}) \geq  pt_1^-  +MN^{\alpha/2}  \right)  \geq \frac{1}{3}.
\end{equation}
We prove (\ref{EqK3}) below in Step 2. For now we assume its validity and conclude the proof of (\ref{eqnGMeps}).\\

For $n \in \llbracket 0, t_3^+\rrbracket$ define the events
$$W_n(M) = \big\{ L_1^N(n) -p n > (4r + 17)(M+1)N^{\alpha/2} \big\}\quad \textrm{and}\quad G_n(M) = W_n(M) \cap \big(\cap_{m = n+1}^{t_3^+} W^c_m(M)\big).$$
Notice that $G(M) = \bigcup_{n = 0}^{t_3^+}G_n(M)$ is the disjoint union of the events $G_n(M)$, and on the event $G(M)$, the $n$ for which $G_n(M)$ occurs, is precisely the maximal value of $s$ under which the inequality $ L_1^N(s) -p s >  (4r + 17)(M+1)N^{\alpha/2}$ holds.

Since $\big\{(L^N_1,L^N_2)\big\}_{N=1}^{\infty}$ is a sequence of $(\alpha, p,  r+3)$--good line ensembles, we may make use of the $(H,H^{RW})$-Gibbs property. Let
$\mathcal{F}_n = \mathcal{F}_{ext} \big(\{1\} \times \llbracket t_3^-+1, n - 1 \rrbracket)$ be the external $\sigma$-algebra (generated by the second curve $L^N_2$ and the first curve $L^N_1(s)$ for $s\notin \llbracket t_3^-+1,n-1\rrbracket $), as defined in \eqref{GibbsCond}. Introduce the shorthand
\begin{equation}\label{eqnEnshort}
\mathbb{E}_{H,H^{RW},n}=\mathbb{E}_{H,H^{RW}}^{ t_3^-, n, L_1^N(t_3^-), L_1^N(n), L^N_2\llbracket t_3^-, n\rrbracket} \mbox{ and }  \mathbb{E}_{H^{RW},n} = \mathbb{E}_{H^{RW}}^{ t_3^-, n, L_1^N(t_3^-), L_1^N(n)}.
\end{equation}
With this, we may make the key deduction that
\begin{equation}\label{eqnGFEM}
\begin{split}
& \mathbb{P}\big( G(M) \cap F(M) \cap E^c(M)\big) = \sum_{n = 0}^{t_3^+} \mathbb{E} \Big[ { \bf 1}_{ G_n(M)} \cdot {\bf 1}_{E^c(M)} \cdot {\bf 1}_{ F(M)} \Big] =  \\
&\sum_{n = 0}^{t_3^+} \mathbb{E} \Big[ \mathbb{E}\big[ { \bf 1}_{ G_n(M)} \cdot {\bf 1}_{E^c(M)}  {\bf 1}\{ L^N_1 (t_1^-)> pt_1^- + MN^{\alpha/2} \} \vert \mathcal{F}_n \big] \Big] =\\
& \sum_{n = 0}^{t_3^+} \mathbb{E} \Big[ { \bf 1}_{ G_n(M)} \cdot {\bf 1}_{E^c(M)} \cdot \mathbb{E}_{H,H^{RW},n}\big[ {\bf 1}\{ \ell (t_1^-)> pt_1^- + MN^{\alpha/2} \}\big] \Big] \geq\\
&  \sum_{n = 0}^{t_3^+} \mathbb{E} \Big[ { \bf 1}_{ G_n(M)} \cdot {\bf 1}_{E^c(M)} \cdot \mathbb{E}_{H^{RW},n}\big[ {\bf 1}\{ \ell (t_1^-)> pt_1^- + MN^{\alpha/2} \}\big] \Big] \geq\\
& \sum_{n = 0}^{t_3^+} \mathbb{E} \Big[ { \bf 1}_{ G_n(M)} \cdot {\bf 1}_{E^c(M)} \cdot \frac{1}{3} \Big] = \frac{1}{3}  \cdot \mathbb{P}\big( G(M) \cap  E^c(M)\big),
\end{split}
\end{equation}
In the above equation we have that the first and last equality follow from the fact that $G(M)$ is a disjoint union of the events $G_n(M)$. The second equality follows from the tower property for conditional expectation and the definition of $F(M)$. In the third equality, we use that ${\bf 1}_{G_n(M)}$ and ${\bf 1}_{E^c(M)}$ are $\mathcal{F}_n$ measurable and so can be taken out of the conditional expectation, and then we apply the $(H,H^{RW})$-Gibbs property (\ref{GibbsEq}) to the function $F(\ell) = {\bf 1}\{ \ell (t_1^-)> pt_1^- + MN^{\alpha/2} \}$. The inequality on the third line uses Lemma \ref{MonCoup} with $x = x' = L_1^N(t_3^-)$, $y = y'  = L_1^N(n)$, $\vec{z} = (-\infty)^{n - t_3^- + 1}$ and $\vec{z}' = L_2^N\llbracket t_3^-, n \rrbracket$. The inequality on the fourth line uses (\ref{EqK3}) and the fact that on the event $G_n(M) \cap E^c(M)$ we have that $L_1^N(n)$ and $L_1^N(t_3^-)$ (which play the roles of $y$ and $x$ in  (\ref{EqK3}) ) satisfy the inequalities $L_1^N(n)\geq (4r + 17)(M+1)N^{\alpha/2}$ and $|L_1^N(t_3^-) - p t_3^-| \leq MN^{\alpha/2}.$

From (\ref{eqnGFEM}) we see that
$$
\mathbb{P}\big( G(M) \cap  E^c(M)\big) \leq 3 \mathbb{P}\big( G(M) \cap F(M) \cap E^c(M)\big)  \leq 3\cdot  \mathbb{P}\big( F(M)\big)
$$
Using this, we finally conclude that
$$ \mathbb{P}(G(M)) = \mathbb{P}\big(G(M) \cap E(M)\big) + \mathbb{P}\big(G(M) \cap E^c(M)\big) \leq \mathbb{P}\big(E(M)\big) + 3 \cdot \mathbb{P} \big(F(M)\big) < \epsilon /4,$$
where in the last inequality we used (\ref{S3eqN1}). The last equation implies \eqref{eqnGMeps}. \\

{\bf \raggedleft Step 2.} In this step, we prove (\ref{EqK3}). Using \eqref{eqnabM}, we see that
$$y -x \geq p( s- t_3^- ) + (4r + 16)(M+1)N^{\alpha/2} \geq p(s - t_3^{-}) + 2\sqrt{2r+8} (M+1)(s-t_3^-)^{1/2},$$
where the last inequality used that $(s- t_3^-) \leq (t_3^+ - t_3^-) \leq N^{\alpha} (2r + 8) $ -see (\ref{IneqCoeffN}). Using this and applying Lemma \ref{LemmaHalfS4} for
$$M_1 = 0,\hspace{2mm} M_2 = 2\sqrt{2r+8}(M+1), \hspace{2mm} T = s - t_3^-, \hspace{2mm} s = t_1^- - t_3^-$$
and $p$ as in (\ref{blah}), together with the fact that $N^{\alpha} \geq W_0(p,M_2)$ by assumption in (\ref{IneqCoeffN}), we see that
\begin{equation*}
\mathbb{P}^{0, s - t_3^-, 0,y - x}_{H^{RW}} \Big(\ell(t_1^- - t_3^-) \geq  \frac{t_1^- - t_3^-}{s - t_3^-} \big(p( s -t_3^-)  + 2\sqrt{2r+8}(M+1)(s-t_3^-)^{1/2} \big) - (s - t_3^-)^{1/4}  \Big) \geq \frac{1}{3},
\end{equation*}
which upon simplification and the shift-invariance of the measure implies
\begin{equation*}
\mathbb{P}^{t_3^-, s , x,y }_{H^{RW}} \Big(\ell(t_1^- ) - x \geq  p( t_1^- -t_3^-) + 2\sqrt{2r+8}(M+1)\frac{t_1^- - t_3^-}{(s - t_3^-)^{1/2}} - (s - t_3^-)^{1/4}  \Big) \geq \frac{1}{3}.
\end{equation*}
Since by assumption (\ref{eqnabM}) we have $x \geq pt_3^-  - MN^{\alpha/2}$, the last inequality implies
\begin{equation}\label{EqK2}
\mathbb{P}^{t_3^-, s , x,y }_{H^{RW}} \Big(\ell(t_1^- )  \geq  p t_1^- -  MN^{\alpha/2}+ 2\sqrt{2r+8}(M+1)\frac{t_1^- - t_3^-}{(s - t_3^-)^{1/2}} - (s - t_3^-)^{1/4}  \Big) \geq \frac{1}{3}.
\end{equation}
We now observe that the inequalities in (\ref{IneqCoeffN}) imply
$$2\sqrt{2r+8}(M+1)\frac{t_1^- - t_3^-}{(s - t_3^-)^{1/2}} - (s - t_3^-)^{1/4} \geq  \frac{2\sqrt{2r+8}(M+1)N^{\alpha}}{(2r+ 8)^{1/2} N^{\alpha/2} } - (2r+8)^{1/4} N^{\alpha/4} \geq 2MN^{\alpha/2}.$$
The latter and (\ref{EqK2}) imply (\ref{EqK3}), which concludes the proof of the second step. 

\medskip

{\bf \raggedleft Proof of \eqref{EqK6}} This proof follows a similar scheme as that of \eqref{EqK5}. Define the $N$-indexed events (and then drop the $N$ superscript below)
\begin{align*}
E_-(M)&=E^N_-(M)  = \Big\{ \big| {L}^N_1(t_3^-)  - p t_3^-\big| > MN^{\alpha/2} \Big\},\\
E_+(M)&=E^N_+(M) = \Big\{ \big| {L}^N_1(t_3^+) - p t_3^+\big| > MN^{\alpha/2} \Big\},\\
G(C) &=G^N(C) =  \Big\{  \inf_{s \in [ t_3^-,t_3^+ ] } \big( L_1^N(s) -p s \big) < -  C N^{\alpha/2}\Big\}.
\end{align*}
To prove \eqref{EqK6} it suffices to show that there exists $C$ sufficiently large so that for all large $N$
\begin{equation}\label{eqnPGAes}
\mathbb{P}\big(G(C)\big) <\epsilon/2.
\end{equation}
Clearly, \eqref{EqK6} follows immediately from this by setting $R_1''=C$. In what remains, we prove \eqref{eqnPGAes}. As before, we split the proof into two steps for clarity. In Step 1, we specify the choice of $C, N_2$, such that (\ref{eqnPGAes}) holds for $N \geq N_2$. The way we deduce (\ref{eqnPGAes}) is by noting that the events $E_{\pm}(M)$ are unlikely if we pick $M$ sufficiently large. This means that it is very likely that the endpoints $L_1^N(t_3^{\pm})$ are well-behaved. We know that a free $H^{RW}$ random walk bridge is unlikely to dip too low if its endpoints are well-behaved -- this appears as (\ref{eqntsAB}) below and it is proved in Step 2, by appealing to Lemma \ref{LemmaMinFreeS4}. Since $ L_1^N$ interacts with the second curve of the ensemble, which pushes it even higher, we see that (\ref{eqntsAB}) through stochastic monotonicity, Lemma \ref{MonCoup}, allows us to show that on $E^c_+(M) \cap E^c_-(M)$, it is unlikely that $L_1^N$ dips low (i.e. $G(C)$ occurs), and this is carefully explained in (\ref{eqnGFEM2}). Overall, equation (\ref{eqnGFEM2}) shows that on the event $E^c_+(M) \cap E^c_-(M)$ it is quite likely that $G(C)$ occurs. Since $E_{\pm}(M)$ are unlikely, we conclude that $G(C)$ is unlikely. \\

{\bf \raggedleft Step 1.} In this step we specify $C$ and $N_2$, and prove (\ref{eqnPGAes}) for this choice of $C,M$ and $N \geq N_2$, modulo a certain statement given in (\ref{eqntsAB}), whose proof is postponed until the next step.

We pick $M$ sufficiently large, so that for every $N \geq N_0$ (as in (\ref{blah}) we have
\begin{equation}\label{S3S2eqN1}
\mathbb{P}\big( E_{+}(M) \cup E_{-}(M)  \big)  <  \epsilon/4.
\end{equation}
Observe that such a choice is possible by (\ref{TailFun}). This fixes our choice of $M$. We next pick $C$ sufficiently large, so that
\begin{equation}\label{Cchoice}
 C - M \geq  A(M, p ,\epsilon/4) (2r+8)^{1/2},
\end{equation}
where $A$ is as in Lemma \ref{LemmaMinFreeS4} and $p$ is as in (\ref{blah}). This fixes our choice of $C$. Next we pick $N_2 \in \mathbb{N}$ sufficiently large, depending on $M$ and the constants in (\ref{blah}), so that $N_2 \geq N_0$ and for $N \geq N_2$ the following inequalities all hold
\begin{equation}\label{IneqCoeffN2}
 t_3^+ - t_3^- \geq 4N^{\alpha}, \hspace{2mm} t_3^+ - t_3^- \leq (2r+8) N^{\alpha}, N^{\alpha} \geq W_1(p, M, \epsilon/4),
\end{equation}
where $W_1$ is as in Lemma \ref{LemmaMinFreeS4} and $p$ is as in (\ref{blah}). With this, our choice of $C,M$ and $N_2$ is fixed.

Now, suppose that $x,y\in \mathbb{R}$ satisfy the inequalities
\begin{equation}\label{eqnabineq2}
|x - pt_3^-| \leq MN^{\alpha/2}\quad \textrm{and}\quad |y - pt_3^+| \leq MN^{\alpha/2}.
\end{equation}
We claim that
\begin{equation}\label{eqntsAB}
\mathbb{P}^{t_3^-, t_3^+, x,y}_{H^{RW}} \Big( \inf_{s \in [ t_3^-, t_3^+ ]} \big( \ell(s) - ps \big) < - C N^{\alpha/2}  \Big)\leq \epsilon/4.
\end{equation}
We prove (\ref{eqntsAB}) below in Step 2. For now we assume its validity and conclude the proof of (\ref{eqnPGAes}).\\

We proceed much in the same way as in \eqref{eqnGFEM}, using the $(H,H^{RW})$-Gibbs property. Using the notation $\mathcal{F}_n$, $\mathbb{E}_{H, H^{RW},n}$ and $\mathbb{E}_{H^{RW},n}$, defined in \eqref{eqnEnshort}, and the paragraph before it with $n=t_3^+$, we find that
\begin{equation}\label{eqnGFEM2}
\begin{split}
&\mathbb{P}\big( G(C) \cap E^c_+(M) \cap E^c_-(M)\big) = \mathbb{E} \left[ \mathbb{E} \left[  {\bf 1}_{E^c_-(M)} \cdot {\bf 1}_{E^c_+(M)} \cdot {\bf 1}_{ G(C)} \vert \mathcal{F}_{t_3^+} \right] \right] = \\
&\mathbb{E}\Big[ {\bf 1}_{E^c_-(M)} \cdot {\bf 1}_{E^c_+(M)} \cdot  \mathbb{E} \Big[  {\bf 1}\big\{  \inf_{s \in [ t_3^-,t_3^+ ] } \big( L_1^N(s) -p s \big) < -  CN^{\alpha/2}\big\} \vert \mathcal{F}_{t_3^+}\Big] \Big] = \\
&\mathbb{E} \Big[ {\bf 1}_{E^c_-(M)} \cdot {\bf 1}_{E^c_+(M)} \cdot \mathbb{E}_{H, H^{RW}, t_3^+} \big[  {\bf 1}\big\{  \inf_{s \in [ t_3^-,t_3^+ ] } \big( \ell(s) -p s \big) < -  CN^{\alpha/2}\big\} \big] \Big] \leq \\
&\mathbb{E} \Big[ {\bf 1}_{E^c_-(M)} \cdot {\bf 1}_{E^c_+(M)} \cdot \mathbb{E}_{ H^{RW}, t_3^+} \big[  {\bf 1}\big\{  \inf_{s \in [ t_3^-,t_3^+ ] } \big( \ell(s) -p s \big) < -  C N^{\alpha/2}\big\} \big] \Big] \leq \\
& \mathbb{E} \Big[ {\bf 1}_{E^c_-(M)} \cdot {\bf 1}_{E^c_+(M)} \cdot \frac{\epsilon}{4}   \Big] \leq  \frac{\epsilon}{4}.
\end{split}
\end{equation}
The first equality follows from the tower property for conditional expectations. The second equality uses the fact that ${\bf 1}_{E^c_{\pm}(M)}$  are $\mathcal{F}_{t_3^+}$-measurable and so can be taken out of the conditional expectation, as well as the definition of $G(C)$. The third equality uses the $(H,H^{RW})$-Gibbs property (\ref{GibbsEq}), applied to the function $F(\ell) = {\bf 1}\big\{  \inf_{s \in [ t_3^-,t_3^+ ] } \big( \ell(s) -p s \big) < -  C N^{\alpha/2}\big\}$. The inequality on the fourth line uses Lemma \ref{MonCoup} with $x =x' =  L_1^N(t_3^-)$, $y  = y' = L_1^N(t_3^+)$, $\vec{z} = (-\infty)^{t_3^+ - t_3^- + 1}$, and $\vec{z}' = L_2^N\llbracket t_3^-, t_3^+ \rrbracket$. The inequality on the fourth line uses (\ref{eqntsAB}), and the fact that on the event $E^c_-(M) \cap E^c_+(M)$ we have that $L_1^N(t_3^-)$, and $L_1^N(t_3^+)$ (which play the roles of $x$ and $y$ in (\ref{eqntsAB})) satisfy the inequalities $\big| {L}^N_1(t_3^\pm)  - p t_3^\pm\big| \leq MN^{\alpha/2})$. The last inequality is trivial.

From (\ref{eqnGFEM2}) and \eqref{S3S2eqN1} we see that
\begin{align*}
\mathbb{P} \big( G(C) \big) &=  \mathbb{P}\big( G(C) \cap E^c_+(M) \cap E^c_-(M)\big) + \mathbb{P}\big(  G(C) \cap \big(E_{+}(M) \cup E_{-}(M)\big) \big)\\
& \leq  \mathbb{P}\big( G(C) \cap E^c_+(M) \cap E^c_-(M)\big) + \mathbb{P}\big( E_{+}(M) \cup E_{-}(M) \big)< \epsilon/2,
\end{align*}
which completes the proof of \eqref{eqnPGAes}.\\

{\bf \raggedleft Step 2.} To show \eqref{eqntsAB}, first note that by the shift invariance of the measure we have
\begin{equation}\label{eqnptasf}
\begin{split}
&\mathbb{P}^{t_3^-, t_3^+, x,y}_{H^{RW}} \Big( \inf_{s \in [ t_3^-, t_3^+]} \big( \ell(s) - ps \big) < - C N^{\alpha/2}  \Big) = \\
& \mathbb{P}^{0, t_3^+ - t_3^- , 0,y - x}_{H^{RW}} \Big( \inf_{s \in[ 0,t_3^+ - t_3^-]} \big( x + \ell(s) - p(s + t_3^-) \big) < - C N^{\alpha/2} \Big) \leq \\
& \mathbb{P}^{0, t_3^+ - t_3^-, 0,y - x}_{H^{RW}} \Big(\inf_{s \in [ 0,t_3^+ - t_3^- ]} \big( \ell(s) - ps \big)  \leq -(C-M) N^{\alpha/2}  \Big),
\end{split}
\end{equation}
where in the last inequality we used that $x \geq pt_3^- - MN^{\alpha/2}$.

The inequalities \eqref{eqnabineq2} imply that $y -x \geq 2 p (t_3^+ - t_3^-)  - 2MN^{\alpha/2} \geq 2p(t_3^+ - t_3^-) - M (t_3^+ - t_3^-)^{1/2}$, where the last inequality used (\ref{IneqCoeffN2}). Using this, and applying  Lemma \ref{LemmaMinFreeS4} for
$T = t_3^+ - t_3^-$, $\epsilon = \epsilon/4$, $M$ as our choice in Step 1 above and $p$ as in (\ref{blah}), together with the fact that $N^{\alpha} \geq W_1(M,p, \epsilon/4)$ by assumption in (\ref{IneqCoeffN2}), we see that
\begin{equation*}
 \mathbb{P}^{0, t_3^+ - t_3^-, 0,y - x}_{H^{RW}} \Big(\inf_{s \in [ 0,t_3^+ - t_3^- ]} \big( \ell(s) - ps \big)  \leq - A(M,p, \epsilon/4) (t_3^+ - t_3^-)^{1/2}  \Big) \leq \epsilon/4.
\end{equation*}
Notice that by our choice of $C$ in (\ref{Cchoice}) we have that $(C- M)N^{\alpha/2} \geq A(M, p, \epsilon/4) [2r+8]^{1/2} N^{\alpha/2} \geq A(M,p, \epsilon/4) (t_3^+ - t_3^-)^{1/2}$, where we also used (\ref{IneqCoeffN2}). This shows that the last inequality implies
$$\mathbb{P}^{0, t_3^+ - t_3^-, 0,y - x}_{H^{RW}} \Big(\inf_{s \in [ 0,t_3^+ - t_3^- ]} \big( \ell(s) - ps \big)  \leq -(C-M) N^{\alpha/2}  \Big) \leq \epsilon/4,$$
which together with \eqref{eqnptasf} yields \eqref{eqntsAB}, as claimed.


\subsection{Proof of Lemma \ref{PropSup2}}\label{Section6.2}
Let $\epsilon > 0$ be given, and put $n = \lfloor N^{\alpha} \rfloor  - 2$. We first specify our choice of $N_3$ as in the statement of the lemma. We assume that $N_3'$ is sufficiently large, so that $2n > N^{\alpha}$ for $N \geq N_3'$.  We then let $N_2(\epsilon/2)$ and $R_1(\epsilon/2)$ be as in the statement of Lemma \ref{PropSup}. With this choice, we know that if $N \geq \max( N_2, N_3')$, then
$$\mathbb{P}\left( \sup_{s \in [t_3^- , t_3^+]}\left|L^N_1(s) - p s \right| >  2R_1 n^{1/2} \right) < \epsilon/2.$$
We let $N_3$ be sufficiently large, so that all of the following inequalities hold for $N \geq N_3$
\begin{equation}\label{IneqNL3}
N\geq N_0 \mbox{ as in (\ref{blah})}, \hspace{2mm} N \geq N_3', \hspace{2mm} N \geq N_2, \hspace{2mm} n \geq W_5(2R_1, p, \epsilon/2),
\end{equation}
where $ W_5(2R_1, p, \epsilon/2)$ is as in Lemma \ref{NoExplodeS4}. This fixes our choice of $N_3$.

Define the events
$$E = \left\{ \sup_{s \in [t_3^- , t_3^+]}\left|L^N_1(s) - p s \right| >  2R_1 n^{1/2} \right\}, G= \left\{ \sup_{s \in [t_2^-, t_2^+]}\left[L^N_2(s) - p s \right] \geq  4R_1n^{1/2} \right\}$$
$$W_m = \{ L^N_2(m) - p m \geq 4R_1  n^{1/2} \} \mbox{ and } G_v = W_v \cap \bigcap_{m = v+1}^{t_2^+} W^c_m  \mbox{ for $v \in \llbracket t_2^- , t_2^+ \rrbracket $}.$$
We claim that for all $N\geq N_3$ we have
\begin{equation}\label{GNTP}
\mathbb{P}(G) < \epsilon,
\end{equation}
which implies the lemma with $R_2 = 4R_1$. In the remainder we establish (\ref{GNTP}). \\

By Lemma \ref{NoExplodeS4}, applied to $M = 2R_1, \epsilon = \epsilon/2$ and $p$ as in (\ref{blah}), we know that for any $N \geq N_3$, $\vec{z} \in [-\infty, \infty)^{2n+1}$ with $z_{n+1} \geq pn + 4R_1n^{1/2}$ and $x,y \in \mathbb{R}$ with $x \geq -2R_1n^{1/2}$ and $y \geq  - 2R_1n^{1/2} + 2pn$
\begin{equation*}
\mathbb{P}_{H,H^{RW}}^{0,2n,x,y, \vec{z}} \left( \ell(n) \leq pn+ 2R_1n^{1/2} \right) \leq \epsilon/2,
\end{equation*}
which by the shift-invariance of the measure implies that for each $m \in \mathbb{N}$, $z_{m+1}  \geq pn + 4R_1n^{1/2}$, $x, y \in \mathbb{R}$ with $x - p(m-n) \geq -2R_1n^{1/2}$, and $y - p(m+n) \geq  - 2R_1n^{1/2}$ we have
\begin{equation}\label{NE1v3}
\mathbb{P}_{H,H^{RW}}^{m-n,m+n,x ,y , \vec{z}} \left( \ell(m) \leq pm+ 2R_1n^{1/2} \right) \leq \epsilon/2.
\end{equation}
Here we used (\ref{IneqNL3}), which ensures that $n \geq W_5(2R_1, p, \epsilon/2)$ as in Lemma \ref{NoExplodeS4}.

For every $m \in \llbracket t_2^- -1, t_2^+- 1\rrbracket$, we define
$$F_{m} =\{ \left|L^N_1(m - n) - p(m-n) \right| \leq 2R_1 n^{1/2} \} \cap \{ \left|L^N_1(m + n) - p(m+ n) \right| \leq 2R_1 n^{1/2} \} \mbox{ and } $$
$$H_m = \{ \left|L^N_1(m) - pm\right| \leq 2R_1 n^{1/2} \} ,$$
and observe that $E^c \subset  F_m \cap H_m  $ . We also let $\mathcal{F}_{ext}^m = \mathcal{F}_{ext}(\{1\} \times \llbracket m - n + 1, m + n - 1\rrbracket)$ as in (\ref{GibbsCond}). We now make the following deduction
\begin{equation*}
\begin{split}
&\mathbb{P} \left( F_m \cap G_{m+1} \cap H_m \right) = \mathbb{E} \left[ \mathbb{E} \left[ {\bf 1}_{G_{m+1}} \cdot {\bf 1}_{F_m} \cdot{\bf 1}_{H_m} \vert \mathcal{F}_{ext}^m \right]  \right] = \\
&\mathbb{E} \left[  {\bf 1}_{G_{m+1}} \cdot {\bf 1}_{F_m} \cdot  \mathbb{E} \left[{\bf 1} \{ \left|L^N_1(m) - pm\right| \leq 2R_1 n^{1/2} \} \vert \mathcal{F}_{ext}^m \right]  \right]  = \\
&\mathbb{E} \left[  {\bf 1}_{G_{m+1}} \cdot {\bf 1}_{F_m} \cdot  \mathbb{E}_{H,H^{RW}} \left[{\bf 1} \{ \left|\ell(m) - pm\right| \leq 2R_1 n^{1/2} \} \right]  \right] \leq \\
&\mathbb{E} \left[  {\bf 1}_{G_{m+1}} \cdot {\bf 1}_{F_m} \cdot  (\epsilon/2)  \right] = (\epsilon/2) \cdot \mathbb{P}( F_m \cap G_{m+1}),
\end{split}
\end{equation*}
where $\mathbb{E}_{H,H^{RW}}$ stands for $\mathbb{E}_{H,H^{RW}}^{m-n, m+n, L_1^N(m-n), L_1^N(m+n), L_N^2\llbracket m-n,m+n \rrbracket}$. The first equality follows from the tower property for conditional expectations.
The second equality uses the fact that ${\bf 1}_{G_{m+1}}$ and ${\bf 1}_{F_m}$ are $\mathcal{F}_{ext}^m$- measurable and can thus be taken out of the conditional expectation, as well as the definition of $H_m$. The third equality uses the $(H,H^{RW})$-Gibbs property (\ref{GibbsEq}) applied to $F(\ell) = {\bf 1} \{ \left|\ell(m) - pm\right| \leq 2R_1 n^{1/2} \}$. The inequality on the third line uses (\ref{NE1v3}) and the fact that on the event $G_{m+1} \cap F_m$ we have that the random variables $L_2^N(m+1)$, $L_1^N(m-n)$, $L_1^N(m+n)$ (which play the roles of $z_{m+1}, x, y$ in (\ref{NE1v3}) ) satisfy the inequalities $L_2^N(m+1) \geq pn + 4R_1n^{1/2}$, $L_1^N(m-n) - p(m-n) \geq -2R_1n^{1/2}$ and $L_1^N(m+n) - p(m+n) \geq  - 2R_1n^{1/2}$.  The last equality is trivial. The above inequality and the fact that $E^c \subset  F_m \cap H_m  $ imply
$$ \mathbb{P} \left( E^c \cap G_{m+1}\right) \leq  \mathbb{P} \left( F_m \cap G_{m+1} \cap H_m \right)  \leq  (\epsilon/2) \cdot \mathbb{P} \left( G_{m+1} \cap F_m \right) \leq  (\epsilon/2) \cdot \mathbb{P} \left( G_{m+1}\right) .$$
Taking the sum over $m$ in $\llbracket t_2^- - 1, t_2^+-1 \rrbracket $, and using that $G = \cup_{m = t_2^-}^{t_2^+} G_m$ is a disjoint union, we get
$$ \mathbb{P} \left( E^c \cap G\right) \leq  (\epsilon/2) \cdot \mathbb{P} \left( G \right) \leq \epsilon/2.$$
On the other hand,
$$ \mathbb{P} \left( E \cap G\right) \leq  \mathbb{P} \left( E \right) < \epsilon/2,$$
by our choice of $R_1$. The above two inequalities imply (\ref{GNTP}).

\subsection{Proof of Lemma \ref{LemmaAP1}}\label{Section6.3}

For clarity, we split the proof into four steps. In the first step we use the idea of size-biasing and reduce the proof of the lemma to establishing a certain lower bound, see (\ref{eqn57S6}) -- this is the easy part of the proof. Establishing the lower bound in (\ref{eqn57S6})  is done in Steps 2, 3 and 4 and we describe our approach within those steps.\\

{\bf \raggedleft Step 1.} We claim that we can find $N_4 \in \mathbb{N}$, such that if $N \geq N_4$ and $\vec{z} \in [-\infty, \infty)^{t_2^+ - t_2^- + 1}$ we have
\begin{equation}\label{eqn57S6}
\mathbb{P}^{t^-_2, t^+_2, x,y, \vec{z}}_{H,H^{RW}} \Big( {Z_{H,H^{RW}}}\big(  t_1^-, t_1^+, \ell(t_1^-) ,\ell(t^+_1), \ell_{bot}\llbracket t_1^-, t_1^+\rrbracket\big) \geq g    \Big) \geq h,
\end{equation}
where in the above equation the random variable over which we are taking the expectation is denoted by $\ell$ and $g,h$ are as in the statement of the lemma. We prove (\ref{eqn57S6}) in the steps below. Here we assume its validity and conclude the proof of the lemma.\\

Let $\ell_{bot}$ be as in the statement of the lemma and $\tilde{\ell}_{bot} \in [-\infty, \infty)^{t_2^+ - t_2^- + 1}$ be such that $\tilde{\ell}_{bot}(i) = \ell_{bot}(i)$ for $i \not \in \llbracket t_1^- + 1, t_1^+\rrbracket$ and $\tilde{\ell}_{bot}(i) = -\infty$ if $i \in \llbracket t^-_1 + 1, t^+_1\rrbracket$. Let $L$ be a $Y(\llbracket t_2^-, t_2^+ \rrbracket)$-valued random variable, whose law is given by $\mathbb{P}_L : = \mathbb{P}^{t_2^-, t_2^+, x,y, \ell_{bot}}_{H,H^{RW}}$ and $\tilde{L}$ be a $Y(\llbracket t_2^-, t_2^+ \rrbracket)$-valued random variable whose law is given by $\mathbb{P}_{\tilde{L}} : = \mathbb{P}^{t_2^-, t_2^+, x,y, \tilde{\ell}_{bot}}_{H,H^{RW}}$.

Define further $\mathbb{P}_{{L}'}$ and $\mathbb{P}_{\tilde{L}'}$ as the projection of $\mathbb{P}_L$ and $\mathbb{P}_{\tilde{L}}$, respectively, to the coordinates $\llbracket t_2^-, t_1^- \rrbracket \cup \llbracket t_1^+, t_2^+\rrbracket$. It follows from (\ref{RND}) that the Radon-Nikodym derivative between these two restricted measures is given on $Y( \llbracket t_2^-, t_1^- \rrbracket \cup \llbracket t_1^+, t_2^+\rrbracket)$-valued random variables $\fB$ by
\begin{equation}\label{eqn61}
\frac{d\mathbb{P}_{{L}'}}{d\mathbb{P}_{\tilde{{L}}'}} (\fB) = (Z')^{-1}Z_{H,H^{RW}}( t_1^-, t^+_1,\fB(t_1^-),\fB(t_1^+),\ell_{bot} \llbracket t_1^-, t_1^+ \rrbracket ) ,
\end{equation}
where $Z' = \mathbb{E}_{\tilde{L}'} \left[ Z_{H,H^{RW}}( t_1^-, t^+_1,\fB(t_1^-),\fB(t_1^+),\ell_{bot} \llbracket t_1^-, t_1^+ \rrbracket ) \right]$. Let us briefly explain why (\ref{eqn61}) holds.

Let us set $S =\llbracket t_2^-, t_1^- \rrbracket \cup \llbracket t_1^+, t_2^+\rrbracket$, and fix a Borel set $F_0 \subseteq Y( S)$. To prove (\ref{eqn61}) we only need
\begin{equation}\label{eqn61H1}
\mathbb{P}_{L'}( \fB \in F_0)  = \frac{\mathbb{E}_{\tilde{L}'} \left[Z_{H,H^{RW}}( t_1^-, t^+_1,\fB(t_1^-),\fB(t_1^+),\ell_{bot} \llbracket t_1^-, t_1^+ \rrbracket )  \cdot {\bf 1}\{ \fB \in F_0\}  \right] }{\mathbb{E}_{\tilde{L}'} \left[Z_{H,H^{RW}}( t_1^-, t^+_1,\fB(t_1^-),\fB(t_1^+),\ell_{bot} \llbracket t_1^-, t_1^+ \rrbracket )  \right]}.
\end{equation}
Using that $L'$ is the projection of $L$ to the set $S$ and (\ref{RND}), we have
\begin{equation*}
\begin{split}
&\mathbb{P}_{L'}( \fB \in F_0) = \mathbb{P}_L( \ell |_S \in F_0) = \frac{\mathbb{E}^{t_2^-, t_2^+,x,y}_{H^{RW}}\left[ W_H^{1,1,t_2^-, t_2^+, \infty, \ell_{bot}}(\ell) \cdot {\bf 1}\{ \ell |_S \in F_0\} \right]}{\mathbb{E}^{t_2^-, t_2^+,x,y}_{H^{RW}}\left[ W_H^{1,1,t_2^-, t_2^+, \infty, \ell_{bot}}(\ell)  \right]}.
\end{split}
\end{equation*}
Also, using that $\tilde{L}'$ is the projection of $\tilde{L}$ to the set $S$, $Z_{H,H^{RW}}(t_1^-, t^+_1,\fB(t^-_1),\fB(t^+_1),\ell_{bot} \llbracket t_1^-, t_1^+ \rrbracket )$ is a (deterministic) bounded measurable function of $\left( \fB(t_1^-),\fB(t_1^+)\right)$ (see Remark \ref{ZMeas}) and (\ref{RND}), we get
\begin{equation*}
\begin{split}
& \frac{\mathbb{E}_{\tilde{L}'} \left[Z_{H,H^{RW}}( t_1^-, t^+_1,\fB(t_1^-),\fB(t_1^+),\ell_{bot} \llbracket t_1^-, t_1^+ \rrbracket )  \cdot {\bf 1}\{ \fB \in F_0\}  \right] }{\mathbb{E}_{\tilde{L}'} \left[Z_{H,H^{RW}}( t_1^-, t^+_1,\fB(t_1^-),\fB(t_1^+),\ell_{bot} \llbracket t_1^-, t_1^+ \rrbracket )  \right]} = \\
& \frac{\mathbb{E}^{t_2^-, t_2^+,x,y}_{H^{RW}}\left[ W_H^{1,1,t_2^-, t_2^+, \infty, \tilde{\ell}_{bot}}(\ell) \cdot Z_{H,H^{RW}}( t_1^-, t^+_1,\ell(t_1^-),\ell(t_1^+),\ell_{bot} \llbracket t_1^-, t_1^+ \rrbracket ) \cdot {\bf 1}\{ \ell |_S \in F_0\} \right]}{\mathbb{E}^{t_2^-, t_2^+,x,y}_{H^{RW}}\left[ W_H^{1,1,t_2^-, t_2^+, \infty, \tilde{\ell}_{bot}}(\ell) \cdot Z_{H,H^{RW}}( t_1^-, t^+_1,\ell(t_1^-),\ell(t_1^+),\ell_{bot} \llbracket t_1^-, t_1^+ \rrbracket ) \right]}.
\end{split}
\end{equation*}
The last two equations show that to prove (\ref{eqn61H1}) it suffices to show that for any Borel set $F_1 \subseteq Y( S)$
\begin{equation}\label{eqn61H2}
\begin{split}
&\mathbb{E}^{t_2^-, t_2^+,x,y}_{H^{RW}}\left[ W_H^{1,1,t_2^-, t_2^+, \infty, \ell_{bot}}(\ell) \cdot {\bf 1}\{ \ell |_S \in F_1\} \right] = \\
&\mathbb{E}^{t_2^-, t_2^+,x,y}_{H^{RW}}\left[ W_H^{1,1,t_2^-, t_2^+, \infty, \tilde{\ell}_{bot}}(\ell) \cdot Z_{H,H^{RW}}( t_1^-, t^+_1,\ell(t_1^-),\ell(t_1^+),\ell_{bot} \llbracket t_1^-, t_1^+ \rrbracket ) \cdot {\bf 1}\{ \ell |_S \in F_1\} \right]
\end{split}
\end{equation}
Using (\ref{WH}), the definition of $\tilde{\ell}_{bot}$ and that $\mathbb{P}^{t_2^-, t_2^+,x,y}_{H^{RW}}$ satisfies the $(H,H^{RW})$-Gibbs property with $H \equiv 0$ (as follows from Lemma \ref{S4AltGibbs}), we have
\begin{equation*}
\begin{split}
&\mathbb{E}^{t_2^-, t_2^+,x,y}_{H^{RW}}\left[ W_H^{1,1,t_2^-, t_2^+, \infty, \ell_{bot}}(\ell) \cdot {\bf 1}\{ \ell |_S \in F_1\} \right] = \\
&\mathbb{E}^{t_2^-, t_2^+,x,y}_{H^{RW}}\left[ \mathbb{E}^{t_2^-, t_2^+,x,y}_{H^{RW}}\left[  W_H^{1,1,t_2^-, t_2^+, \infty, \ell_{bot}}(\ell) \cdot {\bf 1}\{ \ell |_S \in F_1\}\Big{\vert} \mathcal{F}_{ext}(\{1\} \times \llbracket t_1^-+1, t_1^+-1 \rrbracket) \right] \right] = \\
&\mathbb{E}^{t_2^-, t_2^+,x,y}_{H^{RW}}\Bigg[ {\bf 1}\{ \ell |_S \in F_1\} \cdot W_H^{1,1,t_2^-, t_2^+, \infty, \tilde{\ell}_{bot}}(\ell) \cdot \\
& \mathbb{E}^{t_2^-, t_2^+,x,y}_{H^{RW}}\left[  W_H^{1,1,t_1^-, t_1^+, \infty, \ell_{bot} \llbracket t_1^-, t_1^+ \rrbracket}(\ell \llbracket t_1^-, t_1^+\rrbracket)  \Big{\vert} \mathcal{F}_{ext}(\{1\} \times \llbracket t_1^-+1, t_1^+-1 \rrbracket) \right] \Bigg] = \\
&\mathbb{E}^{t_2^-, t_2^+,x,y}_{H^{RW}}\left[ {\bf 1}\{ \ell |_S \in F_1\} \cdot W_H^{1,1,t_2^-, t_2^+, \infty, \tilde{\ell}_{bot}}(\ell) \cdot \mathbb{E}^{t_1^-, t_1^+,\ell(t_1^-),\ell(t_1^+)}_{H^{RW}}\left[  W_H^{1,1,t_1^-, t_1^+, \infty, \ell_{bot} \llbracket t_1^-, t_1^+ \rrbracket}(\tilde{\ell}) \right] \right],
\end{split}
\end{equation*}
where $\ell$ is $\mathbb{P}^{t_2^-, t_2^+,x,y}_{H^{RW}}$-distributed and $\tilde{\ell}$ is $\mathbb{P}^{t_1^-, t_1^+,\ell(t_1^-),\ell(t_1^+)}_{H^{RW}}$-distributed. The last equality proves (\ref{eqn61H2}), once we use that from Definition \ref{DefLGGP} we have
\begin{equation}\label{away2}
Z_{H,H^{RW}}\left( t_1^-,t^+_1,\ell(t^-_1),\ell(t^+_1),\ell_{bot}\llbracket t_1^- , t_1^+ \rrbracket \right)  = \mathbb{E}_{H^{RW}}^{t_1^-, t_1^+, \ell(t_1^-), \ell(t^+_1)} \left[W_H(t^-_1,t^+_1, \tilde{\ell},\ell_{bot}\llbracket t_1^- , t_1^+ \rrbracket ) \right],
\end{equation}
where $\tilde{\ell}$ is $\mathbb{P}_{H^{RW}}^{t_1^-, t_1^+, \ell(t_1^-), \ell(t^+_1)}$-distributed. Overall, we conclude that (\ref{eqn61}) holds.

Now that we have (\ref{eqn61}) we turn back to the proof of the lemma. Observe that the law of $\left( \fB(t_1^-),\fB(t^+_1)\right)$ under $\mathbb{P}_{\tilde L'}$ is the same as the law of $(\tilde{L}(t^-_1), \tilde{L}(t^+_1))$ under $\mathbb{P}_{\tilde L}$ (this is because $\mathbb{P}_{\tilde L'}$ is the projection of $\mathbb{P}_{\tilde L}$ to a set containing $t^{\pm }_1$).  The latter and (\ref{eqn57S6}) together imply
$$Z' = \mathbb{E}_{\tilde{{L}}'} \left[ Z_{H,H^{RW}}( t_1^-,t_1^+,\fB(t^-_1),\fB(t^+_1),\ell_{bot} \llbracket t_1^-, t_1^+ \rrbracket) \right] = $$
$$ \mathbb{E}_{\tilde{{L}}} \left[ Z_{H,H^{RW}}(t_1^-, t_1^+,\tilde{L}(t_1^-),\tilde{L}(t^+_1),\ell_{bot} \llbracket t_1^-, t_1^+ \rrbracket ) \right] \geq g h.$$

Let us denote the set $E = \left\{ Z_{H,H^{RW}}( t_1^-, t_1^+,\fB(t^-_1),\fB(t^+_1),\ell_{bot}\llbracket t_1^-, t_1^+ \rrbracket ) \leq gh \tilde \epsilon) \right\}$. Then, we have
$$ \mathbb{P}_{L'} (E)= \mathbb{E}_{L'} \left[ {\bf 1}_E \right] = \mathbb{E}_{\tilde{L}'} \left[\frac{{\bf 1}_E \cdot  Z_{H,H^{RW}}( t^-_1, t^+_1,\fB(t^-_1),\fB(t^+_1),\ell_{bot}\llbracket t_1^-, t_1^+ \rrbracket)  }{Z'}\right] \leq \frac{gh \tilde \epsilon}{gh} =\tilde \epsilon,$$
where in the second equality we used (\ref{eqn61}). Finally, since $ Z_{H,H^{RW}}( t^-_1, t^+_1,\fB(t^-_1),\fB(t_1^+),\ell_{bot}\llbracket t_1^-, t_1^+ \rrbracket )$ is a bounded measurable function of $\left( \fB(t^-_1),\fB(t^+_1)\right)$ whose law under $\mathbb{P}_{L'}$ is the same as that of $(L(t^-_1), L(t^+_1))$ under $\mathbb{P}_{ L}$, we see that the above implies (\ref{eqn60}), concluding the proof of the lemma.\\

{\bf \raggedleft Step 2.} Define $F = \left\{\min\left( \ell(t_1^-) - p t^-_1, \ell(t^+_1)  -pt^+_1 \right)\geq (M_2 + 2)(t_1^+ - t_1^-)^{1/2}  \right\}$. We claim that for all $N$ sufficiently large and $\vec{z} \in [-\infty, \infty)^{t_2^+ - t_2^- + 1}$ we have
\begin{equation}\label{LAPeq4}
\mathbb{P}^{t_2^-, t_2^+, x,y, \vec{z}}_{H,H^{RW}}\left (F \right) \geq (1/18) \cdot \left(1 - \Phi^{v}\left(10(2 +r)^2(M_1 + M_2 + 10) \right) \right).
\end{equation}
Establishing the validity of (\ref{LAPeq4}) will be done in the third and fourth steps below, and in what follows we assume it is true and finish the proof of (\ref{eqn57S6}).\\

We assert that if $N_4$ is sufficiently large and $N \geq N_4$, we have
\begin{equation}\label{LAPeqS}
F \subset  \left \{  Z_{H,H^{RW}}\left( t_1^-,t^+_1,\ell(t^-_1), \ell(t^+_1),{\ell}_{bot}\llbracket t_1^-, t_1^+ \rrbracket \right) > \frac{1}{4} \left( 1 - \exp \left( \frac{-2}{\sigma_p^2}\right) \right)  \right\}.
\end{equation}
Observe that (\ref{LAPeqS}) and (\ref{LAPeq4}) prove (\ref{eqn57S6}) and so it suffices to verify (\ref{LAPeqS}). The details are presented below (see also Figure \ref{S6S6_1}).
\begin{figure}[h]
\centering
\scalebox{0.66}{\includegraphics{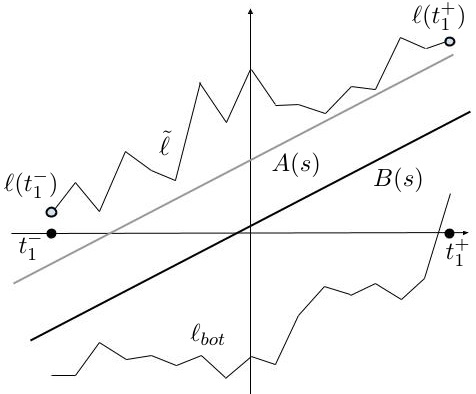}}
\caption{Overview of the arguments in Step 2: We want to prove that on the event $F$ we have a lower bound on $Z_{H,H^{RW}}:= Z_{H,H^{RW}}\left( t^-_1,t^+_1,\ell(t^-_1), \ell(t^-_1),{\ell}_{bot}\llbracket t_1^-, t_1^+ \rrbracket  \right) $. As explained in (\ref{away2}), the random variable $Z_{H,H^{RW}}$ is just the average of the weights $W_H(\tilde{\ell}) :=W_H(t_1^-,t_1^+, \tilde{\ell}, {\ell}_{bot}\llbracket t_1^-, t_1^+ \rrbracket)$ over a free $H^{RW}$ bridge $\tilde{\ell}$. Consequently, to show that $Z_{H,H^{RW}}$ is lower-bounded it suffices to find a big subset $\Omega' $, such that the weights $W_H(\tilde{\ell})$ on $ \Omega'$ are lower-bounded.
 Let $A(s)$ and $B(s)$ denote the lines $ps + (M_2 + 1)(t_1^+ - t_1^-)^{1/2}- (t_1^+- t_1^-)^{1/4}$ and $ps + M_2(t_1^+ - t_1^-)^{1/2}$, drawn in grey and black respectively above. Then $\Omega'$ denotes the event that curve $\tilde{\ell}$ lies above $A(s)$ on $[t_1^-, t_1^+]$. On the event $F$ we have that $\ell(t_1^{\pm})$ are at least a distance $(t_1^+ - t_1^-)^{1/2} + (t_1^+ - t_1^-)^{1/4}$ above the points $A( t^{\pm}_1)$, respectively. Since the endpoints of the bridges are well above those of $A(s)$, this means that some positive fraction of these bridges will stay above $A(s)$ on the entire interval $[t_1^-, t^+_1]$; i.e. $\mathbb{P}_{H^{RW}}^{t_1^-, t^+_1, \ell(t^-_1), \ell(t^+_1)} \left(\Omega' \right)  $ is lower bounded. This is what we mean by $\Omega'$ being big and the exact relation is given in (\ref{away3}).
To see that $W_H(\tilde{\ell})$ on $ \Omega'$ is lower bounded, we notice that on $ \Omega'$ the bridges $\tilde{\ell}$ are well-above $B(s)$, which dominates $\ell_{bot}$ by assumption. This means that $\tilde{\ell}$ is well above $\ell_{bot}$ and for such paths $W_H(\tilde{\ell})$  is lower bounded. The exact relation is given in (\ref{away4}).
}\label{S6S6_1}
\end{figure}

Denote
$$\Omega'  = \left \{  \tilde{\ell}(s) - ps \geq (M_2 + 1) (t_1^+ - t_1^-)^{1/2} - (t_1^+ - t_1^-)^{1/4} \mbox{ for } s \in \llbracket t_1^-, t_1^+\rrbracket \right\}.$$
 It follows from Lemma \ref{LemmaAwayS4} applied to $T = (t_1^+ - t_1^-)$, $x= \ell(t_1^-) - (M_2 + 1) (t_1^+ - t_1^-)^{1/2}$ and $y = \ell(t_1^+) - (M_2 + 1) (t_1^+ - t_1^-)^{1/2}$ that if $N_4$ is sufficiently large, so that for $N \geq N_4$, we have $t_1^+ - t_1^- \geq W_3(p)$ as in Lemma \ref{LemmaAwayS4}, then
\begin{equation}\label{away3}
{\bf 1}_F \cdot  \mathbb{P}_{H^{RW}}^{t^-_1, t^+_1, \ell(t^-_1), \ell(t^+_1)} \left(\Omega' \right)  \geq {\bf 1}_F \cdot  \frac{1}{2} \left( 1 - \exp \left( \frac{-2}{\sigma_p^2}\right)\right).
\end{equation}
In deriving the above equation we also used the shift-invariance of $\mathbb{P}_{H^{RW}}^{t^-_1, t^+_1, \ell(t^-_1), \ell(t^+_1)}$, as well as the definitions of $\Omega'$ and $F$.

Since $(t_1^+ - t_1^-) \geq  N^{\alpha} $, we know that for $N_4$ sufficiently large (depending on $r, \alpha$) and $N \geq N_4$ we have on $ \Omega'$ for all $s \in \llbracket t_1^-, t_1^+\rrbracket$ that
$$\tilde{\ell}(s) - ps \geq (M_2 + 1/2)(t_1^+ - t_1^-)^{1/2}  \geq \ell_{bot}(s) - ps + (1/2)(t_1^+ - t_1^-)^{1/2},$$
 where the last inequality holds true by our assumption on $\ell_{bot}$. The conclusion is that on $ \Omega'$, we have for $s \in \llbracket t_1^-, t_1^+ \rrbracket$ that $\tilde{\ell}(s) - \ell_{bot}(s) \geq m$, where $m =  (1/2) N^{\alpha/2}$. Using that $H$ is increasing and $\lim_{x \rightarrow \infty} x^2H(-x) = 0$, we have that on the event $\Omega'$  the following holds
\begin{equation}\label{away4}
W_H(t_1^-,t_1^+, \tilde{\ell}, \ell_{bot}\llbracket t_1^-, t_1^+ \rrbracket ) \geq e^{- (t_1^+ - t_1^-) H(-m)} \geq e^{-H(-N^{\alpha/2}/2) \cdot (2r + 4)N^{\alpha}} \geq \frac{1}{2},
\end{equation}
where the last inequality holds for all large enough $N_4$ (depending on $r, \alpha$ and $H$) and $N \geq N_4$. Combining (\ref{away2}), (\ref{away3}) and (\ref{away4}), we conclude that  if $N_4$ is sufficiently large and $N \geq N_4$, then on the event $F$ we have
$$Z_{H,H^{RW}}\left( t^-_1,t^+_1,\ell(t^-_1),\ell(t^+_1),\ell_{bot} \llbracket t_1^-, t_1^+ \rrbracket \right) \geq  \mathbb{E}_{H^{RW}}^{t^-_1, t^+_1, \ell(t^-_1), \ell(t^+_1)} \left[{\bf 1}_{\Omega'} \cdot W_H(t^-_1,t^+_1, \tilde{\ell}, \ell_{bot}\llbracket t_1^-, t_1^+ \rrbracket) \right] $$
$$ \geq \frac{1}{2} \cdot  \mathbb{P}_{H^{RW}}^{t^-_1, t^+_1, \ell(t^-_1), \ell(t^+_1)} \left(\Omega' \right) \geq \frac{1}{4} \left( 1 - \exp \left( \frac{-2}{\sigma_p^2}\right) \right),$$
which establishes (\ref{LAPeqS}).\\

{\bf \raggedleft Step 3.} In this step, we prove (\ref{LAPeq4}). We refer the reader to Figure \ref{S6_2} for an overview of the main ideas in this and the next step and a graphical representation of the notation we use.

\begin{figure}[h]
\centering
\scalebox{0.7}{\includegraphics{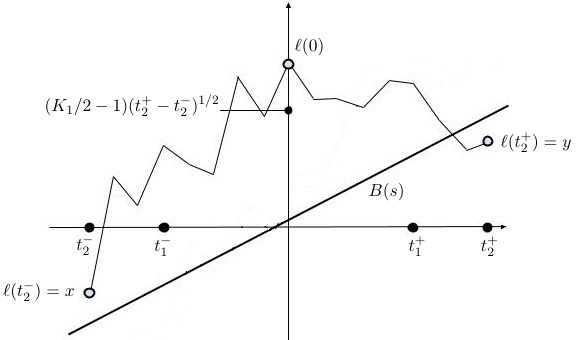}}
\caption{Overview of the arguments in Steps 3 and 4:
 Let $B(s)$ denote the line $ps + M_2(t_1^+ - t_1^-)^{1/2}$, drawn in black above. We have that $F$ denotes the event that $\ell$ is at least a distance $ 2 (t_1^+ - t_1^-)^{1/2}$ above the line $B(s)$ at the points $t^{\pm}_1$ and we want to find a lower bound on $\mathbb{P}^{t_2^-, t_2^+, x,y, \vec{z}}_{H,H^{RW}}(F)$. Using Lemma \ref{MonCoup}, it suffices to lower bound $\mathbb{P}^{t_2^-, t_2^+, x,y}_{H^{RW}}(F \cap E)$ where $E = \{ \ell(0) \geq B(0) + (K_1/2 - 1)(t_2^+ - t_2^-)^{1/2}\}$ for some suitably large $K_1$, depending on $M_1,M_2, r$ -- this reduction is made in (\ref{BIGS53}). One utilizes the fact that conditional on $\ell(0)$ the events that $\{ \ell(t_1^{\pm}) \geq B(t^{\pm}_1) + 2(t_1^+ - t_1^-)^{1/2}\}$ become independent and themselves lower bounded. The lower bound for the latter events is a consequence of the fact that $x,y$ are not too low while on $E$ the variable $\ell(0)$ is very high, which makes $\ell(t_1^{\pm})$ also high with at least probability $1/3$ as follows from an application of Lemma \ref{LemmaHalfS4}. The conditioning described in this paragraph is rephrased in terms of the $(H,H^{RW})$-Gibbs property in (\ref{SeqStatements}). The necessary lower bound of the event that $\ell(0)$ is much higher than $B(0)$ is the first line in (\ref{Redeq1}). The necessary statement required to establish that $\ell(t_1^{\pm})$ are high at least with probability $1/3$ if $\ell(0)$ is very high is the second line in (\ref{Redeq1}). The proof of (\ref{Redeq1}) is the content of Step 4 and essentially follows from Lemmas \ref{LemmaTailS4} and \ref{LemmaHalfS4}.
}\label{S6_2}
\end{figure}

Let $K_1 = 8(2+r)^2(M_1 + M_2 + 10)$ and define the events
$$E =  \left\{ \ell(0) \geq  (K_1/2-1)  (t_2^+ - t_2^-)^{1/2}  \right\},$$
$$E_1  =\left\{ \ell(t_1^-) - pt^-_1  \geq (M_2 + 2) (t_1^+ - t_1^-)^{1/2} \right\} \mbox{, } E_2 = \left\{ \ell(t^+_1) - pt^+_1  \geq(M_2 + 2)  (t_1^+ - t_1^-)^{1/2} \right\}.$$
We assert that if $x \geq  p t_2^- - M_1(t_2^+ -  t_2^-)^{1/2}$, $y \geq pt_2^+ - M_1(t_2^+ - t_2^-)^{1/2}$, $z \geq (K_1/2 - 1) (t_2^+ - t_2^-)^{1/2}$
\begin{equation}\label{Redeq1}
\begin{split}
&\mathbb{P}^{t_2^-, t_2^+, x,y}_{H^{RW}} \left( E \right) \geq (1/2) \cdot [1 - \Phi^v(M_1 + K_1)],\\
& \mathbb{P}^{t_2^-,0,x,z}_{H^{RW}}\left( E_1 \right) \geq 1/3 , \hspace{2mm} \mathbb{P}_{H^{RW}}^{0, t_2^+, z, y} \left (E_2\right) \geq 1/3.
\end{split}
\end{equation}
We will prove (\ref{Redeq1}) in Step 4 below. Here we assume its validity and conclude the proof of (\ref{LAPeq4}). \\

In view of Lemma \ref{MonCoup}, we see that to prove (\ref{LAPeq4}) it suffices to show that for all large $N$
\begin{equation}\label{BIGS53}
\mathbb{P}^{t_2^-, t_2^+, x,y}_{H^{RW}}  (E_1 \cap E_2) \geq  (1/18)  [1 - \Phi^v(M_1 + K_1)].
\end{equation}
From Lemma \ref{S4AltGibbs} we know that $\mathbb{P}^{t_2^-, t_2^+, x,y}_{H^{RW}}$ satisfies the $(H,H^{RW})$-Gibbs property (here we use the second part of the lemma with $\vec{z} = (-\infty)^{t_2^+ - t_2^-}$).
Let $\mathcal{F}^+_{ext} = \mathcal{F}_{ext} \left( \{1 \} \times \llbracket 1, t_2^+ - 1 \rrbracket \right)$ and $\mathcal{F}^-_{ext} = \mathcal{F}_{ext} \left( \{1 \} \times \llbracket t_2^- + 1,  - 1 \rrbracket \right)$ be as in (\ref{GibbsCond}). Then, we have the following sequence of statements
\begin{equation}\label{SeqStatements}
\begin{split}
&\mathbb{P}^{t_2^-, t_2^+, x,y}_{H^{RW}}  (E_1 \cap E_2 \cap E) = \mathbb{E} \left[\mathbb{E}\left[ \mathbb{E}\left[ {\bf 1}_E \cdot {\bf 1}_{E_1} \cdot {\bf 1}_{E_2} \vert \mathcal{F}_{ext}^-\right]  \vert \mathcal{F}_{ext}^+\right] \right] = \\
&\mathbb{E} \left[{\bf 1}_E \cdot \mathbb{E}\left[ {\bf 1}_{E_2} \cdot  \mathbb{E}\left[   {\bf 1}_{E_1} \vert \mathcal{F}_{ext}^-\right]  \vert \mathcal{F}_{ext}^+\right] \right] = \mathbb{E} \left[{\bf 1}_E \cdot \mathbb{E}\left[ {\bf 1}_{E_2}\cdot  \mathbb{E}^{- } \left[ {\bf 1}_{E_1}  \right]  \vert \mathcal{F}_{ext}^+ \right] \right] = \\
& \mathbb{E} \left[{\bf 1}_E \cdot \mathbb{E}^{- } \left[ {\bf 1}_{E_1}  \right]   \cdot  \mathbb{E}\left[ {\bf 1}_{E_2}  \vert \mathcal{F}_{ext}^+ \right] \right] =  \mathbb{E} \left[{\bf 1}_E \cdot \mathbb{E}^{- } \left[ {\bf 1}_{E_1}  \right]   \cdot  \mathbb{E}^+\left[ {\bf 1}_{E_2} \right] \right] \geq \\
& \mathbb{E} \left[{\bf 1}_E \cdot (1/3)  \cdot (1/3)\right]  \geq (1/18) \cdot [1 - \Phi^v(M_1 + K_1)],
\end{split}
\end{equation}
where we have written $\mathbb{E}$ in place of $\mathbb{E}^{t_2^-, t_2^+, x,y}_{H^{RW}}$, $\mathbb{E}^-$ in place of $\mathbb{E}_{H^{RW}}^{x, \ell(0), t_2^-, 0}$ and $\mathbb{E}^+$ in place of $\mathbb{E}_{H^{RW}}^{\ell(0), y, 0, t_2^+}$ to ease the notation.

Let us explain (\ref{SeqStatements}) briefly. The first equality follows from the tower property for conditional expectations. The second equality follows from the fact that ${\bf 1}_E$ is measurable with respect to both $\mathcal{F}_{ext}^{+}$ and $\mathcal{F}^-_{ext}$, while ${\bf 1}_{E_2}$ is measurable with respect to $\mathcal{F}^-_{ext}$, and so these functions can be taken outside of the conditional expectations. The first equality in the second line follows from an application of the $(H,H^{RW})$-Gibbs property (\ref{GibbsEq}) to the function $F(\ell) = {\bf 1} \left\{ \ell(t_1^-) - pt^-_1  \geq (M_2 + 2) (t_1^+ - t_1^-)^{1/2} \right\}.$ We next observe that $\mathbb{E}^{- } \left[ {\bf 1}_{E_1}  \right]$ is a deterministic measurable function of $\ell(0)$ (see also Lemma \ref{ContinuousGibbsCond}) and so in particular it is $\mathcal{F}^+_{ext}$ measurable. This allows us to move $\mathbb{E}^{- } \left[ {\bf 1}_{E_1}  \right]$ outside of the conditional expectation, which explains the second equality on the second line of (\ref{SeqStatements}). The first equality on the third line follows from an application of the $(H,H^{RW})$-Gibbs property (\ref{GibbsEq}) to the function $F(\ell) = {\bf 1} \left\{ \ell(t_1^+) - pt^+_1  \geq (M_2 + 2) (t_1^+ - t_1^-)^{1/2} \right\}.$ Next, we lower bound $\mathbb{E}^{- } \left[ {\bf 1}_{E_1}  \right]   \cdot  \mathbb{E}^+\left[ {\bf 1}_{E_2} \right]$  by $(1/3)\cdot (1/3)$ in view of the second line of (\ref{Redeq1}), since on $E$ we know that $\ell(0)$ (which plays the role of $z$ in (\ref{Redeq1})) is lower bounded by $(K_1/2-1)  (t_2^+ - t_2^-)^{1/2}$. This explains the inequality on the third line of (\ref{SeqStatements}). Finally, the inequality on the fourth line of (\ref{SeqStatements}) follows from the first line in  (\ref{Redeq1}). This justifies equation (\ref{SeqStatements}). Also, it is clear that (\ref{SeqStatements}) implies (\ref{BIGS53}). \\

{\bf \raggedleft Step 4.} In this final step, we establish (\ref{Redeq1}). By Lemma \ref{LemmaTailS4}, applied to $T=  t_2^+ - t_2^-$, $M_1$ as above, $M_2 = K_1$ as above, $p$ as in (\ref{blah}), we have for all sufficiently large $N$, so that $t_2^+ - t_2^- \geq W_2(M_1, K_1,p)$ as in  Lemma \ref{LemmaTailS4}, $\tilde{x} \geq -M_1 (t_2^+ - t_2^-)^{1/2}$, $\tilde{y} \geq p (t_2^+ - t_2^-) - M_1 (t_2^+ - t_2^-)^{1/2}$ and $\rho \in \{-1, 0, 1\}$ that
\begin{equation*}
\mathbb{P}^{0,T,\tilde{x},\tilde{y}}_{H^{RW}}\bigg( \ell(\lfloor T/2\rfloor  + \rho)  \geq \frac{K_1T^{1/2} + p T}{2} - T^{1/4} \bigg) \geq (1/2) (1 - \Phi^{v}(M_1 + K_1) ).
\end{equation*}
By assumption we have that $x \geq p t_2^- - M_1 (t_2^+ - t_2^-)^{1/2}$ and $y \geq pt_2^+ - M_1 (t_2^+ - t_2^-)^{1/2}$. This means that $\tilde{x} = x - pt_2^- \geq - M_1 (t_2^+ - t_2^-)^{1/2}$ and $\tilde{y} = y - pt_2^- \geq p (t_2^+ - t_2^-) - M_1 (t_2^+ - t_2^-)^{1/2}$, which imply from the translation invariance of $\mathbb{P}^{0,T,\tilde{x},\tilde{y}}_{H^{RW}}$ and the above inequality that
$$
\mathbb{P}^{t_2^-, t_2^+, x,y}_{H^{RW}} \left(\ell(0) \geq (K_1/2 - 1) (t_2^+ - t_2^-)^{1/2} \right) \geq (1/2) \cdot [1 - \Phi^v(M_1 + K_1)].$$
This proves the first line in (\ref{Redeq1}).\\

From Lemma \ref{LemmaHalfS4}, applied to $T = -t_2^-$, $M_1 = \tilde{M}_1 := (-2-2M_1)$ and $M_2 = \tilde{M}_2 := (K_1/2 - 1)$, we know that for $\tilde{x} \geq pt^-_2 + \tilde{M}_1 |t_2^-|^{1/2}$ , $\tilde{z} \geq  \tilde{M}_2 |t_2^-|^{1/2}$ and $N$ sufficiently large so that $-t_2^- \geq W_2(\tilde{M}_2  - \tilde{M}_1 ,p)$ we have for any $s \in \llbracket 0, T \rrbracket$ that
\begin{equation*}
\mathbb{P}^{0,T,\tilde{x},\tilde{z}}_{H^{RW}}\Big( \ell(s)  \geq \frac{T-s}{T} \cdot \tilde{M}_1 T^{1/2} + \frac{s}{T} \cdot \big(p T + \tilde{M}_2  T^{1/2}\big) - T^{1/4} \Big) \geq \frac{1}{3}.
\end{equation*}
The latter and the translation invariance of $\mathbb{P}^{0,T,\tilde{x},\tilde{y}}_{H^{RW}}$ shows that if $x \geq pt_2^- - M_1(t_2^+ - t_2^-)^{1/2}$ and $z \geq \tilde{M}_2  (t_2^+ - t_2^-)^{1/2}$ then for any $s \in \llbracket  0, -t_2^- \rrbracket$ we have
$$\mathbb{P}^{t_2^-,0,x,z}_{H^{RW}}\left( \ell(s + t_2^-) - pt_2^-  \geq \frac{-t_2^- - s}{-t_2^-} \cdot \tilde{M}_1 |t_2^-|^{1/2} + \frac{ s}{-t_2^-} \cdot [ p(-t_2^-) + \tilde{M}_2  |t_2^-|^{1/2}] - |t_2^-|^{1/4} \right) \geq 1/3.$$
Setting $s = -t_2^- + t_1^-$ above, we conclude that if $x \geq pt_2^- - M_1(t_2^+ - t_2^-)^{1/2}$ and $z \geq \tilde{M}_2 (t_2^+ - t_2^-)^{1/2}$
\begin{equation*}
\mathbb{P}^{t_2^-,0,x,z}_{H^{RW}}\left( \ell(t_1^-) - pt_1^-  \geq \frac{- t_1^-}{-t_2^-} \cdot \tilde{M}_1 |t_2^-|^{1/2} + \frac{-t_2^- + t_1^-}{-t_2^-} \cdot \tilde{M}_2  |t_2^-|^{1/2} - |t_2^-|^{1/4} \right) \geq 1/3.
\end{equation*}
In particular, from our definition of $\tilde{M}_1$ and $\tilde{M}_2$ we conclude that for $x \geq  p t_2^- - M_1(t_2^+ -  t_2^-)^{1/2}$ and $z \geq (K_1/2 - 1) (t_2^+ - t_2^-)^{1/2}$ we have
\begin{equation*}
\mathbb{P}^{t_2^-,0,x,z}_{H^{RW}}\left( \ell(t_1^-) - pt_1^-  \geq (M_2  + 2)(t_1^+ - t_1^-)^{1/2} \right) \geq 1/3.
\end{equation*}
Similar arguments show that if $y \geq pt_2^+ - M_1(t_2^+ - t_2^-)^{1/2}$ and $z \geq (K_1/2 - 1) (t_2^+ - t_2^-)^{1/2}$, we have
\begin{equation*}
 \mathbb{P}_{H^{RW}}^{0, t_2^+, z, y} \left (\ell(s_1) - ps_1  \geq(M_2 + 2) (t_1^+ - t_1^-)^{1/2} \right) \geq 1/3.
\end{equation*}
The last two equations now imply the second line in (\ref{Redeq1}), which concludes its proof.

%

\section{Absolute continuity with respect to Brownian bridges }\label{Section7}

In Theorem \ref{PropTightGood} we showed that under suitable shifts and scalings $(\alpha, p, r+3)$--good sequences give rise to tight sequences of continuous random curves. In this section, we aim to obtain some qualitative information about their subsequential limits and we show that any subsequential limit is absolutely continuous with respect to a Brownian bridge with appropriate variance. In particular, this demonstrates that we have non-trivial limits and do not kill fluctuations with our rescaling. In Section \ref{Section7.1}, we introduce some useful notation and present the main result of the section -- Theorem \ref{ACBB}. The proof of Theorem \ref{ACBB} is given in Section \ref{Section7.2}, and relies on Proposition \ref{PropMain}, and the strong coupling afforded by Proposition \ref{KMT}.

%
\subsection{Formulation of result and applications}\label{Section7.1} We introduce some relevant notation and define what it means to be absolutely continuous with respect to a Brownian bridge.

\begin{definition} Let $X = C([0,1])$ and $Y = C([-r,r])$ be the spaces of continuous functions on $[0,1]$ and $[-r,r]$ respectively with the uniform topology. Denote by $d_X$ and $d_Y$ the metrics on the two spaces, and by $\mathcal{B}(X)$, and $\mathcal{B}(Y)$ their Borel $\sigma$-algebras. Given $z_1,z_2 \in \mathbb{R}$, we define
$F_{z_1, z_2} : X \rightarrow Y$ and $G_{z_1,z_2}: Y \rightarrow X$ by
\begin{equation}\label{G1}
[F_{z_1,z_2} (g)] (x) =  z_1 + g \left( \frac{x + r}{2r} \right) + \frac{x + r}{2r} (z_2 - z_1) \hspace{3mm} [G_{z_1,z_2} (h) ](\xi) =  h \left( 2r\xi - r \right) -z_1 - (z_2 - z_1)  \xi,
\end{equation}
for $x \in [-r,r]$ and $\xi \in [0,1]$.
\end{definition}

One observes that $F_{z_1,z_2}$ and $G_{z_1,z_2}$ are bijective homomorphisms between $X$ and $Y$ that are mutual inverses. In words, $F_{z_1,z_2}$ gives a general affine way to stretch a continuous function on $[-r,r]$ to one on $[0,1]$, and $G_{z_1,z_2}$ is its inverse. 

Let $X_0 = \{ f \in X : f(0) = f(1) = 0\}$, with the subspace topology, and define $G: Y \rightarrow X$ through $G(h) = G_{h(-r), h(r)} ( h)$. Let us make some observations.
\begin{enumerate}
\item $G$ is a continuous function. Indeed, from the triangle inequality we have \\$d_X \left(G_{h_1(-r), h_1(r) }(h_1),  G_{h_2(-r), h_2(r) }(h_2) \right) \leq 2d_Y(h_1, h_2).$
\item If $L$ is a random variable in $(Y, \mathcal{B}(Y))$ then $G(L)$ is a random variable in $(X, \mathcal{B}(X))$, which belongs to $X_0$ with probability $1$. The measurability of $G(L)$ follows from the continuity of $G$, everything else is clearly true.
\end{enumerate}
The function $G$ above gives a way to affinely map a general curve on $[-r,r]$ (where our usual variables lie) to $X_0$ (where a Brownian bridge resides) so that an honest comparison between the two can be made. Recall from Section \ref{Section4.3} that $B^\sigma$ stands for the Brownian bridge on $[0,1]$, with variance $\sigma^2$ -- this is a random variable in $(X, \mathcal{B}(X))$, which belongs to $X_0$ with probability $1$.

With the above notation we make the following definition.
\begin{definition}\label{DACB}
Let $L$ be a random variable in $(Y, \mathcal{B}(Y))$ with law $\mathbb{P}_L$. We say that $L$ is {\em absolutely continuous } with respect to a Brownian bridge with variance $\sigma^2$ if for any $K \in \mathcal{B}(X)$ we have
$$\mathbb{P}(B^\sigma \in K) = 0 \implies \mathbb{P}_L(G(L) \in K) = 0.$$
\end{definition}

The main result of this section is as follows.
\begin{theorem}\label{ACBB}
Under the same assumptions and notation as in Theorem \ref{PropTightGood} let $\mathbb{P}_\infty$ be any subsequential limit of $\mathbb{P}_N$. If $f_\infty$ has law $\mathbb{P}_\infty$, then it is absolutely continuous with respect to a Brownian bridge with variance $2r\sigma_p^2$ in the sense of Definition \ref{DACB}, where $\sigma_p^2$ is as in Definition \ref{AssHR}.
\end{theorem}

%
\subsection{Proof of Theorem \ref{ACBB}}\label{Section7.2} In this section, we give the proof of Theorem \ref{ACBB}, which for clarity is split into four steps. Before we go into the main argument we introduce some useful notation and give an outline of our main ideas.\\

Throughout we assume we have the same notation as in the statement of Theorem \ref{PropTightGood}, as well as the notation from Section \ref{Section7.1} above.  Since $\mathbb{P}_\infty$ is a subseqential limit of $\mathbb{P}_N$, we know that we can find an increasing sequence $N_j$, such that $\mathbb{P}_{N_j}$ weakly converge to $\mathbb{P}_\infty$. By Skorohod's embedding theorem (see e.g. \cite[Theorem 3.30]{Ka}) we can find a probability space $(\Omega^1, \mathcal{F}^1, \mathbb{P}^1)$, on which are defined random variables $\tilde{f}_{N_j}$ and $\tilde{f}_\infty$ that take values in $(Y, \mathcal{B}(Y))$, such that the laws of $\tilde{f}_{N_j}$ and $\tilde{f}_{\infty}$ are $\mathbb{P}_{N_j}$ and $\mathbb{P}_\infty$, respectively, and such that $d_Y\left(\tilde{f}_{N_j}(\omega^1),\tilde{f}_{\infty}(\omega^1) \right) \rightarrow 0 $ as $j \rightarrow \infty$ for each $\omega^1 \in \Omega^1$.

We consider a probability space $(\Omega^2, \mathcal{F}^2, \mathbb{P}^2)$, on which we have defined the original $(\alpha, p, r+3)$--good sequence $\big\{\mathfrak{L}^N = (L^N_1,L^N_2)\big\}_{N=1}^{\infty}$ and so
$$f_N(s) = N^{-\alpha/2}(L_1^N(sN^{\alpha}) - p s N^{\alpha}), \mbox{ for $s\in [-r,r]$}$$
has law $\mathbb{P}_{N}$ for each $N \geq N_0$ as in Definition \ref{def:intro}. Let us briefly explain the difference between $\mathbb{P}^1$ and $\mathbb{P}^2$ and why we need both. The space  $(\Omega^1, \mathcal{F}^1, \mathbb{P}^1)$ carries the random variables $\tilde{f}_{N_j}$ of law $\mathbb{P}_{N_j}$ and what is crucial is that the latter converge {\em almost surely} to $\tilde{f}_\infty$, whose law is $\mathbb{P}_\infty$. The space $(\Omega^2, \mathcal{F}^2, \mathbb{P}^2)$ carries the {\em entire} discrete line ensembles  $\mathfrak{L}^N = (L_1^N, L_2^N)$ (and not just the top curve), which is needed to apply the $(H,H^{RW})$-Gibbs property. \\

At this time we give a brief outline of the steps in our proof. In the first step, we fix $K \in \mathcal{B}(X)$ such that $\mathbb{P}(B^{\sqrt{2r}\sigma_p} \in K) = 0$ and find an open set $O$, which contains $K$, and such that $B^{\sqrt{2r}\sigma_p} $ is {\em extremely} unlikely to belong to $O$. Our goal is then to show that $G(\tilde{f}_\infty)$ is also unlikely to belong to $O$, the exact statement is given in (\ref{Z3}) below. Using that $O$ is open and that $\tilde{f}_{N_j}$ converge to $\tilde{f}_\infty$ almost surely, we can reduce our goal to showing that it is unlikely that $G(\tilde{f}_{N_j})$ belongs to $O$ {\em and } $\tilde{f}_{N_j}$ is at least a small distance away from the complement of $G^{-1}(O)$ for large $j$. Our gain from the almost sure convergence is that  we have bounded ourselves away from $G^{-1}(O)^c$, and by performing small perturbations we do not leave $G^{-1}(O)$. As the laws of $\tilde{f}_{N_j}$ and $f_{N_j}$ are the same we can switch from $(\Omega^1, \mathcal{F}^1, \mathbb{P}^1)$ to $(\Omega^2, \mathcal{F}^2, \mathbb{P}^2)$, reducing the goal to showing that it is unlikely that $G(f_{N})$ belongs to $O$ and $f_{N}$ is at least a small distance away from $G^{-1}(O)^c$ for large $N$. The exact statement is given in (\ref{Z4}) and the reduction happens in Step 2. The benefit of this switch is that we can use the $(H,H^{RW})$-Gibbs property from Section \ref{Section4.1} in $(\Omega^2, \mathcal{F}^2, \mathbb{P}^2)$ as the latter carries an entire line ensemble.

In Step 3 we apply the $(H,H^{RW})$-Gibbs property and reduce the proof to showing that it is unlikely that a certain rescaled $H^{RW}$-random walk bridge with well-behaved end-points is in $G^{-1}(O)$, and is at least a small distance away from $G^{-1}(O)^c$ for large $N$. The exact statement is given in (\ref{ZV1}). In Step 4, we prove (\ref{ZV1}) by approximating the rescaled $H^{RW}$-random walk bridge by a Brownian bridge using Proposition \ref{KMT}. Since we are bounded a small distance from $G^{-1}(O)^c$, the error in the approximation asymptotically does not matter and we are left with showing that a Brownian bridge is unlikely to be in $G^{-1}(O)$, which is true by the way $O$ is defined.

We now turn to the proof of the theorem.\\

{\raggedleft {\bf Step 1.}} Suppose that $K \in \mathcal{B}(X)$ is given, such that $\mathbb{P}(B^{\sqrt{2r} \sigma_p} \in K) = 0$. We wish to show that
\begin{equation}\label{Z1}
\mathbb{P}^1\left( G(\tilde{f}_\infty) \in K \right) = 0.
\end{equation}
Let $\epsilon \in (0,1)$ be given, and note that by Proposition \ref{PropMain} and Lemma \ref{PropSup}, we can find $\delta \in (0,1)$, $ M > 0$ and $N_1 \geq N_0$ (here $N_0$ is as in Definition \ref{def:intro}), such that for all $N \geq N_1$ we have
\begin{equation}\label{Z2}
\begin{split}
\mathbb{P}^2\left( E(\delta, M, N) \right) < \epsilon, \mbox{ where } &E(\delta, M, N) = \Big\{ \max_{j \in \{+, -\}} \left| L^N_1(t_1^j) - p t_1^j \right| \geq   MN^{\alpha/2} \Big\}\cup \\
&    \Big\{  Z_{H,H^{RW}}( t_1^-,t_1^+,L_1^N(t_1^-),L_1^N(t_1^+), L_2^N \llbracket t_1^- , t_1^+ \rrbracket) < \delta   \Big\} ,
\end{split}
\end{equation}
where we recall from (\ref{eqsts}) that $t^{\pm}_1 =\lfloor \pm (r+1) N^{\alpha} \rfloor$.
We observe that since $C([-r,r])$ is a metric space we have by \cite[Theorem II.2.1]{Parth}  that the measure of $B^{\sqrt{2r}\sigma_p}$ is outer-regular. In particular, we can find an open set $O$ such that $K \subset O$ and
\begin{equation}\label{ZO}
\mathbb{P}(B^{\sqrt{2r}\sigma_p} \in O) < \epsilon \delta/2.
\end{equation}
The set $O$ will not be constructed explicitly and we will not require other properties from it other than it is open and contains $K$. We will show that
\begin{equation}\label{Z3}
 \mathbb{P}^1\left( G(\tilde{f}_\infty) \in O \right) \leq 2\epsilon.
\end{equation}
Notice that the above implies that $\mathbb{P}^1\left( G(\tilde{f}_\infty) \in K \right) \leq 2\epsilon$ and, hence, we have reduced the proof of the theorem to establishing (\ref{Z3}).\\

{\raggedleft {\bf Step 2.}} Our goal in this step is to reduce (\ref{Z3}) to a statement involving finite indexed curves.

We first observe that $G^{-1}(O)$ is open since $G$ is continuous (see Section \ref{Section7.1}) and so
\begin{equation*}
\begin{split}
 &\mathbb{P}^1\left( G(\tilde{f}_\infty) \in O \right) = \mathbb{P}^1\left( \tilde{f}_\infty \in G^{-1}(O) \right)  \leq \limsup_{j \rightarrow \infty} \mathbb{P}^1 \hspace{-0.5mm}\left( \hspace{-0.5mm} \left\{\hspace{-0.5mm}\tilde{f}_{N_j} \in G^{-1}(O) \hspace{-0.5mm} \right\} \hspace{-0.5mm} \cap \hspace{-0.5mm} \left\{ \hspace{-0.5mm} d_Y(\tilde{f}_{N_j}, G^{-1}(O)^c) > \rho_j\hspace{-0.5mm} \right\} \hspace{-0.5mm}\right)\hspace{-1mm},
\end{split}
\end{equation*}
where $\rho_j$ is any sequence that converges to $0$  as $j \rightarrow \infty$. The first equality is by definition. The second one follows from the fact that $\tilde{f}_{N_j}$ converge to $\tilde{f}_\infty$  in the uniform topology $\mathbb{P}^1$-almost surely and that $G^{-1}(O)$ is open. To be more specific, we take $\rho_j = N_j^{-\alpha/8}$ for the sequel.

Since $f_N$ has law $\mathbb{P}_{N}$ for each $N \geq N_1$, we observe that to get (\ref{Z3}) it suffices to show that
\begin{equation}\label{Z4}
\limsup_{N \rightarrow \infty}  \mathbb{P}^2\left(  \left\{ f_{N} \in G^{-1}(O) \right\} \cap \left \{ d_Y(f_{N}, G^{-1}(O)^c) > N^{-\alpha/8} \right\} \right) \leq 2\epsilon .
\end{equation}

{\raggedleft {\bf Step 3.}} For $N \geq N_1$ as in Step 1 and $\ell \in C([t_1^-, t_1^+])$ we define $h_N(\cdot ;\ell) \in C([-r, r])$  through
$$h_N(s;\ell) = N^{-\alpha/2} (\ell(sN^{\alpha}) - ps N^{\alpha}).$$
 In this notation we have that $h_N(L_1^N [t_1^-, t_1^+ ])$ has the same distribution as $f_N$. Recall that $ L_1^N [t_1^-, t_1^+ ]$ is the restriction of the continuous interpretation of $L^N_1$ to the interval $[t_1^- ,t_1^+]$ as in Section \ref{Section4.1}.

We now claim that we can find $N_2 \in \mathbb{N}$ sufficiently large, so that $N_2 \geq N_1$, and if $N \geq N_2$ and $x,y \in \mathbb{R}$ satisfy $\max( |x- p t_1^-|, |y - pt_1^+|) \leq M N^{\alpha/2}$, then
\begin{equation}\label{ZV1}
\mathbb{E}^{t_1^-, t_1^+,x,y }_{H^{RW}}\left[ g_N(h_N(\ell))  \right] \leq \delta \epsilon, \mbox{ where }g_N(h) = {\bf 1} \left\{ h \in G^{-1}(O) \right\} \cdot {\bf 1} \left \{ d_Y(h, G^{-1}(O)^c) > N^{-\alpha/8}\right\}
\end{equation}
and on the left $\ell$ is a $\mathbb{P}_{H^{RW}}^{t_1^-, t_1^+,x,y }$-distrubuted random curve. We will prove (\ref{ZV1}) in the next step. Here we assume its validity and conclude the proof of (\ref{Z4}).\\

By the $(H,H^{RW})$-Gibbs property we can deduce the following statements for $N \geq N_2$
\begin{equation}\label{ZV2}
\begin{split}
&\mathbb{P}^2\left( E(\delta, M, N)^c\cap  \left\{ f_{N} \in G^{-1}(O) \right\} \cap \left \{ d_Y(f_{N}, G^{-1}(O)^c) > N^{-\alpha/8} \right\} \right) = \\
& \mathbb{E}_{\mathbb{P}^2}\left[ {\bf 1}_{E(\delta, M, N)^c} \cdot g_N(h_N(L^N_1[ t_1^- , t_1^+ ] ) )  \right] = \mathbb{E}_{\mathbb{P}^2}\left[  \mathbb{E}_{\mathbb{P}^2} \left[{\bf 1}_{E(\delta, M, N)^c} \cdot g_N(h_N(L^N_1[ t_1^- , t_1^+ ] ) )\vert \mathcal{F}_{ext} \right]   \right] = \\
&\mathbb{E}_{\mathbb{P}^2}\left[{\bf 1}_{E(\delta, M, N)^c} \cdot \mathbb{E}_{H,H^{RW}}^{t_1^- , t_1^+, L^N_1(t_1^-), L^N_1(t_1^+), L^N_2\llbracket t_1^-, t_1^+ \rrbracket} \left[ g_N(h_N(\ell) ) \right]   \right] = \\
&\mathbb{E}_{\mathbb{P}^2}\left[{\bf 1}_{E(\delta, M, N)^c} \cdot \frac{\mathbb{E}_{H^{RW}}^{t_1^- , t_1^+, L^N_1(t_1^-), L^N_1(t_1^+)} \left[ g_N(h_N(\ell )) \cdot W_{H}( t_1^-,t_1^+, \ell , L_2^N \llbracket t_1^- , t_1^+ \rrbracket) \right] }{Z_{H,H^{RW}}( t_1^-,t_1^+,L_1^N(t_1^-),L_1^N(t_1^+), L_2^N \llbracket t_1^- , t_1^+ \rrbracket)}  \right] \leq \\
&\mathbb{E}_{\mathbb{P}^2}\left[{\bf 1}_{E(\delta, M, N)^c} \cdot \frac{\mathbb{E}_{H^{RW}}^{t_1^- , t_1^+, L^N_1(t_1^-), L^N_1(t_1^+)} \left[ g_N(h_N(\ell ) )\right] }{\delta}  \right] \leq \mathbb{E}_{\mathbb{P}^2}\left[{\bf 1}_{E(\delta, M, N)^c} \cdot \frac{\delta \epsilon}{\delta}  \right] \leq \epsilon,
\end{split}
\end{equation}
where we have written $\mathcal{F}_{ext}$ in place of $\mathcal{F}_{ext}(\{1\} \times \llbracket t_1^- + 1, t_1^+ - 1\rrbracket)$ as in (\ref{GibbsCond}) to ease the notation. The first equality in (\ref{ZV2}) follows from the definition of $h_N$, $g_N$ and the distributional equality of $f_N$ and $h_N( L_1^N [t_1^-, t_1^+ ])$. The second equality is a consequence of the tower property for conditional expectations. In the third equality, we use that ${\bf 1}_{E^c(\delta, M, N)}$ is $\mathcal{F}_{ext}$-measurable and can thus be taken outside of the conditional expectation, in addition we applied the $(H,H^{RW})$-Gibbs property (\ref{GibbsEq}) to the function $F(\ell) = g_N(h_N(\ell))$. In the fourth equality, we used (\ref{RND}). The inequality on the fourth line follows from the fact that $0 \leq W_H \leq 1$ and  $Z_{H,H^{RW}} \geq \delta$ on $E(\delta, M, N)^c$. In the first inequality on the fifth line, we used (\ref{ZV1}) and the fact that on $E(\delta, M, N)^c$ the random variables $L^N_1(t_1^-), L_1^N(t_1^+)$ (which play the role of $x,y$ in (\ref{ZV1})) satisfy the inequalities
$$ |L^N_1(t_1^-)- p t_1^-| \leq  M N^{\alpha/2} \mbox{ and }  |L^N_1(t_1^+)- p t_1^+| \leq  M N^{\alpha/2}.$$
The last inequality is trivial.

Combining (\ref{ZV2}) with the fact that for $N \geq N_2$ we have from (\ref{Z2}) that $\mathbb{P}^2\left( E(\delta, M, N) \right) < \epsilon$, we conclude that for $N \geq N_2$ we have
$$ \mathbb{P}^2\left(  \left\{ f_{N} \in G^{-1}(O) \right\} \cap \left \{ d_Y(f_{N}, G^{-1}(O)^c) > N^{-\alpha/8} \right\} \right)  \leq $$
$$ \mathbb{P}^2\left( E(\delta, M, N)^c\cap  \left\{ f_{N} \in G^{-1}(O) \right\} \cap \left \{ d_Y(f_{N}, G^{-1}(O)^c) > N^{-\alpha/8} \right\} \right)  + \mathbb{P}^2\left( E(\delta, M, N) \right) < 2 \epsilon,$$
which certainly implies (\ref{Z4}). \\

{\bf \raggedleft Step 4.} In this step we establish (\ref{ZV1}). From Proposition \ref{KMT} we know that we can find constants $0 < C, a, \tilde{\alpha} < \infty$ (depending on $p$ and $H^{RW}$) and a probability space with measure $\mathbb{P}$ on which are defined a Brownian bridge $B^{\sigma_p}$ with variance $\sigma_p^2$ and a family of random curves $\ell^{z}$ on $[0, T]$, which is parameterized by $z \in \mathbb{R}$, such that $\ell^{z}$  has law $\mathbb{P}^{0,T,0,z}_{H^{RW}}$ and
\begin{equation*}
\mathbb{E}_{\mathbb{P}}\big[ e^{a \Delta(T,z)} \big] \leq C e^{\tilde{\alpha} (\log T)^2}e^{|z- p T|^2/T}, \mbox{ where $\Delta(T,z)=  \sup_{0 \leq t \leq T} \big| \sqrt{T} B^{\sigma_p}_{t/T} + \frac{t}{T}z - \ell^{z}(t ) \big|.$}
\end{equation*}

 If $x,y \in \mathbb{R}$ and $T = t_1^+ - t_1^-$, we observe that if $\ell^{y-x}$ has law $\mathbb{P}^{0,T,0,y-x}_{H^{RW}}$, then the random curve $\ell^{x,y}$ in $C([t_1^-, t_1^+])$ defined by
$$\ell^{x,y}(t) = x+ \ell^{y-x}(t - t_1^-)$$
has law $\mathbb{P}^{t_1^-,t_1^+,x,y}_{H^{RW}}$. The latter implies that
\begin{equation}\label{KMTS7}
\begin{split}
&\mathbb{E}_{\mathbb{P}}\big[ e^{a \Delta(T,x,y)} \big] \leq C e^{\tilde{\alpha} (\log T)^2}e^{|y - x- p T|^2/T}, \mbox{ where } \\
&\Delta(T,x,y)=  \sup_{t_1^- \leq t \leq t_1^+} \left| \sqrt{T} B^{\sigma_p}_{t -t_1^-/T} \hspace{-1mm} + (t - t_1^-/T)(y-x) - \ell^{x,y}(t) + x \right|, \mbox{ and } T = t_1^+ - t_1^-.
\end{split}
\end{equation}
Let us denote for $t \in [t_1^-, t_1^+]$
$$\tilde{B}^{x,y}(t) = \sqrt{T} B^{\sigma_p}_{t -t_1^-/T}  + (t - t_1^-/T)(y-x) + x,$$
and observe that the latter is a random variable in $C([t_1^-, t_1^+])$. We further observe that
\begin{equation}\label{ZV3}
d_Y(h_N(\tilde{B}^{x,y}), h_N(\ell^{x,y})) \leq N^{-\alpha/2} \cdot \Delta(T,x,y).
\end{equation}
From the above equations, we conclude that
\begin{equation}\label{ZV4}
\begin{split}
&\mathbb{E}^{t_1^-, t_1^+,x,y }_{H^{RW}}\left[ g_N(h_N(\ell))  \right] = \mathbb{E}_{\mathbb{P}} \left[g_N(h_N(\ell^{x,y}))  \right] = \\
&\mathbb{P} \left(   \left\{ h_N(\ell^{x,y}) \in G^{-1}(O) \right\} \cap  \left \{ d_Y(h_N(\ell^{x,y}), G^{-1}(O)^c) > N^{-\alpha/8}\right\}\right) \leq \\
& \mathbb{P} \left(   \left\{ h_N(\tilde{B}^{x,y}) \in G^{-1}(O) \right\} \right) +  \mathbb{P} \left( d_Y(h_N(\tilde{B}^{x,y}), h_N(\ell^{x,y})) > (1/2) N^{-\alpha/8}\right) \leq \\
&\mathbb{P} \left(    h_N(\tilde{B}^{x,y}) \in G^{-1}(O) \right) + \mathbb{P} \left( \Delta(T,x,y) > (1/2) N^{3\alpha/8}\right).
\end{split}
\end{equation}

We now notice by (\ref{KMTS7}) and Chebyshev's inequality that if $\max( |x- p t_1^-|, |y - pt_1^+|) \leq M N^{\alpha/2}$, then
\begin{equation}\label{ZV5}
\begin{split}
\mathbb{P} \left( \Delta(T,x,y) > (1/2) N^{3\alpha/8}\right) \leq C e^{-(1/2)N^{3\alpha/8}} \cdot e^{\tilde{\alpha} (\log T)^2}e^{4M^2 N^{\alpha} /T}  \leq \delta \epsilon/2,
\end{split}
\end{equation}
where the latter inequality holds provided that $N_2$ is sufficiently large and $N \geq N_2$. Here we also used that $T \geq N^{\alpha}$ by definition.

Next, observe that
\begin{equation}\label{ZV6}
\mbox{ if $\tilde{h}(x) = ah(x) + bx + c$ for $a,b,c \in \mathbb{R}$ then $G(\tilde{h}) = a \cdot G(h)$.}
\end{equation}
The latter observation shows that
\begin{equation}\label{ZV7}
\mathbb{P} \left(    h_N(\tilde{B}^{x,y}) \in G^{-1}(O) \right) = \mathbb{P} \left(  \frac{\sqrt{T}}{N^{\alpha/2} } \cdot B_N  \in G^{-1}(O) \right),
\end{equation}
where for $t \in [-r,r]$ we have $B_N(t) = B^{\sigma_p}_{tN^{\alpha} -t_1^-/T}.$ From basic properties of Brownian bridges we know that if $B$ is a standard Brownian bridge on $[0,1]$ that is independent of $B_N$, then the process
$$\tilde{B}_N(t) =  B_N( -r) \cdot \left( \frac{r - t}{2r} \right) + B_N( r) \cdot \left( \frac{t +r}{2r} \right) + \frac{ \sqrt{2r} \sigma_p N^{\alpha/2}}{\sqrt{T}} B\left( \frac{t +r}{2r} \right),$$
defined on $[-r,r]$ has the same distribution as $B_N(t)$. Combining the latter with (\ref{ZV6}) and (\ref{ZV7}), we conclude that
\begin{equation}\label{ZV8}
\mathbb{P} \left(    h_N(\tilde{B}^{x,y}) \in G^{-1}(O) \right) = \mathbb{P} \left(  \sqrt{2r}\sigma_p \cdot B  \in O \right) \leq \epsilon \delta/2,
\end{equation}
where in the last inequality we used that $ \sqrt{2r}\sigma_p \cdot B$ has the same law as $ B^{\sqrt{2r}\sigma_p }$ and (\ref{ZO}). Combining (\ref{ZV8}) and (\ref{ZV5}) with (\ref{ZV4}), we conclude (\ref{ZV1}) and thus the proof of the theorem.

%

\section{The log-gamma polymer as a line ensemble}\label{Section8}
 In Section \ref{Section8.1} we present a certain Markov chain formulation of the log-gamma polymer, which is a consequence of the geometric RSK correspondence, following \cite{COSZ}. In Section \ref{Section8.2}, we use the Markov chain formulation to prove that the log-gamma polymer has a Gibbsian line ensemble structure of the type discussed in Section \ref{Section4}. In Section \ref{Section8.3}, we prove Theorem \ref{ThmTight} by appealing to Theorems \ref{PropTightGood} and \ref{ACBB}.

%
\subsection{Markovian dynamics}\label{Section8.1}

 Recall that a continuous random variable $X$ is said have the inverse-gamma distribution with parameter $\theta > 0 $ if its density is given by
\begin{equation}\label{invGammaDens}
f_\theta(x) = {\bf 1} \{ x > 0 \} \cdot \Gamma(\theta)^{-1} \cdot x^{-\theta - 1} \cdot \exp( - x^{-1}).
\end{equation}

Let us fix $N \in \mathbb{N}$ and $\theta > 0$. We let $d = \left( d_{i,j} : i \geq 1, 1\leq j \leq N \right)$ denote the semi-infinite random matrix such that $d_{i,j}$ are i.i.d. random variables with density $f_\theta$ as in (\ref{invGammaDens}). In addition, for $n \geq 1$ we denote by $d^{[1,n]}$ the $n \times N$ matrix $\left( d_{i,j} : 1\leq i \leq n, 1\leq j \leq N \right)$. A {\em directed lattice path} is a sequence of vertices $(x_1, y_1), \dots, (x_k, y_k) \in \mathbb{Z}^2$, such that $x_1 \leq x_2 \leq \cdots \leq x_k$, $y_1 \leq y_2 \leq \cdots \leq y_k$ and $\big(x_i - x_{i-1}\big) + \big(y_i - y_{i-1}\big) = 1$ for $i = 2, \dots, k$. In words, a directed lattice path is an up-right path on $\mathbb{Z}^2$, which makes unit steps in the coordinate directions. A collection of paths $\pi = (\pi_1, \dots, \pi_{\ell})$ is said to be {\em  non-intersecting} if the paths $\pi_1, \dots, \pi_\ell$ are pairwise vertex-disjoint. For $1 \leq \ell \leq k \leq N$, we let $\Pi_{n,k}^\ell$ denote the set of $\ell$-tuples $\pi = (\pi_1, \dots, \pi_{\ell})$ of non-intersecting directed lattice paths in $\mathbb{Z}^2$, such that for $1 \leq r \leq \ell$, $\pi_r$ is a lattice path from $(1,r)$ to $(n, k+ r - \ell)$.

Given an $\ell$-tuple $\pi = (\pi_1, \dots, \pi_{\ell})$, we define its {\em weight} to be
\begin{equation}\label{PathWeight}
w(\pi) = \prod_{r = 1}^\ell \prod_{(i,j) \in \pi_r} d_{i,j}.
\end{equation}
For $1 \leq \ell \leq k \leq N$, we define
\begin{equation}\label{PartitionFunct}
\tau_{k, \ell}(n) = \sum_{ \pi \in \Pi_{n,k}^\ell} w(\pi).
\end{equation}
Note that if $0 \leq n < \ell \leq k \leq N$, then $ \Pi_{n,k}^\ell = \varnothing$ and so, as by convention, we set $\tau_{k, \ell}(n)  = 0$. If $\ell = k$ then $\Pi_{n,k}^\ell $ consists of a unique element, and, in fact, we have
$$\tau_{k,\ell}(n) = \delta_{k,\ell} \cdot \tau_{k, n} (n) \mbox{ for } 0 \leq n < \ell \leq k \leq N,$$
where $\delta_{k,\ell}$ is the Kronecker delta.

Given $\tau_{k,\ell}(n)$, we define the array $z(n) = \{z_{k, \ell}(n): 1 \leq k \leq N \mbox{ and } 1 \leq \ell \leq \min( k, n) \}$ through the equations
\begin{equation}\label{ZfromTau}
z_{k,1}(n) z_{k,2}(n) \cdots z_{k, \ell}(n) = \tau_{k, \ell }(n).
\end{equation}
We next proceed to define a certain Markovian dynamics on triangular arrays of positive reals, which will be ultimately related to the random variables $z_{k, \ell}(n)$ in (\ref{ZfromTau}).\\

Let us introduce some notation. For each $k \in \mathbb{N}$, we let $\mathbb{Y}_k = (0, \infty)^k$ and
$$\mathbb{T}_k = \mathbb{Y}_1 \times \cdots \times \mathbb{Y}_k = \{(z^{[1]}, \dots, z^{[k]}): z^{[r]} \in \mathbb{Y}_r \mbox{ for } 1 \leq r \leq k \}.$$
For each $r \in \{1, \dots, k\}$, we have that $z^{[r]} = (z_{r,1}, \dots, z_{r,r}) \in \mathbb{Y}_r$, and also we can naturally identify $\mathbb{T}_k$ with $(0, \infty)^{k (k+1)/2}$, where the coordinates are labeled by $z_{i,j}$ for $1 \leq i \leq j \leq k$.
In particular, we can view $(\mathbb{T}_k, \mathcal{B}) $ as a measurable space, where $\mathcal{B}$ is the usual  Borel $\sigma$-algebra of $(0, \infty)^{k (k+1)/2}$. If $1 \leq k \leq N$ and $z \in \mathbb{T}_N$, we define
$$z^{[1, k]} = (z_{t,s} )_{1 \leq s \leq t \leq k} \in\mathbb{T}_k.$$

Let us fix $N \in \mathbb{N}$. The measurable space $(\mathbb{T}_N, \mathcal{B})$ is the state space of our Markov chain, whose transition kernel we define next. Define the kernel $P_\theta^1: \mathbb{Y}_1 \rightarrow \mathbb{Y}_1$ through
\begin{equation}\label{P1}
P^1_\theta(y, d\tilde{y}) = \frac{1}{\Gamma(\theta)} \left( \frac{y}{\tilde{y}}\right)^{\theta} \cdot \exp \left( - \frac{y}{\tilde{y}} \right) \frac{d \tilde{y}}{\tilde{y}}.
\end{equation}
In words, the above kernel encodes the transition from $y$ to $\tilde{y} = d \cdot y$, where $d$ is an independent random variable with density $f_\theta$. In particular, $P_\theta^1$ is indeed a stochastic transition kernel. For $k \geq 2$ we let $L_\theta^k : \mathbb{Y}_{k-1} \times \mathbb{Y}_k  \times \mathbb{Y}_{k-1} \rightarrow \mathbb{Y}_k$ be defined through
\begin{equation}\label{Lk}
\begin{split}
&\int_{(0,\infty)^k} h(\tilde{y}) L^k_\theta ((x,y, \tilde{x}), d\tilde{y}) = \int_{(0,\infty)} \frac{d\tilde{y}_1}{\tilde{y}_1} \frac{1}{\Gamma(\theta)} \cdot \left( \frac{y_1 + \tilde{x}_1}{\tilde{y}_1} \right)^\theta \exp \left( - \frac{y_1 + \tilde{x}_1}{\tilde{y}_1}\right) \times \\
& h\left(\tilde{y}_1 , \left \{ \frac{y_{\ell-1} \tilde{x}_{\ell - 1}}{x_{\ell-1}} \cdot \frac{y_\ell + \tilde{x}_\ell}{y_{\ell - 1} + \tilde{x}_{\ell - 1}} \right\}_{2 \leq \ell \leq k-1}, \frac{y_k y_{k-1} \tilde{x}_{k-1}}{x_{k-1}(y_{k-1} + \tilde{x}_{k-1})} \right),
\end{split}
\end{equation}
where $h(\cdot)$ is a bounded continuous function. In other words, the kernel $L^k_\theta$ encodes the transition from the vector $(x,y,\tilde{x}) \in \mathbb{Y}_{k-1} \times \mathbb{Y}_k  \times \mathbb{Y}_{k-1}$ to the random vector $\tilde{y} \in \mathbb{Y}_k$, given by
\begin{equation}\label{LkProb}
\begin{split}
&\tilde{y}_1 = d \cdot (y_1 + \tilde{x}_1), \\
&\tilde{y}_\ell = \frac{y_{\ell -1} \cdot \tilde{x}_{\ell - 1}}{x_{\ell - 1}} \cdot \frac{y_{\ell} + \tilde{x}_{\ell}}{y_{\ell - 1} + \tilde{x}_{\ell - 1}} \mbox{ for } 2 \leq \ell \leq k-1,\\
&\tilde{y}_{k} = \frac{y_k \cdot y_{k-1} }{x_{k-1}(y_{k-1} + \tilde{x}_{k-1})},
\end{split}
\end{equation}
where $d$ is an independent random variable with density $f_\theta$. In particular, $L_\theta^k$ is indeed a stochastic transition kernel. We define the following transition kernels $\Pi^N_\theta: \mathbb{T}_N \rightarrow \mathbb{T}_N$ by induction on $N \geq 1$. For $N = 1$ we let $\Pi^1_\theta = P^1_\theta$ as in (\ref{P1}), and for $N \geq 2$ we let
\begin{equation}\label{PiN}
\Pi^N_\theta (z, d\tilde{z}) = \Pi^{N-1}_\theta \left(z^{[1, N-1]}, d \tilde{z}^{[1, N-1]} \right) L_\theta^N \left( \left(z^{[N-1]}, z^{[N]}, \tilde{z}^{[N-1]} \right), d\tilde{z}^{[N]} \right).
\end{equation}
We write $\{ z(n) \}_{n \geq 0}$ to denote the Markov chain on $(\mathbb{T}_N, \mathcal{B})$, whose transition kernel is $\Pi^N_\theta$.\\

As it turns out, if the chain $\{ z(n) \}_{n \geq 0}$ is started from certain initial conditions $z(0)$, then the process $y(n) = \phi(z(n))$, with $\phi(z) = z^{[N]}$, will be Markovian in its own filtration. Let us elaborate this point further. We define a positive kernel $P^N_\theta$ on $\mathbb{Y}_N$ through
\begin{equation}\label{PN}
P^N_\theta(y, d\tilde{y}) = \prod_{i = 1}^{N-1} \exp \left( - \frac{\tilde{y}_{i +1}}{y_i} \right) \prod_{j = 1}^N \left( \frac{1}{\Gamma(\theta)} \left( \frac{y_j}{\tilde{y}_j} \right)^{\theta} \cdot \exp \left( - \frac{y_j}{\tilde{y}_j}\right)\frac{d\tilde{y}_j}{\tilde{y}_j} \right).
\end{equation}
In addition, we define a positive (intertwining) kernel from $\mathbb{Y}_N$ to $\mathbb{T}_N$ by
\begin{equation}\label{KN}
K_\theta^N(y, dz) = \prod_{1 \leq \ell \leq k < N} \exp \left( - \frac{z_{k, \ell}}{z_{k+1, \ell}} - \frac{z_{k+1, \ell+1}}{z_{k, \ell}} \right) \frac{dz_{k, \ell}}{z_{k, \ell}} \prod_{ \ell = 1}^N \delta_{y_\ell}(dz_{N, \ell}),
\end{equation}
where $\delta_y(dz_{i,j})$ denotes the Dirac delta measure at $y$. Furthermore, we define
\begin{equation}\label{NormalizedKernels}
\bar{P}_\theta^N (y, d\tilde{y}) = \frac{w_\theta^N(\tilde{y})}{w_\theta^N(y)} P_\theta^N(y, d\tilde{y}),\quad \bar{K}^N_\theta(y, dz) = \frac{K_\theta^N(y, dz)}{w_\theta^N(y)}, \mbox{ with }
w_\theta^N(y) = \int_{\mathbb{T}_N} K_\theta^N(y, dz).
\end{equation}
As explained in \cite[Section 3.1]{COSZ}, the renormalization in (\ref{NormalizedKernels}) is made so that the kernels $\bar{P}_\theta^N$ and $\bar{K}_\theta^N$ become stochastic. The following is the main algebraic result of \cite{COSZ} (see Proposition 3.4 and Corollary 3.6 in that paper with $\theta_i = \theta$ and $\hat{\theta}_i = 0$ for $i \in \mathbb{N}$).
\begin{proposition}\label{intertwineProp}
For any $\theta > 0$ and $N \in \mathbb{N}$, we have the following intertwining relation
\begin{equation}\label{intertwineE}
\bar{P}_\theta^N K_\theta^N = \bar{K}^N_\theta \Pi_\theta^N.
\end{equation}
\end{proposition}
The intertwining relation (\ref{intertwineE}), through a certain general formalism for Markov chains, is essentially sufficient to prove the following statement.
\begin{theorem}\label{markovProj}\cite[Theorem 3.7]{COSZ}
Let $\theta > 0$ and $N \in \mathbb{N}$. Let $y(0)$ be a random or deterministic initial state in $\mathbb{Y}_N$ and let $\{ z(n) \}_{n \geq 0}$ be the Markov chain, whose initial state $z(0)$ has the distribution $\bar{K}_{\theta}^N(y(0), \cdot)$ and whose transition kernel is $\Pi^N_\theta$. Then the sequence of random variables $y(n) = \phi(z(n)), n \geq 0$ is a Markov chain with respect to its own filtration with state space $\mathbb{Y}_N$, initial state $y(0)$, and transition kernel $\bar{P}_\theta^N$.
\end{theorem}

We conclude this section by relating the Markov chain $\{ z(n) \}_{n \geq 0}$, we just defined, to the variables $z_{k, \ell}$ in (\ref{ZfromTau}) in the following statement, which is \cite[Proposition 5.3]{COSZ} for $\theta_i = \theta$ and $\hat{\theta}_i = 0$ for $i \in \mathbb{N}$.
\begin{proposition}\label{MCConv}Let $\theta > 0$ and $N \in \mathbb{N}$.  For each $M \in \mathbb{N}$, we define
$$y^{0, M} = \left(\exp \left( -M \cdot \frac{N- \ell}{2}\right)\right)_{1 \leq \ell \leq N}.$$
 We also let $\mathbb{P}_M$ denote the probability distribution of the Markov chain $\{ z(n) \}_{n \geq 0}$, whose initial state $z(0)$ has the distribution $\bar{K}_{\theta}^N(y^{0, M}, \cdot)$, and whose transition kernel is $\Pi^N_\theta$. Let $n \in \mathbb{N}$ be given and denote by $\mathbb{P}_{\varnothing}$ the probability distribution of $z(s) = \{z_{k, \ell}(s): 1 \leq k \leq N \mbox{ and } 1 \leq \ell \leq \min( k, s) \}$ for $s =1, \dots, n$ as in (\ref{ZfromTau}). If $f$ is a bounded continuous function of the $(0,\infty)$-valued coordinates $\{ z_{k, \ell}(s): 1\leq s \leq n, 1 \leq k \leq N, 1\leq \ell \leq \min (k, s)\}$, then we have
\begin{equation}\label{LimitObs}
\lim_{M \rightarrow \infty} \mathbb{E}_{\mathbb{P}_M} \left[ f(z(1), \dots, z(n)) \right] = \mathbb{E}_\varnothing \left[ f(z(1), \dots, z(n)) \right].
\end{equation}
\end{proposition}

%
\subsection{The log-gamma polymer as a line ensemble}\label{Section8.2} In this section, we prove that the log-gamma polymer can be understood as a discrete line ensemble that satisfies the $(H,H^{RW})$-Gibbs property, where
\begin{equation}\label{Hamiltonian}
\g_\theta(x) = e^{-H^{RW}(x)}, \hspace{3mm}H^{RW}(x) = \theta x + e^{-x} + \log \Gamma (\theta) \mbox{ and } H(x) = e^x,
\end{equation}
with $\theta > 0$ as in the definition of the polymer model in Section \ref{Section8.1}.

 In this section, we prove the following result.
\begin{proposition}\label{PropLOG}
Let $H,H^{RW}$ be as in (\ref{Hamiltonian}). Fix $\topc, N \in \mathbb{N}$ with $N \geq \topc \geq 2$. Let $T_0, T_1 \in \mathbb{N}$ be such that $T_0 < T_1$ and $T_0 \geq \topc$. Then, we can construct a probability space $\mathbb{P}$ that supports a $\llbracket 1, \topc \rrbracket \times \llbracket T_0, T_1 \rrbracket$-indexed line ensemble $\mathfrak{L} = (L_{1}, \dots, L_\topc)$, such that:
\begin{enumerate}
\item the $\mathbb{P}$-distribution of $(L_i(j): (i,j) \in \llbracket 1, \topc \rrbracket \times \llbracket T_0, T_1 \rrbracket )$ is the same as that of $ ( \log(z_{N,i}(j)): (i,j) \in \llbracket 1, \topc \rrbracket \times \llbracket T_0, T_1 \rrbracket )$ as in (\ref{ZfromTau});
\item $\mathfrak{L}$ satisfies the $(H,H^{RW})$-Gibbs property .
\end{enumerate}
\end{proposition}
\begin{remark}
We mention that the embedding of $\log(z_{N,i})$ as the lowest-indexed curve of an $(H,H^{RW})$-Gibbsian line ensemble is a special feature of the log-gamma polymer model. A similar property is known to hold for other integrable polymer models such as the strict-weak polymer model, mentioned in Remark \ref{S1RemWS}. For general directed polymer models such an embedding is not known.
\end{remark}
\begin{remark}
We mention that analogues of Proposition \ref{PropLOG} can be found in \cite[Section 3.4]{Wu19} and \cite{JO20}. As the statements and notations from those papers are a bit different than here, we provide the fairly short proof of this result for the sake of completeness.
\end{remark}

\begin{proof} We split the proof into two steps, for clarity.

{\bf \raggedleft Step 1.}  We assume the same notation as in Theorem \ref{markovProj} and Proposition \ref{MCConv}. From these results we know that for each $M \in \mathbb{N}$ there exists a probability space with measure $\mathbb{P}_M$ that supports a Markov chain $z(n), n \geq 0$, whose initial state $z(0)$ has the distribution $\overline{K}_{\theta}^N(y^{0,M}, \cdot)$, and whose transition kernel is $\Pi_\theta^N$. Moreover, the process $y(n) = \phi(z(n)), n\geq 0$ is a Markov chain with respect to its own filtration with state space $\mathbb{Y}_N$, initial state $y(0) = y^{0,M}$, and transition kernel $\overline{P}_\theta^N$. We write $y(n) = (y_1(n), \dots, y_N(n))$.

Let us define $(L^M_i(j): i = 1, \dots, N, j \geq 0)$ through $L^M_i(j) = \log \big(y_i(j)\big)$. By a simple change of variables, using (\ref{PN}), (\ref{KN}) and (\ref{NormalizedKernels}), we see that the sequence $L^M(j)$, $j \geq 0$ of $\mathbb{R}^N$-valued random variables (the $i$-th coordinate of $L^M(j)$ is $L^M_i(j)$) is also Markov in its own filtration and its transition kernel is
\begin{equation}\label{TransLogChain}
\hat{P}^N_\theta(z, d\tilde{z}) = \frac{w_\theta^N(e^{\tilde{z}}) }{w_\theta^N(e^z)} \cdot \prod_{i = 1}^{N-1} \exp \left( -H (\tilde{z}_{i+1} - z_i) \right) \cdot \prod_{j = 1}^N \g_{\theta}(z_j - \tilde{z}_j) d\tilde{z}_j,
\end{equation}
where for $z \in \mathbb{R}^N$ we write $e^z = (e^{z_1}, \dots, e^{z_N})$.

We claim that for each $M \geq 1$ and $T_1 \in \mathbb{N}$ as in the statement of the proposition, we have that the line ensemble $(L^M_i(j): (i,j) \in \llbracket 1, N \rrbracket \times \llbracket 0, T_1 \rrbracket)$ satisfies the $(H,H^{RW})$-Gibbs property. We prove this claim the next step. For now we assume its validity and conclude the proof of the proposition. \\

Since $(L^M_i(j): (i,j) \in \llbracket 1, N \rrbracket \times \llbracket 0, T_1 \rrbracket)$ satisfies the $(H,H^{RW})$-Gibbs property we know that $(L^M_i(j): (i,j) \in \llbracket 1, \topc \rrbracket \times \llbracket T_0, T_1 \rrbracket)$ satisfies the $(H,H^{RW})$-Gibbs property as a $\llbracket 1, \topc \rrbracket \times \llbracket T_0, T_1 \rrbracket$-indexed line ensemble (cf. Remark \ref{restrict}). Furthermore, by Proposition \ref{MCConv} we know that $\big(L^M_i(j): (i,j) \in \llbracket 1, \topc \rrbracket \times \llbracket T_0, T_1 \rrbracket\big)$ weakly converge to $ \big( \log(z_{N,i}(j)): (i,j) \in \llbracket 1, \topc \rrbracket \times \llbracket T_0, T_1 \rrbracket \big)$  (here we used that $T_0 \geq \topc$).  Since $\big(L^M_i(j): (i,j) \in \llbracket 1, \topc \rrbracket \times \llbracket T_0, T_1 \rrbracket\big)$ each satisfy the $(H,H^{RW})$-Gibbs property, we conclude the same is true for $\log(z_{N,i}(j)): (i,j) \in \llbracket 1, \topc \rrbracket \times \llbracket T_0, T_1 \rrbracket \big)$ by Lemma \ref{S4WeakGibbs}. This concludes the proof, modulo verifying the claimed $(H,H^{RW})$-Gibbs property, which is done in the next step.\\

{\bf \raggedleft Step 2.} To prove that $\big(L^M_i(j): (i,j) \in \llbracket 1, N \rrbracket \times \llbracket 0, T_1 \rrbracket\big)$ satisfies the $(H,H^{RW})$-Gibbs property, we appeal to Lemma \ref{S4AltGibbs} and we use the same notation as in that lemma. To simplify the expressions below we drop $M$, which is fixed in this step, from the notation. Let $f_{i,j}$ for $(i,j) \in \llbracket 1, N \rrbracket \times \llbracket 0, T_1 \rrbracket$  be bounded continuous functions on $\mathbb{R}$. In view of Lemma \ref{S4AltGibbs} it suffices to show
\begin{equation}\label{S8towerFunV2}
\begin{split}
& \mathbb{E}\Big[ \prod_{i = 1}^N \prod_{j = 0}^{T_1} f_{i,j}(  L_i(j)  ) \Big] = \mathbb{E} \Big[ \prod_{(i,j) \in B}f_{i,j}( L_i(j) ) \cdot  \mathbb{E}_{H,H^{RW}}^{1, N-1, 0, T_1, \vec{x}, \vec{y},\infty,L_N} \Big[\prod_{(i,j) \in A} f_{i,j}( L_i(j)  ) \Big]  \Big],
\end{split}
\end{equation}
Using the transition probability in (\ref{TransLogChain}), we obtain that
\begin{equation}\label{LeftTower}
\begin{split}
& \mathbb{E}_{\mathbb{P}}\Big[\prod_{i = 1}^N \prod_{j = 0}^{T_1} f_{i,j}(  L_i(j)  )\Big]  =\prod_{i = 1}^N f_{i,0}( z_i^0  ) \cdot \int_{\mathbb{R}^{NT_1}}  \prod_{i = 1}^N \prod_{j = 1}^{T_1} f_{i,j}( z_i^j  ) \prod_{i = 1}^{N-1} \prod_{j = 1}^{T_1} e^{- H( z_{i+1}^j - z_i^{j-1})} \cdot \\
&\prod_{i = 1}^N \prod_{j = 1}^{T_1}  \g_\theta(z^j_i -z^{j-1}_i) \cdot \frac{w_\theta^N(e^{z^{T_1}})}{w^N_\theta(e^{z^0})} \prod_{i = 1}^N \prod_{j = 1}^{T_1} d z_i^j,
\end{split}
\end{equation}
where $z_i^0 = \log [y^{0,M}_i] = -M \cdot \frac{N-i}{2}$.

On the other hand, we have by definition that the right-hand side of (\ref{S8towerFunV2}) equals
\begin{equation}\label{RightTower}
\begin{split}
&\prod_{i = 1}^N f_{i,0}( z_i^0  ) \cdot \int_{\mathbb{R}^{NT_1}}  \int_{\mathbb{R}^{(N-1)(T_1-1)}}   \prod_{i = 1}^N f_{i,T_1}( z_i^{T_1}  )  \cdot  \prod_{j = 1}^{T_1-1} f_{N,j}( z_N^j  )\prod_{i = 1}^{N-1} \prod_{j = 1}^{T_1} e^{ - H( z_{i+1}^j - z_i^{j-1})} \cdot \\
&\prod_{i = 1}^N \prod_{j = 1}^{T_1}\g_\theta(z^j_i -z^{j-1}_i) \cdot \frac{w_\theta^N(e^{z^{T_1}})}{w^N_\theta(e^{z^0})}  \prod_{i = 1}^{N-1} \prod_{j = 1}^{T_1-1} f_{i,j}( \tilde{z}_i^j  )  \cdot  \prod_{i = 1}^{N-1} \prod_{j = 1}^{T_1} e^{ -H( \tilde{z}_{i+1}^j - \tilde{z}_i^{j-1}) }\cdot  \\
&\prod_{i = 1}^{N-1} \prod_{j = 1}^{T_1} \g_\theta(\tilde{z}^j_i -\tilde{z}^{j-1}_i)  \cdot \prod_{i = 1}^{N-1} \frac{1}{\g^{z^0_N}_{T_1}(z^{T_1}_i)} \cdot \frac{1}{Z_{H}^{1, N-1, z^0, z^{T_1},\infty,z_N}}  \cdot  \prod_{i = 1}^{N-1} \prod_{j = 1}^{T_1-1} d\tilde{z}_i^j  \prod_{i = 1}^N \prod_{j = 1}^{T_1} d z_i^j,
\end{split}
\end{equation}
where $z_N$ stands for the vector $(z_N^0, \dots, z_N^{T_1})$, $\tilde{z}_i^0 = z_i^0$ for $i = 1, \dots, N$, $\tilde{z}_i^{T_1} = z_i^{T_1}$ for $i = 1, \dots, N$, $\tilde{z}_N^i = z_N^i$ for $i = 1, \dots, T_1$, and
\begin{equation*}
\begin{split}
&Z_{H}^{1, N-1, z^0, z^{T_1},\infty,z_N} =  \int_{\mathbb{R}^{(N-1)(T_1-1)}}   \prod_{i = 1}^{N-1} \prod_{j = 1}^{T_1} e^{ - H( z_{i+1}^j - z_i^{j-1}) } \cdot \\
&   \prod_{i = 1}^{N-1} \prod_{j = 1}^{T_1} \g_\theta(\tilde{z}^j_i -\tilde{z}^{j-1}_i) \cdot \prod_{i = 1}^{N-1} \frac{1}{\g^{z^0_N}_{T_1}(z^{T_1}_i)} \prod_{i = 1}^{N-1} \prod_{j = 1}^{T_1-1} d z_i^j.
\end{split}
\end{equation*}
In the above, we have also used that $\g_n^x(y)$ is as in (\ref{RWN}) for $\g = \g_\theta$. We remark that the integration over $\tilde{z}$ corresponds to the expectation $\mathbb{E}_{H,H^{RW}}^{1, N-1, 0, T_1, \vec{x}, \vec{y},\infty,L_N}$ on the right side of (\ref{S8towerFunV2}), while the integration over $z$ corresponds to the outer expectation on the right side of (\ref{S8towerFunV2}).

We may now integrate in (\ref{RightTower}) over the variables $z_i^j$ with $(i,j) \in \llbracket 1, N-1 \rrbracket \times \llbracket 1, T_1 - 1 \rrbracket$ and cancel the resulting factor with
$$\prod_{i = 1}^{N-1} \frac{1}{\g^{z^0_N}_{T_1}(z^{T_1}_i)} \cdot \frac{1}{Z_{H}^{1, N-1, z^0, z^{T_1},\infty,z_N}}.$$
 The resulting expression will then equal (\ref{LeftTower}) upon relabeling $z_i^j$ to $\tilde{z}_i^j$ for $(i,j) \in \llbracket 1, N-1 \rrbracket \times \llbracket 1, T_1 - 1 \rrbracket$. This proves (\ref{S8towerFunV2}) and, hence, the proposition.
\end{proof}

%
\subsection{Spatial tightness of the log-gamma polymer}\label{Section8.3} In this section, we prove Theorem \ref{ThmTight} by appealing to Theorem \ref{informalintrotightness}. In what follows, we fix $\theta > 0$ and let $H^{RW}, H$ be as in (\ref{Hamiltonian}). We will use much of the notation of Section \ref{sec:introlog} (e.g. $d_\theta, h_\theta, \mathcal{F}(n,N),$). For convenience, we denote
$$ \LB = \lfloor rN + (T+3)N^{2/3} + 2 \rfloor.$$

Fix some $\topc\geq 2$. For each $N\geq \topc$, Corollary \ref{cor:loggammalineensemble} provides us with a $\llbracket 1, \topc \rrbracket \times \llbracket K, \LB \rrbracket$-indexed line ensemble, which we will denote $\mathfrak{\tilde{L}}^N$, whose lowest labeled curve $\big(\tilde{L}^N_1(n):n\in \llbracket \topc,\LB \rrbracket \big)$ has the same law as $\big(\log z_{N,1}(n):n\in \llbracket\topc,\LB \rrbracket \big)$ in the notation from Section \ref{Section8.1} or, equivalently, $\big(\log Z^{n,N}:n\in \llbracket\topc, \LB \rrbracket \big)$ in the notation of Section \ref{sec:introlog}. Moreover, this line ensemble enjoys the  $(H,H^{RW})$-Gibbs property with $H$ and $H^{RW}$ given in \eqref{Hamiltonianintro}.

Let $T_N = \lfloor (T+3)N^{2/3} + 2 \rfloor$ and assume that $N_0 \geq 2$ is sufficiently large, so that $M - 2T_N -2 \geq K$ for all $N \geq N_0$. Such a choice of $N_0$ is possible by our assumption that $r > 0$, and depends only on $T$ and $r$. Provided $N \geq N_0$ as above, we define the $\llbracket 1, 2 \rrbracket \times \llbracket -T_N, T_N \rrbracket$-indexed line ensemble $\mathfrak{L}^N$ by setting
$$  L_i^N(x) = \tilde{L}_i^N(x+ \lfloor rN\rfloor) + N h_\theta(r) \mbox{ for $i = 1,2$ and $x \in \llbracket -T_N, T_N\rrbracket$},$$
where $\mathfrak{\tilde{L}}^N$ is as above. The condition that $N \geq N_0$ ensures that the argument in $\tilde{L}_i^N$ stays in $\llbracket K, M \rrbracket$, so that $\mathfrak{L}^N$ is well-defined. 

We claim that the sequence of line ensembles $\mathfrak{L}^N$ defined just now is $(\alpha, p, T+3)$-good in the sense of Definition \ref{def:intro} with $\alpha = 2/3$ and $p = -h'_\theta(r)$. Assuming this for the moment, we see that if $f_N(x)$ is as in Theorem \ref{informalintrotightness} then $f_N(x) = f_N^{LG}(x)$ with $f_N^{LG}$ as in the statement of Theorem \ref{ThmTight}. Consequently, by Theorem \ref{informalintrotightness} we see that the sequence of random functions $f_N^{LG}$ is a tight sequence of $(C[-T,T], \mathcal{C})$-valued random variables, and also that any subsequential limit $\mathbb{P}_\infty$ of the laws of $f_N^{LG}$ is absolutely continuous with respect to a Brownian bridge with variance $2T \sigma_p^2$ in the sense of Definition \ref{DACB}, where $\sigma_p^2$ is as in Definition \ref{AssHR}. Consequently, to conclude the proof of Theorem \ref{ThmTight} it remains to show
\begin{enumerate}
\item the sequence of line ensembles $\mathfrak{L}^N$  is $(2/3, -h'_\theta(r), T+3)$-good in the sense of Definition \ref{def:intro};
\item $\sigma_p^2 = \Psi'(g_{\theta}^{-1}(r))$.
\end{enumerate}

Note that by the definition of $N_0$ and $T_N$ above we know that $T_N > TN^{\alpha} + 1$ for $N \geq N_0$. Furthermore, since $\mathfrak{\tilde{L}}^N$ satisfies the $(H,H^{RW})$-Gibbs property, the same can be deduced for $\mathfrak{L}^N$, as the latter was obtained from the former by a horizontal shift (by $\lfloor rN \rfloor$) and a vertical shift (by $h_\theta(r)N $) followed by a projection to the coordinates $\llbracket 1, 2 \rrbracket \times \llbracket -T_N, T_N \rrbracket$ and all of these operations preserve the $(H,H^{RW})$-Gibbs property. This establishes the first condition of Definition \ref{def:intro}. To see why the second condition holds, let us fix $s \in [-T,T]$ and note that for $n =\lfloor rN \rfloor + \lfloor sN^{2/3} \rfloor$, and $\mathcal{F}(n,N)$ as in (\ref{ResFE}) we have
\begin{equation*}
\begin{split}
&f_N^{LG} (\lfloor sN^{2/3} \rfloor) - d_\theta(n/N) \mathcal{F}( n ,N) = N^{2/3}( h_{\theta}(n/N) - h_\theta(r) )+ h'_\theta(r) N^{-1/3} \lfloor sN^{2/3} \rfloor = O(1)\\
\end{split}
\end{equation*}
where the last equality follows from basic Taylor expansion and the constant in the big $O$ notation depends on $\theta, T$ and $r$. From Proposition \ref{thm:BCDA} we know that $ \mathcal{F}( n ,N) $ is tight (in fact it converges to the Tracy-Widom distribution) and since $d_\theta(n/N)$ converges to $d_\theta(r)$ we conclude that $f_N^{LG} (\lfloor sN^{2/3} \rfloor)$ is also tight. Thus, the second condition of Definition \ref{def:intro} is also satisfied.

Since $H(x) = e^x$ we have that $H$ is convex, increasing and $\lim_{x \rightarrow \infty} x^2 H(-x) = 0$, which shows that $H$ satisfies the conditions in Definition \ref{AssH}. What remains is to show that $H^{RW} = \theta x + e^{-x} + \log \Gamma (\theta)$ satisfies the five assumptions in Definition \ref{AssHR}. For Assumption 1, $H^{RW}(x)$ is immediately seen to be continuous and convex, and $G(x) = e^{-H^{RW}(x)}$ is bounded and integrates to $1$. For Assumption 2, the moment generating function is evaluated to be $M(t) = \Gamma(\theta-t)/\Gamma(\theta)$ provided that $t<\theta$. Thus, the cumulant generating function $\Lambda(t) = \log M(t) = \log \Gamma(\theta-t) - \log\Gamma(\theta)$ is defined on a domain $\mathcal{D}_{\Lambda}= (-\infty, \theta)$ and Assumption 2 is verified. From the exact formula, Assumption 3 follows immediately. Assumption 4 follows from the fact that for any $-\infty<a<b < \theta$, there exists constants $c,C>0$ such that for all $z$  with $\Re(z)\in (a,b)$, $\big\vert \Gamma(z)\big\vert \leq C e^{-c|z|}$. Assumption 5 -- namely, the second bound in \eqref{S2S1E2} -- follows from the double exponential decay of $\g(x)$ for negative $x$. The above two paragraphs verify all the conditions of Definition \ref{def:intro}.

To see why $\sigma_p^2 =\Psi'(g_{\theta}^{-1}(r))$, note that from Definition \ref{AssHR} we have $\sigma_p^2 := \Lambda''( (\Lambda')^{-1}(p))$. From the explicit formula for $\Lambda$ we may compute $\Lambda'(t) = -\Psi(\theta-t)$ and $\Lambda''(t) = \Psi'(\theta-t)$. Using that $h'_\theta(r) = \Psi(g^{-1}(r)) = - p$, we see that 
$$(\Lambda')^{-1}(p) = (\Lambda')^{-1}(-h_\theta'(r) ) = \theta - g^{-1}_\theta(r),$$
and so $\sigma_p^2 = \Lambda''( (\Lambda')^{-1}(p))= \Psi'(g^{-1}(r))$ as desired. This concludes the proof of the theorem.

%
\section{Proof of results from Section \ref{Section4}}\label{Section11} In this section, we give the proof of various results from Section \ref{Section4}, and we will use much of the same notation as in that section. 

%
\subsection{Proof of Lemmas \ref{S4AltGibbs} and \ref{S4WeakGibbs}}\label{Section11.1}
We begin by giving an analogue of Definition \ref{YVec}, and  proving a useful auxiliary result. 
\begin{definition}\label{YVecBar} For a finite set $J \subset \mathbb{Z}^2$, we let $Y^+(J)$ denote the space of functions $f: J \rightarrow (-\infty, \infty]$ with the Borel $\sigma$-algebra $\mathcal{D}^+$, coming from the natural identification of ${Y}(J)$ with $(-\infty, \infty]^{|J|}$. Similarly, we let $Y^-(J)$ denote the space of functions $f: J \rightarrow [-\infty, \infty)$ with the Borel $\sigma$-algebra $\mathcal{D}^-$ coming from the natural identification of ${Y}^-(J)$ with $[-\infty, \infty)^{|J|}$. We think of an element of $Y^{\pm}(J)$ as a $|J|$-dimensional vector whose coordinates are indexed by $J$. 
\end{definition}

\begin{lemma}\label{ContinuousGibbsCond} Let $H$ and $H^{RW}$ be as in Definition \ref{Pfree}. Suppose that $a,b,k_1, k_2 \in \mathbb{Z}$ with $a < b $  and $ k_1 \leq k_2 $. In addition, suppose that $h: Y( \llbracket k_1 ,k_2 \rrbracket \times \llbracket a,b \rrbracket) \rightarrow \mathbb{R}$ is a bounded Borel-measurable function (recall that $Y(J)$ was defined in Definition \ref{YVec}). Let $V_L =  \llbracket k_1, k_2 \rrbracket  \times \{a \}$, $V_R = \llbracket k_1, k_2 \rrbracket  \times \{b \}$, $V_T =  \{k_1 - 1\} \times \llbracket a,b \rrbracket$ and $V_B = \{k_2 + 1 \} \times \llbracket a, b\rrbracket$, and define the set
\begin{equation*}
\begin{split}
 S = \left\{ (\vec{x}, \vec{y}, \vec{u},\vec{v}) \in Y(V_L) \times Y(V_R) \times Y^+(V_T) \times Y^-(V_B) \right \},
\end{split}
\end{equation*}
where we endow $S$ with the product topology and corresponding Borel $\sigma$-algebra. Then, the function $G_h: S\rightarrow \mathbb{R}$, given by
\begin{equation}
G_h(\vec{x}, \vec{y}, \vec{u},\vec{v}) = \mathbb{E}_{H, H^{RW}}^{a,b,\vec{x}, \vec{y}, \vec{u} ,\vec{v}} \left[ h(\mathfrak{L})\right],
\end{equation}
is bounded and measurable. Moreover, if $h$ is also continuous, then so is $G_h$. In the above equations the random variable over which we are taking the expectation is denoted by $\mathfrak{L}$.
\end{lemma}
\begin{proof} Let us briefly explain the main ides behind the proof. It is clear that $|G_h|$ is bounded by $\|h\|_{\infty}$, which implies its boundedness. In equation (\ref{S11GRat}) below, we express $G_h$ as a ratio of two functions, with a strictly positive denominator. These functions themselves are integrals over a suitable Euclidean space of functions that are jointly measurable in the variables $ (\vec{x}, \vec{y}, \vec{u},\vec{v}) $ and the variables, over which we are integrating. We then deduce the measurability of $G_h$ from the measurability of the numerator and denominator in (\ref{S11GRat}), which in turn is a direct consequence of Fubini's theorem. The continuity of $G_h$, when $h$ is continuous, is obtained by showing that the integrals in the numerator and denominator in (\ref{S11GRat}) are continuous in $ (\vec{x}, \vec{y}, \vec{u},\vec{v}) $. As these functions are integrals of functions that are already continuous in $ (\vec{x}, \vec{y}, \vec{u},\vec{v}) $, their continuity follows from being able to exchange the order of integration and a limit in the arguments $(\vec{x}, \vec{y}, \vec{u},\vec{v})$. The latter is then a consequence of the Generalized dominated convergence theorem (see \cite[Theorem 4.17]{Royden}). We now turn to the details.\\

We denote $B  = \llbracket k_1, k_2 \rrbracket \times \llbracket a ,b \rrbracket$, and $A = \llbracket k_1, k_2 \rrbracket \times \llbracket a + 1, b-1 \rrbracket$. By definition of $\mathbb{P}_{H, H^{RW}}^{a,b,\vec{x}, \vec{y}, \vec{u} , \vec{v}}$ (see (\ref{RWN}), (\ref{RWB}) and (\ref{RND}) ), we know that
\begin{equation}\label{S11GRat}
G_h(\vec{x}, \vec{y}, \vec{u},\vec{v}) = \frac{F_h(\vec{x}, \vec{y}, \vec{u}, \vec{v})}{F_1(\vec{x}, \vec{y}, \vec{u}, \vec{v})}.
\end{equation}
In the above equation, $1$ stands for the constant function that is equal to $1$, and
\begin{equation}\label{S11E1}
\begin{split}
F_h(\vec{x}, \vec{y}, \vec{u}, \vec{v})  = \int_{Y(A)} &h( x_{i,j} : (i,j) \in B )   \cdot \frac{ P\big(\vec{x}, \vec{y};  x_{i,j}: (i,j) \in A \big)}{\prod_{i = k_1}^{k_2} \g^{x_{i,a}}_{b - a }(y_{i,b}) } \times \\
&Q\big(\vec{x}, \vec{y},\vec{u}, \vec{v} ;  x_{i,j}: (i,j) \in A \big)  \hspace{-3mm}\prod_{(i,j) \in A}\hspace{-3mm} dx_{i,j} ,
\end{split}
\end{equation}
where
\begin{equation}\label{S11GibbsPQ}
\begin{split}
&Q\big(\vec{x}, \vec{y}, \vec{u}, \vec{v};  x_{i,j}: (i,j) \in A \big) =  \exp \left( - \sum_{i = k_1 - 1}^{k_2}  \sum_{ m = a}^{ b - 1} H(x_{i + 1, m+1} - x_{i,m}) \right) ,\\
& P\big(\vec{x}, \vec{y};  x_{i,j}: (i,j) \in A \big) = \prod_{i = k_1}^{k_2} \prod_{m = a + 1}^{b} \g(x_{i,m} - x_{i, m-1}),
\end{split}
\end{equation}
and also $x_{k_1-1, j} = u_{k_1 - 1, j}$, $x_{k_2 + 1,j} = v_{k_2 + 1, j}$ for $j \in \llbracket a, b \rrbracket$, and $x_{i,b} = y_{i,b}$ for $i \in \llbracket k_1, k_2 \rrbracket$.
If $b = a+1$ the function $F_h$ takes the form
\begin{equation}\label{S11E2}
F_h(\vec{x}, \vec{y}, \vec{u}, \vec{v})= h( x_{i,j} : (i,j) \in B ) \cdot \exp \left( - \sum_{i = k_1}^{k_2}  H (x_{i+1,b} - x_{i,a}) \right).
\end{equation}
We mention here that our assumption that $\vec{u} \in Y^+(V_T)$ and $\vec{v} \in Y^-(V_B)$ ensures that the arguments in $H$ are all well-defined (i.e. we do not have $\infty -\infty$), and moreover they lie in $[-\infty, \infty)$. In particular, all the functions above are well-defined and finite.

Since $H \geq 0$, we see that
$$Q\big(\vec{x}, \vec{y}, \vec{u}, \vec{v};  x_{i,j}: (i,j) \in A \big)   \leq 1 \hspace{2mm} \mbox{ and }\hspace{2mm}  \int_{Y(A)}   \frac{P\big(\vec{x}, \vec{y};  x_{i,j}: (i,j) \in A \big)  }{\prod_{i = k_1}^{k_2} \g^{x_{i}}_{b - a }(y_{i}) } \hspace{-3mm} \prod_{(i,j) \in A} \hspace{-3mm} dx_{i,j} = 1,$$
and so $|F_h| \leq \|h\|_{\infty}$. By Fubini's theorem we conclude $F_h$ is measurable. This implies the measurability of $G_h$, which is the ratio of two measurable functions with a strictly positive denominator.\\

What remains is to show $G_h$ is continuous if $h$ is continuous and bounded, which we assume in the sequel. Since $G_h$ is the ratio of two functions, we see that it suffices to prove that $F_h(\vec{x}, \vec{y}, \vec{u}, \vec{v}) $ is continuous if $h$ is bounded and continuous.

Fix some point $(\vec{x}_{\infty},\vec{y}_{\infty}, \vec{u}_{\infty}, \vec{v}_{\infty}) \in Y(V_L) \times Y(V_R) \times Y^+(V_T) \times Y^-(V_B) $, and suppose that we are given any sequence $(\vec{x}_n, \vec{y}_n, \vec{u}_n, \vec{v}_n) \in Y(V_L) \times Y(V_R) \times Y^+(V_T) \times Y^-(V_B)$, which converges to $(\vec{x}_{\infty},\vec{y}_{\infty}, \vec{u}_{\infty}, \vec{v}_{\infty}) $. Then, we wish to establish that
\begin{equation}\label{S11ContGLim}
\begin{split}
& \lim_{n \rightarrow \infty} F_h(\vec{x}_n, \vec{y}_n,\vec{u}_n, \vec{v}_n) =  F_h(\vec{x}_{\infty}, \vec{y}_{\infty},\vec{u}_{\infty}, \vec{v}_{\infty}).
\end{split}
\end{equation}
We know that for $d \geq 1$, $f \in L^1(\mathbb{R}^d)$ and $g \in L^\infty(\mathbb{R}^d)$ we have that
$$U(x) := \int_{\mathbb{R}^d} f(y) g(x - y) dy$$
is bounded and continuous, see e.g. \cite[Proposition 2.39]{Folland}. The latter implies, in view of (\ref{RWN}), that $\g^{x_{i,a}}_{b- a}(y_{i,b})$ are positive and continuous functions of $\vec{x}$ and $\vec{y}$. In particular, we see that to prove (\ref{S11ContGLim}) it suffices to show that
\begin{equation}\label{S11GDCT}
\begin{split}
&\lim_{n \rightarrow \infty} \hspace{-1mm}\int_{Y(A)} Q\big(\vec{x}_n, \vec{y}_n, \vec{u}_n, \vec{v}_n;  A \big)   P\big(\vec{x}_n, \vec{y}_n;  A \big) h\big( x_{i,j} : (i,j) \in A  \big)  \prod_{(i,j) \in A}\hspace{-3mm} dx_{i,j} = \\
& \int_{Y(A)}  Q\big(\vec{x}_{\infty}, \vec{y}_{\infty}, \vec{u}_\infty, \vec{v}_\infty; A \big)   P\big(\vec{x}_{\infty}, \vec{y}_{\infty}; A \big)  h\big(  x_{i,j} : (i,j) \in A  \big)\hspace{-2mm} \prod_{(i,j) \in A}  dx_{i,j},
\end{split}
\end{equation}
where $P,Q$ are as in (\ref{S11GibbsPQ}) (we have replaced $ x_{i,j}: (i,j) \in A$ with $A$ above to ease the notation).

By continuity of $H$ and $\g$, we know that the integrand in the top line of (\ref{S11GDCT}) converges pointwise to the integrand in the second line. The fact that the integrals also converge then follows from the Generalized dominated convergence theorem (see \cite[Theorem 4.17]{Royden}) with dominating functions
$$ \g_n \big(  x_{i,j}: (i,j) \in A\big) = \| h\|_\infty \cdot M^{\topc - 1} \cdot \prod_{i = 1}^{\topc-1} \prod_{m = T_0 + 1}^{T_1} \g\big(x^n_{i,m} - x^n_{i, m-1}\big), \mbox{ where $M =  \| \g\|_{\infty}$,}$$
where $x_{i,j}^n = x_{i,j}$ if $(i,j) \in A$, $x_{i,j}^n$ equals the $(i,j)$-th coordinate of $\vec{x}_n$ if $(i,j) \in V_L$ and $x_{i,j}^n$ equals the $(i,j)$-th coordinate of $\vec{y}_n$ if $(i,j) \in V_R$.

Let us elaborate the last argument briefly. Since $H \geq 0$ by assumption we know that $|Q| \leq 1$ and then it is clear that $\g_n$ as above dominate the integrands in the top line of (\ref{S11GDCT}). Furthermore, by the continuity of the integrands we conclude that $\g_n$ converge pointwise to $\g_\infty$, which has the same form as $\g_n$ with $\vec{x}_n$ and $\vec{y}_n$ replaced with $\vec{x}_\infty, \vec{y}_\infty$. To conclude the application of the Generalized dominated convergence theorem we need to show
$$\lim_{n \rightarrow \infty} \int_{Y(A)}\prod_{i = k_1}^{k_2} \prod_{m = a + 1}^{b} \g\big(x^n_{i,m} - x^n_{i, m-1}\big) \prod_{(i,j) \in A} dx^n_{i,j}= \int_{Y(A)}\prod_{i = k_1}^{k_2} \prod_{m = a+ 1}^{b} \g\big(x^{\infty}_{i,m} - x^{\infty}_{i, m-1}\big)\prod_{(i,j) \in A} dx^{\infty}_{i,j},$$
which after the change of variables $\tilde{x}_{i,j} = x^n_{i,j} - x^n_{i,j-1}$ for $(i,j) \in A$ is equivalent to
\begin{equation}\label{S11Grewritten}
\begin{split}
& \lim_{n \rightarrow \infty}  \int_{Y(A)}\prod_{i = k_1}^{k_2} \prod_{m = a + 1}^{b-1} \g(\tilde{x}_{i,m} ) \cdot \prod_{i = k_1}^{k_2} \g \Big(y^n_{i,b} - x^n_{i, a} - \sum_{j = a+1}^{b - 1} x_{i, j}\Big) \prod_{(i,j) \in A} d\tilde{x}_{i,j} = \\
&\int_{Y(A)}\prod_{i = k_1}^{k_2} \prod_{m = a + 1}^{b-1} \g(\tilde{x}_{i,m} ) \cdot \prod_{i = k_1}^{k_2} \g \Big(y^\infty_{i,b} - x^\infty_{i, a} - \sum_{j = a+1}^{b- 1} x_{i, j}\Big) \prod_{(i,j) \in A} d\tilde{x}_{i,j}.
\end{split}
\end{equation}
Equation (\ref{S11Grewritten}) is now a consequence of the dominated convergence theorem (see \cite[Theorem 4.16]{Royden}) with dominating function $\| \g\|_{\infty}  \cdot \prod_{i = k_1}^{k_2} \prod_{m = a + 1}^{b-1} \g(\tilde{x}_{i,m} )$. Thus the Generalized dominated convergence theorem is applicable and implies (\ref{S11GDCT}).
\end{proof}

We next prove Lemma \ref{S4AltGibbs}, whose statement is recalled here for the reader's convenience.
\begin{lemma}\label{AltGibbs} Let $H^{RW}$ and $H$ be as in Definition \ref{Pfree}. Fix $\topc \geq 2$, two integers $T_0 < T_1$, and set $\Sigma = \{1,\dots,\topc\}$. Define sets $A =  \llbracket 1, \topc-1 \rrbracket \times \llbracket T_0 + 1, T_1 - 1\rrbracket$, and $B =\Sigma\times \llbracket T_0, T_1 \rrbracket \setminus A$.  Suppose that $\mathbb{P}$ is a probability distribution on  a $\Sigma\times \llbracket T_0, T_1 \rrbracket$-indexed discrete line ensemble $\mathfrak{L} = (L_1, \dots, L_\topc)$. Then, the following two statements are equivalent:
\begin{enumerate}
\item $\mathbb{P}$ satisfies the $(H, H^{RW})$-Gibbs property;
\item For any bounded continuous functions $f_{i,j}$ on $\mathbb{R}$ with $(i,j) \in \llbracket 1, \topc \rrbracket \times \llbracket T_0 , T_1 \rrbracket$, we have
\begin{equation}\label{towerFun}
\begin{split}
& \mathbb{E}\Bigg[\prod_{i = 1}^\topc \prod_{j = T_0}^{T_1} f_{i,j}(  L_i(j)  )\Bigg]  = \\
& \mathbb{E} \Bigg[ \prod_{(i,j) \in B}f_{i,j}( L_i(j) ) \cdot  \mathbb{E}_{H, H^{RW}}^{1, \topc-1, T_0, T_1, \vec{x}, \vec{y},\infty,L_\topc\llbracket T_0,T_1\rrbracket} \bigg[\prod_{(i,j) \in A} f_{i,j}( \tilde{L}_i(j)  ) \bigg]  \Bigg],
\end{split}
\end{equation}
where $\vec{x} =  (L_{1}(T_0), \dots L_{\topc-1}(T_0))$, $\vec{y} = (L_{1}(T_1), \dots L_{\topc-1}(T_1))$ and $\tilde{\mathfrak{L}} = (\tilde{L}_1, \dots, \tilde{L}_{\topc-1})$ is distributed according to
 $ \mathbb{P}_{H, H^{RW}}^{1, \topc-1, T_0, T_1, \vec{x}, \vec{y},\infty,L_\topc\llbracket T_0,T_1\rrbracket} $.
\end{enumerate}
Moreover, if $\vec{z} \in [-\infty, \infty)^{T_1 - T_0 +1}$ and $\vec{x}, \vec{y} \in \mathbb{R}^{\topc-1}$ then $\mathbb{P}_{H, H^{RW}}^{1, \topc-1, T_0, T_1, \vec{x}, \vec{y},\infty,\vec{z}}$ from Definition \ref{Pfree} satisfies the $(H,H^{RW})$-Gibbs property in the sense that (\ref{GibbsEq}) holds for all $1 \leq k_1 \leq k_2 \leq \topc-1$, $T_0 \leq a < b \leq T_1$ and bounded Borel-measurable $F$ on $Y( \llbracket k_1, k_2 \rrbracket \times \llbracket a, b \rrbracket)$.
\end{lemma}
\begin{proof} Throughout the proof we write $\mathbb{E}_{H, H^{RW}}$ in place of $\mathbb{E}_{H, H^{RW}}^{1, \topc-1, T_0, T_1, \vec{x}, \vec{y},\infty,L_\topc \llbracket T_0, T_1 \rrbracket}$ to ease the notation.  For clarity we split the proof into several steps. In the first step, we show that (1) $\implies$ (2), which is the easy part of the lemma. In Step 2, we reduce the proof of the lemma to establishing a certain equality of expectations of products of indicator functions -- this is (\ref{GibbsRed2}). In Step 3, we show that one can replace in (\ref{towerFun}) the functions $f_{i,j}$ by indicators of half-infinite lines and still have the equality. In Step 4, we use our result from Step 3 to find a suitable integral representation for a tower of conditional expectations as in the second line of (\ref{towerFun}). In Steps 5 and 6, we use our result from Step 4 to find integral representations of the right and left sides of (\ref{GibbsRed2}) and show they are equal. In Step 7, we prove the second part of the lemma, which essentially follows the same approach from Steps 2, 4, 5 and 6 with a few slight modifications. \\

{\bf \raggedleft Step 1.} In this step, we show that (1) $\implies$ (2). Let us fix any bounded continuous function $f$ on $Y(A)$  (here $Y$ is as in Definition \ref{YVec})and let $\mathcal{H}$ denote the set of bounded Borel functions $h$ on $Y(B)$, which satisfy
\begin{equation}\label{towerFunV2}
\begin{split}
& \mathbb{E}\Big[ h\big( \mathfrak{L}\vert_{B} \big)\cdot f\big(  \tilde{\mathfrak{L}}\vert_{A}\big) \Big]  = \mathbb{E} \bigg[ h\big(  \mathfrak{L}\vert_{B} \big) \cdot  \mathbb{E}_{H, H^{RW}}\Big[f\big(  \tilde{\mathfrak{L}}\vert_{A}\big) \Big]  \bigg].
\end{split}
\end{equation}
We recall that $\mathfrak{L}\vert_{B}$ was introduced in Section \ref{Section4.1}, and denoted the restriction of the vector to the coordinates indexed by the set $B$.

Using (\ref{GibbsEq}) with $k_1 = 1, k_2 = \topc-1, a= T_0, b = T_1, F = f$ and the defining properties of conditional expectations we know that
$$1 \in \mathcal{H} \mbox{ and for any numbers $a_{i,j} \in \mathbb{R}$ the function }  h(x_{i,j} :  (i,j) \in B) = \prod_{(i,j) \in B} {\bf 1} \{ x_{i,j} \leq a_{i,j}\} \in \mathcal{H}.$$
 Furthermore, by linearity of expectations we have that if $h_1, h_2 \in \mathcal{H}$ then $h_1 + h_2 \in \mathcal{H}$ and $c h_1 \in \mathcal{H}$ for any $c \in \mathbb{R}$. Finally, suppose that $h_n \in \mathcal{H}$ and $h_n$ is an increasing sequence of non-negative functions that converges to a bounded function $h$. This means that $h_n( \mathfrak{L}\vert_{B}) $ increases almost surely to $h(\mathfrak{L}\vert_{B}) $ and so by the bounded convergence theorem we conclude that $h \in \mathcal{H}$. An application of the Monotone class theorem (see e.g. \cite[Theorem 5.2.2]{Durrett}) shows that $\mathcal{H}$ contains all bounded Borel functions, which in particular proves (\ref{towerFun}). \\

{\bf \raggedleft Step 2.} In the next steps we show that (2) $\implies$  (1), which is the hard part of the proof.  Let us fix $1 \leq k_1 < k_2 \leq \topc-1$ and $T_0 \leq a < b \leq T_1$. We set $D =\llbracket k_1, k_2 \rrbracket \times \llbracket a, b \rrbracket$ and $C = \llbracket k_1, k_2 \rrbracket \times \llbracket a + 1, b - 1 \rrbracket$. Also, for an arbitrary set $J \subset \Sigma \times \llbracket T_0, T_1 \rrbracket$, we put $\mathcal{F}_J = \sigma( L_i(s): (i,s) \in J)$. Then, we want to show that if $F: Y(\llbracket k_1, k_2 \rrbracket \times \llbracket a, b \rrbracket) \rightarrow \mathbb{R}$ is a bounded Borel-measurable function, then $\mathbb{P}$-almost surely
$$\mathbb{E} \Big[ F \big( \mathfrak{L}\vert_{D} \big)  \big{\vert} \mathcal{F}_{A\cup B \setminus C} \Big]  = \mathbb{E}_{H,H^{RW}}^{k_1, k_2, a ,b, \vec{u}, \vec{v},L_{k_1 - 1}\llbracket a ,b \rrbracket,L_{k_2+  1}\llbracket a ,b \rrbracket}  \big[F( \mathfrak{L}') \big],$$
with the convention that $L_0 = \infty$ if $k_1 = 1$. In the above, the $D$-indexed discrete line ensemble $\mathfrak{L}' = (L'_{k_1}, \dots, L'_{k_2})$ is distributed according to $\mathbb{P}_{H,H^{RW}}^{k_1, k_2, a ,b, \vec{u}, \vec{v},L_{k_1 - 1}\llbracket a ,b \rrbracket,L_{k_2+  1}\llbracket a ,b \rrbracket}$, where $\vec{u} = (L_{k_1}(a), \dots, L_{k_2}(a))$, $\vec{v} =  (L_{k_1}(b), \dots, L_{k_2}(b))$. We will write $\mathbb{P}_{H, H^{RW}}'$ for this measure and $\mathbb{E}_{H, H^{RW}}'$ for the corresponding expectation. From the defining properties of conditional expectation, we see that it suffices to prove that for $R \in \mathcal{F}_{A\cup B \setminus C}$ we have
\begin{equation}\label{GibbsRed1}
\mathbb{E} \Big[ {\bf 1}_R \cdot F\big( \mathfrak{L}\vert_{D} \big)  \Big]  = \mathbb{E} \left[{\bf 1}_R  \cdot \mathbb{E}_{H,H^{RW}}'  \big[F( \mathfrak{L}') \big] \right].
\end{equation}

We claim that if we fix $a_{i,j} \in \mathbb{R}$ for $(i,j) \in A \cup B \setminus C$ and $b_{i,j} \in \mathbb{R}$ for $(i,j) \in D$, then
\begin{equation}\label{GibbsRed2}
\begin{split}
&\mathbb{E} \left[ \prod_{(i,j) \in A \cup B \setminus C} {\bf 1}\{ L_i(j) \leq a_{i,j} \} \cdot \prod_{(i,j) \in D} {\bf 1}\{ L_i(j) \leq b_{i,j} \} \right]  = \\
& \mathbb{E} \left[ \prod_{(i,j) \in A \cup B \setminus C} {\bf 1}\{ L_i(j) \leq a_{i,j} \}   \cdot \mathbb{E}_{H,H^{RW}}'  \left[\prod_{(i,j) \in D} {\bf 1}\{ L'_i(j) \leq b_{i,j} \} \right] \right].
\end{split}
\end{equation}
We prove (\ref{GibbsRed2}) in the steps below. Here we assume its validity and conclude the proof of (\ref{GibbsRed1}).\\

Fix $a_{i,j} \in \mathbb{R}$ for $(i, j) \in A\cup B \setminus C$. Let $\mathcal{H}$ denote the set of bounded Borel functions $h$ on $Y(D)$ that satisfy
\begin{equation}\label{S2CDFC5}
\begin{split}
&  \mathbb{E} \left[  \prod_{(i,j) \in A \cup B \setminus C}  \hspace{-4mm} {\bf 1} \{L_i(j) \leq a_{i,j} \} \cdot h\big( \mathfrak{L} \vert_D\big)   \right] = \mathbb{E} \left[  \prod_{(i,j) \in A \cup B \setminus C} \hspace{-4mm}{\bf 1} \{L_i(j) \leq a_{i,j} \} \cdot \mathbb{E}'_{H, H^{RW}}  \Big[ h\big( \mathfrak{L}'\big)  \Big] \right].
\end{split}
\end{equation}
 From (\ref{GibbsRed2}), we know that
$$1 \in \mathcal{H} \mbox{ and for any numbers $b_{i,j} \in \mathbb{R}$ the function }  h( x_{i,j} :  (i,j) \in D) = \prod_{(i,j) \in D} {\bf 1} \{ x_{i,j} \leq b_{i,j}\} \in \mathcal{H}.$$
By linearity of expectation, we have that if $h_1, h_2 \in \mathcal{H}$ then $h_1 + h_2 \in \mathcal{H}$ and $c h_1 \in \mathcal{H}$ for any $c \in \mathbb{R}$. Moreover, if $h_n$ is an increasing sequence of non-negative functions that converges to a bounded function $h$, then by the bounded convergence theorem $\mathbb{P}$-almost surely
\begin{equation*}
\begin{split}
&\lim_{n \rightarrow \infty} \mathbb{E}'_{H, H^{RW}} \Big[ h_n\big( \mathfrak{L}' \big) \Big] = \mathbb{E}'_{H, H^{RW}} \Big[ h\big( \mathfrak{L}' \big) \Big],
\end{split}
\end{equation*}
 which after a second application of the bounded convergence theorem implies that $h \in \mathcal{H}$. By the Monotone class theorem, we conclude that $\mathcal{H}$ contains all bounded Borel functions. In particular, it contains the function $F$.

Finally, let us fix a bounded Borel function $F$ on $Y(D)$. Suppose that $\mathcal{R}$ denotes the collection of sets $R  \in \mathcal{F}_{A \cup B \setminus C} $, such that (\ref{GibbsRed1}) holds. Then, it is clear that $\Omega \in \mathcal{R}$, and by the bounded convergence theorem if $R_n \uparrow R$ with $R_n \in \mathcal{R}$, then $R \in \mathcal{R}$. Moreover, from (\ref{S2CDFC5}), applied to $h = F$, we know that $\mathcal{R}$ contains the $\pi$-system of sets $\{ L_i(j) \leq a_{i,j}: (i,j) \in A\cup B \setminus C\}$ with $a_{i,j} \in \mathbb{R}$ for $(i,j) \in  A\cup B \setminus C$. It follows from the $\pi-\lambda$ Theorem (see e.g. \cite[Theorem 2.1.6]{Durrett}) that $\mathcal{R} = \mathcal{F}_{A \cup B \setminus C} $, which proves (\ref{GibbsRed1}). Since $k_1, k_2, a ,b$ and $F$ were arbitrary we conclude that $\mathbb{P}$ satisfies the $(H,H^{RW})$-Gibbs property.\\

{\bf \raggedleft Step 3.} In this and the next steps, we establish (\ref{GibbsRed2}). Fix $c_{i,j} \in \mathbb{R}$ for $(i,j) \in A \cup B$ and define
$$\hat{f}\big( x_{i,j}: (i,j) \in A \big) = \prod_{(i,j) \in A} {\bf 1} \{x_{i,j} \leq c_{i,j} \} ,\qquad \hat{g}\big( x_{i,j}: (i,j) \in B \big) = \prod_{(i,j) \in B} {\bf 1} \{x_{i,j} \leq c_{i,j} \}.$$
Observe that $\hat{f}$ (resp. $\hat{g}$) is a bounded measurable function on $Y(A)$ (resp. $Y(B)$). The purpose of this step is to establish that
\begin{equation}\label{S2BCT2}
\begin{split}
& \mathbb{E}\Big[ \hat{f}\big(  \mathfrak{L}\vert_A\big) \hat{g}\big( \mathfrak{L}\vert_B \big) \Big]  =  \mathbb{E} \bigg[ \hat{g}\big( \mathfrak{L}\vert_B\big) \cdot  \mathbb{E}_{H, H^{RW}}\Big[ \hat{f}\big( \tilde{ \mathfrak{L}}\vert_A \big) \Big]  \bigg].
\end{split}
\end{equation}
If $\hat{f}$ and $\hat{g}$ were products of bounded continuous functions, the above equation would follow from (\ref{towerFun}). However, the indicator functions in the definition of $\hat{f}$ and $\hat{g}$ are not continuous but can be approximated by continuous functions $\hat{f}_n$ and $\hat{g}_n$, and then (\ref{S2BCT2}) would follow from taking the limit of (\ref{towerFun}) applied to $\hat{f}_n$ and $\hat{g}_n$. We supply the details below.\\

For $r \in \mathbb{R}$ and $n \in \mathbb{N}$ we define
$$h_n(x;r) = \begin{cases} 0 &\mbox{ if $x > r + n^{-1}$} \\ 1 - n(x - r) &\mbox{ if $x \in [r, r+ n^{-1}]$} \\ 1 &\mbox{ if $x < r$}.   \end{cases}$$
Let us set
$$\hat{f}_n\big(x_{i,j}: (i,j) \in A \big) = \prod_{(i,j) \in A} h_n(x_{i,j}; c_{i,j}) ,\qquad \hat{g}_n\big( x_{i,j}: (i,j) \in B \big) = \prod_{(i,j) \in B} h_n(x_{i,j}; c_{i,j}),$$
Clearly, $\hat{f}_n$ (resp. $\hat{g}_n$) are bounded continuous functions on $Y(A)$ (resp. $Y(B)$). From (\ref{towerFun}) we get
\begin{equation}\label{S2BCT}
\begin{split}
& \mathbb{E}\Big[  \hat{g}_n\big(\mathfrak{L}\vert_B\big) \hat{f}_n\big(  \mathfrak{L}\vert_A\big) \Big]  = \mathbb{E} \bigg[ \hat{g}_n\big(\mathfrak{L}\vert_B\big) \cdot  \mathbb{E}_{H, H^{RW}}\Big[ \hat{f}_n\big(  \tilde{\mathfrak{L}}\vert_A \big)\Big]  \bigg].
\end{split}
\end{equation}
 We first observe that for any deterministic $ \vec{x} \in Y( \llbracket 1, \topc-1 \rrbracket \times \{T_0\})$, $\vec{y} \in Y( \llbracket 1, \topc-1 \rrbracket \times \{T_1\})$ and $\vec{z} \in Y(\{\topc\} \times \llbracket T_0, T_1 \rrbracket)$  we have that
$$\lim_{n \rightarrow \infty} \mathbb{E}_{H, H^{RW}}^{1, \topc-1, T_0, T_1, \vec{x}, \vec{y},\infty,\vec{z}} \Big[ \hat{f}_n\big( \tilde{\mathfrak{L}}'\vert_A \big) \Big] = \mathbb{E}_{H,H^{RW}}^{1, \topc-1, T_0, T_1, \vec{x}, \vec{y},\infty,\vec{z}} \Big[ \hat{f}\big( \tilde{\mathfrak{L}}'\vert_A \big) \Big],$$
as a consequence of the bounded convergence theorem. In the above $\tilde{\mathfrak{L}}'$ is $ \mathbb{P}_{H, H^{RW}}^{1, \topc-1, T_0, T_1, \vec{x}, \vec{y},\infty, \vec{z}} $-distributed. It follows that $\mathbb{P}$-almost surely
\begin{equation*}
\begin{split}
\lim_{n \rightarrow \infty} &\hat{f}_n\big( \mathfrak{L}\vert_A \big) \hat{g}_n\big( \mathfrak{L}\vert_B\big) = \hat{f}\big(  \mathfrak{L}\vert_A \big) \hat{g}\big( \mathfrak{L}\vert_B\big),\\
\lim_{n \rightarrow \infty} &\hat{g}_n\big( \mathfrak{L}\vert_B\big) \cdot  \mathbb{E}_{H, H^{RW}} \Big[ \hat{f}_n\big(  \tilde{\mathfrak{L}}\vert_A\big) \Big]=  \hat{g}\big( \mathfrak{L}\vert_B\big) \cdot  \mathbb{E}_{H, H^{RW}}\Big[ \hat{f}\big(  \tilde{\mathfrak{L}}\vert_A\big) \Big].
\end{split}
\end{equation*}
Since all of the above functions are uniformly bounded in $[0,1]$ we conclude by (\ref{S2BCT}) and the bounded convergence theorem that (\ref{S2BCT2}) holds.\\

{\bf \raggedleft Step 4.} In this step we find an integral representation of the right side of (\ref{S2BCT2}). Suppose we are given $c_{i,j} \in \mathbb{R}$ for $(i,j) \in A$, and then define the function $G_1: Y(B) \rightarrow \mathbb{R}$ via
\begin{equation}\label{GOMFG}
G_1( x_{i,j} : (i,j) \in B) = \mathbb{E}_{H,H^{RW}}^{1, \topc-1, T_0 ,T_1, \vec{x}, \vec{y}, \infty ,L_{\topc}\llbracket T_0 ,T_1 \rrbracket}  \left[\prod_{(i,j) \in A} {\bf 1}\{ L'_i(j) \leq c_{i,j} \} \right],
\end{equation}
where $\vec{x} = (x_{1,T_0}, \dots, x_{\topc-1,T_0})$, $\vec{y} = (x_{1,T_1}, \dots, x_{\topc-1,T_1})$, and $L_{\topc}\llbracket T_0 ,T_1 \rrbracket = (x_{\topc,T_0}, \dots, x_{\topc,T_1})$. Also $\mathfrak{L}' = (L'_{1}, \dots, L'_{\topc-1})$ is $ \mathbb{P}_{H,H^{RW}}^{1, \topc-1, T_0 ,T_1, \vec{x}, \vec{y}, \infty ,L_{\topc}\llbracket T_0,T_1 \rrbracket}  $-distributed. By definition of $\mathbb{E}_{H,H^{RW}}$, we see that
\begin{equation}\label{GOMFG2}
\begin{split}
&G_1(x_{i,j}: (i,j) \in B) = \int_{Y(A)} \hspace{-5mm} \frac{\exp \left( - \sum_{i =1}^{\topc-1}  \sum_{ m = T_0}^{T_1-1}{ H} (y_{i + 1, m+1} - y_{i,m}) \right) }{Z_{H, H^{RW}}^{1, \topc-1,T_0 ,T_1, \vec{x}, \vec{y},\infty, L_{\topc}\llbracket T_0,T_1\rrbracket }}     \\
&  \prod_{i = 1}^{\topc-1} \frac{ \prod_{m = T_0 + 1}^{T_1} \g(y_{i,m} - y_{i, m-1}) }{\g^{x_{i,T_0}}_{b - a }(x_{i, T_1}) } \prod_{(i,j) \in A}{\bf 1} \{y_{i,j} \leq c_{i,j} \}   \prod_{(i,j) \in A} dy_{i,j},
\end{split}
\end{equation}
where we recall that $\g$ and $\g_n^x$ were defined in Definition \ref{DefBridge} and also that $Y(A)$ is naturally identified with $\mathbb{R}^{|A|}$ and then $\prod_{(i,j) \in A} dy_{i,j}$ stands for the usual Lebesgue measure on this space. Also, in (\ref{GOMFG2}) we use the convention that $y_{i,j} = x_{i,j}$ for $(i,j) \in B$.

From Lemma \ref{ContinuousGibbsCond} we know that $G_1$ is bounded and measurable on $Y(B)$. We thus conclude that the right side of (\ref{S2BCT2}) can be rewritten as
\begin{equation}\label{GOMFG3}
\begin{split}
&\mathbb{E} \bigg[ \hat{g}\big( \mathfrak{L}\vert_B\big) \cdot  \mathbb{E}_{H, H^{RW}}\Big[ \hat{f}\big( \tilde{ \mathfrak{L}}\vert_A \big) \Big]  \bigg] = \mathbb{E} \left[\hat{g}\big( \mathfrak{L}\vert_B\big) \cdot G_1(L_i(j): (i,j) \in B)  \right] = \\
& \int_{Y(B)} \int_{Y(A)} \prod_{i = 1}^{\topc-1} \frac{\prod_{m = T_0 + 1}^{T_1} \g(x_{i,m} - x_{i, m-1}) }{\g^{x_{i,T_0}}_{T_1 - T_0 }(x_{i, T_1}) }  \frac{\exp \left( - \sum_{i = 1}^{\topc-1}  \sum_{ m = T_0}^{T_1-1}{ H} (x_{i + 1, m+1} - x_{i,m}) \right) }{Z_{H, H^{RW}}^{1, \topc-1,T_0 ,T_1, \vec{x}, \vec{y},\infty,  L_\topc\llbracket T_0, T_1\rrbracket} } \\
& \prod_{(i,j) \in A }{\bf 1} \{x_{i,j} \leq c_{i,j} \}   \prod_{(i,j) \in A} dx_{i,j} \times  \prod_{(i,j) \in  B}{\bf 1} \{x_{i,j} \leq c_{i,j} \}  \mu_B(dx_{i,j} : (i,j) \in B),
\end{split}
\end{equation}
where $\mu_B$ the push-forward measure of $\mathbb{P}$ on $Y(B)$ obtained by the projection $ \mathfrak{L} \rightarrow \mathfrak{L}\vert_B$. This is the integral representation we were after.\\

{\bf \raggedleft Step 5.} In this step we find an integral representation of the second line of (\ref{GibbsRed2}) by utilizing (\ref{S2BCT2}) and (\ref{GOMFG3}).
If we take the limit $c_{i,j}\rightarrow \infty$ for $(i,j) \in C$ on the left side of (\ref{S2BCT2}), and in (\ref{GOMFG3}), and apply the monotone convergence theorem we conclude
\begin{equation*}
\begin{split}
&\mathbb{P}\big( \{ L_i(j) \leq c_{i,j}: (i,j) \in A \cup B \setminus C \}\big) =  \int_{Y(B)} \int_{Y(A)} \prod_{i = 1}^{\topc-1} \frac{\prod_{m = T_0 + 1}^{T_1} \g(x_{i,m} - x_{i, m-1}) }{\g^{x_{i,T_0}}_{T_1 - T_0 }(x_{i, T_1}) } \times \\
& \frac{\exp \left( - \sum_{i = 1}^{\topc-1}  \sum_{ m = T_0}^{T_1-1}{ H} (x_{i + 1, m+1} - x_{i,m}) \right) }{Z_{H, H^{RW}}^{1, \topc-1,T_0 ,T_1, \vec{x}, \vec{y},\infty,  L_\topc\llbracket T_0, T_1\rrbracket} }  \prod_{(i,j) \in A \cup B \setminus C} \hspace{-5mm}{\bf 1} \{x_{i,j} \leq c_{i,j} \} \prod_{(i,j) \in A} \hspace{-2mm}dx_{i,j} \cdot \mu_B(dx_{i,j} : (i,j) \in B).
\end{split}
\end{equation*}

The above formula gives us an expression for the joint cumulative distribution of $\mathfrak{L} \vert_{A\cup B \setminus C}$. In particular, say by the Monotone class theorem, we conclude that for any bounded Borel-measurable $G_2$ on $Y(A \cup B \setminus C)$ we have
\begin{equation}\label{S2CDFC}
\begin{split}
&\mathbb{E}\big[G_2 \big( \mathfrak{L} \vert_{A \cup B \setminus C} \big)\big] =  \int_{Y(B)} \int_{Y(A)} \prod_{i = 1}^{\topc-1} \frac{\prod_{m = T_0 + 1}^{T_1} \g(x_{i,m} - x_{i, m-1}) }{\g^{x_{i,T_0}}_{T_1 - T_0 }(x_{i, T_1}) } \times \\
& \frac{\exp \left( - \sum_{i = 1}^{\topc-1}  \sum_{ m = T_0}^{T_1-1}{ H} (x_{i + 1, m+1} - x_{i,m}) \right) }{Z_{H, H^{RW}}^{1, \topc-1,T_0 ,T_1, \vec{x}, \vec{y},\infty,  L_\topc\llbracket T_0, T_1\rrbracket} }  \times   \\
&  G_2(x_{i,j} : (i,j) \in A \cup B \setminus C) \prod_{(i,j) \in A} dx_{i,j} \cdot \mu_B(dx_{i,j} : (i,j) \in B).
\end{split}
\end{equation}

Let us fix $b_{i,j} \in \mathbb{R}$ for $(i,j) \in D$, and define the function $G_3: Y(A \cup B \setminus C) \rightarrow \mathbb{R}$ through
\begin{equation}\label{S2CDFC3}
\begin{split}
&G_3(x_{i,j}: (i,j) \in A \cup B \setminus C) = \mathbb{E}_{H,H^{RW}}^{k_1, k_2, a ,b, \vec{u}, \vec{v},L_{k_1 - 1}\llbracket a ,b \rrbracket,L_{k_2+  1}\llbracket a ,b \rrbracket}  \left[\prod_{(i,j) \in D} {\bf 1}\{ L'_i(j) \leq b_{i,j} \} \right],
\end{split}
\end{equation}
where $\vec{u} = (x_{k_1,a}, \dots, x_{k_2,a})$, $\vec{v} = (x_{k_1,b}, \dots, x_{k_2,b})$, $L_{k_1 - 1}\llbracket a ,b \rrbracket = (x_{k_1-1, a}, \dots, x_{k_1-1,b})$ if $k_1 \geq 2$ and $L_{k_1 - 1} = (\infty)^{b-a+1}$ if $k_1 = 1$, and $L_{k_2+  1}\llbracket a ,b \rrbracket = (x_{k_2 + 1,a}, \dots, x_{k_2 + 1,b})$. Also $\mathfrak{L}' = (L'_{k_1}, \dots, L'_{k_2})$ is $ \mathbb{P}_{H,H^{RW}}^{k_1, k_2, a ,b, \vec{u}, \vec{v},L_{k_1 - 1}\llbracket a ,b \rrbracket,L_{k_2+  1}\llbracket a ,b \rrbracket}  $-distributed. By definition of $\mathbb{E}_{H,H^{RW}}$ we see that
\begin{equation}\label{S2CDFC4}
\begin{split}
&G_3(x_{i,j}: (i,j) \in A \cup B \setminus C) = \int_{Y(C)} \hspace{-5mm} \frac{\exp \left( - \sum_{i = k_1-1}^{k_2}  \sum_{ m = a}^{b-1}{ H} (y_{i + 1, m+1} - y_{i,m}) \right) }{Z_{H, H^{RW}}^{k_1, k_2,a ,b, \vec{u}, \vec{v},L_{k_1 - 1}\llbracket a,b\rrbracket, L_{k_2 +1}\llbracket a,b\rrbracket }}     \\
&  \prod_{i = k_1}^{k_2} \frac{ \prod_{m = a + 1}^{b} \g(y_{i,m} - y_{i, m-1}) }{\g^{x_{i,a}}_{b - a }(x_{i, b}) } \prod_{(i,j) \in D}{\bf 1} \{y_{i,j} \leq b_{i,j} \}   \prod_{(i,j) \in C} dy_{i,j}.
\end{split}
\end{equation}
From Lemma \ref{ContinuousGibbsCond} we know that $G_3$ is bounded and measurable on $Y(A \cup B \setminus C)$. It follows from (\ref{S2CDFC}), applied to the function
$$G_2 (x_{i,j} : (i,j) \in A \cup B \setminus C) = G_3(x_{i,j} : (i,j) \in A \cup B \setminus C) \cdot \prod_{(i,j) \in A \cup B \setminus C} {\bf 1} \{x_{i,j} \leq a_{i,j}\},$$
for $a_{i,j} \in \mathbb{R}$ for $(i,j) \in A \cup B \setminus C$ that the second line of (\ref{GibbsRed2}) is equal to
\begin{equation}\label{S2CDFC5W2}
\begin{split}
& \int_{Y(B)} \int_{Y(A)} \int_{Y(C)} \hspace{-5mm} \frac{\exp \left( - \sum_{i = k_1-1}^{k_2}  \sum_{ m = a}^{b-1}{ H} (y_{i + 1, m+1} - y_{i,m}) \right) }{Z_{H, H^{RW}}^{k_1, k_2,a ,b, \vec{u}, \vec{v},L_{k_1 - 1}\llbracket a,b\rrbracket, L_{k_2 +1}\llbracket a,b\rrbracket }}    \prod_{i = k_1}^{k_2} \frac{ \prod_{m = a + 1}^{b} \g(y_{i,m} - y_{i, m-1}) }{\g^{x_{i,a}}_{b - a }(x_{i, b}) }  \\
&    \frac{\exp \left( - \sum_{i = 1}^{\topc-1}  \sum_{ m = T_0}^{T_1-1}{ H} (x_{i + 1, m + 1} - x_{i,m}) \right) }{Z_{H, H^{RW}}^{1, \topc-1,T_0 ,T_1, \vec{x}, \vec{y},\infty, L_\topc\llbracket T_0, T_1\rrbracket} }  \prod_{i = 1}^{\topc-1} \frac{\prod_{m = T_0 + 1}^{T_1} \g(x_{i,m} - x_{i, m-1}) }{\g^{x_{i,T_0}}_{T_1 - T_0 }(x_{i, T_1}) }    \\
&  \prod_{(i,j) \in D}{\bf 1} \{y_{i,j} \leq b_{i,j} \} \hspace{-3mm}   \prod_{(i,j) \in A \cup B \setminus C}  \hspace{-3mm}   {\bf 1} \{x_{i,j} \leq a_{i,j}\}  \prod_{(i,j) \in C} dy_{i,j} \prod_{(i,j) \in A} dx_{i,j}  \cdot \mu_B( dx_{i,j} : (i,j) \in B),
\end{split}
\end{equation}
where we use the convention that $y_{i,j} = x_{i,j}$ if $(i,j) \not \in C$ and $x_{0,j} = \infty$, $\vec{u} = (x_{k_1, a}, \dots, x_{k_2, a})$, $\vec{v} = (x_{k_1, b}, \dots, x_{k_2, b})$, $L_{k}\llbracket a,b \rrbracket = (x_{k,a}, \dots x_{k, b})$ if $k \geq 1$ and $L_{k}\llbracket a,b \rrbracket = (\infty)^{b-a+1}$ if $k = 0$. Also, $\vec{x} = (x_{1, T_0}, \dots, x_{\topc-1, T_0})$ and $\vec{y} = (x_{1, T_1}, \dots, x_{\topc -1,T_1})$. This is our desired form of the second line of (\ref{GibbsRed2}).\\

{\bf \raggedleft Step 6.} In this step we find a suitable representation of the first line of (\ref{GibbsRed2}), by utilizing (\ref{S2BCT2}) and (\ref{GOMFG3}), and then see that it agrees with (\ref{S2CDFC5W2}). We apply (\ref{S2BCT2})  for $c_{i,j} = a_{i,j}$ if $(i,j) \in A \cup B\setminus D$, $c_{i,j} = b_{i,j}$ for $(i,j) \in C$ and $c_{i,j} = \min (a_{i,j} ,b_{i,j})$ for $(i,j) \in  D \setminus C$. This allows us to rewrite the first line of (\ref{GibbsRed2}) as (see also (\ref{GOMFG3}) )
\begin{equation}\label{S2CDFCV2}
\begin{split}
&  \int_{Y(B)} \int_{Y(A)} \prod_{i = 1}^{\topc-1} \frac{\prod_{m = T_0 + 1}^{T_1} \g(x_{i,m} - x_{i, m-1}) }{\g^{x_{i,T_0}}_{T_1 - T_0 }(x_{i, T_1}) }  \frac{\exp \left( - \sum_{i = 1}^{\topc-1}  \sum_{ m = T_0}^{T_1-1}{ H} (x_{i + 1, m+1} - x_{i,m}) \right) }{Z_{H, H^{RW}}^{1, \topc-1,T_0 ,T_1, \vec{x}, \vec{y},\infty,  L_\topc\llbracket T_0, T_1\rrbracket} } \\
&\prod_{(i,j) \in D}{\bf 1} \{x_{i,j} \leq b_{i,j} \}   \prod_{(i,j) \in A \cup B \setminus C}{\bf 1} \{x_{i,j} \leq a_{i,j} \}  \prod_{(i,j) \in A} dx_{i,j} \cdot \mu_B( dx_{i,j} : (i,j) \in B).
\end{split}
\end{equation}
where $\vec{x} = (x_{1, T_0}, \dots, x_{\topc-1, T_0})$, $\vec{y} = (x_{1, T_1}, \dots, x_{\topc-1 ,T_1})$ and $ L_\topc\llbracket T_0, T_1\rrbracket = (x_{\topc,T_0}, \dots, x_{\topc, T_1})$. In deriving the above, we also used that ${\bf 1} \{ x\leq\min(a,b) \} = {\bf 1} \{ x\leq a \}\cdot {\bf 1} \{ x\leq b \}$.

What remains to prove (\ref{GibbsRed2}) is to show that the expressions in (\ref{S2CDFCV2}) and (\ref{S2CDFC5W2}) are equal. We start by performing the integration in (\ref{S2CDFC5W2})  over $x_{i,j}$ with $(i,j) \in C$. This gives
\begin{equation*}
\begin{split}
& \int_{Y(B)} \int_{Y(A \setminus C)} \int_{Y(C)}   \hspace{-5mm}  \frac{\exp \left( - \sum_{i = k_1-1}^{k_2}  \sum_{ m =a}^{b-1}{ H} (y_{i + 1, m+1} - y_{i,m}) \right) }{Z_{H, H^{RW}}^{k_1, k_2,a ,b, \vec{u}, \vec{v},L_{k_1 - 1}\llbracket a,b \rrbracket, L_{k_2 +1}\llbracket a,b \rrbracket }} \hspace{-1mm}  \prod_{i = k_1}^{k_2} \hspace{-1mm}\frac{ \prod_{m = a + 1}^{b} \g(y_{i,m} - y_{i, m-1}) }{\g^{x_{i,a}}_{b - a }(x_{i, b}) } \\
&    \prod_{i = 1}^{\topc-1} \frac{1}{\g^{x_{i,T_0}}_{T_1 - T_0 }(x_{i, T_1}) } \cdot \prod_{(i,j) \in W_1}  \g(x_{i,j} - x_{i, j-1})   \frac{\prod_{i = k_1}^{k_2} \g^{x_{i,a}}_{b - a }(x_{i, b})\cdot  Z_{H, H^{RW}}^{k_1, k_2,a ,b, \vec{u}, \vec{v},L_{k_1 - 1}\llbracket a,b \rrbracket, L_{k_2 +1}\llbracket a,b \rrbracket } }{Z_{H, H^{RW}}^{1, \topc-1,T_0 ,T_1, \vec{x}, \vec{y},\infty, L_\topc\llbracket T_0, T_1\rrbracket} }  \\
&\exp \Big( - \sum_{ (i,j) \in W_2} { H} (x_{i + 1, j + 1} - x_{i,j})  \Big)  \prod_{(i,j) \in D}{\bf 1} \{y_{i,j} \leq b_{i,j} \}  \prod_{(i,j) \in A \cup B \setminus C} \hspace{-3mm} {\bf 1} \{x_{i,j} \leq a_{i,j} \}  \\
& \prod_{(i,j) \in C} dy_{i,j}  \prod_{(i,j) \in A \setminus C} dx_{i,j} \cdot \mu_B( dx_{i,j} : (i,j) \in B).
\end{split}
\end{equation*}
where $W_1 = \llbracket 1, \topc-1 \rrbracket \times \llbracket T_0 + 1, T_1 \rrbracket \setminus \llbracket k_1, k_2 \rrbracket \times \llbracket a+1, b \rrbracket$, and $W_2 = \llbracket 1, \topc -1 \rrbracket \times \llbracket T_0, T_1 -1 \rrbracket \setminus \llbracket k_1-1 , k_2 \rrbracket \times \llbracket a ,b -1 \rrbracket$. Upon relabeling $y_{i,j}$ with $x_{i,j}$ for $(i,j) \in C$ in the last expression and performing a bit of cancellations, we recognize (\ref{S2CDFCV2}).\\

{\bf \raggedleft Step 7.} In this step we prove the second part of the lemma. We use the same notation as in the previous steps (e.g. the sets $A,B,C,D$). Below we write $\mathbb{E}$ to stand for $\mathbb{E}_{H,H^{RW}}^{1, \topc-1, T_0 ,T_1, \vec{x}, \vec{y}, \infty ,\vec{z}} $. The essential argument below is to find an analogue of (\ref{GibbsRed2}), and then the arguments in Step 2 can be repeated to conclude the second part of the lemma. This analogue is given in (\ref{GibbsRed2v2}) and the way it is obtained is by repeating many of the arguments from Step 3-6 above. The only difference that is happening is that the formulas we derive below will not depend on $B$, and the reason behind this is that the $K$-th curve $L_K$, as well as the time $T_0$ and time $T_1$ distribution of the ensemble are deterministically fixed (the left boundary is $\vec{x}$, the right boundary is $\vec{y}$ and the $K$-th curve equals $\vec{z}$) in this part of the lemma, while before they were random. Even though there are no new ideas in the derivation, we still supply the formulas as they are somewhat complicated and possibly not immediate from our previous work. After a careful comparison of the formulas below with their earlier counterparts, the reader should be able to see that they are the same except that $\mu_B$ is being replaced by a delta function at the deterministic boundary formed by $\vec{x}, \vec{y}, \vec{z}$.

Let $\hat{f}(x_{i,j} :(i,j) \in A) = \prod_{(i,j) \in A} {\bf 1}\{ x_{i,j} \leq c_{i,j} \}$ be as in Step 3. Analogously to (\ref{GOMFG2}), we find
\begin{equation}\label{GOMFGV2}
\begin{split}
& \mathbb{E}  \left[\prod_{(i,j) \in A} {\bf 1}\{ L_i(j) \leq c_{i,j} \} \right] =\int_{Y(A)}\prod_{i = 1}^{\topc-1} \frac{ \prod_{m = T_0 + 1}^{T_1} \g(y_{i,m} - y_{i, m-1}) }{\g^{x_{i,T_0}}_{b - a }(x_{i, T_1}) }    \\
&  \frac{\exp \left( - \sum_{i =1}^{\topc-1}  \sum_{ m = T_0}^{T_1-1}{ H} (y_{i + 1, m+1} - y_{i,m}) \right) }{Z_{H, H^{RW}}^{1, \topc-1,T_0 ,T_1, \vec{x}, \vec{y},\infty, \vec{z}} }   \prod_{(i,j) \in A}{\bf 1} \{y_{i,j} \leq c_{i,j} \}   \prod_{(i,j) \in A} dy_{i,j},
\end{split}
\end{equation}
where $y_{\topc, m+1} = z_{m+1}$ (the coordinates of $\vec{z}$ are indexed by $\llbracket T_0, T_1 \rrbracket$). Repeating the arguments in the beginning of Step 5, we conclude that for any bounded Borel-measurable $G_2$ on $Y(A \setminus C)$
\begin{equation}\label{S2CDFCV2V2}
\begin{split}
&\mathbb{E}\big[G_2 \big( \mathfrak{L} \vert_{A \setminus C} \big)\big] =   \int_{Y(A)} \prod_{i = 1}^{\topc-1} \frac{\prod_{m = T_0 + 1}^{T_1} \g(x_{i,m} - x_{i, m-1}) }{\g^{x_{i,T_0}}_{T_1 - T_0 }(x_{i, T_1}) } \times \\
& \frac{\exp \left( - \sum_{i = 1}^{\topc-1}  \sum_{ m = T_0}^{T_1-1}{ H} (x_{i + 1, m+1} - x_{i,m}) \right) }{Z_{H, H^{RW}}^{1, \topc-1,T_0 ,T_1, \vec{x}, \vec{y},\infty,  \vec{z}} }  G_2(x_{i,j} : (i,j) \in A  \setminus C) \prod_{(i,j) \in A} dx_{i,j} ,
\end{split}
\end{equation}
where again $x_{\topc, m+1} = z_{m+1}$. Repeating the same arguments in the rest of Step 5, we arrive at the following analogue of (\ref{S2CDFC5W2}) for any $a_{i,j}  \in \mathbb{R}$ for $(i,j) \in A\setminus C$ and $b_{i,j} \in \mathbb{R}$ for $(i,j) \in D$

\begin{equation}\label{S2CDFC5V2}
\begin{split}
& \mathbb{E} \left[ \prod_{(i,j) \in A \setminus C} {\bf 1}\{ L_i(j) \leq a_{i,j} \}   \cdot \mathbb{E}_{H,H^{RW}}'  \left[\prod_{(i,j) \in D} {\bf 1}\{ L'_i(j) \leq b_{i,j} \} \right] \right]  = \\
&\int_{Y(A)} \int_{Y(C)} \hspace{-5mm} \frac{\exp \left( - \sum_{i = k_1-1}^{k_2}  \sum_{ m = a}^{b-1}{ H} (y_{i + 1, m+1} - y_{i,m}) \right) }{Z_{H, H^{RW}}^{k_1, k_2,a ,b, \vec{u}, \vec{v},L_{k_1 - 1}\llbracket a,b\rrbracket, L_{k_2 +1}\llbracket a,b\rrbracket }}    \prod_{i = k_1}^{k_2} \frac{ \prod_{m = a + 1}^{b} \g(y_{i,m} - y_{i, m-1}) }{\g^{x_{i,a}}_{b - a }(x_{i, b}) }  \\
&    \frac{\exp \left( - \sum_{i = 1}^{\topc-1}  \sum_{ m = T_0}^{T_1-1}{ H} (x_{i + 1, m + 1} - x_{i,m}) \right) }{Z_{H, H^{RW}}^{1, \topc-1,T_0 ,T_1, \vec{x}, \vec{y},\infty, \vec{z}} }  \prod_{i = 1}^{\topc-1} \frac{\prod_{m = T_0 + 1}^{T_1} \g(x_{i,m} - x_{i, m-1}) }{\g^{x_{i,T_0}}_{T_1 - T_0 }(x_{i, T_1}) }    \\
&  \prod_{(i,j) \in D}{\bf 1} \{y_{i,j} \leq b_{i,j} \} \hspace{-3mm}   \prod_{(i,j) \in A \setminus C}  \hspace{-3mm}   {\bf 1} \{x_{i,j} \leq a_{i,j}\}  \prod_{(i,j) \in C} dy_{i,j} \prod_{(i,j) \in A} dx_{i,j}  ,
\end{split}
\end{equation}
where we use the convention that $y_{i,j} = x_{i,j}$ if $(i,j) \not \in C$ and $x_{0,j} = \infty$, $\vec{u} = (x_{k_1, a}, \dots, x_{k_2, a})$, $\vec{v} = (x_{k_1, b}, \dots, x_{k_2, b})$, $L_{k}\llbracket a,b \rrbracket = (x_{k,a}, \dots x_{k, b})$ if $k \geq 1$ and $L_{k}\llbracket a,b \rrbracket = (\infty)^{b-a+1}$ if $k = 0$. Also, $\vec{x} = (x_{1, T_0}, \dots, x_{\topc-1, T_0})$ and $\vec{y} = (x_{1, T_1}, \dots, x_{\topc-1 ,T_1})$ and $x_{\topc,i} = z_i$ for $i \in \llbracket T_0, T_1 \rrbracket$. On the left side of (\ref{S2CDFC5V2}) we have that $\mathfrak{L}' = (L'_{k_1}, \dots, L'_{k_2})$ is distributed according to $\mathbb{P}_{H,H^{RW}}^{k_1, k_2, a ,b, \vec{u}, \vec{v},L_{k_1 - 1}\llbracket a ,b \rrbracket,L_{k_2+  1}\llbracket a ,b \rrbracket}$, where $\vec{u} = (L_{k_1}(a), \dots, L_{k_2}(a))$, $\vec{v} =  (L_{k_1}(b), \dots, L_{k_2}(b))$ and we write $\mathbb{P}_{H, H^{RW}}'$ for this measure and $\mathbb{E}_{H, H^{RW}}'$ for the corresponding expectation. In the latter, we have adopted the convention that $L_0 \llbracket T_0, T_1 \rrbracket = (\infty)^{T_1 - T_0 + 1}$ and $L_\topc\llbracket T_0, T_1 \rrbracket = \vec{z}$.

Performing the integration in (\ref{S2CDFC5V2}) over $x_{i,j}$ with $(i,j) \in C$ gives
\begin{equation*}
\begin{split}
& \int_{Y(A \setminus C)} \int_{Y(C)}   \hspace{-5mm}  \frac{\exp \left( - \sum_{i = k_1-1}^{k_2}  \sum_{ m =a}^{b-1}{ H} (y_{i + 1, m+1} - y_{i,m}) \right) }{Z_{H, H^{RW}}^{k_1, k_2,a ,b, \vec{u}, \vec{v},L_{k_1 - 1}\llbracket a,b \rrbracket, L_{k_2 +1}\llbracket a,b \rrbracket }} \hspace{-1mm}  \prod_{i = k_1}^{k_2} \hspace{-1mm}\frac{ \prod_{m = a + 1}^{b} \g(y_{i,m} - y_{i, m-1}) }{\g^{x_{i,a}}_{b - a }(x_{i, b}) } \\
&    \prod_{i = 1}^{\topc-1} \frac{1}{\g^{x_{i,T_0}}_{T_1 - T_0 }(x_{i, T_1}) } \cdot \prod_{(i,j) \in W_1}  \g(x_{i,j} - x_{i, j-1})   \frac{\prod_{i = k_1}^{k_2} \g^{x_{i,a}}_{b - a }(x_{i, b})\cdot  Z_{H, H^{RW}}^{k_1, k_2,a ,b, \vec{u}, \vec{v},L_{k_1 - 1}\llbracket a,b \rrbracket, L_{k_2 +1}\llbracket a,b \rrbracket } }{Z_{H, H^{RW}}^{1, \topc-1,T_0 ,T_1, \vec{x}, \vec{y},\infty, \vec{z}} }  \\
&\exp \Big( - \sum_{ (i,j) \in W_2} { H} (x_{i + 1, j + 1} - x_{i,j})  \Big)  \prod_{(i,j) \in D}  \hspace{-3mm} {\bf 1} \{y_{i,j} \leq b_{i,j} \}  \prod_{(i,j) \in A  \setminus C} \hspace{-3mm} {\bf 1} \{x_{i,j} \leq a_{i,j} \}  \prod_{(i,j) \in C} dy_{i,j}  \prod_{(i,j) \in A \setminus C}  \hspace{-3mm} dx_{i,j} .
\end{split}
\end{equation*}
where $W_1 = \llbracket 1, \topc-1 \rrbracket \times \llbracket T_0 + 1, T_1 \rrbracket \setminus \llbracket k_1, k_2 \rrbracket \times \llbracket a+1, b \rrbracket$ and $W_2 = \llbracket 1, \topc -1 \rrbracket \times \llbracket T_0, T_1 -1 \rrbracket \setminus \llbracket k_1-1 , k_2 \rrbracket \times \llbracket a ,b -1 \rrbracket$. Upon relabeling $y_{i,j}$ with $x_{i,j}$ for $(i,j) \in C$ in the last expression and performing a bit of cancellations we obtain
\begin{equation*}
\begin{split}
&\int_{Y(A)}\prod_{i = 1}^{\topc-1} \frac{ \prod_{m = T_0 + 1}^{T_1} \g(x_{i,m} - x_{i, m-1}) }{\g^{x_{i,T_0}}_{b - a }(x_{i, T_1}) }   \frac{\exp \left( - \sum_{i =1}^{\topc-1}  \sum_{ m = T_0}^{T_1-1}{ H} (x_{i + 1, m+1} - x_{i,m}) \right) }{Z_{H, H^{RW}}^{1, \topc-1,T_0 ,T_1, \vec{x}, \vec{y},\infty, \vec{z}} } \hspace{-3mm} \prod_{(i,j) \in A}{\bf 1} \{x_{i,j} \leq a_{i,j} \}  \\
&    \prod_{(i,j) \in D}{\bf 1} \{x_{i,j} \leq b_{i,j} \}  \prod_{(i,j) \in A} dx_{i,j} = \mathbb{E} \left[\prod_{(i,j) \in A}{\bf 1} \{L_i(j) \leq a_{i,j} \}  \prod_{(i,j) \in D}{\bf 1} \{L_i(j)\leq b_{i,j} \} \right],
\end{split}
\end{equation*}
where in the last equality we used (\ref{GOMFGV2}). Combining our work, we see that for any $a_{i,j} \in \mathbb{R}$ for $(i,j) \in A \setminus C$ and $b_{i,j} \in \mathbb{R}$ for $(i,j) \in D$ we have
\begin{equation}\label{GibbsRed2v2}
\begin{split}
&\mathbb{E} \left[ \prod_{(i,j) \in A  \setminus C} {\bf 1}\{ L_i(j) \leq a_{i,j} \} \cdot \prod_{(i,j) \in D} {\bf 1}\{ L_i(j) \leq b_{i,j} \} \right]  = \\
& \mathbb{E} \left[ \prod_{(i,j) \in A  \setminus C} {\bf 1}\{ L_i(j) \leq a_{i,j} \}   \cdot \mathbb{E}_{H,H^{RW}}'  \left[\prod_{(i,j) \in D} {\bf 1}\{ L'_i(j) \leq b_{i,j} \} \right] \right].
\end{split}
\end{equation}
The monotone class argument in Step 2 can now be repeated verbatim to show that for any bounded Borel-measurable $F$ on $Y(D)$ we have $\mathbb{P}$-almost surely that
$$\mathbb{E} \Big[ F \big( \mathfrak{L}\vert_{D} \big)  \big{\vert} \mathcal{F}_{A \setminus C} \Big]  = \mathbb{E}_{H,H^{RW}}^{k_1, k_2, a ,b, \vec{u}, \vec{v},L_{k_1 - 1}\llbracket a ,b \rrbracket,L_{k_2+  1}\llbracket a ,b \rrbracket}  \big[F( \mathfrak{L}') \big],$$
which concludes the proof of the second part of the lemma.
\end{proof}

We end this section by proving Lemma \ref{S4WeakGibbs}.
\begin{lemma}\label{WeakGibbs}[Lemma \ref{S4WeakGibbs}] Let $H$ and $H^{RW}$ be as in Definition \ref{Pfree}. Fix $\topc \geq 2$, two integers $T_0 < T_1$, and set $\Sigma = \llbracket 1,\topc\rrbracket$. Suppose that $\mathbb{P}_n$ is a sequence of probability distributions on  $\Sigma\times \llbracket T_0, T_1 \rrbracket$-indexed discrete line ensembles, such that for each $n$ we have that $\mathbb{P}_n$ satisfies the $(H,H^{RW})$-Gibbs property. If $\mathbb{P}_n$ converges weakly to a measure $\mathbb{P}$, then $\mathbb{P}$ also satisfies the $(H,H^{RW})$-Gibbs property.
\end{lemma}
\begin{proof}[Proof of Lemma \ref{S4WeakGibbs}] Let $A$, $B$, $Y(A)$ and $Y(B)$ be as in Lemma \ref{S4AltGibbs} and Definition \ref{YVec}. Suppose that $f, h$ are defined by
$$f\big( x_{i,j} : (i,j) \in A \big): = \prod_{(i,j) \in A} f_{i,j}(x_{i,j}), \mbox{and } h\big(x_{i,j} : (i,j) \in B \big): = \prod_{(i,j) \in B} f_{i,j}(x_{i,j}),$$
where $f_{i,j}$ are bounded, continuous real functions.  From Lemma \ref{S4AltGibbs} we know that
\begin{equation*}
\begin{split}
& \mathbb{E}_{\mathbb{P}_n}\Big[   h\big( L_i(j) : (i,j) \in B \big) \cdot  f\big( L_i(j) : (i,j) \in A \big)  \Big]  = \\
&\mathbb{E}_{\mathbb{P}_n} \bigg[ h\big( L_i(j) : (i,j) \in B \big) \cdot G_f(\vec{x}, \vec{y}, (\infty)^{T_1 - T_0 +1}, L_\topc\llbracket T_0, T_1 \rrbracket )  \bigg],
\end{split}
\end{equation*}
where $\vec{x} =  (L_{1}(T_0), \dots L_{\topc-1}(T_0))$, $\vec{y} = (L_{1}(T_1), \dots L_{\topc-1}(T_1))$, and $G_f$ is as in Lemma \ref{ContinuousGibbsCond}.

From Lemma \ref{ContinuousGibbsCond} we know that $ G_f$ is a bounded, continuous function. Consequently, we can take the limit as $n \rightarrow \infty$ above and using the weak convergence of $\mathbb{P}_n$ to $\mathbb{P}$ conclude that
\begin{equation*}
\begin{split}
& \mathbb{E}_{\mathbb{P}}\Big[   h\big( L_i(j) : (i,j) \in B \big) \cdot  f\big( L_i(j) : (i,j) \in A \big)  \Big]  = \\
&\mathbb{E}_{\mathbb{P}} \bigg[ h\big( L_i(j) : (i,j) \in B \big) \cdot G_f(\vec{x}, \vec{y}, (\infty)^{T_1 - T_0 +1}, L_\topc\llbracket T_0, T_1 \rrbracket )    \bigg],
\end{split}
\end{equation*}
which in view of Lemma \ref{S4AltGibbs} concludes the proof of this lemma.
\end{proof}

%

\subsection{Proof of the lemmas in Section \ref{Section4.3} }\label{Section11.2} In what follows, we prove the lemmas in Section \ref{Section4.3}, whose statements are recalled for the readers convenience. Before we begin, we note the following immediate consequence of Proposition \ref{KMT} and Chebyshev's inequality 
\begin{equation}\label{ChebyshevIneq}
\mathbb{P}\left(\Delta(T,z) > T^{1/4}\right) \leq Ce^{\tilde{\alpha} (\log T)^2} e^{(z - pT)^2/T} e^{-a T^{1/4}},
\end{equation}
which will be used several times in the proofs below. 
\begin{lemma}\label{LemmaHalf} [Lemma \ref{LemmaHalfS4}] Let $\ell$ have distribution $\mathbb{P}^{0,T,x,y}_{H^{RW}}$, with $H^{RW}$ satisfying the assumptions in Definition \ref{AssHR}. Let $M_1, M_2 \in \mathbb{R}$ and $p \in \mathbb{R}$ be given. Then, we can find $W_0 = W_0(p,M_2 - M_1) \in \mathbb{N}$, such that for $T \geq W_0$, $x \geq M_1 T^{1/2}$, $y \geq pT + M_2 T^{1/2}$ and $s \in [0,T]$ we have
\begin{equation}\label{halfEq1}
\mathbb{P}^{0,T,x,y}_{H^{RW}}\Big( \ell(s)  \geq \frac{T-s}{T} \cdot M_1 T^{1/2} + \frac{s}{T} \cdot \big(p T + M_2 T^{1/2}\big) - T^{1/4} \Big) \geq \frac{1}{3}.
\end{equation}
\end{lemma}
\begin{proof}
In view of Lemma \ref{MonCoup} with $\vec{z} = (-\infty)^n$, we know that
$$ \mathbb{P}^{0,T,c,d}_{H^{RW}}\left( \ell(s)  \geq  \frac{T-s}{T} \cdot M_1 T^{1/2} + \frac{s}{T} \cdot \big(p T + M_2 T^{1/2}\big) - T^{1/4} \right) \geq  $$
$$  \mathbb{P}^{0,T,x,y}_{H^{RW}}\left( \ell(s)  \geq  \frac{T-s}{T} \cdot M_1 T^{1/2} + \frac{s}{T} \cdot \big(p T + M_2 T^{1/2}\big) - T^{1/4} \right),$$
whenever $c \geq x$ and $d \geq y$ and so it suffices to prove the lemma when $x =   M_1 T^{1/2}$ and $y =pT + M_2 T^{1/2}$, which we assume in the sequel. Suppose we have the same coupling as in Proposition \ref{KMT} and let $\mathbb{P}$ denote the probability measure on the space afforded by that proposition. Then, the left side of (\ref{halfEq1}) equals
$$\mathbb{P}\Big( x + \ell^{(T,y-x)}(s) \geq \frac{T-s}{T} \cdot M_1 T^{1/2} + \frac{s}{T} \cdot \big(p T + M_2 T^{1/2}\big) - T^{1/4}\Big) \geq$$
$$\geq \mathbb{P}\Big( T^{1/2}B^{\sigma}_{s/T} \geq 0 \mbox{ and } \Delta(T,y-x) \leq T^{1/4} \Big) \geq 1/2 - \mathbb{P}\big(\Delta(T,y-x) > T^{1/4}\big).$$
To get the first expression from (\ref{halfEq1}) we used the fact that $\ell(s)$ and $x + \ell^{(T,y-x)}(s)$ have the same law. The first inequality follows from the coupling to a Brownian bridge, and the last inequality uses that $\mathbb{P}(B^{v}_{s/T} \geq 0) = 1/2$ for every $v > 0$ and $s \in [ 0,T]$. From (\ref{ChebyshevIneq}) we have
\begin{equation*}
\mathbb{P}\left(\Delta(T,y-x) > T^{1/4}\right) \leq Ce^{\tilde{\alpha} (\log T)^2} e^{(M_2 - M_1)^2} e^{-a T^{1/4}},
\end{equation*}
which is at most $1/6$ if we take $W_0$ sufficiently large and $T \geq W_0$, which implies (\ref{halfEq1}).
\end{proof}

\begin{lemma}\label{LemmaMinFree}[Lemma \ref{LemmaMinFreeS4}] Let $\ell$ have distribution $\mathbb{P}^{0,T,0,y}_{H^{RW}}$, with $H^{RW}$ satisfying the assumptions in Definition \ref{AssHR}. Let $M > 0$, $p \in \mathbb{R}$ and $\epsilon > 0$ be given. Then, we can find $W_1=W_1(M,p, \epsilon) \in \mathbb{N}$ and $A=A(M,p, \epsilon) > 0$, such that for $T \geq W_1$, $ y \geq p T -  MT^{1/2}$ we have
\begin{equation}\label{minFree1}
\mathbb{P}^{0,T,0,y}_{H^{RW}}\Big( \inf_{s \in [ 0, T]}\big( \ell(s) -  ps \big) \leq -AT^{1/2} \Big) \leq \epsilon.
\end{equation}
\end{lemma}
\begin{proof}
Fix $\epsilon > 0$. In view of Lemma \ref{MonCoup} with $\vec{z} = (-\infty)^n$, we know that whenever $z_2 \geq z_1$
\begin{align*}
&\mathbb{P}^{0,T,0,z_2}_{H^{RW}}\Big( \min_{s \in [ 0, T]}\big( \ell(s) -  ps \big) \leq - AT^{1/2} \Big)  \leq \mathbb{P}^{0,T,0,z_1}_{H^{RW}}\Big( \min_{s \in [0,T]}\big( \ell(s) -  ps \big) \leq - AT^{1/2} \Big)\leq\\
&\mathbb{P}^{0,T,0,z_1}_{H^{RW}}\Big(  \min_{s \in [0,T]}\big( \ell(s) -  \frac{s}{T}( pT- MT^{1/2}) \big) \leq (M- A)T^{1/2} \Big),
\end{align*}
and so it suffices to prove that the latter probability with $z_1 =   pT -  MT^{1/2}$ is less than $\epsilon$. Suppose we have the same coupling as in Proposition \ref{KMT} and let $\mathbb{P}$ denote the probability measure on the space afforded by that proposition. Below we set $z = pT -  MT^{1/2}$. Then, we have
\begin{align*}
&\mathbb{P}^{0,T,0,z}_{H^{RW}}\Big( \min_{s \in  [0,T]}\big( \ell(s) -  \frac{s}{T}( pT- MT^{1/2}) \big) \leq  (M- A)T^{1/2} \Big) =\\
&\mathbb{P}\Big(  \min_{s \in [0,T]}\big(\ell^{(T,z)}(s) -  \frac{s}{T}( pT- MT^{1/2}) \big) \leq  (M- A)T^{1/2}\Big) \leq \\
& \mathbb{P}\Big(  \min_{s \in  [0,T]} T^{1/2} B^{\sigma}_{s/T}\leq   (M- A - 1)T^{1/2}\Big) + \mathbb{P} \big( \Delta(T,z) > T^{1/2} \big).
\end{align*}
 Using (\ref{ChebyshevIneq}), we can make the second probability smaller than $\epsilon/2$ by choosing $W_1$ large.
By basic properties of Brownian bridges, we know that the first probability (for $A \geq M+1$) is given by
$$  \mathbb{P}\Big( \min_{s \in [0,1]} B^{1}_s \leq - \sigma^{-1} (M- A - 1)  \Big) =  \mathbb{P}\Big( \max_{s \in [0,1]} B^{1}_s \geq \sigma^{-1}(M- A - 1) \Big) =e^{-2\sigma^{-2}(M- A - 1)^2},$$
where the last equality can be found in \cite[Chapter 4, (3.40)]{KS}. Thus by making $A$ sufficiently large we can make the above less than $\epsilon/2$.
\end{proof}

\begin{lemma}\label{LemmaTail}[Lemma \ref{LemmaTailS4}] Let $\ell$ have distribution $\mathbb{P}^{0,T,x,y}_{H^{RW}}$, with $H^{RW}$ satisfying the assumptions in Definition \ref{AssHR}. Let $M_1,M_2 > 0$ and $p \in \mathbb{R}$ be given. Then, we can find $W_2 = W_2(M_1,M_2,p) \in \mathbb{N}$, such that for $T \geq W_2$, $ x \geq -M_1T^{1/2}$, $ y \geq pT -  M_1T^{1/2}$ and $\rho \in \{-1, 0 , 1\}$ we have
\begin{equation}\label{halfEq2}
\mathbb{P}^{0,T,x,y}_{H^{RW}}\bigg( \ell(\lfloor T/2\rfloor + \rho)  \geq \frac{M_2T^{1/2} + p T}{2} - T^{1/4} \bigg) \geq (1/2) (1 - \Phi^{v}(M_1 + M_2) ),
\end{equation}
where $\Phi^{v}$ is the cumulative distribution function  of a Gaussian random variable with mean $0$ and variance $v = \sigma_p^2/4$.
\end{lemma}
\begin{proof}
 In view of Lemma \ref{MonCoup} it suffices to prove the lemma when $z_1 =  -M_1T^{1/2} $, and $ z_2  =  p T -  M_1T^{1/2} $. Set $\Delta z = z_2 - z_1$, $t_{\rho} = \frac{\lfloor T/2 \rfloor + \rho}{T}$, and observe that
$$ \mathbb{P}^{0,N,z_1,z_2}_{H^{RW}}\bigg(\ell(T \cdot t_\rho)  \geq \frac{M_2T^{1/2} + p T}{2}- T^{1/4} \bigg) =  \mathbb{P}^{0,T,0,\Delta z}_{H^{RW}}\bigg(\ell(T \cdot t_\rho)  \geq \frac{M_2T^{1/2} + p T}{2}-z_1 - T^{1/4} \bigg).$$
Suppose we have the same coupling as in Proposition \ref{KMT} and let $\mathbb{P}$ denote the probability measure on the space afforded by that proposition. Then, we have
\begin{align*}
&\mathbb{P}^{0,T,0,\Delta z}_{H^{RW}}\bigg(\ell(T  t_\rho)  \geq \frac{M_2T^{1/2} + p T}{2}- z_1 - T^{1/4} \bigg) =
\mathbb{P}\bigg( \ell^{(T,\Delta z)}(T  t_\rho) \geq   \frac{M_2T^{1/2} + p T}{2} - z_1 - T^{1/4} \bigg)=\\
&\mathbb{P}\bigg( \ell^{(T,\Delta z)}(T  t_\rho) \geq   \frac{(2M_1 + M_2)T^{1/2} + \Delta z}{2}- T^{1/4}  \bigg)\geq
\mathbb{P}\bigg( B^{\sigma}_{t_{\rho}} \geq  \frac{M_2 + 2M_1}{2} \mbox{ and } \Delta(T,\Delta z) \leq T^{1/4}  \bigg).
\end{align*}
Since $B^{\sigma}_{t_\rho}$ has the distribution of a normal random variable with mean $0$ and variance $v_{\rho} = \sigma_p^2 t_{\rho} (1 - t_{\rho})$, and $\Phi^{v}$ is decreasing on $\mathbb{R}_{ > 0}$  we conclude that the last expression is bounded from below by
$$ 1 - \Phi^{v_{\rho}}(M_1 + M_2) - \mathbb{P}\big(\Delta(T,z) > T^{1/4}\big) \geq 1 - \Phi^{v_{\rho}}(M_1 + M_2)  - Ce^{\tilde{\alpha}(\log T)^2} e^{-a T^{1/4}}.$$
In the last inequality we used (\ref{ChebyshevIneq}). The above is at least $(1/2) \big(1 - \Phi^{v}(M_1 + M_2) \big) $ if $W_2$ is taken sufficiently large and $T \geq W_2$.
\end{proof}

\begin{lemma}\label{LemmaAway}[Lemma \ref{LemmaAwayS4}] Let $\ell$ have distribution $\mathbb{P}^{0,T,x,y}_{H^{RW}}$, with $H^{RW}$ satisfying the assumptions in Definition \ref{AssHR}. Let $p \in \mathbb{R}$ be given. Then, we can find $W_3 = W_3(p) \in \mathbb{N}$, such that for $T \geq W_3$, $ x \geq T^{1/2}$, $ y \geq pT +  T^{1/2}$
\begin{equation}\label{away}
\mathbb{P}^{0,T,x,y}_{H^{RW}}\Big( \inf_{s \in [0,T]} \big( \ell(s) -ps \big)+ T^{1/4} \geq 0 \Big) \geq \frac{1}{2} \left(1 - \exp\left(\frac{-2}{\sigma_p^2}\right)\right).
\end{equation}
\end{lemma}
\begin{proof}
In view of Lemma \ref{MonCoup} it suffices to prove the lemma when $z_1 =  T^{1/2} $ and $z_2 =    pT +  T^{1/2} $. Set $\Delta z = z_2 - z_1$, and observe that
$$\mathbb{P}^{0,T,z_1,z_2}_{H^{RW}}\Big( \min_{s \in [ 0,T]}\big( \ell(s) -ps \big)+ T^{1/4} \geq 0 \Big) = \mathbb{P}^{0,T,0,\Delta z}_{H^{RW}}\Big( \min_{s \in [0,T]} \big( \ell(s) -ps \big)+ T^{1/4} \geq -z_1 \Big).$$
Suppose we have the same coupling as in Proposition \ref{KMT} and let $\mathbb{P}$ denote the probability measure on the space afforded by that proposition. Then, we have
\begin{align*}
&\mathbb{P}^{0,T,0,\Delta z}_{H^{RW}}\Big( \min_{s \in [ 0,T]} \big( \ell(s) -ps \big) + T^{1/4} \geq -z_1 \Big) = \mathbb{P}\Big( \min_{s \in [ 0,T ]} \big( \ell^{(T, \Delta z)}(s) - ps \big)\geq  - T^{1/4}  -z_1 \Big) =\\
&\mathbb{P}\Big( \min_{s \in [ 0,T ]} \big( \ell^{(T, \Delta z)}(s) - \frac{s}{T} \Delta z \big)  \geq  - T^{1/4} - T^{1/2} \Big)\geq
\mathbb{P}\Big( \min_{s \in [0,1]} B^{\sigma}_s \geq - 1 \mbox{ and } \Delta(T,\Delta z) \leq T^{1/4}  \Big).
\end{align*}
We can lower-bound the above expression by
$$\mathbb{P}\Big( \min_{s \in [0,1]} B^{\sigma}_s \geq - 1 \Big) - \mathbb{P}\big(\Delta(T,\Delta z) > T^{1/4}   \big).$$
 By basic properties of Brownian bridges, we know that
$$ \mathbb{P}\Big( \min_{s \in [0,1]} B^{\sigma}_s \geq - 1 \Big)  = \mathbb{P}\Big( \min_{s \in [0,1]} B^{1}_s \geq - \sigma^{-1} \Big) =  \mathbb{P}\Big( \max_{s \in [0,1]} B^{1}_s \leq \sigma^{-1} \Big) = 1 - e^{-2\sigma^{-2}},$$
where the last equality can be found, for example, in \cite[Chapter 4, (3.40)]{KS}. Also by  (\ref{ChebyshevIneq})
$$\mathbb{P}\big(\Delta(T,\Delta z) > T^{1/4}\big)  \leq Ce^{\tilde{\alpha}(\log T)^2} e^{1} e^{-a T^{1/4}},$$
and the latter is at most $(1/2)(1 - e^{-2\sigma^{-2}})$ if $W_3$ is taken sufficiently large and $N \geq W_3$. Combining the above estimates, we conclude (\ref{away}).
\end{proof}

\begin{lemma}\label{MOCLemma}[Lemma \ref{MOCLemmaS4}]  Let $\ell$ have distribution $\mathbb{P}^{0,T,0,y}_{H^{RW}}$, with $H^{RW}$ satisfying the assumptions in Definition \ref{AssHR}. Let $M > 0$ and $p \in \mathbb{R}$ be given. For each positive $\epsilon$ and $\eta$, there exist a $\delta > 0$ and $W_4 = W_4(M, p, \epsilon, \eta) \in \mathbb{N}$, such that  for $T \geq W_4$ and $|y - pT| \leq MT^{1/2}$ we have
\begin{equation}\label{MOCeq}
\mathbb{P}^{0,T,0,y}_{H^{RW}}\Big( w\big({f^\ell},\delta\big) \geq \epsilon \Big) \leq \eta,
\end{equation}
where $f^\ell(u) = T^{-1/2}\big(\ell(uT) - puT\big)$  for $u \in [0,1]$.
\end{lemma}
\begin{proof}
 The strategy is to use the strong coupling between $\ell$ and a Brownian bridge afforded by Proposition \ref{KMT}. This will allow us to argue that with high probability the modulus of continuity of $f^\ell$ is close to that of a Brownian bridge, and since the latter is continuous a.s., this will lead to the desired statement of the lemma. We now turn to providing the necessary details.

Let $\epsilon, \eta > 0$ be given and fix $\delta \in (0,1)$, which will be determined later. Suppose we have the same coupling as in Proposition \ref{KMT} and let $\mathbb{P}$ denote the probability measure on the space afforded by that proposition. Then, we have
\begin{equation}\label{MOC1}
\mathbb{P}^{0,T,0,z}_{H^{RW}}\Big( w\big({f^\ell},\delta\big) \geq \epsilon \Big) = \mathbb{P}\Big(w\big({f^{\ell^{(T,z)}}} ,\delta\big) \geq \epsilon \Big).
\end{equation}
By definition, we have
\begin{equation*}
w\big({f^{\ell^{(N,z)}}} ,\delta\big) = T^{-1/2} \sup_{\substack{x,y \in [0,1]\\ |x-y| \leq \delta}} \Big|\ell^{(T,z)}(xT) - pxT- \ell^{(T,z)}(yT) + pyT \Big|.
\end{equation*}
From Proposition \ref{KMT} and  the above, we conclude that
\begin{equation}\label{MOC2}
w\big({f^{\ell^{(T,z)}}} ,\delta\big) \leq T^{-1/2} \sup_{\substack{x,y \in [0,1]\\ |x-y| \leq \delta}} \Big|T^{1/2}B^\sigma_{x}  -T^{1/2}B^\sigma_{y} +(z- pT)(x-y)  \Big| +2 T^{-1/2}\Delta(T,z) .
\end{equation}
From (\ref{MOC1}), (\ref{MOC2}), the triangle inequality and the assumption $|z - pT| \leq MT^{1/2}$, we see that
\begin{equation}\label{MOC3}
\mathbb{P}^{0,T,0,z}_{H^{RW}}\Big( w\big({f^\ell},\delta\big) \geq \epsilon \Big) \leq  \mathbb{P}\Big(w\big({B^\sigma},\delta\big) + \delta M +2 T^{-1/2}\Delta(T,z) \geq \epsilon \Big).
\end{equation}
If $(I) = \mathbb{P}\Big(w\big({B^\sigma},\delta\big) \geq \epsilon /3\Big)$ , $(II) =  \mathbb{P}\big(\delta M  \geq \epsilon /3\big) $ and $(III) = \mathbb{P}\big(2 T^{-1/2}\Delta(T,z)  \geq \epsilon /3 \big) $ we have
$$\mathbb{P}\Big(w\big({B^\sigma},\delta\big) + \delta M +2 T^{-1/2}\Delta(T,z) \geq \epsilon \Big) \leq (I) + (II) + (III).$$
By (\ref{ChebyshevIneq}), we have
$$\mathbb{P}\big(\Delta(T,z) > T^{1/4}\big) \leq Ce^{\tilde{\alpha}(\log T)^2} e^{M^2} e^{-a T^{1/4}}.$$
Consequently, if we pick $W_4$ sufficiently large and $T \geq W_4$, we can ensure that $2T^{-1/4} < \epsilon/3$ and $ Ce^{\tilde{\alpha}(\log T)^2} e^{M^2} e^{-a T^{1/4}} < \eta/3$, which would imply $(III) \leq \eta /3$.

Since $B^{\sigma}$ is a.s. continuous, we know that $w({B^\sigma},\delta) $ goes to $0$ as $\delta$ goes to $0$, hence we can find $\delta_0$ sufficiently small so that if $\delta < \delta_0$, we have $(I) < \eta/3$. Finally, if $\delta M < \epsilon /3$ then $(II) = 0$. Combining all the above estimates with (\ref{MOC3}), we see that for $\delta$ sufficiently small, $W_4$ sufficiently large and $T \geq W_4$, we have $\mathbb{P}^{0,T,0,z}_{H^{RW}}\left( w\big(f^\ell,\delta\big) \geq \epsilon \right) \leq (2/3) \eta < \eta$, as desired.

\end{proof}

\begin{lemma}\label{NoExplode}[Lemma \ref{NoExplodeS4}] Let $H$ be as in Definition \ref{Pfree}, and suppose it is convex,  increasing and  $\lim_{x \rightarrow \infty} H(x) = \infty$. For such a choice of $H$, we let $\ell$ have law $\mathbb{P}_{H,H^{RW}}^{0, 2T, x ,y, \vec{z}}$ as in Section \ref{Section4.2}, where $H^{RW}$ satisfies the assumptions in Definition \ref{AssHR}. Let $M, \epsilon > 0$ and $p \in \mathbb{R}$ be given. Then, we can find a constant $W_5 = W_5(M,p,\epsilon) \in \mathbb{N}$, so that the following holds. If $T \geq W_5$, $\vec{z} \in [-\infty, \infty)^{2T+1}$ with $z_{ T+1} \geq pT + 2MT^{1/2}$ and $x,y \in \mathbb{R}$ with $x \geq -MT^{1/2}$ and $y \geq  - MT^{1/2} + 2pT$, then we have
\begin{equation}\label{NE1v2}
\mathbb{P}_{H,H^{RW}}^{0,2T,x,y, \vec{z}} \big( \ell(T) \leq pT+ MT^{1/2} \big) \leq \epsilon.
\end{equation}
\end{lemma}
\begin{proof}
Let $\epsilon , M > 0 $ and $p \in \mathbb{R}$ be given. Notice that if $\ell \leq \tilde{\ell}$ (meaning $\ell(i) \leq \tilde{\ell}(i)$ for $i \in \llbracket 0, 2T\rrbracket $), then ${\bf 1}\{ \ell(T) \leq pT+ MT^{1/2} \} \geq {\bf 1} \{\tilde{\ell}(T) \leq pT+ MT^{1/2}  \}$, which means in view of Lemma \ref{MonCoup} that it suffices to prove (\ref{NE1v2}) when $x = -MT^{1/2}$, $y = -MT^{1/2} + 2pT$ and $z_{T+1} = pT + 2MT^{1/2}$ and $z_i = -\infty$ for all $i \neq T+1$, which we assume in the sequel.

 One can rewrite (\ref{NE1v2}) as
\begin{equation}\label{NEe2}
\mathbb{E}_{H^{RW}}^{0,2T,x,y} \Big[  {\bf 1} \big\{ \ell(T) \leq pT+ M T^{1/2}\big\}  e^{-H( pT + 2MT^{1/2} - \ell(T))} \Big] \leq \epsilon  \mathbb{E}_{H^{RW}}^{0,2T,x,y} \Big[ e^{-H( pT + 2MT^{1/2}- \ell(T))} \Big].
\end{equation}
Using that $H$ is monotonically increasing, we see that
$$\mathbb{E}_{H^{RW}}^{0,2T,x,y} \Big[  {\bf 1} \big\{ \ell(T) \leq pT+ M T^{1/2}\big\} \cdot e^{-H( pT + 2MT^{1/2} - \ell(T))} \Big]  \leq  e^{-H(  MT^{1/2} )}. $$
And so it suffices to show that
\begin{equation}\label{NEe3}
\epsilon^{-1} \leq   \mathbb{E}_{H^{RW}}^{0,2T,x,y} \Big[ e^{H(  MT^{1/2} ) -H( pT + 2MT^{1/2}- \ell(T))} \Big].
\end{equation}

Suppose we have the same coupling as in Proposition \ref{KMT} and let $\mathbb{P}$ denote the probability measure on the space afforded by that proposition. Setting $z = 2pT$, we have
\begin{align*}
&\mathbb{E}_{H^{RW}}^{0,2T,x,y} \Big[  e^{H(  MT^{1/2} ) -H( pT + 2MT^{1/2}- \ell(T))} \Big]  = \mathbb{E}_{\mathbb{P}} \Big[   e^{H(MT^{1/2} ) -H( pT + 3MT^{1/2} - \ell^{(2T,z)}(T))} \Big] \geq \\
&  e^{H(MT^{1/2})}  \mathbb{E}_{\mathbb{P}} \Big[   e^{-H( pT + 3MT^{1/2} - \ell^{(2T,z)}(T))} {\bf 1} \big\{ \sqrt{2T}B^{\sigma}_{1/2} \geq (3M + 1) T^{1/2}\big\} \cdot {\bf 1} \big\{ \Delta(2T,z) \leq T^{1/2}\big\}  \Big] \geq\\
&e^{H(  M T^{1/2} ) -H( 0)} \cdot \mathbb{E}_{\mathbb{P}} \Big[  {\bf 1} \big\{ \sqrt{2T}B^{\sigma}_{1/2} \geq (3M + 1) T^{1/2}\big\} \cdot {\bf 1} \big\{ \Delta(2T,z) \leq T^{1/2}\big\}   \Big],
\end{align*}
where in the last inequality we used that $H$ is increasing. Note that the factor $3MT^{1/2}$ is introduced above because $\ell(T)$ and $\ell^{(2T,z)}(T) - MT^{1/2}$ are equal in law by the formulation of $\ell^{(2T,z)}$ in Proposition \ref{KMT}.
We now observe that
$$\mathbb{E}_{\mathbb{P}} \Big[  {\bf 1} \big\{ \sqrt{2T}B^{\sigma}_{1/2} \geq (3M + 1) T^{1/2}\big\} \cdot {\bf 1} \big\{ \Delta(2T,z) \leq T^{1/2}\big\}   \Big] \geq \big(1 - \Phi^v(3M+1)\big) - \mathbb{P}\big(\Delta(2T,z) > T^{1/2} \big),$$
where $\Phi^v$ is the cumulative distribution function of a mean $0$ variance $v = \sigma_p^2/2$ Gaussian variable. Consequently, by (\ref{ChebyshevIneq}), we know that we can choose $W_5$ sufficiently large so that if $T \geq W_5$
$$ \big(1 - \Phi^v(3M+1)\big) - \mathbb{P}\big(\Delta(2T,z) > T^{1/2} \big) \geq (1/2)  \big(1 - \Phi^v(3M+1)\big).$$
 Combining all of the above inequalities, we conclude that if $T \geq W_5$ then
$$ \mathbb{E}_{H^{RW}}^{0,2T,x,y} \Big[ e^{H(  MT^{1/2} ) -H( pT + 2MT^{1/2}- \ell(T))} \Big] \geq e^{H(  M T^{1/2} ) -H( 0)}  \cdot (1/2) \cdot \big(1 - \Phi^v(3M+1)\big).$$
Finally, by possibly making $W_5$ bigger we see that the above implies (\ref{NEe3}) as we can make $ e^{H(  M T^{1/2} ) -H( 0)} $ arbitrarily large in view of our assumption that $\lim_{x \rightarrow \infty} H(x) = \infty$.
\end{proof}

%
\section{KPZ scaling theory for the log-gamma polymer}
\label{appendixKPZscaling}

The KPZ universality conjecture posits that a wide range of stochastic interface models converge to a universal limit, the KPZ fixed point, under $3:2:1$ scaling of time, space and fluctuations. There are two non-universal, model-dependent coefficients in this conjecture. The physics work \cite{krug1992amplitude} provides a conjecture for how to compute these non-universal constants. In this appendix, we explain how, based on this theory, the free energy of the log-gamma polymer should be scaled to obtain the Airy$_2$ process \cite{prahofer2002scale}. We will follow the exposition of KPZ scaling theory from \cite{spohn2012kpz}.

For any interface model in the KPZ universality class, described by a height function $H(x,t)$, it is believed that there should exist deterministic model-dependent functions $h,\kappa$ and $d$, such that for fixed $r\in \mathbb{R}$,
\begin{equation}
x\mapsto \frac{H\left(r N+x \kappa(r) N^{2/3},N\right)+N h\left(r+ x N^{-1/3}  \kappa(r)\right) }{d(r) N^{1/3}},
\label{eq:kappad}
\end{equation}
converges as $N$ goes to infinity to a universal limit process (in $x$), depending only on the type of initial condition one starts with. For example, with narrow-wedge type initial data, this limit process should be the Airy$_2$ process. We will explain below how to compute $h,\kappa$ and $d$ for the log-gamma polymer (though the approach applies more generally). This will parallel the exposition in \cite{spohn2012kpz}. 

\medskip

From the partition function \eqref{PartitionFunct} of the log-gamma polymer with parameter $\theta$, we define the \emph{height function}
$$ H(n,N) = \log Z^{n,N},$$
where the coordinate $N$ will play the role of time in KPZ scaling theory, while the coordinate $n$ will play the role of space. We define the \emph{slope field} associated to $H$ as
$$ u(n,N) = H(n,N) - H(n-1,N).$$
An important input of KPZ scaling theory is a family of translation invariant and stationary distributions of the slope field, usually parameterized by the density of the slope field $\rho=\mathbb{E}[u(0,N)]$. A priori, the slope field $u(n,N)$ can take any value in $\mathbb{R}$, but from the definition of the log-gamma polymer model, $\mathbb E[u(n,N)]>\mathbb E[\log d_{n,N}]=-\Psi(\theta)$, where $d_{n,N}$ is the Boltzmann weight at location $(n,N)$ as in Section \ref{sec:introlog}. Hence, we will consider densities $\rho$ in the range $(-\Psi(\theta), \infty)$. In the case of the log-gamma polymer, such stationary measures were discovered in \cite[Theorem 3.3]{Sep12} and we will denote them by $\mu_{\rho}$ for any $\rho \in (-\Psi(\theta), \infty)$. For the log-gamma polymer, these measures are more conveniently parameterized by a real number $\alpha\in (0, \theta)$ such that, under $\mu_{\rho}$, the slope field is i.i.d as $n$ varies, and distributed as $-\log G(\alpha)$ where $G(\alpha)$ is a random variable following the gamma distribution with shape parameter $\alpha$. Hence, the density $\rho$ is related to the parameter $\alpha$ via
\begin{equation}
\rho=-\Psi(\alpha), \;\;\;\alpha\in (0, \theta), \;\;\rho\in (-\Psi(\theta), \infty).
\label{eq:relationrhoalpha}
\end{equation}
Partial theoretical justifications for the validity of the KPZ scaling theory are based on the uniqueness of spatially ergodic stationary measures of the slope field (see \cite{JanRas} and \cite{spohn2012kpz}), but we will not delve into that. We will simply assume that necessary hypotheses hold and compute the constants.

We introduce the \emph{instantaneous current} $j(\rho)$, which equals the average increment of the height function per unit of time under the stationary slope field $\mu_{\rho}$. In the case of the log-gamma polymer,  it was shown in \cite[Theorem 3.3]{Sep12} that if for any $N$, the time $N$ slope field is distributed as $\mu_{\rho}$, then for each $n$, $H(n,N+1)-H(n,N)$ is distributed as $-\log G(\theta-\alpha)$ where $G(\theta-\alpha)$ is a random variable following the gamma distribution with shape parameter $\theta-\alpha$. This implies that
\begin{equation} \label{eq:defj}
j(\rho) = -\Psi(\theta-\alpha),
\end{equation}
where we recall that $\alpha\in (0,\theta)$ is bijectively related to $\rho\in (-\Psi(\theta), \infty)$ via \eqref{eq:relationrhoalpha}.
From this formula, the KPZ scaling theory (see \cite{spohn2012kpz}) produces the law of large numbers via the Legendre transform of the function  $-j(\rho)$:
\begin{equation}
\lim_{N\to\infty} \frac{H(r N,N)}{N} =  \inf_{\rho\in (-\Psi(\theta), \infty)} \lbrace r\rho + j(\rho) \rbrace =  - \sup_{\alpha\in (0, \theta)} \lbrace  r \Psi(\alpha) + \Psi(\theta-\alpha) \rbrace.
\label{eq:Legendre}
\end{equation}
The supremum is attained for $\alpha$ solving
\begin{equation}
r=g_{\theta}(\alpha) \quad\textrm{where}\quad g_{\theta}(\alpha) = \Psi'(\theta-\alpha)/\Psi'(\alpha)
\label{eq:rgrelation}
\end{equation}  as in \eqref{DefLittleg}. Thus, we define the function (as in \eqref{HDefLLN})
\begin{equation}
h_{\theta}(x)  =  \lim_{N\to\infty} \frac{-H(x N,N)}{N}= x \Psi(g_{\theta}^{-1}(x)) + \Psi(\theta- g_{\theta}^{-1}(x)).
\label{eq:defhtheta}
\end{equation}
This law of large numbers was rigorously proved for the log-gamma polymer \cite[Eq. (2.7)]{Sep12}.

\medskip

We now introduce two functions $\lambda$ and $A$ through which the $\kappa$ and $d$ in \eqref{eq:kappad} are defined. We first define $\lambda(\rho) = j''(\rho)$, and assume (as will be the case for the log-gamma polymer) that $\lambda(\rho)\neq 0$. We can think of $\rho$ as a function of $r$ by combining \eqref{eq:relationrhoalpha} and \eqref{eq:rgrelation}; $\rho$ takes on the meaning of the local density of the slope field around the location $(n,N)$ with $n=rN$, $N\to \infty$. Note that this local density is precisely the minimizer in \eqref{eq:Legendre}, this was rigorously proved in \cite[Theorem 4.1]{georgiou2013ratios}.

The function $A(\rho)$ is defined as the \emph{integrated covariance of the slope field} by
\begin{equation}
A(\rho) =  \left(\sum_{j\in \mathbb Z} \mathbb E_{\mu_{\rho}} \left[  u(0;0)u(j;0)\right]  - \rho^2 \right).
\label{eq:defA}
\end{equation}
In the case of the log-gamma polymer, the stationary slope field is i.i.d. and distributed as $-\log G(\alpha)$ at each point, so that the integrated covariance of the slope field is simply the variance at one point, that is
$A(\rho)=\Psi'(\alpha)$ where $\alpha$ is related to $\rho$ from \eqref{eq:relationrhoalpha}.

As described above, we will parameterize $\lambda$ and $A$ by the slope $r$ rather than the density $\rho$ and we will denote them by $\lambda_{\theta}(r)$ and  $A_{\theta}(r)$. Hence, we set
$$ A_{\theta}(r) = \Psi'(g_\theta^{-1}(r)).$$
In order to determine  $\lambda_{\theta}(r)$, we will compute the second derivative of $j(\rho)$ from \eqref{eq:defj} in terms of $\alpha=g_{\theta}^{-1}(r)$, so that 
$$ \lambda_{\theta}(r) =  \frac{1}{\partial_{\alpha} \rho} \partial_{\alpha} \left( \frac{\partial_{\alpha} j}{\partial_{\alpha} \rho}  \right) = \frac{\Psi''(\theta-\alpha)\Psi'(\alpha)+\Psi'(\theta-\alpha)\Psi''(\alpha)}{-\Psi'(\alpha)^3} = \frac{g_{\theta}'(\alpha)}{\Psi'(\alpha)}$$ 
Using that $h_\theta'(r) = \Psi(g_{\theta}^{-1}(r))$, we find that 
$$ \lambda_{\theta}(r)= \frac{1}{h_\theta''(r)}.$$
Note that the fact that $\lambda_{\theta}$ is the inverse of $h_{\theta}''$ could also be deduced from the fact that the functions $-h_{\theta}(x)$ and $-j(\rho)$ are Legendre transforms of each other.

Following \cite{spohn2012kpz}, from the $\lambda_\theta$ and $A_\theta$ as above, we define 
\begin{equation}
d_{\theta}(r) = \left( \frac{1}{2} \lambda_{\theta}(r) A_{\theta}(r)^2 \right)^{1/3}\qquad \textrm{and}\qquad \kappa_{\theta}(r) = \left( 2 \lambda_{\theta}(r)^2 A_{\theta}(r) \right)^{1/3}
\end{equation} 
(this is equivalent to the definitions from \eqref{Defdvariant} and \eqref{DefKappa}). (We note that there is a misprint in the arxiv version of \cite{spohn2012kpz} in the definition of the spatial scaling $\left(2 \lambda^2 A t^2 \right)^{1/3}$, and this was corrected in the published version of the paper \cite{spohn2012kpz}.) The KPZ scaling theory predicts that for any fixed $r\in \mathbb R$, the function in \eqref{eq:kappad} converges (as $N\to\infty$) to the Airy$_2$ process. Using the expressions for $d_{\theta}$, $\kappa_{\theta}$ and Taylor expansion of the function $h_{\theta}$, we see that the function \eqref{eq:kappad} is the same as the function $x\mapsto 2^{1/2}\tilde{f}_N^{LG}(x)  + x^2$  up to an $O(N^{-1/3})$ deterministic error,  where $\tilde{f}_N^{LG}$ is as in Definition \ref{DefFLG2}. In particular, we see that the KPZ scaling theory predicts that $\tilde{f}_N^{LG}$ converges to $2^{-1/2}(\mathcal{A}(x) - x^2)$, where $\mathcal{A}$ is the Airy$_2$ process, as claimed in Conjecture \ref{Conj}.

\bibliographystyle{alpha} 
\bibliography{PD}

\end{document}